\documentclass[reqno,11pt]{amsart}

\usepackage{times}
\usepackage[latin1]{inputenc}
\usepackage[T1]{fontenc}
\usepackage{graphics, color, amsmath, amsfonts, amssymb, amsthm, amscd, latexsym, euscript}

\usepackage[dvips]{graphicx}
\usepackage[margin=1in]{geometry}

\usepackage{hyperref}

\def\C{{\mathbb C}}
\def\bC{{\bf T}}
\def\bQ{{\bf Q}}
\def\bK{{\bf K}}
\def\Z{{\mathbb Z}}
\def\N{{\mathcal N}}
\def\R{{\mathbb R}}

\def\e{{\varepsilon}}
\def\g{{\gamma}}
\def\G{{\Gamma}}

\def\a{{\alpha}}
\def\b{{\beta}}
\def\d{{\delta}}
\def\l{{\lambda}}
\def\La{{|\partial_x|}}

\def\z{{\zeta}}

\def\p{{\prime}}

\def\K{ {\mathcal{K} } }
\def\H{ {\mathcal{H} } }

\def\P{{\mathcal{P}}}
\def\F{{\mathcal{F}}}

\def\Im{\mbox{Im}}
\def\Re{\mbox{Re}}

\def\id{\,\mbox{id} \,}

\def\+R{+_{_{ \!\! \R}}}

\def\curl{{\,\mbox{curl}\,}}
\def\nn{\nonumber}
\def\wt{\widetilde}
\def\what{\widehat}
\def\bar{\overline}
\def\id{\mbox{id}}

\numberwithin{equation}{section}

\begin{document}

\newtheorem{theo}{Theorem}[section]
\newtheorem{pro}[theo]{Proposition}
\newtheorem{lem}[theo]{Lemma}
\newtheorem{defin}[theo]{Definition}
\newtheorem{rem}[theo]{Remark}
\newtheorem{cor}[theo]{Corollary}

\title{Global solutions for the gravity water waves system in 2d}


\author{Alexandru D. Ionescu}\address{Princeton University}\email{aionescu@math.princeton.edu}

\author{Fabio Pusateri}\address{Princeton University}\email{fabiop@math.princeton.edu}


\begin{abstract} 
We consider the gravity water waves system in the case of a one dimensional interface, for sufficiently smooth and localized initial data,
and prove global existence of small solutions. This improves the almost global existence result of Wu \cite{WuAG}.
We also prove that the asymptotic behavior of solutions as time goes to infinity is different from linear,
unlike the three dimensional case \cite{GMS2,Wu3DWW}.
In particular, we identify a suitable nonlinear logarithmic correction and show modified scattering.
The solutions we construct in this paper appear to be the first global smooth nontrivial solutions of the gravity water waves system in 2d.
\end{abstract}

\maketitle

\tableofcontents

\section{Introduction}

\subsection{The problem}
The evolution of an inviscid perfect fluid that occupies a domain $\Omega_t$ in $\R^n$ ($n \geq 2$) at time $t$,
is described by the free boundary incompressible Euler equations.
If $v$ and $p$ denote respectively the velocity and the pressure of the fluid (which is assumed here to have constant density equal to $1$),
these equations are:
\begin{equation}
\tag{E}
\label{E}
\left\{
\begin{array}{ll}
v_t + v \cdot \nabla v = - \nabla p - g e_n  &   x \in \Omega_t
\\
\nabla \cdot v = 0    &   x \in \Omega_t
\\
v (0,x) = v_0 (x)     &   x \in \Omega_0 \, , 
\end{array}
\right.
\end{equation}
where $g$ is the gravitational constant, which we will assume to be $1$ from now on. 
The free surface $S_t := \partial \Omega_t$ moves with the normal component of the velocity,
and, in absence of surface tension, the pressure vanishes on the boundary:
\begin{equation}
\tag{BC}
\label{BC}
\left\{
\begin{array}{l}
\partial_t + v \cdot \nabla  \,\, \mbox{is tangent to} \,\, \bigcup_t S_t \subset \R^{n+1}
\\
p (t,x) = 0  \,\, , \,\,\, x \in S_t  \, .
\end{array}
\right.
\end{equation}

In the case of irrotational flows, i.e.
\begin{equation}
\label{irro}
\curl v = 0 \, ,                                   
\end{equation}
one can reduce \eqref{E}-\eqref{BC} to a system on the boundary.
Although this reduction can be performed identically regardless of the number of spatial dimensions, 
we only focus on the two dimensional case which is the one we are interested in.
Assume that $\Omega_t \subset \R^2$ is the region below the graph of a function $h : \R_t \times \R_x \rightarrow \R$,
that is $\Omega_t = \{ (x,y) \in \R^2 \, : y \leq h(t,x) \}$ and $S_t = \{ (x,y) : y = h(t,x) \}$.
Let us denote by $\Phi$ the velocity potential: $\nabla \Phi(t,x,y) = v (t,x,y)$, for $(x,y) \in \Omega_t$.
If $\phi(t,x) := \Phi (t, x, h(t,x))$ is the restriction of $\Phi$ to the boundary $S_t$, 
the equations of motion reduce to the following system for the unknowns $h, \phi : \R_t \times \R_x \rightarrow \R$:
\begin{equation}
\label{WWE}
\left\{
\begin{array}{l}
\partial_t h = G(h) \phi
\\
\\
\partial_t \phi = -h - \frac{1}{2} {|\phi_x|}^2 + \frac{1}{2(1+{|h_x|}^2)} {\left( G(h)\phi + h_x \phi_x \right)}^2
\end{array}
\right.
\end{equation}
with
\begin{equation}
\label{defG0}
G(h) := \sqrt{1+{|h_x|}^2} \N(h)
\end{equation}
where $\N(h)$ is the Dirichlet-Neumann operator associated to the domain $\Omega_t$. 
We refer to  \cite[chap. 11]{SulemBook} or \cite{CraSul} for the derivation of the  water wave equations \eqref{WWE}.

Another possible description for \eqref{E}-\eqref{BC} can be given in Lagrangian coordinates
again by deriving, in the case of irrotational flows \eqref{irro}, a system of equations on the boundary $S_t$.
More precisely, following \cite{WuAG}, let $z(t,\a)$, for $\a \in \R$, be the equation of the free interface $S_t$ at time $t$
in Lagrangian coordinates, i.e. $z_t(t,\a) = v (t, z(t,\a))$.
Identifying $\R^2$ with the complex plane we use the same notation for a point $z=(x,y)$ and its complex form $z = x + i y$.
We will then denote $\bar{z} = x - i y \sim (x,-y)$.
The divergence and curl free condition on the velocity $v$ imply that $\bar{v}$ is holomorphic in $\Omega_t$.
Therefore $\bar{z}_t = \H_z \bar{z}_t$, where $\H_\g$ denotes the Hilbert transform\footnote{Integrals like the one in \eqref{HT}
are to be understood in the principal value sense, but for simplicity we will often omit the p.v. notation.} along a curve $\g$:
\begin{align}
\label{HT}
(\H_\g f )(t,\a) := \frac{1}{i \pi} \mbox{ p.v.} \int_\R \frac{ f(t,\b) }{\g(t,\a) - \g(t,\b)} \, \g_\b (t,\b) \, d\b \, .
\end{align}
The vanishing of the pressure in \eqref{BC} implies that $\nabla p$ is perpendicular to $S_t$ and therefore
$- \nabla p = i a z_\a$, with $a := - \frac{\partial P}{\partial n} \frac{1}{|z_\a|}$.
Since $z_{tt}(t,\a) = \left(v_t + v\cdot \nabla v \right) (t,z(t,\a))$, 
one see that \eqref{E}-\eqref{BC}-\eqref{irro} in two dimensions are equivalent to
\begin{align}
\label{WWL}
\left\{
\begin{array}{l}
z_{tt} + i = i a z_\a
\\
\bar{z}_t = \H_z \bar{z}_t \, .
\end{array}
\right.
\end{align}

In \cite{Wu1} Wu was able to reduce \eqref{WWL} to a quasilinear system, and to exploit the weakly hyperbolic structure 
of the new system to obtain local-in-time existence of solutions in Sobolev spaces by energy methods. 
Earlier results for small initial data in two dimensions were proven in \cite{Nalimov,Yosi} and \cite{CraigLim}.
In \cite{Wu2} Wu was also able to prove local existence for the three dimensional problem (two dimensional interface).
Following the breakthrough of \cite{Wu1,Wu2}, there has been considerable amount of work on the local well-posedness of \eqref{E}-\eqref{BC},
also including other effects on the wave motion, such as surface tension on the interface or a finite bottom.
We refer the reader to \cite{ABZ1,ABZ2,AM,BG,CHS,CL,CoutShko2,Lindblad,Lannes,ShZ3} for some of the works on the local well-posedness of \eqref{E}.
Recently, blow-up solutions in the form of turning waves \cite{CCFGL} and ``splash'' solutions \cite{CCFGG,CoutShkoSplash} have been constructed.

The question of existence of global-in-time solutions for small, smooth, 
and suitably localized data, has also received attention in recent years.
In the case of one dimensional interfaces, the only work investigating the 
long time behavior of small gravity waves is that of Wu \cite{WuAG},
who was able to show almost global existence of solutions for \eqref{WWL}.
To do this, the author proposed some new unknowns, which we denote here by $F$, 
and a fully nonlinear change of coordinates, reducing \eqref{WWL} to a system of the form
\begin{align}
\label{cubicintro}
\partial_t^2 F + i\partial_\a F = G 
\end{align}
where $G$ are quasilinear nonlinearities of cubic and higher order with suitable structure. 
Thanks to the cubic nature of this new system Wu was then able to perform (almost optimal) energy estimates 
and obtain existence of solutions up to times of order $e^{c/\e}$, where $\e$ is the size of the initial data.

On the other hand, in the case of two dimensional interfaces, Germain, Masmoudi and Shatah \cite{GMS2} and Wu \cite{Wu3DWW} obtained global solutions. 
The result of \cite{GMS2} relied on the energy method of \cite{ShZ1,ShZ3} and on the space-time resonance method introduced in \cite{GMS1}.
In \cite{Wu3DWW} the author used instead a three dimensional version of the arguments of \cite{WuAG}
to derive a set of equations similar to \eqref{cubicintro}, perform weighted energy estimates on them, and obtain decay
via $L^2 - L^\infty$ type estimates.
Recently, Germain, Masmoudi and Shatah \cite{GMSC} obtained global solutions in three dimensions for capillary waves,
i.e. with surface tension on the interface and no gravitational force.

Here we are interested in the gravity water waves system \eqref{E}-\eqref{BC}-\eqref{irro}
in the case of one dimensional interfaces which are a perturbation of the flat one,
and initial velocity potentials which are suitably small in an appropriate norm.
We aim to prove the existence of global-in-time and pointwise decaying solutions,
and determine their asymptotic behavior as $t \rightarrow \infty$.

\subsection{The main theorem} 
We define first our main spaces of functions. Fix\footnote{We assume a large number of derivatives mostly to simplify the exposition. 
However one can likely reduce this number to, say, $N_0$ between $10$ and $100$ by a slightly more careful analysis. Similarly, 
the parameter $\beta$ in \eqref{defZ}, which is related to the size of the small frequencies, can be allowed to take other values in the 
interval $(0,1/2)$.} 
$N_0 = 10^4$ and define $N_1 := N_0/2 + 4$.
Let $S=\frac{1}{2} t \partial_t + \a \partial_\a$ be the scaling vector field. Given a time interval $I$ and a function 
$f:I\times\mathbb{R}\to\mathbb{C}$ we define the norm
\begin{align}
\label{defX_k}
{\| f (t)\|}_{X_k} :=  {\| f (t)\|}_{H^k} + {\| S f(t) \|}_{H^{k/2}}.
\end{align}
$X_{N_0}$ will be the weighted energy-space for the solution,
expressed in some appropriate modified Lagrangian coordinates, as well as in Eulerian coordinates.

Given the height function $h$ and the velocity potential $\phi$ in Eulerian coordinates, we define the $Z^\p$ norm
\begin{align}
\label{defZp}
{\| (h(t),\phi(t)) \|}_{Z^\p} :=  {\| h(t) \|}_{W^{N_1+4,\infty}} + {\| \Lambda \phi(t) \|}_{W^{N_1+4,\infty}} \, ,
  \quad \mbox{with} \quad \Lambda := {|\partial_x|}^{\frac{1}{2}} \, .
\end{align}
This is the decaying norm that we will estimate.
Decay of this norm at the rate of $t^{-1/2}$ will give us a small global solution to the water wave problem.

Finally, we define the space $Z$ by the norm
\begin{align}
\label{defZ}
{\| f(t) \|}_{Z} :=  \sup_{\xi\in\R} \left| \left( {|\xi|}^\b + |\xi|^{N_1 + 15} \right) \what{f}(\xi,t)  \right| 
\end{align}
where $\beta = 1/100$, and
\begin{align*}
\what{f}(\xi,t) := \int_\R e^{-i x\xi} f(t,x) \, dx
\end{align*}
is the partial Fourier transform in the spatial variable.
This space plays a key role in obtaining decay of the $Z^\p$ norm of solutions, see the linear estimate \eqref{disperseintro}.
The $Z$ norm also plays an important role in proving modified scattering of solutions in Eulerian coordinates.

The paper is concerned with the proof of the following Main Theorem:

\begin{theo}\label{maintheo}
Let $h_0(x) = h(0,x)$ be the initial height of the surface $S_0$,
and let $\phi_0 (x) = \phi(0,x)$ be the restriction to $S_0$ of the initial velocity potential.
Assume that at the initial time one has
\begin{subequations}
\label{initdata}
\begin{align} 
\label{initdataa}
& {\| (h_0, \Lambda \phi_0) \|}_{H^{N_0+2}} + {\left\| x \partial_x \left( h_0, \Lambda \phi_0 \right)  \right\|}_{H^{N_0/2+1}}
  + {\| h_0 + i \Lambda \phi_0 \|}_{Z}  \leq \e_0 \, ,
\end{align}
where $Z$ is defined in \eqref{defZ}.
Moreover, for $x \in \Omega_0$ let $v_0(x) = v(0,x)$, where $v$ is the irrotational and divergence free velocity field of the fluid,
and assume that
\begin{align}
\label{initdatab}
{\| |x|\nabla v_0 \|}_{H^{N_0/2} (\Omega_0)} 
\leq \e_0 \, .
\end{align}
\end{subequations}

\setlength{\leftmargini}{2.0em}
\begin{itemize}
 
\item[(i)] (Global existence) Then there exists $\bar{\e}_0$  small enough, such that for any $\e_0 \leq \bar{\e}_0$,
the initial value problem associated to \eqref{WWE} admits a unique global solution with
\begin{align*}
\sup_{t}  \left[ (1+t)^{-p_0} {\| (h(t), \phi_x (t)) \|}_{X_{N_0}} + 
  {\| h(t) + i\Lambda \phi(t)  \|}_{H^{N_1 + 10}}+\sqrt{1+t} {\| (h(t), \phi(t)) \|}_{Z^\p} \right] \lesssim \e_0 \, ,
\end{align*}
where $p_0 = 10^{-4}$.

\item[(ii)] (Modified scattering) Let $u(t) := h(t) + i \Lambda \phi(t)$, with $\Lambda := {|\partial_x|}^{1/2}$.
Define
\begin{equation*}
G(\xi,t):= \frac{{|\xi|}^{4}}{\pi} \int_0^t|\widehat{u}(\xi,s)|^2 \frac{ds}{s+1} \, , \qquad t\in[0,T] \, .
\end{equation*}
Then there is $p_1 > 0$ such that
\begin{equation}
\label{Zbound}
(1+t_1)^{p_1}\Big\|(1+|\xi|)^{N_1} \big[e^{iG(\xi,t_2)}e^{it_2\Lambda(\xi)}\widehat{u}(\xi,t_2)
  -e^{iG(\xi,t_1)}e^{it_1\Lambda(\xi)} \widehat{u}(\xi,t_1)\big]\Big\|_{L^2_\xi}
  \lesssim \varepsilon_0 \, ,
\end{equation}
for any $t_1\leq t_2\in [0,T]$. In particular, there is $w_\infty\in L^2((1+|\xi|)^{2N_1}d\xi)$ with the property that
\begin{equation}\label{scatt}
\sup_{t\in[0,\infty)}(1+t)^{p_1}{ \left\|   
  (1+|\xi|)^{N_1} \big( e^{iG(\xi,t)} e^{it\Lambda(\xi)} \what{u}(\xi,t)  -  w_\infty(\xi) \big) \right\|}_{L^2_\xi}\lesssim \varepsilon_0 \, .
\end{equation}
\end{itemize}
\end{theo}

\begin{rem} The first norm in \eqref{initdataa} ensures that our initial data is small and smooth in Sobolev spaces of high regularity.
Notice that we are assuming that the vertical variation from equilibrium of the interface, given by the graph of $h$, 
as well as half derivative of the velocity potential are small.
This is consistent with the conserved energy (and Hamiltonian) 
\begin{equation*}
E_0(h,\phi) := \frac{1}{2} \int \phi G(h) \phi \, dx + \frac{1}{2} \int h^2 \, dx\approx \|h+i\Lambda\phi\|_{L^2}^2,
\end{equation*}
for solutions of \eqref{WWE}.
The second norm in \eqref{initdataa}, properly evolved in time, gives some control of certain weighted norms of the solution.

The key new ingredient in \eqref{initdataa} is the smallness of $h_0$ and $\Lambda \phi_0$ in the $Z$-norm defined in \eqref{defZ}. 
The $Z$-norm is the key new component of our global argument: it is the only strong norm of the solution that we can 
control uniformly in time, while the other energy-type norms are allowed to increase slowly in time. Furthermore, 
the $Z$-norm allows us to describe properly the modified scattering of the solution.

Finally, \eqref{initdatab} is similar to a condition imposed by Wu in \cite{WuAG}.
We use it as in the cited paper to guarantee that the energy functional 
on which energy estimates are based is small at time $t=0$.
\end{rem}

\begin{rem} 
The solutions can also be defined on the time interval $(-\infty,0]$, since the equations are time-reversible. 
The global solutions we construct in Theorem \ref{maintheo} appear to be the first smooth nontrivial global solutions 
of the gravity water waves system \eqref{E}--\eqref{BC} in 2 dimensions. 
\end{rem}

\begin{rem}
A more precise statement of modified scattering can be found in Lemma \ref{bigbound2}, 
in terms of certain modified Eulerian variables. Also, more precise bounds on the solution, 
both in terms of the Eulerian variables $(h,\phi)$ and the Lagrangian variable $z$ can be found in section \ref{secproofmain},
see \eqref{apriori0i}--\eqref{aprioriki}.
\end{rem}

\subsection{Main ideas in the proof}\label{secintro2}

If one is interested in the long-time existence of small smooth solutions to quasilinear dispersive and wave equations, 
such as \eqref{WWE} or \eqref{WWL}, 
there are two main aspects one needs to consider: controlling high frequencies and proving dispersion.
The first aspect is generally connected to the construction of high order energies which control the Sobolev norm of a solution.
The second aspect is related to $L^p$ decay estimates, and to estimates of weighted norms.
When dealing with the water waves system both of these aspects are extremely delicate. 

\subsubsection{Supercriticality, energy estimates and normal forms}
The general strategy for obtaining a global small solution 
usually starts with local-in-time energy estimates.
The aim of the energy method is to construct an energy functional $E(t)$ such that
\begin{equation}
\label{Ein}
 E(t) \sim {\| u(t) \|}_{H^N}^2  \qquad \mbox{and} \qquad \frac{d}{dt} E(t) \lesssim E(t)^{3/2} \, .
\end{equation}
Here, the power $3/2$ is dictated by the quadratic nature of the nonlinearities in \eqref{WWE} or \eqref{WWL}.
The estimates \eqref{Ein} are often the key ingredient in obtaining local solutions,
and for initial data of size $\e$ they give existence for times of order $1/\e$.
We remark here that the construction of an energy satisfying \eqref{Ein}  for the water waves system is particularly challenging.
Nevertheless it has been done in several works, such as the already cited \cite{ABZ1,BG,CL,Lindblad,Lannes,ShZ3,Wu1,Wu2}, 
thanks to considerable insight into the structure of the equations.

To extend a local solution for longer times one needs to engage the dispersive effects of the equation.
One possibility is to try to upgrade \eqref{Ein} to
\begin{equation}
\label{Ein2}
{\| u(t) \|}^2_{H^N} +  {\| u (t) \|}^2_W  \lesssim E(t) \qquad \mbox{and} \qquad  \frac{d}{dt} E(t) \lesssim \frac{\e}{t^a}  \, E(t) \, ,
\end{equation}
provided the solution decays like $t^{-a}$ in some $L^\infty$-based space. 
The W-norm in \eqref{Ein2} is supposed to encode some information about the localization of the flow.
In the best case scenario a bound of the form ${\| u (t) \|}_W \lesssim 1$ implies the desired $t^{-a}$ decay.
Thus, if one can prove \eqref{Ein2} with $a > 1$ small solutions will exist globally and scatter to a linear solution.
If $a = 1$ solutions will automatically exist almost globally and further analysis\footnote{
Examples of such analysis are the classical vector fields method of Klainerman \cite{K0,K1}, 
or the more recent papers \cite{GMS1,GMS2,GNT1,zakharov,IP1,IP2,nullcondition} on global regularity results for certain physical systems.}
is needed in this critical case to show global existence and determine the asymptotic behavior.
If $a < 1$ the problem of global existence and scattering is much more difficult.
This case is referred to as scattering supercritical and it is the case of the $2d$ water waves problem, 
since solutions of the linear equation $i \partial_t u - \Lambda u = 0$, $\Lambda = {|\partial_x|}^{1/2}$, decay at the rate $t^{-1/2}$.

In the work of Wu \cite{WuAG} on the almost global existence in the two dimensional problem, 
the author relied on a nonlinear version of a normal form transformation.
Starting from the Lagrangian formulation \eqref{WWL}, Wu proposed some new quantities,
and a diffeomorphism depending fully nonlinearly on the solution, 
such that the set of equations obtained in the new coordinates
admit a certain type of energy estimates, consistent with cubic nonlinearities. More precisely, let us denote, schematically, by $F$ the new transformed unknowns, after the nonlinear change of variables. The point of Wu's remarkable construction is that the variables $F$ satisfy nonlinear equations like \eqref{cubicintro} with {\it{cubic}}-type nonlinearities. As a consequence, Wu defines a suitable energy functional $E(t)$ and proves the following type of energy estimates:
\begin{align}
 \label{EinWu}
 {\| F(t) \|}^2_{H^N} + {\| F(t) \|}^2_{W} \lesssim E (t) \qquad 
    \mbox{and} \qquad \frac{d}{dt} E (t) \lesssim  {\| F(t) \|}^2_{W^{\frac{N}{2} ,\infty}} \, E(t)  \log t  +  \frac{1}{t}  E^2 (t)\, .
\end{align}
These estimates can then be combined with $L^2-L^\infty$ estimates, which show that $F(t)$ decays pointwise like $t^{-1/2}$. 
Therefore $E(t) \lesssim \e^2$ as long as $\log t \lesssim \e^{-1}$, which concludes the proof of almost global existence.

\subsubsection{Dispersion and asymptotic behavior}

To pass to global existence, our first concern is to improve \eqref{EinWu} by eliminating logarithmic losses.
We will use the same approach of \cite{WuAG}, and the same equations derived there, to show an estimate of the form
\begin{align}
 \label{EinWuimp}
 {\| F(t) \|}^2_{H^N} + {\| F(t) \|}^2_{W} \lesssim E (t) \qquad 
    \mbox{and} \qquad \frac{d}{dt} E (t) \lesssim  {\| F (t) \|}^2_{Z_\infty} \, E(t)  \, ,
\end{align}
for some $L^\infty$ based space $Z_\infty$ which is stronger than $W^{\frac{N}{2},\infty}$.
Such an estimate is achieved by carefully analyzing the singular integrals (Calder\'{o}n commutators) appearing in the cubic nonlinearities,
and exploiting special structures present in some of them. 

Another important point is that in order to justify the existence of Wu's diffeomorphism $k$ 
and of the new unknowns $F$ for all times $t$, we need appropriate a priori control on Wu's change of coordinates.
This is obtained by taking advantage of a certain null structure present in the transformation.

Thanks to \eqref{EinWuimp} one can guarantee $E(t) \lesssim_{p_0} \e^2 t^{2p_0}$, for any fixed $p_0>0$, and for all $t \in [0,\infty)$,
provided  ${\| F(t) \|}_{Z_\infty} \lesssim \e t^{-1/2}$.
However, since $E(t)$ is forced to grow in time, although just slightly, 
one cannot obtain the desired sharp $Z_\infty$ decay through energy estimates like those in \cite{WuAG}.

Our main idea in this paper  is to use also the Eulerian formulation of the equations \eqref{WWE} 
for the purpose of proving decay and modified scattering. More precisely, we achieve this by bootstrapping at the same time several bounds,
which control the Eulerian variables, the Lagrangian variables, and Wu's diffeomorphism $k$ (see \eqref{apriori0}--\eqref{apriorik}). 
This consists of several steps:

\setlength{\leftmargini}{2.0em}
\begin{itemize}

\item We show first that the Eulerian variables $h$ and $\phi_x$ in \eqref{WWE} are controlled in the energy norms $H^{N_0}$
and in $S^{-1}H^{N_0/2}$ by the energy $E(t)$ of the modified Lagrangian variables. Here $S=(1/2)t\partial_t+x\partial_x$ is the scaling vector field. In other words we transfer the energy and weighted energy bounds (which are expected to increase slowly in time as a result of the bootstrap argument) to the Eulerian variables $h,\phi_x$.

\item The variables $h$ and $\phi$ satisfy the equations \eqref{WWE}, which can be written schematically in the form
\begin{equation*}
\begin{split}
&\partial_t h  = |\partial_x| \phi  - \partial_x( h  \partial_x \phi) - |\partial_x| (h |\partial_x| \phi)+\mathrm{Cubic}(h,\phi_x);\\
&\partial_t \phi   = - h - (1/2){|\phi_x|}^2 + (1/2){||\partial_x|\phi|}^2 + \mathrm{Cubic}(h,\phi_x),
\end{split}
\end{equation*}
where the remainders are cubic expression of $h,\phi_x$.  We use a normal form transformation $H=h+A(h,h)$, $\Psi=\phi+B(h,\phi)$, 
for suitable bilinear operators $A$ and $B$ to eliminate the quadratic nonlinearities and reduce this to an evolution equation of a 
complex variable with a cubic nonlinearity. More precisely, letting $V=H+i\Lambda\Psi$, we show that $V$ satisfies an equation of the form
\begin{equation}
\label{eqvintro}
\partial_t V+i\Lambda V = C(V,\bar{V}),\qquad\Lambda=|\partial_x|^{1/2},
\end{equation}
where $C$ is a nonlocal cubic quasilinear nonlinearity depending on all possible combination of $V$ and $\bar{V}$, and some of their derivatives. 
Moreover, we show that the variable $V$ satisfies similar energy and weighted energy estimates as the functions $h$ and $\phi_x$, 
i.e. with a slow increase in time.

\item Finally, to analyze\footnote{This step was considered, as a model case, in \cite{FNLS}. More precisely, we considered 
the semilinear Cauchy problem
\begin{equation*}
\partial_t u +i\Lambda u = ic_0{|u|}^2 u + c_1 u^3 + c_2 u \bar{u}^2  + c_3 \bar{u}^3, 
\end{equation*}
$c_0\in\mathbb{R}$, and $c_1,c_2,c_3 \in \C$. This is a simplified semilinear version of the quasilinear equation \eqref{eqvintro}, 
and energy and weighted energy estimates are not an issue. However, to prove global existence and pointwise decay, one still needs 
to identify an appropriate logarithmic correction, and prove modified scattering using a norm similar to the $Z$ norm. 
A similar argument was used in \cite{KP} in the case of scattering critical semilinear Schr\"{o}dinger equations (see also \cite{ozawa,HN}). 
For more works on modified scattering we refer the reader to the papers of Delort \cite{DelortKG1d},
Hayashi and Naumkin \cite{HN,HNKdV}, Lindblad and Soffer \cite{LSKG}, Deift and Zhou \cite{DZmKdV,DZNLS},
and references therein.
}
the equation \eqref{eqvintro} we use the key $Z$ norm defined in \eqref{defZ}. Unlike all the other energy and weighted energy norms, 
which are allowed to increase slowly at $t\to \infty$, the $Z$ norm of the solution is not based on $L^2$-type of spaces and is the 
only strong norm we are able to control uniformly in time. Letting $f(t)=e^{it\Lambda}V(t)$ and using the Fourier transform method, 
we identify an appropriate nonlinear correction $L=L(\widehat{f})$ and show that the function 
$t\to \mathcal{F}^{-1}[e^{iL(\xi,t)}\widehat{f}(\xi,t)]$ converges, at a polynomial rate, as $t\to\infty$. 
This suffices to prove global existence and modified scattering.

\end{itemize}
The crucial $t^{-1/2}$ pointwise decay of the solution, which is needed to close the energy estimates, 
is then a consequence of the linear bound in Lemma \ref{dispersive},
\begin{equation}
\label{disperseintro}
{\|e^{i t\Lambda}h\|}_{L^\infty} \lesssim (1+|t|)^{-1/2} {\|\,|\xi|^{3/4}\widehat{h}(\xi)\|}_{L^\infty_\xi}
  +(1+|t|)^{-5/8}\big[ {\|x\cdot\partial_x h\|}_{L^2} + {\|h\|}_{H^2} \big].
\end{equation}
Notice that this pointwise bound requires sharp control of the $Z$ norm, 
but can tolerate slow increase of the energy norms in time. This is consistent with the information we have on our solutions.

\subsubsection{Plan of the paper}
In section \ref{secproof} we describe in detail the strategy of our proof through Propositions  \ref{prolocal2}-\ref{prodecay}.
In section \ref{secproofmain} we prove the main Theorem \ref{maintheo} assuming these propositions.
In section \ref{secEuler} we state Propositions \ref{proE1}-\ref{proE4} and show how they imply the decay 
and the control of lower Sobolev norms stated in Proposition \ref{prodecay}.
Propositions \ref{proE1}-\ref{proE4} are then proved in sections \ref{secproE1} and \ref{secproE4}.
In section \ref{secL} we describe the change of coordinates used by Wu in \cite{WuAG}, the cubic equations obtained there,
and the associated energy functional.
In section \ref{secproenergy} we prove the energy estimates contained in 
Proposition \ref{proenergy} via Propositions \ref{proenergy1}-\ref{proenergy3}.
We then prove Proposition \ref{prok}. i.e. that the change of coordinates used is a diffeomorphism,
on any time interval where one has a small solution satisfying certain a priori bounds.
The transition of energy norms to Eulerian coordinates is done in section \ref{secproLE}, where we prove Proposition \ref{proLE}.
In appendix \ref{appWu} we first give some variants of the estimates used in \cite{WuAG} that are compatible with our energy estimates.
Section \ref{secop} contains some estimates for singular integral operators of ``Calder\'{o}n commutators'' type
that are used in the course of the energy estimates.
In appendix \ref{secsym} we calculate the resonant contribution of the cubic nonlinearities in Eulerian coordinates, 
after the application of the normal form. 
Appendix \ref{appR} contains estimates for the quartic and higher order remainders in the equation \eqref{WWE}.

\subsection*{Acknowledgements}
The authors are grateful to the anonymous referees for carefully reading the manuscript and for their many useful comments.
After submission of this manuscript a different proof of a similar result was given by Alazard and Delort \cite{AD13}.
The first author was partially supported by a Packard Fellowship and NSF Grant DMS 1265818.
The second author was partially supported by a Simons Postdoctoral Fellowship and NSF Grant DMS 1265875.

\section{Strategy of the proof}\label{secproof}
The proof of Theorem \ref{maintheo} relies on a set of different Propositions.
We state these key Propositions below and make some comments.

\subsection{Local Existence} Our strategy for controlling high Sobolev norms of solutions relies on the energy method of Wu \cite{WuAG},
which is developed starting from the Lagrangian formulation of the problem.
Therefore we begin by describing the local existence theory in Lagrangian coordinates.
Assume that at the initial time the interface $S_0$ is given by the graph of a function $h_0:\R \rightarrow \R$, 
with $h_0(\a) \rightarrow 0$ as $|\a| \rightarrow \infty$.
Let $z_0(\a) = \a + i h_0(\a)$ be a parametrization of $S_0 \subset \C$. 
Assume that for some $\mu > 0$
\begin{align}
|z_0(\a) - z_0(\b)| \geq \mu |\a-\b| \qquad \forall \, \a, \, \b \in \R \, .
\end{align}
Let $z = z(t,\a)$ be the equation of the free surface $S_t$ at time $t$, in the Lagrangian coordinate $\a$, with $z(0,\a) = z_0(\a)$.
The following local existence result holds:

\begin{pro}[{Local existence in Lagrangian coordinates \cite[Theorem 5.1]{WuAG}, \cite{Wu1}}]\label{prolocal2}

Let $N \geq 4$ be an integer. Assume that
\begin{align}
\label{localinitdataz}
\sum_{ 0 \leq j \leq N} {\left\| \partial_\a^j 
  \left( z_\a(0) - 1, z_t(0), \partial_\a z_t(0) \right) \right\|}_{H^{1/2}} +
  {\left\| \left( \partial_\a^j  (z_{tt}(0), \partial_\a z_{tt}(0) \right) \right\|}_{L^2} \leq \e_0 \, .
\end{align}
Then there exists a time $T>0$, depending only on the norm of the initial data,
such that the initial value problem for \eqref{WWL}\footnote{
Recall that the Taylor coefficient $a$ can be expressed in terms of $z$ and $z_t$, see formula (5.13) in \cite{WuAG}.}
has unique solution $z = z(t,\a)$ for $t \in [0,T]$, satisfying for all $j \leq N$
\begin{align*}
& \partial_\a^j \left( z_\a - 1, z_t, \partial_\a z_t \right) \in C \left( [0,T],H^{1/2}(\R) \right)
\\
& \partial_\a^j \left( z_{tt}, \partial_\a z_{tt}  \right) \in C \left( [0,T],L^2(\R) \right)  \, ,
\end{align*}
and $|z(t,\a) - z(t,\b)| \geq \nu |\a-\b|$, for all $\a,\b \in \R$, $t \in [0,T]$, and some $\nu > 0$.

Moreover, one has the following continuation criterion:
if $T^\ast$ is the supremum of all such times $T$, then either $T^\ast = \infty$ or 
\begin{align}
\label{contcrit}
\sup_{t\in[0,T^\ast)} \Big( \sum_{ 0 \leq j \leq \left[\frac{N+1}{2}\right] + 2}
  {\| \partial^j_\a z_{tt}(t)  \|}_{L^2} + {\| \partial^j_\a z_t(t)  \|}_{H^{1/2}} 
  +  \sup_{\a \neq \b} \Big| \frac{\a-\b}{z(t,\a) - z(t,\b)} \Big| \Big) = \infty\, .
\end{align}
\end{pro}

Given $N = N_0$, and a local solution on $[0,T]$, with initial data as in the main Theorem, we assume that
\begin{align}
\label{apriori0L}
\sup_{[0,T]} \Big( \sum_{ 0 \leq j \leq \frac{N_0}{2} + 2}
  {\| \partial^j_\a z_{tt} (t) \|}_{L^2} + {\| \partial^j_\a z_t (t) \|}_{H^{1/2}} 
  +  {\| z_\a(t)  - 1 \|}_{L^\infty} \Big) \leq \e_0^{3/4} \, .
\end{align}
To obtain a global solution it suffices to prove that the quantity in the left-hand side of the inequality above is bounded by $C \e_0$.

\subsection{The main a priori assumptions}
Under the a priori assumption \eqref{apriori0L}, we have that $\Re z$ is a diffeomorphism on $[0,T]$.
We can then relate the Eulerian variables $h$ and $\phi$ to the Lagrangian map $\a \rightarrow z(t,\a)$, for $t\in[0,T]$,
via the following identities
\begin{align}
\label{h00}
 h (t, \Re z(t,\a) ) = \Im z(t,\a),\qquad  \phi (t, \Re z(t,\a) ) = \psi (t,\a) \,.
\end{align}
Here $\psi(t,\alpha)$ is the trace of the velocity potential in Lagrangian coordinates, which can be obtained from the map $z$, 
for example, using the Bernoulli equation $\psi_t=-\Im z+(1/2)|z_t|^2$, $\psi(0,\alpha)=\phi_0(\alpha)$.

In addition, we also need Wu's change of coordinates $k$ as in \cite{WuAG},
to obtain cubic equations amenable to energy estimates.
The explicit form for $k$ is given in \eqref{defk}, and is the same as the one used by Totz and Wu in \cite{WuNLS}, see formula (2.3) there.
As long as this transformation $k$ is a well-defined diffeomorphism, 
one can associate to the Lagrangian map $z$ a modified Lagrangian map $\z = z \circ k^{-1}$,
and the following vector associated to the new coordinates:
\begin{equation}
\label{defwtL}
 \wt{L}(t,\a) :=  \left( \z_\a(t,\a) - 1, u(t,\a), w(t,\a), \Im \z(t,\a) \right) \, ,
\end{equation}
with 
\begin{align}
\z := z \circ k^{-1}
\quad , \quad
u := z_t \circ k^{-1}
\quad , \quad
w := z_{tt} \circ k^{-1} \, .
\label{w_0}
\end{align}

Our main bootstrap assumptions on the solution are:
\begin{align}
\label{apriori0}
\sup_{t\in [0,T]}  \left[ (1+t)^{-p_0} {\| (h(t), \phi_x (t)) \|}_{X_{N_0}} + 
  {\| h(t) + i\Lambda \phi(t)  \|}_{H^{N_1 + 10}}+\sqrt{1+t} {\| (h(t), \phi(t)) \|}_{Z^\p} \right] \leq \e_1 \, ,
\end{align}
\begin{align}
\label{aprioriL1}
\sup_{t \in [0,T]}  \left[(1+t)^{-p_0} {\| \wt{L}(t) \|}_{X_{N_0}}+{\| \wt{L}(t) \|}_{H^{N_1+5}} + \sqrt{1+t}  {\| \wt{L}(t) \|}_{W^{N_1,\infty}} \right] \leq \e_1 \,,
\end{align}
and
\begin{align}
\label{apriorik}
\sup_{t\in [0,T]}  {\| k_\a (t) - 1 \|}_{W^{N_0/2+3,\infty}}  \leq \e_1 \, ,
\end{align}
where $X_{N_0}$ is defined by \eqref{defX_k} and $Z^\p$ is defined in \eqref{defZp}, $\e_1\leq \e_0^{3/4}$, and $p_0=10^{-4}$. 
In other words we assume a priori control on the Eulerian variables $(h,\phi)$, 
on the modified Lagrangian variable $\zeta$, and on the diffeomorphism $k$. 
To close the bootstrap argument we need to prove improved control on these quantities; 
this is the content of Propositions \ref{prok}--\ref{prodecay} below.

\subsection{The main propositions} Our first proposition, which is proved in section \ref{secprok}, provides improved control on the diffeomorphism $k$.

\begin{pro}[Control on the diffeomorphism $k$]\label{prok}
Assume that \eqref{apriori0}--\eqref{apriorik} hold and, in addition,
\begin{align}
\label{kinit}
{\| k_\a(0) - 1 \|}_{W^{N_0/2+3,\infty}}  \lesssim \e_0 \,.
\end{align}
Then
\begin{align}
\label{apriorikconc}
\sup_{t\in [0,T]}  {\| k_\a (t) - 1 \|}_{W^{N_0/2+3,\infty}}  \lesssim \e_0 +\e_1^2.
\end{align}
\end{pro}
The proof of the above proposition relies crucially on the exploitation of a special null structure present in the transformation $k$.

Our second proposition concerns improved control of the highest energy norm in modified Lagrangian variables. More precisely:

\begin{pro}[{Energy estimates for the modified Lagrangian variables}]\label{proenergy}
Assume that \eqref{apriori0}--\eqref{apriorik} hold. Then
\begin{align}
\label{concproenergy}
\sup_{t\in [0,T]} (1+t)^{-p_0}{\| \wt{L}(t) \|}_{X_{N_0}} \lesssim \e_0 + \e_1^2.
\end{align}
\end{pro}

This is proved in section \ref{secproenergy}. We follow, to a large extent, the proof of Wu in \cite{WuAG},  using cubic equations for the  ``good unknowns'' related to $\wt{L}$, and performing energy estimates. However, some arguments need to be adjusted in order to avoid the logarithmic losses in the energy bound (compare with \eqref{EinWu}).

The next step consists in translating the bounds given by the energy estimates in terms of the modified Lagrangian coordinates,
to bounds on the norms of the Eulerian variables $h$ and $\partial_x \phi$.

\begin{pro}[{Transition to Eulerian coordinates}]\label{proLE}
Assume that \eqref{apriori0}--\eqref{apriorik} hold. Then, for $t\in[0,T]$,
\begin{align}
\label{conc1proLE}
& {\| (h(t) , \partial_x \phi(t)) \|}_{X_{N_0}}  \lesssim {\| \wt{L}(t) \|}_{X_{N_0}} ,
\end{align}
and
\begin{subequations}
\begin{align}
\label{conc2proLE}
& {\| \wt{L}(t) \|}_{H^{N_1+5}}  \lesssim {\| (h(t) , \partial_x \phi(t)) \|}_{H^{N_1+7}}  + \e_1^2\,,
\\
\label{conc3proLE}
& {\| \wt{L}(t) \|}_{W^{N_1,\infty}} \lesssim {\| (h(t), \phi(t)) \|}_{Z^\p}  \, .
\end{align}
\end{subequations}
\end{pro}

The proof of Proposition \ref{proLE} is given in section \ref{secproLE}.

Finally, we use the Eulerian formulation of the equations to show the decay of the $Z^\p$-norm and bound lower Sobolev norms:

\begin{pro}[{Control of dispersive and lower Sobolev norms}]\label{prodecay}
Assume that \eqref{apriori0}--\eqref{apriorik} hold. Then
\begin{align}
\label{conc1prodecay}
\sup_{[0,T]} \sqrt{1+t} {\| (h(t),\phi(t)) \|}_{Z^\p}  \lesssim \e_0 + \e_1^2
\end{align}
and
\begin{align}
\label{conc3prodecay}
{\| h(t) + i\Lambda \phi(t)  \|}_{H^{N_1 + 10}} \lesssim \e_0 +\e_1^2.
\end{align}
\end{pro}

The detailed strategy for the proof of Proposition \ref{prodecay} can be found in section \ref{secEuler},
and relies on Propositions \ref{proE1}-\ref{proE4}, which are proven in sections \ref{secproE1} and \ref{secproE4}.

Proposition \ref{prodecay} is the main new ingredient in our proof of global regularity. 
We start from the equations \eqref{WWE}, written schematically in \eqref{Eeq} and perform first a normal form transformation 
to eliminate the quadratic terms in the nonlinearity. We then analyze the resulting cubic equation, given in Lemma \ref{proE1+}.
A crucial step in the analysis consists in applying a phase correction to the transformed solution, and estimating it in the auxiliary 
$Z$-norm defined in \eqref{defZ}.

\section{Proof of the main Theorem}\label{secproofmain}

Using Propositions \ref{prok}--\ref{prodecay} we can now complete the proof of the main theorem. 
Set $\e_1=\e_0^{3/4}$ and assume that \eqref{apriori0}--\eqref{apriorik} hold. 
We apply Propositions \ref{prok}--\ref{prodecay} to conclude that
\begin{align}
\label{apriori0i}
\sup_{t\in [0,T]}  \left[ (1+t)^{-p_0} {\| (h(t), \phi_x (t)) \|}_{X_{N_0}} + {\| h(t) + i\Lambda \phi(t)  \|}_{H^{N_1 + 10}}+\sqrt{1+t} {\| (h(t), \phi(t)) \|}_{Z^\p} \right] \lesssim \e_0 \, ,
\end{align}
\begin{align}
\label{aprioriL1i}
\sup_{t \in [0,T]}  \left[(1+t)^{-p_0} {\| \wt{L}(t) \|}_{X_{N_0}}+{\| \wt{L}(t) \|}_{H^{N_1+5}} + \sqrt{1+t}  {\| \wt{L}(t) \|}_{W^{N_1,\infty}} \right] \lesssim \e_0\,,
\end{align}
and
\begin{align}
\label{aprioriki}
\sup_{t\in [0,T]}  {\| k_\a (t) - 1 \|}_{W^{N_0/2+3,\infty}}  \lesssim \e_0.
\end{align}
This provides the desired improvement of the a priori assumptions \eqref{apriori0}--\eqref{apriorik}. By continuity, the improved bounds \eqref{apriori0i}--\eqref{aprioriki} hold on the time interval $[0,T]$.  

We consider now the a priori assumption \ref{apriori0L}. To improve it, we need to show that 
\begin{align}
\label{apriori0Li}
\sup_{[0,T]} \Big( \sum_{ 0 \leq j \leq \frac{N_0}{2} + 2}
  {\| \partial^j_\a z_{tt} (t) \|}_{L^2} + {\| \partial^j_\a z_t (t) \|}_{H^{1/2}} 
  +  {\| z_\a(t)  - 1 \|}_{L^\infty} \Big) \lesssim\e_0\, .
\end{align}
This follows using the chain rule from the identities $z=\zeta\circ k$, $z_t=u\circ k$, and $z_{tt}=w\circ k$, see \eqref{w_0}, and the bounds \eqref{aprioriL1i} and \eqref{aprioriki}. As a consequence of Proposition \ref{prolocal2}, the solution extends globally, and part (i) of the Main Theorem follows.

To prove the modified scattering statement in part (ii) we may assume that $t_1\leq t_2\leq 2t_1$ and use 
Lemma \ref{bigbound2}, more precisely the bound \eqref{bn8}
\begin{equation}\label{bm1}
\|(|\xi|^\beta+|\xi|^{N_1+15})(g(\xi,t_2)-g(\xi,t_1))\|_{L^\infty_\xi}\lesssim\varepsilon_0(1+t_1)^{-p_1},
\end{equation}
where, see \eqref{bn2} and Lemma \ref{proE1+},
\begin{equation}\label{bm2}
\begin{split}
&g(\xi,t)=e^{iL(\xi,t)}\widehat{f}(\xi,t),\quad L(\xi,t):=\frac{\xi^4}{\pi}\int_0^t|\widehat{f}(\xi,s)|^2 \frac{ds}{s+1},\\
&f(t)=e^{it\Lambda}V(t),\quad V=H+i\Lambda\Psi=[h+A(h,h)]+i\Lambda[\phi+B(h,\phi)].
\end{split}
\end{equation}
With the notation in Theorem \ref{maintheo}, and using also \eqref{Al1000}, we notice that, for $t\in\{t_1,t_2\}$,
\begin{equation*}
 \Big\|(1+|\xi|)^{N_1} \big[e^{iG(\xi,t)}e^{it\Lambda(\xi)}\widehat{u}(\xi,t)
  -e^{iG(\xi,t)}\widehat{f}(\xi,t)\big]\Big\|_{L^2_\xi}\lesssim  
\Big\|(1+|\xi|)^{N_1} \big[\widehat{u}(\xi,t)-\widehat{V}(\xi,t)\big]\Big\|_{L^2_\xi}\lesssim \e_0(1+t)^{-1/4}.
\end{equation*}
Therefore, for \eqref{Zbound} it suffices to prove that
\begin{equation}\label{bm3}
\Big\|(1+|\xi|)^{N_1} \big[e^{iG(\xi,t_2)}\widehat{f}(\xi,t_2)
  -e^{iG(\xi,t_1)}\widehat{f}(\xi,t_1)\big]\Big\|_{L^2_\xi}\lesssim \varepsilon_0(1+t_1)^{-p_1}.
\end{equation}

Notice that $e^{iG(\xi,t)}\widehat{f}(\xi,t)=g(\xi,t)e^{i[G(\xi,t)-L(\xi,t)]}$, $t\in\{t_1,t_2\}$. Moreover
\begin{equation*}
\begin{split}
\big |e^{i[G(\xi,t_2)-L(\xi,t_2)]}-e^{i[G(\xi,t_1)-L(\xi,t_1)]}\big|
&\lesssim \big |[G(\xi,t_2)-L(\xi,t_2)]-[G(\xi,t_1)-L(\xi,t_1)]\big|\\
&\lesssim \xi^4\int_{t_1}^{t_2}\big|\widehat{u}(\xi,s)-\widehat{V}(\xi,s)\big|(|\widehat{u}(\xi,s)|+|\widehat{V}(\xi,s)|)\frac{ds}{s+1}.
\end{split}
\end{equation*}
In view of Proposition \ref{proE4}, we have, for any $\xi\in\mathbb{R}$ and $s\in [t_1,t_2]$,
\begin{equation*}
\big(|\xi|^\beta+|\xi|^{N_1+15}\big)\big(|\widehat{f}(\xi,s)|+|\widehat{V}(\xi,s)|+|\widehat{g}(\xi,s)|\big)\lesssim\e_0.
\end{equation*}
Also, using Proposition \ref{proE1+} and the bounds $\|h(s)\|_{H^{N_0}}+\|\phi_x(s)\|_{H^{N_0}}\lesssim \e_0(1+s)^{p_0}$, 
see \eqref{apriori0}, it follows that, for any $\xi\in\mathbb{R}$ and $s\in [t_1,t_2]$,
\begin{equation*}
\big(1+|\xi|^{N_1+15}\big)\big|\widehat{u}(\xi,s)-\widehat{V}(\xi,s)\big|\lesssim\e_0(1+s)^{2p_0}.
\end{equation*}
Therefore, using these three bounds and \eqref{bm1}, the left-hand side of \eqref{bm3} is dominated by
\begin{equation*}
\begin{split}
\Big\|(1+|\xi|)^{N_1} \big[g(\xi,t_2)-g(\xi,t_1)\big]\Big\|_{L^2_\xi}+
\Big\|(1+|\xi|)^{N_1} g(\xi,t_2)\big |e^{i[G(\xi,t_2)-L(\xi,t_2)]}-e^{i[G(\xi,t_1)-L(\xi,t_1)]}\big|\Big\|_{L^2_\xi}\\
\lesssim \varepsilon_0(1+t_1)^{-p_1}+\e_0\sup_{s\in[t_1,t_2]}(1+s)^{2p_0}\big\|\widehat{u}(\xi,s)-\widehat{V}(\xi,s)\big\|_{L^2_\xi}\\
\lesssim \e_0(1+t)^{-p_1}.
\end{split}
\end{equation*}
This completes the proof of the desired bound \eqref{bm3} and of the main theorem.

\section{Eulerian formulation and proof of Proposition \ref{prodecay}}\label{secEuler}

In this section we first recall the water waves equations in Eulerian coordinates.
We then explain our strategy for the proof of Proposition \ref{prodecay}.
This will be obtained as a consequence of Propositions \ref{proE1}, \ref{proE2}, \ref{proE3} and \ref{proE4} below.

\subsection{The equations in Eulerian coordinates}
The system of equations in Eulerian coordinates is
\begin{equation}
\left\{
\begin{array}{l}
\partial_t h = G(h) \phi
\\
\\
\partial_t \phi = -h - \frac{1}{2} {|\phi_x|}^2 + \frac{1}{2(1+{|h_x|}^2)} {\left( G(h)\phi + h_x \phi_x \right)}^2,
\end{array}
\right.
\end{equation}
where
\begin{equation}
G(h) := \sqrt{1+{|h_x|}^2} \N(h)
\end{equation}
and $\N$ denotes the Dirichlet-Neumann operator associated to $\Omega_t$.

Given a multilinear expression of $h$ and $\phi$
\begin{align*}
F = F(h,\phi) = \sum_{j\geq 1} F_j(h,\phi),
\end{align*}
where $F_j$ is an homogeneous polynomial of order $j$ in its arguments, we denote
\begin{align}
\label{kterm}
{[F]}_k := F_k (h,\phi) 
\end{align}
and
\begin{align}
\label{k+term}
{[F]}_{\geq k} := \sum_{j \geq k} {[F]}_j, \qquad  {[F]}_{\leq k} := \sum_{1 \leq j \leq k} {[F]}_j \, .
\end{align}

After expanding $\N$ for small displacements of the moving surface, see \cite{SulemBook,GMS2}, one obtains the equations
\begin{equation}
\label{Eeq}
\left\{
\begin{array}{rl}
\partial_t h  & = |\partial_x| \phi  - \partial_x( h  \partial_x \phi) - |\partial_x| (h |\partial_x| \phi)
    \\  & - \frac{1}{2} |\partial_x| \left[  h^2 |\partial_x|^2 \phi + |\partial_x| (h^2 |\partial_x| \phi) 
    - 2 h |\partial_x| (h |\partial_x| \phi )  \right] + R_1(h,\phi),
\\
\\
\partial_t \phi  & = - h - \frac{1}{2} {|\phi_x|}^2 + \frac{1}{2} {||\partial_x|\phi|}^2 + 
      |\partial_x|\phi \left[ h |\partial_x|^2 \phi - |\partial_x| (h |\partial_x| \phi) \right] + R_2(h,\phi),
\end{array}
\right.
\end{equation}
where:
\begin{align}
\label{defR_1}
R_1 (h,\phi) & := {[G(h)\phi]}_{\geq 4}
\\
\label{defR_2}
R_2 (h,\phi) & := {\left[ \frac{(G(h)\phi + h_x\phi_x)^2}{2 (1+{|h_x|}^2)} \right]}_{\geq 4} \, .
\end{align}

Let us denote
\begin{align}
\label{p_2}
M_2(h,\phi) & := -\partial_x( h  \partial_x \phi) - \La (h \La \phi),
\\
\label{p_3}
M_3(h,h,\phi) & := - \frac{1}{2} \La \left[  h^2 \La^2 \phi + \La (h^2 \La \phi)
	  - 2 h \La (h \La \phi ) \right],
\\
\label{q_2}
Q_2(\phi,\phi) & := - \frac{1}{2} {|\phi_x|}^2 + \frac{1}{2} {|\La \phi|}^2,
\\
\label{q_3}
Q_3(\phi,h,\phi) & := \La \phi \left[ h \La^2 \phi - \La (h \La \phi) \right],
\end{align}
so that
\begin{equation}
\label{Eeq1}
\left\{
\begin{array}{l}
\partial_t h  = \La \phi + M_2(h,\phi) + M_3(h,h,\phi) + R_1(h,\phi),
\\
\\
\partial_t \phi  = - h + Q_2(\phi,\phi) + Q_3(\phi,h,\phi) + R_2(h,\phi).
\end{array}
\right.
\end{equation}

\subsection{Strategy for the proof of Proposition \ref{prodecay}}

Recall that in Proposition \ref{prodecay} we are making the following a priori assumptions:
\begin{align}
\label{Eapriori}
\sup_{t\in[0,T]}  \left[ (1+t)^{-p_0} {\| (h(t) , \partial_x \phi(t)) \|}_{X_{N_0}} +\sqrt{1+t}  {\| (h(t) , \phi(t)) \|}_{Z^\p}\right] \leq \e_1 
\end{align}
and
\begin{align}
\label{Eapriori0}
\|h_0+i\Lambda\phi_0\|_{H^{N_0+1}}+\|x\partial_x(h_0+i\Lambda\phi_0)\|_{H^{N_0/2}}+\|h_0+i\Lambda\phi_0\|_{Z}\leq \e_0,
\end{align}
for some $\e_1\in [\e_0,1]$. We then aim to prove 
\begin{align}
\label{Zpconc}
\sup_{t\in[0,T]} \sqrt{1+t} {\| (h(t), \phi(t)) \|}_{Z^\p}  \lesssim \e_0 +\e_1^2 \, 
\end{align}
and
\begin{align}
\label{H^Nconc}
\sup_{t\in[0,T]}{\| h(t) + i\Lambda \phi(t)  \|}_{H^{N_1 + 10}} \lesssim \e_0 + \e_1^2.
\end{align}

To prove \eqref{Zpconc} and \eqref{H^Nconc} the idea to transform the quadratic equations into cubic ones, 
and then apply the strategy of our previous paper \cite{FNLS} to the cubic equations.
We will proceed through several steps. We first perform a bilinear normal form transformation:

\begin{pro}[Poincar\'{e}-Shatah Normal Form]\label{proE1}
There exist bilinear operators $A$ and $B$ such that if
\begin{equation}
\label{HPsi0}
\left\{
\begin{array}{l}
H \stackrel{def}{=} h + A (h,h),
\\
\\
\Psi \stackrel{def}{=} \phi + B(h,\phi),
\end{array}
\right.
\end{equation}
then the function $V$ defined by
\begin{align}
\label{V}
V \stackrel{def}{=} H + i \Lambda \Psi 
\end{align}
satisfies
\begin{equation}
\label{eqV}
\partial_t V + i \Lambda V = C \left(h, |\partial_x| \phi \right) 
\end{equation}
where $C$ is a nonlinearity consisting of cubic and higher order terms.
\end{pro}

A more precise statement of this Proposition, with the explicit form of $A,B$, and $C$
is given in section \ref{secproE1}, followed by its proof.
In section \ref{secproE2} we show the following bounds on the transformation:

\begin{pro}[Bounds for the transformation]\label{proE2}
Under the a priori  assumptions \eqref{Eapriori} on $h = h(t)$ and $\phi = \phi(t)$, we have for any $t \in [0,T]$
\begin{align}
\label{proE2conc1}
& {\|A(h,h) \|}_{W^{N_1+4,\infty}} 
  + {\| \Lambda B(h,\phi) \|}_{W^{N_1+4,\infty}} \lesssim \e_1^2(1+t)^{-1/2},
\\
\label{proE2conc2}
& {\|A(h,h) \|}_{H^{N_0-5}} + {\| \Lambda B(h,\phi)\|}_{H^{N_0-5}}  \lesssim \e_1^2,
\\
\label{proE2conc3}
& {\| SA(h,h) ) \|}_{H^{\frac{N_0}{2}-5}} 
  + {\| S \Lambda B(h,\phi)\|}_{H^{\frac{N_0}{2}-5}} \lesssim \e_1^2.
\end{align}

In particular we have
\begin{align}
\label{hphivZp}
& {\| (h, \phi) \|}_{Z^\p} \lesssim  {\| V \|}_{W^{N_1+4,\infty}} + \e_1^2(1+t)^{-1/2} \, ,
\\
\label{hphivH^N}
& {\| h+i\Lambda\phi \|}_{H^{N_1+10}} \lesssim  {\| V \|}_{H^{N_1+11}} + \e_1^2 \, ,
\end{align}
and
\begin{align}
\label{vhphiZp}
& {\| V \|}_{W^{N_1+4,\infty}} \lesssim \e_1(1+t)^{-1/2}\, ,
\\
\label{vhphiSobolev}
& {\| H+i\partial_x\Psi\|}_{H^{N_0 -6}} \lesssim  \varepsilon_1(1+t)^{p_0},
\\
\label{vhphiS}
& {\| S H \|}_{H^{\frac{N_0}{2}-6}} + {\| S \partial_x \Psi \|}_{H^{\frac{N_0}{2}-6}} \lesssim \e_1(1+t)^{p_0}.
\end{align}
\end{pro}

The above Proposition shows that 
the a priori  smallness assumption 
\eqref{Eapriori} can be suitably transferred to $V$. The next step is to improve \eqref{vhphiS} by using the equation \eqref{eqV} and the specific properties of the nonlinearity:

\begin{pro}[Improvement of the weighted bound on $V$]\label{proE3}
Let $V$ be the function defined by \eqref{V} and satisfying \eqref{eqV}. Then
\begin{align}
\label{vhphiSimp}
& \sup_{t\in[0,T]} {(1+t)}^{-5p_0} \big[{\| S V(t) \|}_{H^{N_0/2-20}}+{\| V(t) \|}_{H^{N_0/2-20}}\big]\lesssim \e_0+\e_1^2 \, .
\end{align}
Furthermore, if we define the profile of $V$ as
\begin{align}
\label{FV}
 f(t,x) := \left( e^{it\Lambda} V(t) \right) (x) \, .
\end{align}
we have
\begin{align}
\label{xd_xF} 
&  \sup_{t\in[0,T]} {(1+t)}^{-5p_0} \big[{\| x\partial_x f(t) \|}_{H^{N_0/2-20}}+{\| f(t) \|}_{H^{N_0/2-20}}\big]\lesssim \e_0+\e_1^2\, .
\end{align}
\end{pro}

This is proved in section \ref{secproE3}. The bound \eqref{vhphiSimp} improves the bound \eqref{vhphiS} by gaining half derivative for low frequencies on the estimate for $\Psi$, at the expense of losing a small amount of decay and some derivatives. This gives us \eqref{xd_xF} and allows us to exploit the bounds obtained in our previous paper \cite{FNLS}.

Using the bounds given by Proposition \ref{proE2} we will work on the scalar cubic equation \eqref{eqV} with the aim of showing:

\begin{pro}[Bound on the $Z$-norm and decay of the $Z^\p$-norm of $V$]\label{proE4}
Let $V$ be defined as above, and satisfying the bounds \eqref{vhphiZp}, \eqref{vhphiSobolev}, \eqref{xd_xF} . 
Assume further that
\begin{align}
\label{VZapriori}
\sup_{t\in[0,T]} {\| V(t) \|}_{Z} \leq \e_1 \, ,
\end{align}
where $Z$ is the norm defined in \eqref{defZ}.
Then
\begin{align}
\label{VZ}
\sup_{t\in[0,T]} {\| V(t) \|}_{Z} \lesssim\e_0  +  \e_1^3 \, .
\end{align}

As a consequence, using also Lemma \ref{dispersive},
\begin{align}
\label{VZp}
& \sup_{t\in[0,T]} \sqrt{1+t}  {\| V(t) \|}_{W^{N_1+4,\infty}} \lesssim \e_0  +  \e_1^2,
\end{align}
and
\begin{align}
\label{VHN_1}
\sup_{t\in[0,T]} {\| V(t)\|}_{H^{N_1 + 11}} \lesssim \e_0  +  \e_1^3\, .
\end{align}
\end{pro}
The proof of \eqref{VZ} constitutes the heart of the proof for the decay in Eulerian coordinates,
and is performed in section \ref{secproE4}, using a construction similar to our paper \cite{FNLS}.

Using \eqref{VZapriori} and an estimate similar to \eqref{proE2conc1}, we can also obtain the following:
\begin{cor}\label{corproE4}
Under the a priori  assumptions \eqref{Eapriori} and \eqref{VZapriori} we have
\begin{align}
\label{estcorproE4}
\sup_{t\in[0,T]} {(1+t)}^{1/8} {\| \phi(t) \|}_{L^\infty} \lesssim \e_0 + \e_1^2 \, .
\end{align}
This shows in particular the validity of the assumption \eqref{lemenergy32hyp} in Lemma \ref{lemenergy32}.
\end{cor}

\begin{proof}[Proof of Corollary \ref{corproE4}] For any $k\in\mathbb{Z}$ let $P_k$ be the Littlewood-Paley projector defined after \eqref{phi_k}.
We estimate, using \eqref{VZp} and \eqref{hphivZp}
\begin{equation*}
{\|P_k\phi(t)\|}_{L^\infty} \lesssim (\e_0 + \e_1^2)2^{-k/2}(1+t)^{-1/2}.
\end{equation*}
At the same time, using \eqref{VZ}, \eqref{proE2conc1}, and the bound
\begin{align*}
{\| P_k B(h(t),\phi(t)) \|}_{L^\infty} \lesssim 2^k {\| \F P_k B(h(t),\phi(t)) \|}_{L^\infty} \lesssim
  2^k {\| h(t) \|}_{L^2}  {\| |\partial_x| \phi(t) \|}_{L^2} \lesssim  2^k \e_1^2 ,
\end{align*}
see the explicit from of the symbol $b$ in \eqref{b_10},
we can obtain
\begin{equation*}
{\|P_k\phi(t)\|}_{L^\infty} \lesssim {\|P_k\Psi(t)\|}_{L^\infty} + 
  {\|P_kB(h(t),\phi(t))\|}_{L^\infty} \lesssim 2^{k/3}(\e_0 + \e_1^2)+2^{k} \e_1^2.
\end{equation*}
The desired conclusion follows from these estimates by considering the cases $2^k \leq {(1+t)}^{-3/8}$ and $2^k \geq {(1+t)}^{-3/8}$.
\end{proof}

Observe that \eqref{VZp} together with \eqref{hphivZp} imply \eqref{Zpconc}.
The bound \eqref{VHN_1} together with \eqref{hphivH^N} implies \eqref{H^Nconc}, 
thereby conluding the proof of Proposition \ref{prodecay}.

\section{Proof of Propositions \ref{proE1}, \ref{proE2}, and \ref{proE3}: normal forms}\label{secproE1}

In this section we aim to transform the quadratic equations \eqref{Eeq} into cubic ones using a normal form transformation,
as in \cite{shatahKGE,GMS2}.
The possibility of doing this relies on the vanishing of the symbols of the quadratic interaction on the time resonant set.
We remark that the structure of the transformation here is very important because we only have information on $\partial_x \phi$
and not on $\phi$ or $\Lambda \phi$.
Therefore we want to find $H$  and $\Psi$ as in \eqref{HPsi0}, with $A$ and $B$ depending nicely on $h$ and $\partial_x \phi$, 
and such that $H+i\Lambda\Psi$ satisfies a cubic equation.

\subsection{Solving the homological equation}
Given a suitable symbol $m:\mathbb{R}\times\mathbb{R}\to\mathbb{C}$ we define the associated bilinear operator $M(f,g)$ by the formula
\begin{equation}\label{al0}
\mathcal{F}\big[M(f,g)\big](\xi)=\frac{1}{2\pi}\int_{\mathbb{R}}m(\xi,\eta)\widehat{f}(\xi-\eta)\widehat{g}(\eta)\,d\eta.
\end{equation}
The following lemma gives the explicit form for the transformation in Proposition \ref{proE1}:

\begin{lem}\label{proE1+}
Let
\begin{align}
\label{a_10}
a(\xi,\eta) &:=-\frac{|\xi|}{2}\frac{\eta}{|\eta|}\frac{\xi-\eta}{|\xi-\eta|},\\
\label{b_10}
b(\xi,\eta) &:=-|\eta|\frac{\xi-\eta}{|\xi-\eta|}\frac{\xi}{|\xi|},
\end{align}
and
\begin{align}
\label{p_20}
m_2(\xi,\eta) &= \xi\eta - |\xi||\eta| \, ,
\\
\label{q_20}
q_2(\xi,\eta) &= \frac{1}{2}(\xi-\eta) \eta + \frac{1}{2}|\xi-\eta||\eta|.
\end{align}
Then, the function $V$ defined as
\begin{align}
\label{V0}
V \stackrel{def}{=} H + i \Lambda \Psi=[h+A(h,h)]+i\Lambda[\phi+B(h,\phi)]
\end{align}
satisfies
\begin{align}
\label{eqV0}
\partial_t V & + i \Lambda V = \N_3 + \N_4 
\end{align}
where $\N_3$, respectively $\N_4$, are cubic, respectively quartic and higher, order terms explicitly given by
\begin{align}
\label{defN_3}
\N_3 &\stackrel{def}{=} M_3(h,h,\phi)+2A(M_2(h,\phi),h)+ i \Lambda \left[ Q_3(\phi,h,\phi) + B(M_2(h,\phi),\phi) + B(h, Q_2(\phi,\phi))\right],
\\
\label{defN_4}
\N_4 &\stackrel{def}{=} R_1(h,\phi) +2A(M_3(h,h,\phi)+R_1(h,\phi),h)
\\
\nn
& + i \Lambda \left[R_2(h,\phi)+B(h,Q_3(\phi,h,\phi)+R_2(h,\phi))+B(M_3(h,h,\phi)+R_1(h,\phi),\phi)\right].
\end{align}
\end{lem}

\begin{proof}[Proof of Lemma \ref{proE1+}]
Given equation \ref{Eeq1} we look for a transformation of the form $(h, \phi) \rightarrow (H, \Psi)$, with
\begin{equation}
\label{HPsi}
\left\{
\begin{array}{l}
H = h + A_1 (\phi,\phi) + A_2 (h,h),
\\
\\
\Psi = \phi + B(h,\phi),
\end{array}
\right.
\end{equation}
where $A_1,A_2$ are symmetric bilinear forms and $B$ is a bilinear form. Our goal is to eliminate the quadratic nonlinear expressions, i. e. 
\begin{equation}
\label{d_tHPsi}
\left\{
\begin{array}{l}
\partial_t H = \La \Psi + \mbox{cubic terms}
\\
\\
\partial_t \Psi = - H + \mbox{cubic terms}.
\end{array}
\right.
\end{equation}

Indeed, using \eqref{HPsi} and \eqref{Eeq1}, we have
\begin{align*}
\partial_t H-|\partial_x|\Psi =& -|\partial_x|\phi-|\partial_x|B(h,\phi)+\partial_t h+2A_1(\partial_t\phi,\phi)+2A_2(\partial_th,h)\\
=&-|\partial_x|B(h,\phi)+M_2(h,\phi)+M_3(h,h,\phi)+R_1(h,\phi)\\
&-2A_1(h,\phi)+2A_1(Q_2(\phi,\phi),\phi) + 2A_1(Q_3(\phi,h,\phi),\phi) + 2A_1(R_2(h,\phi),\phi)\\
&+2A_2(\La \phi,h) + 2A_2(M_2(h,\phi),h) + 2A_2(M_3(h,h,\phi),h) + 2A_2(R_1(h,\phi),h),
\end{align*}
and
\begin{align*}
\partial_t \Psi+H =& h + A_1 (\phi,\phi) + A_2 (h,h)+\partial_t\phi+B(h,\partial_t\phi)+B(\partial_th,\phi)\\
=&+A_1 (\phi,\phi) + A_2 (h,h)+Q_2(\phi,\phi) + Q_3(\phi,h,\phi) + R_2(h,\phi)\\
&-B(h,h)+B(h,Q_2(\phi,\phi)) + B(h,Q_3(\phi,h,\phi)) + B(h,R_2(h,\phi))\\
&+B(\La \phi,\phi) + B(M_2(h,\phi),\phi) + B(M_3(h,h,\phi),\phi) + B(R_1(h,\phi),\phi).
\end{align*}
The condition \eqref{d_tHPsi} is equivalent to
\begin{equation}\label{al0.1}
\begin{split}
-|\partial_x|B(h,\phi)+M_2(h,\phi)-2A_1(h,\phi)+2A_2(\La \phi,h)&=0,\\
A_1 (\phi,\phi)+Q_2(\phi,\phi)+B(\La \phi,\phi)&=0,\\
A_2 (h,h)-B(h,h)&=0.
\end{split}
\end{equation}
Therefore one can define
\begin{equation}\label{al0.2}
\begin{split}
&a_2(\xi,\eta)=\frac{b(\xi,\eta)+b(\xi,\xi-\eta)}{2},\qquad a_1(\xi,\eta)=-q_2(\xi,\eta)-\frac{b(\xi,\eta)|\xi-\eta|+b(\xi,\xi-\eta)|\eta|}{2},\\
&b(\xi,\eta)=\frac{-2(|\xi|+|\eta|-|\xi-\eta|)q_2(\xi,\eta)+(|\eta|+|\xi-\eta|-|\xi|)m_2(\xi,\eta)-2|\eta|m_2(\xi,\xi-\eta)}{D(\xi,\eta)},
\end{split}
\end{equation}
where
\begin{equation*}
D(\xi,\eta) = -{|\xi|}^2 - {|\xi-\eta|}^2 - {|\eta|}^2 +2|\xi||\xi-\eta| +2|\xi||\eta| +2|\eta||\xi-\eta|\,
\end{equation*}
and the identities \eqref{al0.1} are easily seen to be verified.

The formulas can be simplified in the one-dimensional situation. 
Indeed, we notice first that $m(\xi,\eta)=m(-\xi,-\eta)$ for all $m\in\{q_2,m_2,a_1,a_2,b\}$. 
Moreover, for $\xi > 0$, we calculate\footnote{Some of our symbols are discontinuous when $\xi=0$ or $\eta=0$ or $\xi-\eta=0$
due to the vanishing of the denominator $D(\xi,\eta)$.} explicitly:

if $\eta < 0$ then
\begin{equation*}
\begin{split}
&q_2(\xi,\eta)=0,\quad m_2(\xi,\eta)=-2|\xi||\eta|,\quad m_2(\xi,\xi-\eta)=0,\quad D(\xi,\eta)=4|\xi||\eta|,\quad b(\xi,\eta)=-|\eta|;
\end{split}
\end{equation*}

if $\eta \in (0,\xi)$ then
\begin{equation*}
\begin{split}
&q_2(\xi,\eta)=|\xi-\eta||\eta|,\quad m_2(\xi,\eta)=0,\quad m_2(\xi,\xi-\eta)=0,\quad D(\xi,\eta)=4|\xi-\eta||\eta|,\quad b(\xi,\eta)=-|\eta|;
\end{split}
\end{equation*}

if $\eta > \xi$ then
\begin{equation*}
\begin{split}
&q_2(\xi,\eta)=0,\quad m_2(\xi,\eta)=0,\quad m_2(\xi,\xi-\eta)=-2|\xi||\xi-\eta|,\quad D(\xi,\eta)=4|\xi||\xi-\eta|,\quad b(\xi,\eta)=|\eta|.
\end{split}
\end{equation*}

Using also the formulas in the first line of \eqref{al0.2}, we calculate, for $\xi > 0$,
\begin{equation*}
\begin{split}
\text{ if } \quad \eta < 0 \quad&\text{ then }\quad a_1(\xi,\eta)=0,\quad a_2(\xi,\eta)=\xi/2;\\
\text{ if } \quad \eta \in (0,\xi) \quad&\text{ then }\quad a_1(\xi,\eta)=0,\quad a_2(\xi,\eta)=-\xi/2;\\
\text{ if } \quad \eta > \xi \quad&\text{ then }\quad a_1(\xi,\eta)=0,\quad a_2(\xi,\eta)=\xi/2.
\end{split}
\end{equation*}

In particular $A_1=0$ and the desired formulas in the lemma follow.
\end{proof}

\subsection{Analysis of the symbols}
We now want to study the behavior of the symbols that appear in Lemma \ref{proE1+}. We will describe our multipliers in terms of a simple class of symbols $\mathcal{S}^{\infty}$, which is defined by
\begin{equation}\label{Al4}
\mathcal{S}^\infty\stackrel{def}{=}\{m:\mathbb{R}^2\to\mathbb{C}:\,m\text{ continuous and }\|m\|_{\mathcal{S}^\infty}:=\|\mathcal{F}^{-1}(m)\|_{L^1}<\infty\}.
\end{equation}
Clearly, $\mathcal{S}^\infty\hookrightarrow L^\infty(\mathbb{R}\times\mathbb{R})$.
Moreover, $\mathcal{S}^\infty$ symbols are compatible with H\"{o}lder-type bounds on bilinear operators.
Our first lemma summarizes some simple properties of the $\mathcal{S}^{\infty}$ symbols.

\begin{lem}\label{touse}
(i) If $m,m'\in \mathcal{S}^\infty$ then $m\cdot m'\in\mathcal{S}^\infty$ and
\begin{equation}\label{al8}
\|m\cdot m'\|_{\mathcal{S}^\infty}\lesssim \|m\|_{\mathcal{S}^\infty}\|m'\|_{\mathcal{S}^\infty}.
\end{equation}
Moreover, if $m\in \mathcal{S}^\infty$, $A:\mathbb{R}^2\to\mathbb{R}^2$ is a linear transformation, $v\in \mathbb{R}^2$, and $m_{A,v}(\xi,\eta):=m(A(\xi,\eta)+v)$ then
\begin{equation}\label{al9}
\|m_{A,v}\|_{\mathcal{S}^\infty}=\|m\|_{\mathcal{S}^\infty}.
\end{equation}

(ii) Assume $p,q,r\in[1,\infty]$ satisfy $1/p+1/q=1/r$, and $m\in \mathcal{S}^\infty$. Then, for any $f,g\in L^2(\mathbb{R})$,
\begin{equation}\label{mk6}
\|M(f,g)\|_{L^r}\lesssim \|m\|_{S^\infty}\|f\|_{L^p}\|g\|_{L^q},
\end{equation}
where $M$ is defined as in \eqref{al0}.
\end{lem}

\begin{proof}[Proof of Lemma \ref{touse}] Part (i) follows directly from the definition. To prove (ii) let
\begin{equation*}
K(x,y):=(\mathcal{F}^{-1}m)(x,y)=\int_{\mathbb{R}\times\mathbb{R}}m(\xi,\eta)e^{ix\cdot\xi}e^{iy\cdot\eta}\,d\xi d\eta.
\end{equation*}
Then
\begin{equation*}
\begin{split}
M(f,g)(x)=C\int_{\mathbb{R}^2}e^{ix\xi}m(\xi,\eta)\widehat{f}(\xi-\eta)\widehat{g}(\eta)\,d\eta d\xi=C\int_{\mathbb{R}^2}K(u,v)f(x-u)g(x-u-v)\,dudv,
\end{split}
\end{equation*}
and the desired bound \eqref{mk6} follows.
\end{proof}

We fix $\varphi:\mathbb{R}\to[0,1]$ an even smooth function supported in $[-8/5,8/5]$ and 
equal to $1$ in $[-5/4,5/4]$. Let
\begin{equation}
\label{phi_k}
\varphi_k(x):=\varphi(x/2^k)-\varphi(x/2^{k-1}),\qquad k\in\mathbb{Z},\,x\in\mathbb{R}.
\end{equation}
Let $P_k$ denote the operator defined by the Fourier multiplier $\xi\to\varphi_k(\xi)$. Given any multiplier $m:\mathbb{R}^2\to\mathbb{C}$ and any $k,k_1,k_2\in\mathbb{Z}$ we define
\begin{equation}\label{al11}
m^{k,k_1,k_2}(\xi,\eta):=m(\xi,\eta)\cdot\varphi_k(\xi)\varphi_{k_1}(\xi-\eta)\varphi_{k_2}(\eta).
\end{equation}
Our next lemma, which is an easy consequence of the explicit formulas \eqref{a_10}--\eqref{q_20}, describes our main multipliers $m_2,q_2,a,b$ in terms of $\mathcal{S}^\infty$ symbols.

\begin{lem}\label{description}
For any $k,k_1,k_2\in\mathbb{Z}$ we have
\begin{equation}\label{descpq}
\|m_2^{k,k_1,k_2}\|_{\mathcal{S}^\infty}+\|q_2^{k,k_1,k_2}\|_{\mathcal{S}^\infty}\lesssim 2^k2^{\min(k_1,k_2)},
\end{equation}
\begin{equation}\label{desca1}
\|a^{k,k_1,k_2}\|_{\mathcal{S}^\infty}\lesssim 2^k,
\end{equation}
and
\begin{equation}\label{descb}
\|b^{k,k_1,k_2}\|_{\mathcal{S}^\infty}\lesssim 2^{k_2}.
\end{equation}
\end{lem}

\subsection{Proof of Proposition \ref{proE2}: bounds on the normal form}\label{secproE2}
Recall that we are assuming \eqref{Eapriori} and we want to show the three estimates \eqref{proE2conc1}, 
\eqref{proE2conc2} and \eqref{proE2conc3} for the bilinear operators $A,B$ defined through their symbols $a,b$ in Lemma \ref{proE1+}.

As a consequence of \eqref{Eapriori}, we have the following bounds on $h=h(t)$ and $\phi=\phi(t)$, for any $k\in\mathbb{Z}$:
\begin{equation}\label{Al32}
\begin{split}
&\|P_k h\|_{L^2}+2^k\|P_k\phi\|_{L^2}\lesssim \e_1(1+t)^{p_0}2^{-N_0k_+},
\\
&\|P_k h\|_{L^\infty}+2^{k/2}\|P_k\phi\|_{L^\infty}\lesssim \e_1(1+t)^{-1/2}2^{-(N_1+4)k_+},
\\
&\|P_k Sh\|_{L^2}+2^k\|P_k S\phi\|_{L^2}\lesssim \e_1(1+t)^{p_0}2^{-N_0k_+/2},
\end{split}
\end{equation}
where here, and from now on, we denote $k_+ = \max(k,0)$.

For any $k\in\mathbb{Z}$ let
\begin{equation}\label{Al31}
\begin{split}
&\mathcal{X}_k:=\mathcal{X}_k^1\cup \mathcal{X}_k^2,\\
&\mathcal{X}_k^1:=\{(k_1,k_2)\in\mathbb{Z}\times\mathbb{Z}:\min(k_1,k_2)\leq k+4,\,|\max(k_1,k_2)-k|\leq 4\},\\
&\mathcal{X}_k^2:=\{(k_1,k_2)\in\mathbb{Z}\times\mathbb{Z}:\min(k_1,k_2)\geq k-4,\,|k_1-k_2|\leq 4\}.
\end{split}
\end{equation}
Also let
\begin{equation}\label{Al31.5}
\begin{split}
&\mathcal{X}_{k,s}:=\{(k_1,k_2)\in\mathcal{X}_k:2^{\min(k_1,k_2)}\leq\min(2^{k-10},(1+t)^{-10})\},\\
&\mathcal{X}_{k,l}:=\{(k_1,k_2)\in\mathcal{X}_k:2^{\min(k_1,k_2)}\geq\min(2^{k-10},(1+t)^{-10})\}.
\end{split}
\end{equation}

To prove \eqref{proE2conc1}--\eqref{proE2conc3} we estimate for any $k\in\mathbb{Z}$, using \eqref{Al32}, \eqref{desca1}--\eqref{descb}, 
and Lemma \ref{touse} (ii),
\begin{equation*}
\begin{split}
\|P_kA(h,h)\|_{L^2}&\lesssim \sum_{(k_1,k_2)\in\mathcal{X}_k}\|P_kA(P_{k_1}h,P_{k_2}h)\|_{L^2}\\
&\lesssim \sum_{k_1\leq k+4,\,|k_2-k|\leq 4}2^k\|P_{k_1}h\|_{L^\infty}\|P_{k_2}h\|_{L^2}+\sum_{(k_1,k_2)\in\mathcal{X}_k^2}2^k\|P_{k_1}h\|_{L^\infty}\|P_{k_2}h\|_{L^2}\\
&\lesssim \e_1^2(1+t)^{2p_0-1/2}2^{k/4}2^{-(N_0-3)k_+}
\end{split}
\end{equation*}
and
\begin{equation*}
\begin{split}
\|P_k\Lambda B&(h,\phi)\|_{L^2}\lesssim \sum_{(k_1,k_2)\in\mathcal{X}_k}2^{k/2}\|P_kB(P_{k_1}h,P_{k_2}\phi)\|_{L^2}\\
&\lesssim \sum_{k_1\leq k+4,\,|k_2-k|\leq 4}2^{k/2}2^{k_2}\|P_{k_1}h\|_{L^\infty}\|P_{k_2}\phi\|_{L^2}+\sum_{k_2\leq k+4,\,|k_1-k|\leq 4}2^{k/2}2^{k_2}\|P_{k_1}h\|_{L^2}\|P_{k_2}\phi\|_{L^\infty}\\
&+\sum_{(k_1,k_2)\in\mathcal{X}_k^2}2^k2^{k_1}\|P_{k_1}h\|_{L^\infty}\|P_{k_2}\phi\|_{L^2}\\
&\lesssim \e_1^2(1+t)^{2p_0-1/2}2^{k/4}2^{-(N_0-3)k_+}.
\end{split}
\end{equation*}
Therefore, for any $k\in\mathbb{Z}$,
\begin{equation}\label{Al1000}
\|P_kA(h,h)\|_{L^2}+\|P_k\Lambda B(h,\phi)\|_{L^2}\lesssim \e_1^2(1+t)^{2p_0-1/2}2^{k/4}2^{-(N_0-3)k_+},
\end{equation}
and the desired bound \eqref{proE2conc2} follows.

Similarly, we also have the $L^\infty$ bounds,
\begin{equation*}
\begin{split}
\|P_kA(h,h)\|_{L^\infty}&\lesssim \sum_{(k_1,k_2)\in\mathcal{X}_k}\|P_kA(P_{k_1}h,P_{k_2}h)\|_{L^\infty}\\
&\lesssim \sum_{(k_1,k_2)\in\mathcal{X}_k}2^k\|P_{k_1}h\|_{L^\infty}\|P_{k_2}h\|_{L^\infty}\\
&\lesssim \e_1^2(1+t)^{p_0-1}2^{k/4}2^{-N_0k_+/2}
\end{split}
\end{equation*}
and
\begin{equation*}
\begin{split}
\|P_k\Lambda B(h,\phi)\|_{L^\infty}&\lesssim \sum_{(k_1,k_2)\in\mathcal{X}_k}2^{k/2}\|P_kB(P_{k_1}h,P_{k_2}\phi)\|_{L^\infty}\\
&\lesssim \sum_{(k_1,k_2)\in\mathcal{X}_k}2^{k/2}2^{k_2}\|P_{k_1}h\|_{L^\infty}\|P_{k_2}\phi\|_{L^\infty}\\
&\lesssim \e_1^2(1+t)^{p_0-1}2^{k/4}2^{-N_0k_+/2}.
\end{split}
\end{equation*}
Therefore, for any $k\in\mathbb{Z}$,
\begin{equation}\label{Al1001}
\|P_kA(h,h)\|_{L^\infty}+\|P_k\Lambda B(h,\phi)\|_{L^\infty}\lesssim \e_1^2(1+t)^{p_0-1}2^{k/4}2^{-N_0k_+/2},
\end{equation}
and the desired bound \eqref{proE2conc1} follows.

To prove \eqref{proE2conc3} we notice first that the symbol $a$ is homogeneous of degree $1$, i. e.
\begin{equation*}
a(\lambda\xi,\lambda\eta)=\lambda a(\xi,\eta)\qquad\text{ for any }\xi,\eta\in\mathbb{R},\lambda\in(0,\infty).
\end{equation*}
Differentiating this identity with respect to $\lambda$ and then setting $\lambda=1$, we have
\begin{equation*}
(\xi\partial_\xi a)(\xi,\eta)+(\eta\partial_\eta a)(\xi,\eta)=a(\xi,\eta).
\end{equation*}
The symbol to $(\xi,\eta)\to b(\xi,\eta)$ is homogeneous of degree $1$. As a consequence, we have the identities
\begin{equation}\label{Al25}
\begin{split}
&(\xi\partial_\xi a)(\xi,\eta)+(\eta\partial_\eta a)(\xi,\eta)=a(\xi,\eta),\\
&(\xi\partial_\xi b)(\xi,\eta)+(\eta\partial_\eta b)(\xi,\eta)=b(\xi,\eta).
\end{split}
\end{equation}

Using the first formula in \eqref{Al25} we calculate
\begin{equation*}
\begin{split}
\mathcal{F}&\big[SA(h,h)\big](\xi)=\Big[\frac{1}{2}t\partial_t-\xi\partial_\xi-I\Big]\Big[\int_{\mathbb{R}}a(\xi,\eta)\widehat{h}(\xi-\eta,t)\widehat{h}(\eta,t)\,d\eta\Big]\\
&=\int_{\mathbb{R}}a(\xi,\eta)\Big[\frac{1}{2}t(\partial_t\widehat{h})-\widehat{h}\Big](\xi-\eta,t)\widehat{h}(\eta,t)\,d\eta+\int_{\mathbb{R}}a(\xi,\eta)\widehat{h}(\xi-\eta,t)\frac{1}{2}t(\partial_t\widehat{h})(\eta,t)\,d\eta\\
&-\int_{\mathbb{R}}(\xi\partial_\xi a)(\xi,\eta)\widehat{h}(\xi-\eta,t)\widehat{h}(\eta,t)\,d\eta-\int_{\mathbb{R}}a(\xi,\eta)\xi(\partial\widehat{h})(\xi-\eta,t)\widehat{h}(\eta,t)\,d\eta\\
&=\int_{\mathbb{R}}a(\xi,\eta)\widehat{Sh}(\xi-\eta,t)\widehat{h}(\eta,t)\,d\eta-\int_{\mathbb{R}}a(\xi,\eta)\eta(\partial\widehat{h})(\xi-\eta,t)\widehat{h}(\eta,t)\,d\eta\\
&+\int_{\mathbb{R}}a(\xi,\eta)\widehat{h}(\xi-\eta,t)\widehat{Sh}(\eta,t)\,d\eta+\int_{\mathbb{R}}a(\xi,\eta)\widehat{h}(\xi-\eta,t)[\eta(\partial\widehat{h})(\eta,t)+\widehat{h}(\eta,t)]\,d\eta\\
&+\int_{\mathbb{R}}[(\eta\partial_\eta a)(\xi,\eta)-a(\xi,\eta)]\widehat{h}(\xi-\eta,t)\widehat{h}(\eta,t)\,d\eta\\
&=\mathcal{F}\big[A(Sh,h)\big](\xi)+\mathcal{F}\big[A(h,Sh)\big](\xi)-\mathcal{F}\big[A(h,h)\big](\xi).
\end{split}
\end{equation*}
A similar calculation can be applied to the operator $B$, using also \eqref{Al25}. Therefore
\begin{equation}\label{Al26}
\begin{split}
SA(h,h)&=A(Sh,h)+A(h,Sh)-A(h,h),\\
SB(h,\phi)&=B(Sh,\phi)+B(h,S\phi)-B(h,\phi).
\end{split}
\end{equation}

For any $k\in\mathbb{Z}$ we estimate, using \eqref{Al32}, Lemma \ref{touse} (ii), and \eqref{desca1}--\eqref{descb}, and recalling \eqref{Al31.5},
\begin{equation*}
\begin{split}
\|P_kA&(Sh,h)\|_{L^2}\lesssim \sum_{(k_1,k_2)\in\mathcal{X}_k}\|P_kA_2(P_{k_1}Sh,P_{k_2}h)\|_{L^2}\\
&\lesssim \sum_{(k_1,k_2)\in\mathcal{X}_{k,l}}2^k\|P_{k_1}Sh\|_{L^2}\|P_{k_2}h\|_{L^\infty}+\sum_{(k_1,k_2)\in\mathcal{X}_{k,s}}2^k2^{\min(k_1,k_2)/2}\|P_{k_1}Sh\|_{L^2}\|P_{k_2}h\|_{L^2}\\
&\lesssim \e_1^2(1+t)^{2p_0-1/2}2^{k/4}2^{-(N_0/2-3)k_+},
\end{split}
\end{equation*}
and
\begin{equation*}
\begin{split}
\|P_k\Lambda &B(Sh,\phi)\|_{L^2}+\|P_k\Lambda B(h,S\phi)\|_{L^2}\\
&\lesssim 2^{k/2}\sum_{(k_1,k_2)\in\mathcal{X}_k}\big[\|P_kB(P_{k_1}Sh,P_{k_2}\phi)\|_{L^2}+\|P_kB(P_{k_1}h,P_{k_2}S\phi)\|_{L^2}\big]\\
&\lesssim 2^{k/2}\sum_{(k_1,k_2)\in\mathcal{X}_{k,l}}\big[2^{k_2}\|P_{k_1}Sh\|_{L^2}\|P_{k_2}\phi\|_{L^\infty}+2^{k_2}\|P_{k_1}h\|_{L^\infty}\|P_{k_2}S\phi\|_{L^2}\big]\\
&+2^{k/2}\sum_{(k_1,k_2)\in\mathcal{X}_{k,s}}2^{\min(k_1,k_2)/2}\big[2^{k_2}\|P_{k_1}Sh\|_{L^2}\|P_{k_2}\phi\|_{L^2}+2^{k_2}\|P_{k_1}h\|_{L^2}\|P_{k_2}S\phi\|_{L^2}\big]\\
&\lesssim \e_1^2(1+t)^{2p_0-1/2}2^{k/4}2^{-(N_0/2-3)k_+}.
\end{split}
\end{equation*}
Therefore, using also \eqref{Al26} and \eqref{Al1000}, for any $k\in\mathbb{Z}$,
\begin{equation}\label{Al1002}
\|P_kSA(h,h)\|_{L^2}+\|P_kS\Lambda B(h,\phi)\|_{L^2}\lesssim \e_1^2(1+t)^{2p_0-1/2}2^{k/4}2^{-(N_0/2-3)k_+},
\end{equation}
and the desired bound \eqref{proE2conc3} follows.

\subsection{Proof of Proposition \ref{proE3}}\label{secproE3} We start from the formula \eqref{eqV0},
\begin{equation*}
\partial_tV+i\Lambda V=\mathcal{N}_3+\mathcal{N}_4,
\end{equation*}
where $\mathcal{N}_3$ and $\mathcal{N}_4$ are given in \eqref{defN_3} and \eqref{defN_4}. Applying $S$ and commuting we derive the equation
\begin{equation*}
(\partial_t+i\Lambda) SV=S\mathcal{N}_3+S\mathcal{N}_4+(1/2)(\mathcal{N}_3+\mathcal{N}_4).
\end{equation*}
Moreover, with $f(t)=e^{it\Lambda}V(t)$, we have
\begin{equation*}
(x\partial_xf)(t)=e^{it\Lambda}\big[SV-(t/2)(\mathcal{N}_3+\mathcal{N}_4)].
\end{equation*}

It follows from the assumption \eqref{Eapriori0} and Proposition \ref{proE2} that
\begin{equation*}
\|V(0)\|_{H^{N_0/2-5}}+\|SV(0)\|_{H^{N_0/2-5}}\lesssim \e_0+\e_1^2.
\end{equation*}
Therefore, it suffices to prove the following:

\begin{lem}\label{Al30}
For any $t$, 
\begin{equation}\label{SN_3}
\|\mathcal{N}_3\|_{H^{N_0-20}}+\|S\mathcal{N}_3\|_{H^{N_0/2-20}}\lesssim \e_1^3(1+t)^{5p_0-1},
\end{equation}
and
\begin{equation}\label{SN_4}
\|\mathcal{N}_4\|_{H^{N_0-20}}+\|S\mathcal{N}_4\|_{H^{N_0/2-20}}\lesssim \e_1^3(1+t)^{5p_0-1}.
\end{equation}
\end{lem}

\begin{proof}[Proof of Lemma \ref{Al30}] As in the proof of Proposition \ref{proE2}, see \eqref{Al1000}, \eqref{Al1001}, and \eqref{Al1002}, for any $k\in\mathbb{Z}$ we have
\begin{equation}\label{Al40}
\begin{split}
&\|P_kQ_2(\phi,\phi)\|_{L^2}+\|P_kM_2(h,\phi)\|_{L^2}\lesssim \e_1^2(1+t)^{2p_0-1/2}2^{k/4}2^{-(N_0-3)k_+},
\\
&\|P_kSQ_2(\phi,\phi)\|_{L^2}+\|P_kSM_2(h,\phi)\|_{L^2}\lesssim \e_1^2(1+t)^{2p_0-1/2}2^{k/4}2^{-(N_0/2-3)k_+},
\\
&\|P_kQ_2(\phi,\phi)\|_{L^\infty}+\|P_kM_2(h,\phi)\|_{L^\infty}\lesssim \e_1^2(1+t)^{p_0-1}2^{k/4}2^{-(N_0/2-3)k_+}.
\end{split}
\end{equation}

We examine now the trilinear expressions $M_3(h,h,\phi)$ and $Q_3(\phi,h,\phi)$ in \eqref{p_3} and \eqref{q_3}. 
These expressions appear in both nonlinearities $\mathcal{N}_3$ and $\mathcal{N}_4$. 
To estimate them we start by estimating $h|\partial_x|\phi$ and $h|\partial_x|^2\phi$: 
using \eqref{Al32}, for any $k\in\mathbb{Z}$, we obtain as before
\begin{equation}\label{Al42}
\begin{split}
&\|P_k(h|\partial_x|\phi)\|_{L^2}+\|P_k(h|\partial_x|^2\phi)\|_{L^2}\lesssim \e_1^2(1+t)^{2p_0-1/2}2^{-(N_0-3)k_+},\\
&\|P_kS(h|\partial_x|\phi)\|_{L^2}+\|P_kS(h|\partial_x|^2\phi)\|_{L^2}\lesssim \e_1^2(1+t)^{2p_0-1/2}2^{-(N_0/2-3)k_+},\\
&\|P_k(h|\partial_x|\phi)\|_{L^\infty}+\|P_k(h|\partial_x|^2\phi)\|_{L^\infty}\lesssim \e_1^2(1+t)^{p_0-1}2^{-(N_0/2-3)k_+}.
\end{split}
\end{equation}
We examine the formulas \eqref{p_3} and \eqref{q_3}. 
For any $k\in\mathbb{Z}$ we use \eqref{Al32} and \eqref{Al42} and estimate as before, for any $k\in\mathbb{Z}$,
\begin{equation}\label{Al45}
\begin{split}
&2^{-k/4}\|P_kM_3(h,h,\phi)\|_{L^2}+\|P_kQ_3(\phi,h,\phi)\|_{L^2}\lesssim \e_1^3(1+t)^{3p_0-1}2^{-(N_0-6)k_+},\\
&2^{-k/4}\|P_kSM_3(h,h,\phi)\|_{L^2}+\|P_kSQ_3(\phi,h,\phi)\|_{L^2}\lesssim \e_1^3(1+t)^{3p_0-1}2^{-(N_0/2-6)k_+},\\
&2^{-k/4}\|P_kM_3(h,h,\phi)\|_{L^\infty}+\|P_kQ_3(\phi,h,\phi)\|_{L^\infty}\lesssim \e_1^3(1+t)^{3p_0-3/2}2^{-(N_0/2-6)k_+}.
\end{split}
\end{equation}

Recall the formulas \eqref{defN_3} and \eqref{defN_4},
\begin{equation*}
\begin{split}
\N_3 &= M_3(h,h,\phi)+2A(M_2(h,\phi),h)+ i \Lambda \left[ Q_3(\phi,h,\phi) + B(M_2(h,\phi),\phi) + B(h, Q_2(\phi,\phi))\right],\\
\N_4 &= R_1(h,\phi) +2A(M_3(h,h,\phi)+R_1(h,\phi),h)\\
&+ i \Lambda \left[R_2(h,\phi)+B(h,Q_3(\phi,h,\phi)+R_2(h,\phi))+B(M_3(h,h,\phi)+R_1(h,\phi),\phi)\right],
\end{split}
\end{equation*}
and the bounds \eqref{estRL^2}
\begin{equation*}
\|R_1(h,\phi)+i\Lambda R_2(h,\phi)\|_{H^{N_0-10}}+\|S\big(R_1(h,\phi)+i\Lambda R_2(h,\phi)\big)\|_{H^{N_0/2-10}}\lesssim \e_1^4(1+t)^{-5/4}.
\end{equation*}
The desired bounds \eqref{SN_3} and \eqref{SN_4} follow using \eqref{Al40}, \eqref{Al45}, and Lemma \ref{Al60} below, with $G=R_2(h,\phi)$ and
\begin{equation*}
F\in\{Q_2(\phi,\phi),M_2(h,\phi),M_3(h,h,\phi),Q_3(\phi,h,\phi),R_1(h,\phi)\}.
\end{equation*}
\end{proof}

\begin{lem}\label{Al60}
Assume $F$ and $G$ satisfy the bounds, for any $k\in\mathbb{Z}$,
\begin{equation}\label{Al61}
\begin{split}
&\|P_kF\|_{L^2}+2^{k/2}\|P_kG\|_{L^2}\lesssim\e_1^2(1+t)^{3p_0-1/2}2^{-(N_0-12)k_+},\\
&\|P_kSF\|_{L^2}+2^{k/2}\|P_kSG\|_{L^2}\lesssim\e_1^2(1+t)^{3p_0-1/2}2^{-(N_0/2-12)k_+},\\
&\|P_kF\|_{L^\infty}+2^{k/2}\|P_kG\|_{L^\infty}\lesssim\e_1^2(1+t)^{3p_0-1}2^{-(N_0/2-12)k_+}.
\end{split}
\end{equation}
Then, for any $k\in\mathbb{Z}$,
\begin{equation}\label{Al62}
\begin{split}
&\|P_kA(F,h)\|_{L^2}+2^{k/2}\|P_kB(F,\phi)\|_{L^2}+2^{k/2}\|P_kB(h,G)\|_{L^2}\lesssim \e_1^3(1+t)^{5p_0-1}2^{k/4}2^{-(N_0-16)k_+},\\
&\|P_kSA(F,h)\|_{L^2}+2^{k/2}\|P_kSB(F,\phi)\|_{L^2}+2^{k/2}\|P_kSB(h,G)\|_{L^2}\lesssim \e_1^3(1+t)^{5p_0-1}2^{k/4}2^{-(N_0/2-16)k_+}.
\end{split}
\end{equation}
\end{lem}

\begin{proof}[Proof of Lemma \ref{Al60}] 
We estimate, using \eqref{Al32}, \eqref{Al61}, Lemma \ref{description}, and Lemma \ref{touse} (ii),
\begin{equation*}
\begin{split}
\|P_kA(F,h)\|_{L^2}&\lesssim \sum_{k_1\leq k+4,\,|k_2-k|\leq 4}2^k\|P_{k_1}F\|_{L^\infty}\|P_{k_2}h\|_{L^2}+\sum_{k_2\leq k+4,\,|k_1-k|\leq 4}2^k\|P_{k_1}F\|_{L^2}\|P_{k_2}h\|_{L^\infty}\\
&+\sum_{(k_1,k_2)\in\mathcal{X}_k^2}2^k\|P_{k_1}F\|_{L^\infty}\|P_{k_2}h\|_{L^2}\\
&\lesssim \e_1^3(1+t)^{5p_0-1}2^{k/4}2^{-(N_0-16)k_+}
\end{split}
\end{equation*}
and
\begin{equation*}
\begin{split}
\|P_kA(F,Sh)\|_{L^2}+&\|P_kA(SF,h)\|_{L^2}\lesssim \sum_{(k_1,k_2)\in\mathcal{X}_{k,l}}2^k\big[\|P_{k_1}F\|_{L^\infty}\|P_{k_2}Sh\|_{L^2}+\|P_{k_1}SF\|_{L^2}\|P_{k_2}h\|_{L^\infty}\big]\\
&+\sum_{(k_1,k_2)\in\mathcal{X}_{k,s}}2^k2^{\min(k_1,k_2)/2}\big[\|P_{k_1}F\|_{L^2}\|P_{k_2}Sh\|_{L^2}+\|P_{k_1}SF\|_{L^2}\|P_{k_2}h\|_{L^2}\big]\\
&\lesssim \e_1^3(1+t)^{5p_0-1}2^{k/4}2^{-(N_0/2-16)k_+}.
\end{split}
\end{equation*}
This proves the desired bounds for $\|P_kA(F,h)\|_{L^2}$ and $\|P_kSA(F,h)\|_{L^2}$. The other bounds in \eqref{Al62} are similar.
\end{proof}

\section{Proof of Proposition \ref{proE4}: uniform control of the $Z$ norm}\label{secproE4}
In this section we prove Proposition \ref{proE4}, which is our main bootstrap estimate. 
We start by rewriting the equation \eqref{eqV0} in the form 
\begin{equation}\label{fd1}
\partial_t V + i\Lambda V = \widetilde{\N}_3 + R,
\end{equation}
where
\begin{equation}\label{fd2}
\begin{split}
\widetilde{\N}_3:= &M_3(H,H,\Psi) + 2A_2(M_2(H,\Psi),H)\\
&+ i \Lambda \left[ Q_3(\Psi,H,\Psi) + B(M_2(H,\Psi),\Psi) + B(H, Q_2(\Psi,\Psi))\right],
\end{split}
\end{equation}
and
\begin{equation}\label{fd3}
R:=\mathcal{N}_3+\mathcal{N}_4-\widetilde{\mathcal{N}}_3.
\end{equation}
The point of the above decomposition is to express the cubic part of the nonlinearity in terms of $H$ and $\Psi$,
hence of the solution $V$.

Letting $f(t)=e^{it\Lambda}V(t)$ as in \eqref{FV}, we have
\begin{equation}\label{fd3.5}
\partial_tf=e^{it\Lambda}(\widetilde{\N}_3(t)+R(t)).
\end{equation}

Notice that
\begin{equation}\label{fd4}
\begin{split}
&H(t)=\frac{V(t)+\overline{V}(t)}{2}=\frac{e^{-it\Lambda}f(t)+e^{it\Lambda}\overline{f}(t)}{2},\\
&\Psi(t)=\frac{i\Lambda^{-1}(\overline{V}(t)-V(t))}{2}=i\frac{e^{it\Lambda}(\Lambda^{-1}\overline{f})(t)-e^{-it\Lambda}(\Lambda^{-1}f)(t)}{2}.
\end{split}
\end{equation}
Therefore, the trilinear expression $\widetilde{\N}_3$ can be written in terms of $f(t)$ and $\overline{f}(t)$, in the form
\begin{equation}\label{fd5}
\begin{split}
&\mathcal{F}(e^{it\Lambda}\widetilde{\N}_3(t))(\xi)=\frac{i}{(2\pi)^2}\big[I^{++-}(\xi,t)+I^{--+}(\xi,t)+I^{+++}(\xi,t)+I^{---}(\xi,t)\big],\\
&I^{++-}(\xi,t):=\int_{\mathbb{R}\times\mathbb{R}}e^{it[\Lambda(\xi)-\Lambda(\xi-\eta)-\Lambda(\eta-\sigma)+\Lambda(\sigma)]}c^{++-}(\xi,\eta,\sigma)\widehat{f}(\xi-\eta,t)\widehat{f}(\eta-\sigma,t)\widehat{\overline{f}}(\sigma,t)\,d\eta d\sigma,\\
&I^{--+}(\xi,t):=\int_{\mathbb{R}\times\mathbb{R}}e^{it[\Lambda(\xi)-\Lambda(\xi-\eta)+\Lambda(\eta-\sigma)+\Lambda(\sigma)]}c^{--+}(\xi,\eta,\sigma)\widehat{\overline{f}}(\xi-\eta,t)\widehat{\overline{f}}(\eta-\sigma,t)\widehat{f}(\sigma,t)\,d\eta d\sigma,\\
&I^{+++}(\xi,t):=\int_{\mathbb{R}\times\mathbb{R}}e^{it[\Lambda(\xi)-\Lambda(\xi-\eta)-\Lambda(\eta-\sigma)-\Lambda(\sigma)]}c^{+++}(\xi,\eta,\sigma)\widehat{f}(\xi-\eta,t)\widehat{f}(\eta-\sigma,t)\widehat{f}(\sigma,t)\,d\eta d\sigma,\\
&I^{---}(\xi,t):=\int_{\mathbb{R}\times\mathbb{R}}e^{it[\Lambda(\xi)+\Lambda(\xi-\eta)+\Lambda(\eta-\sigma)+\Lambda(\sigma)]}c^{---}(\xi,\eta,\sigma)\widehat{\overline{f}}(\xi-\eta,t)\widehat{\overline{f}}(\eta-\sigma,t)\widehat{\overline{f}}(\sigma,t)\,d\eta d\sigma.
\end{split}
\end{equation}
The symbols $c^{++-}, c^{--+}, c^{+++}, c^{---}$ can be calculated explicitly, using the formulas \eqref{p_3}, \eqref{q_3}, and \eqref{a_10}--\eqref{q_20}, see Appendix \ref{secsym}. For us it is important to notice that these symbols are real-valued, and satisfy the uniform bounds
\begin{equation}\label{csymbols0}
\big\|\mathcal{F}^{-1}[c^{\iota_1\iota_2\iota_3}(\xi,\eta,\sigma)\cdot\varphi_l(\xi)\varphi_{k_1}(\xi-\eta)\varphi_{k_2}(\eta-\sigma)\varphi_{k_3}(\sigma)]\big\|_{L^1(\mathbb{R}^3)}\lesssim 2^{l/2}2^{2\max(k_1,k_2,k_3)},
\end{equation}
for any $(\iota_1\iota_2\iota_3)\in\{(++-),(--+),(+++),(---)\}$ and $l,k_1,k_2,k_3\in\mathbb{Z}$. As a consequence,
\begin{equation}\label{csymbols}
\big\|c^{\iota_1\iota_2\iota_3}(\xi,\eta,\sigma)\cdot\varphi_l(\xi)\varphi_{k_1}(\xi-\eta)\varphi_{k_2}(\eta-\sigma)\varphi_{k_3}(\sigma) \big\|_{\mathcal{S}^\infty_{\eta,\sigma}} \lesssim 2^{l/2}2^{2\max(k_1,k_2,k_3)},
\end{equation}
for any $(\iota_1\iota_2\iota_3)\in\{(++-),(--+),(+++),(---)\}$, $\xi\in\mathbb{R}$, and $l,k_1,k_2,k_3\in\mathbb{Z}$. Moreover, for any $\mathbf{k}=(k_1,k_2,k_3),\mathbf{l}=(l_1,l_2,l_3)\in\mathbb{Z}^3$ let
\begin{equation*}
\begin{split}
&c^\ast_\xi(x,y):=c^{++-}(\xi,-x,-\xi-x-y),\\
&(\partial_xc^\ast_\xi)_{\mathbf{k},\mathbf{l}}(x,y):=(\partial_xc^\ast_\xi)(x,y)\cdot \varphi_{k_1}(\xi+x)\varphi_{k_2}(\xi+y)\varphi_{k_3}(\xi+x+y)\varphi_{l_1}(x)\varphi_{l_2}(y)\varphi_{l_3}(2\xi+x+y),\\
&(\partial_yc^\ast_\xi)_{\mathbf{k},\mathbf{l}}(x,y):=(\partial_yc^\ast_\xi)(x,y)\cdot \varphi_{k_1}(\xi+x)\varphi_{k_2}(\xi+y)\varphi_{k_3}(\xi+x+y)\varphi_{l_1}(x)\varphi_{l_2}(y)\varphi_{l_3}(2\xi+x+y).
\end{split}
\end{equation*}
Then, for any $\mathbf{k},\mathbf{l}\in\mathbb{Z}^3$, and $\xi\in\mathbb{R}$,
\begin{equation}\label{csymbols2}
\begin{split}
&\|(\partial_xc^\ast_\xi)_{\mathbf{k},\mathbf{l}}\|_{\mathcal{S}^\infty}\lesssim 2^{-\min(k_1,k_3)}2^{5\max(k_1,k_2,k_3)/2},\\
&\|(\partial_yc^\ast_\xi)_{\mathbf{k},\mathbf{l}}\|_{\mathcal{S}^\infty}\lesssim 2^{-\min(k_2,k_3)}2^{5\max(k_1,k_2,k_3)/2}.
\end{split}
\end{equation}
These bounds are proved in Lemma \ref{cprop}.

Let (compare with \eqref{c0}),
\begin{equation}\label{bn2}
\begin{split}
&\widetilde{c}(\xi):=-8\pi|\xi|^{3/2}c^\ast_\xi(0,0)=4\pi|\xi|^4,\\
&L(\xi,t):=\frac{\widetilde{c}(\xi)}{4\pi^2}\int_0^t|\widehat{f}(\xi,s)|^2 \frac{ds}{s+1},\\
&g(\xi,t):=e^{iL(\xi,t)}\widehat{f}(\xi,t).
\end{split}
\end{equation}
It follows from \eqref{fd3.5} that
\begin{equation}\label{bn3}
\begin{split}
(\partial_tg)(\xi,t)&=\frac{i}{(2\pi)^2}e^{iL(\xi,t)}\Big[I^{++-}(\xi,t)+\widetilde{c}(\xi)\frac{|\widehat{f}(\xi,t)|^2}{t+1}\widehat{f}(\xi,t)\Big]\\
&+\frac{i}{(2\pi)^2}e^{iL(\xi,t)}[I^{--+}(\xi,t)+I^{+++}(\xi,t)+I^{---}(\xi,t)]+e^{iL(\xi,t)}e^{it\Lambda(\xi)}\widehat{R}(\xi,t).
\end{split}
\end{equation}
Proposition \ref{proE4} clearly follows from Lemma \ref{bigbound2} below.

\begin{lem}\label{bigbound2}
With the same notation as before, recall that $f$ satisfies the bounds
\begin{equation}\label{bn7}
\sup_{t\in[0,T]}\big[(1+t)^{-p_0}\|f(t)\|_{H^{N_0-10}}+(1+t)^{-5p_0}\|(x\partial_xf)(t)\|_{H^{N_0/2-20}}+\|f(t)\|_{Z}\big]\leq\varepsilon_1.
\end{equation}
Then there is $p_1>0$ such that, for any $m\in\{1,2,\ldots\}$ and any $t_1\leq t_2\in[2^{m}-2,2^{m+1}]$,
\begin{equation}\label{bn8}
\|(|\xi|^\beta+|\xi|^{N_1+15})(g(\xi,t_2)-g(\xi,t_1))\|_{L^\infty_\xi}\lesssim\varepsilon_1^32^{-p_1m}.
\end{equation}
\end{lem}

The rest of the section is concerned with the proof of Lemma \ref{bigbound2}. We will use the following dispersive linear estimate from \cite[Lemma 2.3]{FNLS}:

\begin{lem}\label{dispersive}
 For any $t\in\mathbb{R}$ we have
\begin{equation}\label{disperse}
 \|e^{i t\Lambda}h\|_{L^\infty}\lesssim (1+|t|)^{-1/2}\|\,|\xi|^{3/4}\widehat{h}(\xi)\|_{L^\infty_\xi}+(1+|t|)^{-5/8}\big[\|x\cdot\partial_x h\|_{L^2}+\|h\|_{H^2}\big].
\end{equation}
\end{lem}

For any $k\in\mathbb{Z}$ let $f_k^+:=P_kf$, $f_k^-:=P_k\overline{f}$, and decompose, 
\begin{equation*}
I^{\iota_1\iota_2\iota_3}=\sum_{k_1,k_2,k_3\in\mathbb{Z}}I_{k_1,k_2,k_3}^{\iota_1\iota_2\iota_3},
\end{equation*}
for $(\iota_1\iota_2\iota_3)\in\{(++-),(--+),(+++),(---)\}$, where
\begin{equation}\label{bn9}
\begin{split}
I_{k_1,k_2,k_3}^{\iota_1\iota_2\iota_3}(\xi,t):=\int_{\mathbb{R}\times\mathbb{R}}&e^{it[\Lambda(\xi)-\iota_1\Lambda(\xi-\eta)-\iota_2\Lambda(\eta-\sigma)-\iota_3\Lambda(\sigma)]}\\
&\times c^{\iota_1\iota_2\iota_3}(\xi,\eta,\sigma)
\widehat{f_{k_1}^{\iota_1}}(\xi-\eta,t)\widehat{f_{k_2}^{\iota_2}}(\eta-\sigma,t)\widehat{f_{k_3}^{\iota_3}}(\sigma,t)\,d\eta d\sigma.
\end{split}
\end{equation}

Using \eqref{bn3}, for \eqref{bn8} it suffices to prove that if $k\in\mathbb{Z}$, $m\in\{1,2,\ldots\}$, $|\xi|\in[2^k,2^{k+1}]$, 
and $t_1\leq t_2\in[2^m-2,2^{m+1}]\cap[0,T]$ then
\begin{equation}\label{bn11}
\begin{split}
\sum_{k_1,k_2,k_3\in\mathbb{Z}}\Big|\int_{t_1}^{t_2} e^{iL(\xi,s)}\Big[I_{k_1,k_2,k_3}^{++-}(\xi,s)+\widetilde{c}(\xi)\frac{
\widehat{f_{k_1}^+}(\xi,s)\widehat{f_{k_2}^+}(\xi,s)\widehat{f_{k_3}^-}(-\xi,s)}{s+1}\Big]\,ds\Big|\\
\lesssim \varepsilon_1^32^{-p_1m}(2^{\beta k}+2^{(N_1+15)k})^{-1},
\end{split}
\end{equation}
\begin{equation}\label{bn12}
\sum_{k_1,k_2,k_3\in\mathbb{Z}}\Big|\int_{t_1}^{t_2} e^{iL(\xi,s)}I_{k_1,k_2,k_3}^{\iota_1\iota_2\iota_3}(\xi,s)\,ds\Big|
\lesssim \varepsilon_1^32^{-p_1m}(2^{\beta k}+2^{(N_1+15)k})^{-1}
\end{equation}
for any $(\iota_1,\iota_2,\iota_3)\in\{(--+),(+++),(---)\}$, and
\begin{equation}\label{bn12.5}
\Big|\int_{t_1}^{t_2} e^{iL(\xi,s)}e^{is\Lambda(\xi)}\widehat{R}(\xi,s)\,ds\Big|
\lesssim \varepsilon_1^32^{-p_1m}(2^{\beta k}+2^{(N_1+15)k})^{-1}.
\end{equation}

In view of \eqref{bn7} and Lemma \ref{disperse}, we have
\begin{equation}\label{bn13}
\begin{split}
\|\widehat{f_l^{\pm}}(s)\|_{L^2}&\lesssim \varepsilon_12^{p_0 m}2^{-(N_0-10) l_+},\\
\|(\partial\widehat{f_l^{\pm}})(s)\|_{L^2}&\lesssim \varepsilon_12^{5p_0 m}2^{-l}2^{-(N_0/2-20) l_+},\\
\|\widehat{f_l^{\pm}}(s)\|_{L^\infty}&\lesssim \varepsilon_1(2^{\beta l}+2^{(N_1+15)l})^{-1},\\
\|e^{\mp is\Lambda}f_l^{\pm}(s)\|_{L^\infty}&\lesssim \varepsilon_12^{-m/2}2^{-(N_0/2-20) l_+},
\end{split}
\end{equation}
for any $l\in\mathbb{Z}$ and $s\in[2^m-2,2^{m+1}]\cap[0,T]$. Using only the $L^2$ bounds in the first line of \eqref{bn13} it is easy to see that
\begin{equation}\label{bn14}
|I_{k_1,k_2,k_3}^{\iota_1\iota_2\iota_3}(\xi,s)|\lesssim \varepsilon_1^32^{3p_0 m}2^{\min(k_1,k_2,k_3)/2}(1+2^{\max(k_1,k_2,k_3)})^{-(N_0-20)},
\end{equation}
for any $(\iota_1\iota_2\iota_3)\in\{(++-),(--+),(+++),(---)\}$, $k_1,k_2,k_3\in\mathbb{Z}$. Moreover , using the $L^\infty$ bounds in \eqref{bn13},
\begin{equation*}
\Big|\widetilde{c}(\xi)\frac{\widehat{f_{k_1}^+}(\xi,s)\widehat{f_{k_2}^+}(\xi,s)\widehat{f_{k_3}^-}(\xi,s)}{s+1}\Big|
\lesssim 2^{-m}\varepsilon_1^32^{k}2^{-3N_1k_+}\mathbf{1}_{[0,4]}(\max(|k_1-k|,|k_2-k|,|k_3-k|)).
\end{equation*}
Using these two bounds it is easy to see that the sums in \eqref{bn11} and \eqref{bn12} over those $(k_1,k_2,k_3)$ for which 
$\max(k_1,k_2,k_3)\geq 3m/N_0-1000$ or $\min(k_1,k_2,k_3)\leq -4m$ are bounded by $C\varepsilon_1^32^{-p_1m}2^{-(N_1+15)k_+}$, as desired. 
The remaining sums have only $Cm^3$ terms. Therefore it suffices to prove the desired estimates for each $(k_1,k_2,k_3)$ fixed satisfying $k_1,k_2,k_3\in[-4m,3m/N_0-1000]$. 

At the same time, using \eqref{bn13}, together with the symbol estimates \eqref{csymbols}, it follows that 
\begin{equation}\label{bn14.5}
|I_{k_1,k_2,k_3}^{\iota_1\iota_2\iota_3}(\xi,s)|\lesssim \varepsilon_1^3|\xi|^{1/2}2^{\min(k_1,k_2,k_3)(1-\beta)}2^{\mathrm{med}(k_1,k_2,k_3)(1-\beta)}2^{-(N_1+10)\max(k_1,k_2,k_3,0)},
\end{equation}
for any $(\iota_1\iota_2\iota_3)\in\{(++-),(--+),(+++),(---)\}$, $k_1,k_2,k_3\in\mathbb{Z}$, and $\xi\in\mathbb{R}$,
where $\mathrm{med}(k_1,k_2,k_3)$ denotes the second largest frequency among $k_1,k_2,k_3$.

After these reductions, it suffices to prove the following lemma:

\begin{lem}\label{bb1}
Assume that $m\in\{1,2,\ldots\}$, $k\in\mathbb{Z}$, $|\xi|\in[2^k,2^{k+1}]$ and $t_1\leq t_2\in[2^{m}-2,2^{m+1}]\cap[0,T]$. Then, for any $k_1,k_2,k_3$ satisfying
\begin{equation}\label{ksassump}
k_1,k_2,k_3\in[-4m,3m/N_0-1000]\cap\mathbb{Z},\,\,\min(k_1,k_2,k_3)+\mathrm{med}(k_1,k_2,k_3)\geq -m(1+3\beta),
\end{equation}
we have 
\begin{equation}\label{bn20}
\Big|\int_{t_1}^{t_2} e^{iL(\xi,s)}\Big[I_{k_1,k_2,k_3}^{++-}(\xi,s)+\widetilde{c}(\xi)\frac{
\widehat{f_{k_1}^+}(\xi,s)\widehat{f_{k_2}^+}(\xi,s)\widehat{f_{k_3}^-}(-\xi,s)}{s+1}\Big]\,ds\Big|\\
\lesssim \varepsilon_1^32^{-2p_1m}(2^{\beta k}+2^{(N_1+15)k})^{-1},
\end{equation}
and, for any $(\iota_1\iota_2\iota_3)\in\{(--+),(+++),(---)\}$,
\begin{equation}\label{bn21}
\Big|\int_{t_1}^{t_2} e^{iL(\xi,s)}I_{k_1,k_2,k_3}^{\iota_1\iota_2\iota_3}(\xi,s)\,ds\Big|
\lesssim \varepsilon_1^32^{-2p_1m}(2^{\beta k}+2^{(N_1+15)k})^{-1}.
\end{equation}
Moreover
\begin{equation}\label{bn22}
\Big|\int_{t_1}^{t_2} e^{iL(\xi,s)}e^{is\Lambda(\xi)}\widehat{R}(\xi,s)\,ds\Big|
\lesssim \varepsilon_1^32^{-p_1m}(2^{\beta k}+2^{(N_1+15)k})^{-1}.
\end{equation}
\end{lem}

We will prove this main lemma in several steps. The bounds \eqref{bn20} and \eqref{bn21} are proved in the remaining part of this section. 
The main ingredients are the bounds \eqref{bn13}. We will also use the following consequence of Lemma \ref{touse} (ii): if $(p,q,r)\in\{(2,2,\infty),(2,\infty,2),(\infty,2,2)\}$ then
\begin{equation}\label{touse2}
 \Big|\int_{\mathbb{R}\times\mathbb{R}}\widehat{f}(\eta)\widehat{g}(\sigma)\widehat{h}(-\eta-\sigma)m(\eta,\sigma)\,d\eta d\sigma\Big|
\lesssim \|m\|_{\mathcal{S}^\infty}\|f\|_{L^p}\|g\|_{L^q}\|h\|_{L^r}.
\end{equation}
We also need suitable $L^2$ bounds on the derivatives $(\partial_sf^{\pm}_l)$, in order to be able to integrate by parts in time. More precisely, we have
\begin{equation}\label{touse4}
\|(\partial_s\widehat{f^{\pm}_l})(s)\|_{L^2}\lesssim \varepsilon_12^{5p_0m-m}2^{-(N_0-20)l_+},
\end{equation}
which is a consequence of the \eqref{SN_3}--\eqref{SN_4} and the formula $\partial_t f=e^{it\Lambda}[\mathcal{N}_3(t)+\mathcal{N}_4(t)]$, see \eqref{fd3.5}.

The bound \eqref{bn22} is proved in appendix \ref{appR}.

\subsection{Proof of \eqref{bn20}} We divide the proof in several cases.

\begin{lem}\label{bb2}
The bounds \eqref{bn20} hold provided that \eqref{ksassump} holds and, in addition,
\begin{equation}\label{bn25}
k_1,k_2,k_3\in[k-20,k+20]\cap\mathbb{Z}.
\end{equation}
\end{lem}

\begin{proof}[Proof of Lemma \ref{bb2}]
 This is the main case, when the specific correction in the left-hand side of \eqref{bn20} is important. We will prove that
\begin{equation}\label{bn28}
 \Big|I_{k_1,k_2,k_3}^{++-}(\xi,s)+\widetilde{c}(\xi)\frac{\widehat{f_{k_1}^+}(\xi,s)\widehat{f_{k_2}^+}(\xi,s)
\widehat{f_{k_3}^-}(-\xi,s)}{s+1}\Big|\lesssim 2^{-m}\varepsilon_1^32^{-2p_1m}(2^{\beta k}+2^{(N_1+15)k})^{-1},
\end{equation}
for any $s\in [t_1,t_2]$, which is clearly stronger than the desired bound \eqref{bn20}.

The bound \eqref{bn28} follows easily from the bound \eqref{bn14.5} if $k\leq-m/2$. Therefore, in the rest of the proof 
of \eqref{bn28} we may assume that
\begin{equation}\label{bn28.5}
 k\geq -m/2.
\end{equation}

After changes of variables we rewrite\footnote{The point of this change of variables 
is to be able to identify $\eta=\sigma=0$ as the unique critical point of the phase $\Phi$ in \eqref{bn26}.} 
\begin{equation*}
I_{k_1,k_2,k_3}^{++-}(\xi,s)=\int_{\mathbb{R}^2}e^{is\Phi(\xi,\eta,\sigma)}
\widehat{f_{k_1}^+}(\xi+\eta,s)\widehat{f_{k_2}^+}(\xi+\sigma,s)\widehat{f_{k_3}^-}(-\xi-\eta-\sigma,s)c^\ast_\xi(\eta,\sigma)\,d\eta d\sigma,
\end{equation*}
where
\begin{equation}\label{bn26}
\Phi(\xi,\eta,\sigma):=\Lambda(\xi)-\Lambda(\xi+\eta)-\Lambda(\xi+\sigma)+\Lambda(\xi+\eta+\sigma).
\end{equation}
Let $\overline{l}$ denote the smallest integer with the property that $2^{\overline{l}}\geq 2^{3k/4}2^{-49m/100}$ 
(in view of \eqref{bn28.5} $\overline{l}\leq k-10$). For any $m,k\in\mathbb{Z}$, $m\leq k$, we define
\begin{equation}\label{disp0}
 \varphi^{(m)}_k(x):=
\begin{cases}
\varphi(x/2^k)-\varphi(x/2^{k-1}),\qquad &\text{ if }k\geq m+1,\\
\varphi(x/2^k),\qquad &\text{ if }k=m.
\end{cases}
\end{equation}
We decompose
\begin{equation}\label{bn27}
I_{k_1,k_2,k_3}^{++-}(\xi,s)=\sum_{l_1,l_2=\overline{l}}^{k+20}J_{l_1,l_2}(\xi,s),
\end{equation}
where, for any $l_1,l_2\geq \overline{l}$,
\begin{equation}\label{bn29}
J_{l_1,l_2}(\xi,s):=\int_{\mathbb{R}^2}e^{is\Phi(\xi,\eta,\sigma)}\widehat{f_{k_1}^+}(\xi+\eta,s)\widehat{f_{k_2}^+}(\xi+\sigma,s)
\widehat{f_{k_3}^-}(-\xi-\eta-\sigma,s)\varphi^{(\overline{l})}_{l_1}(\eta)\varphi^{(\overline{l})}_{l_2}(\sigma)c^\ast_\xi(\eta,\sigma)\,d\eta d\sigma.
\end{equation}

{\bf{Step 1.}} We show first that 
\begin{equation}\label{bn30}
|J_{l_1,l_2}(\xi,s)|\lesssim 2^{-m}\varepsilon_1^32^{-3p_1m}(2^{\beta k}+2^{(N_1+15)k})^{-1},\qquad\text{ if }l_2\geq\max(l_1,\overline{l}+1).
\end{equation}
For this we integrate by parts in $\eta$ in the formula \eqref{bn29}. Recalling that $\Lambda(\theta)=\sqrt{|\theta|}$, we observe that
\begin{equation}\label{bn30.5}
\big|(\partial_\eta\Phi)(\xi,\eta,\sigma)\big|=\big|\Lambda'(\xi+\eta+\sigma)-\Lambda'(\xi+\eta)\big|\gtrsim 2^{l_2}2^{-3k/2},
\end{equation}
provided that $|\xi+\eta|\approx 2^k, |\xi+\eta+\sigma|\approx 2^k, |\sigma|\approx 2^{l_2}$. After integration by parts in $\eta$ we see that
\begin{equation*}
|J_{l_1,l_2}(\xi,s)|\leq |J^1_{l_1,l_2,1}(\xi,s)|+|J^2_{l_1,l_2,1}(\xi,s)|+|F_{l_1,l_2,1}(\xi,s)|+|G_{l_1,l_2,1}(\xi,s)|,
\end{equation*}
where
\begin{equation}\label{bn31}
\begin{split}
&J^1_{l_1,l_2,1}(\xi,s):=\int_{\mathbb{R}^2}e^{is\Phi(\xi,\eta,\sigma)}\widehat{f_{k_1}^+}(\xi+\eta,s)\widehat{f_{k_2}^+}(\xi+\sigma,s)
\widehat{f_{k_3}^-}(-\xi-\eta-\sigma,s)(\partial_\eta r_1)(\eta,\sigma)c^\ast_\xi(\eta,\sigma)\,d\eta d\sigma,\\
&J^2_{l_1,l_2,1}(\xi,s):=\int_{\mathbb{R}^2}e^{is\Phi(\xi,\eta,\sigma)}\widehat{f_{k_1}^+}(\xi+\eta,s)\widehat{f_{k_2}^+}(\xi+\sigma,s)
\widehat{f_{k_3}^-}(-\xi-\eta-\sigma,s) r_1(\eta,\sigma)(\partial_\eta c^\ast_\xi)(\eta,\sigma)\,d\eta d\sigma,\\
&F_{l_1,l_2,1}(\xi,s):=\int_{\mathbb{R}^2}e^{is\Phi(\xi,\eta,\sigma)}(\partial\widehat{f_{k_1}^+})(\xi+\eta,s)\widehat{f_{k_2}^+}(\xi+\sigma,s)
\widehat{f_{k_3}^-}(-\xi-\eta-\sigma,s)r_1(\eta,\sigma)c^\ast_\xi(\eta,\sigma)\,d\eta d\sigma,\\
&G_{l_1,l_2,1}(\xi,s):=\int_{\mathbb{R}^2}e^{is\Phi(\xi,\eta,\sigma)}\widehat{f_{k_1}^+}(\xi+\eta,s)\widehat{f_{k_2}^+}(\xi+\sigma,s)
(\partial\widehat{f_{k_3}^-})(-\xi-\eta-\sigma,s)r_1(\eta,\sigma)c^\ast_\xi(\eta,\sigma)\,d\eta d\sigma,
\end{split}
\end{equation}
and
\begin{equation*}
r_1(\eta,\sigma):=\frac{\varphi^{(\overline{l})}_{l_1}(\eta)\varphi_{l_2}(\sigma)}{s(\partial_\eta\Phi)(\xi,\eta,\sigma)}\cdot 
\varphi_{[k_1-2,k_1+2]}(\xi+\eta)\varphi_{[k_3-2,k_3+2]}(\xi+\eta+\sigma).
\end{equation*}

To estimate $|F_{l_1,l_2,1}(\xi,s)|$ we recall that $\xi$ and $s$ are fixed and use \eqref{touse2} with
\begin{equation*}
\begin{split}
&\widehat{f}(\eta):=e^{-is\Lambda(\xi+\eta)}(\partial\widehat{f_{k_1}^+})(\xi+\eta,s),\\
&\widehat{g}(\sigma):=e^{-is\Lambda(\xi+\sigma)}\widehat{f_{k_2}^+}(\xi+\sigma,s)\cdot \varphi(\sigma/2^{l_2+4}),\\
&\widehat{h}(\theta):=e^{is\Lambda(\xi-\theta)}\widehat{f_{k_3}^-}(-\xi+\theta,s)\cdot \varphi(\theta/2^{l_2+4}).
\end{split}
\end{equation*}
It is easy to see, compare with \eqref{bn30.5}, that $r_1$ satisfies the symbol-type estimates
\begin{equation}\label{bn32}
|(\partial_\eta^a\partial_\sigma^br_1)(\eta,\sigma)|\lesssim (2^{-m}2^{-l_2}2^{3k/2})(2^{-al_1}2^{-bl_2})\cdot 
\mathbf{1}_{[0,2^{l_1+4}]}(|\eta|)\mathbf{1}_{[2^{l_2-4},2^{l_2+4}]}(|\sigma|),
\end{equation}
for any $a,b\in[0,20]\cap\mathbb{Z}$. It follows from \eqref{bn13}  that
\begin{equation*}
 \|f\|_{L^2}\lesssim \varepsilon_12^{-k}2^{5p_0m}2^{-(N_0/2-20)k_+},\quad \|g\|_{L^\infty}\lesssim\varepsilon_12^{-m/2}2^{-(N_0/2-20)k_+},\quad \|h\|_{L^2}\lesssim \varepsilon_12^{l_2/2}2^{-\beta k}.
\end{equation*}
It follows from \eqref{touse2}, \eqref{bn32}, and \eqref{csymbols} that
\begin{equation*}
 \|\mathcal{F}^{-1}(r_1\cdot c^\ast_\xi)\|_{L^1}\lesssim 2^{-m}2^{-l_2}2^{4k}.
\end{equation*}
Therefore, using \eqref{touse2} and recalling that $2^{-l_2/2}\lesssim 2^{m/4}2^{-3k/8}$ and that $k\leq m/10$,
\begin{equation*}
\begin{split}
 |F_{l_1,l_2,1}(\xi,s)|&\lesssim \varepsilon_1^32^{-k}2^{5p_0m}2^{-(N_0/2-20)k_+}\cdot 2^{-m/2}2^{-(N_0/2-20)k_+}\cdot 2^{l_2/2}2^{-\beta k}\cdot 2^{-m}2^{-l_2}2^{4k}\\
&\lesssim \varepsilon_1^32^{-(N_0/2+40)k_+}2^{-m}\cdot 2^{-m/8}.
\end{split}
\end{equation*}
Similar arguments show that $|G_{l_1,l_2,1}(\xi,s)|\lesssim \varepsilon_1^32^{-(N_0/2+40)k_+}2^{-9m/8}$ and, using also the bound \eqref{csymbols2}\footnote{The use of \eqref{csymbols2} requires additional dyadic decompositions in the variables $\eta$, $\sigma$, and $2\xi+\eta+\sigma$. This leads to an additional polynomial loss $\approx m^3$, which does not change the estimates.},  $|J^2_{l_1,l_2,1}(\xi,s)|\lesssim \varepsilon_1^32^{-(N_0/2+40)k_+}2^{-9m/8}$. Therefore, for \eqref{bn30} it suffices to prove that
\begin{equation}\label{bn35}
|J^1_{l_1,l_2,1}(\xi,s)|\lesssim 2^{-m}\varepsilon_1^32^{-3p_1m}(2^{\beta k}+2^{(N_1+15)k})^{-1}.
\end{equation}

For this we integrate by parts again in $\eta$ and estimate
\begin{equation*}
|J^1_{l_1,l_2,1}(\xi,s)|\leq |J^1_{l_1,l_2,2}(\xi,s)|+|J^2_{l_1,l_2,2}(\xi,s)|+|F_{l_1,l_2,2}(\xi,s)|+|G_{l_1,l_2,2}(\xi,s)|,
\end{equation*}
where
\begin{equation*}
\begin{split}
&J^1_{l_1,l_2,2}(\xi,s):=\int_{\mathbb{R}^2}e^{is\Phi(\xi,\eta,\sigma)}\widehat{f_{k_1}^+}(\xi+\eta,s)\widehat{f_{k_2}^+}(\xi+\sigma,s)
\widehat{f_{k_3}^-}(-\xi-\eta-\sigma,s)(\partial_\eta r_2)(\eta,\sigma)c^\ast_\xi(\eta,\sigma)\,d\eta d\sigma,\\
&J^2_{l_1,l_2,2}(\xi,s):=\int_{\mathbb{R}^2}e^{is\Phi(\xi,\eta,\sigma)}\widehat{f_{k_1}^+}(\xi+\eta,s)\widehat{f_{k_2}^+}(\xi+\sigma,s)
\widehat{f_{k_3}^-}(-\xi-\eta-\sigma,s)r_2(\eta,\sigma)(\partial_\eta c^\ast_\xi)(\eta,\sigma)\,d\eta d\sigma,\\
&F_{l_1,l_2,2}(\xi,s):=\int_{\mathbb{R}^2}e^{is\Phi(\xi,\eta,\sigma)}(\partial\widehat{f_{k_1}^+})(\xi+\eta,s)\widehat{f_{k_2}^+}(\xi+\sigma,s)
\widehat{f_{k_3}^-}(-\xi-\eta-\sigma,s)r_2(\eta,\sigma)c^\ast_\xi(\eta,\sigma)\,d\eta d\sigma,\\
&G_{l_1,l_2,2}(\xi,s):=\int_{\mathbb{R}^2}e^{is\Phi(\xi,\eta,\sigma)}\widehat{f_{k_1}^+}(\xi+\eta,s)\widehat{f_{k_2}^+}(\xi+\sigma,s)
(\partial\widehat{f_{k_3}^-})(-\xi-\eta-\sigma,s)r_2(\eta,\sigma)c^\ast_\xi(\eta,\sigma)\,d\eta d\sigma.
\end{split}
\end{equation*}
and
\begin{equation*}
 r_2(\eta,\sigma):=\frac{(\partial_\eta r_1)(\eta,\sigma)}{s(\partial_\eta\Phi)(\xi,\eta,\sigma)} .
\end{equation*}
It follows from \eqref{bn32} that $r_2$ satisfies the stronger symbol-type bounds
\begin{equation}\label{bn44}
|(\partial_\eta^a\partial_\sigma^br_2)(\eta,\sigma)|\lesssim (2^{-m}2^{-l_1-l_2}2^{3k/2})(2^{-m}2^{-l_2}2^{3k/2})(2^{-al_1}2^{-bl_2})\cdot 
\mathbf{1}_{[0,2^{l_1+4}]}(|\eta|)\mathbf{1}_{[2^{l_2-4},2^{l_2+4}]}(|\sigma|),
\end{equation}
for $a,b\in[0,19]\cap\mathbb{Z}$. Therefore, using Lemma \ref{touse} as before,
\begin{equation*}
 |F_{l_1,l_2,2}(\xi,s)|+|G_{l_1,l_2,2}(\xi,s)|+|J^2_{l_1,l_2,2}(\xi,s)|\lesssim  \varepsilon_1^32^{-(N_0/2+40)k_+}2^{-m}\cdot 2^{-m/8}.
\end{equation*}
Moreover, we can now estimate $|J^1_{l_1,l_2,2}(\xi,s)|$ using only \eqref{bn44} and the first $L^\infty$ bounds in \eqref{bn13},
\begin{equation*}
\begin{split}
|J_{l_1,l_2,2}(\xi,s)|&\lesssim 2^{l_1+l_2}\cdot\varepsilon_1^3(2^{\beta k}+2^{(N_1+15)k})^{-3}\cdot (2^{-m}2^{-l_1-l_2}2^{3k/2})^22^{5k/2}\\
&\lesssim \varepsilon_1^32^{-(N_0/2+40)k_+}2^{-m}\cdot 2^{-m/50}.
\end{split}
\end{equation*}
This completes the proof of \eqref{bn35} and \eqref{bn30}.

A similar argument shows that
\begin{equation*}
|J_{l_1,l_2}(\xi,s)|\lesssim 2^{-m}\varepsilon_1^32^{-3p_1m}(2^{\beta k}+2^{(N_1+15)k})^{-1},\qquad\text{ if }l_1\geq\max(l_2,\overline{l}+1).
\end{equation*}

{\bf{Step 2.}} Using the decomposition \eqref{bn27}, for \eqref{bn28} it suffices to prove that
\begin{equation}\label{bn50}
 \Big|J_{\overline{l},\overline{l}}(\xi,s)+\widetilde{c}(\xi)\frac{\widehat{f_{k_1}^+}(\xi,s)\widehat{f_{k_2}^+}(\xi,s)
\widehat{f_{k_3}^-}(-\xi,s)}{s+1}\Big|\lesssim 2^{-m}\varepsilon_1^32^{-2p_1m}(2^{\beta k}+2^{(N_1+15)k})^{-1}.
\end{equation}

To prove \eqref{bn50} we notice that
\begin{equation*}
\Big|\Phi(\xi,\eta,\sigma)+\frac{\eta\sigma}{4|\xi|^{3/2}}\Big|\lesssim 2^{-5k/2}(|\eta|+|\sigma|)^3,
\end{equation*}
as long as $|\eta|+|\sigma|\leq 2^{k-5}$. Therefore, using the $L^\infty$ bounds in \eqref{bn13}
\begin{equation}\label{bn51}
 \Big|J_{\overline{l},\overline{l}}(\xi,s)-J'_{\overline{l},\overline{l}}(\xi,s)\Big|\lesssim \varepsilon_1^3(2^{\beta k}+2^{(N_1+15)k})^{-3}\cdot 2^m2^{-5k/2}2^{5\overline{l}}\lesssim \varepsilon_1^3(2^{\beta k}+2^{(N_1+15)k})^{-1}2^{-5m/4},
 \end{equation}
where
\begin{equation}\label{bn52}
\begin{split}
J'_{\overline{l},\overline{l}}(\xi,s):=\int_{\mathbb{R}^2}e^{-is\eta\sigma/(4|\xi|^{3/2})}&\widehat{f_{k_1}^+}(\xi+\eta,s)\widehat{f_{k_2}^+}(\xi+\sigma,s)\\
&\times 
\widehat{f_{k_3}^-}(-\xi-\eta-\sigma,s)\varphi(2^{-\overline{l}}\eta)\varphi(2^{-\overline{l}}\sigma)c^\ast_\xi(\eta,\sigma)\,d\eta d\sigma.
\end{split}
\end{equation}
Using the second bound in \eqref{bn13}, we see that
\begin{align*}
| \what{f_l} (\xi+\rho,s) - \what{f_l} (\xi,s)| \lesssim \int_0^\rho | \partial \what{f_l}(\xi+\mu,s)| \,d\mu \lesssim
  {|\rho|}^{1/2} {\| \partial \what{f_l}(s) \|}_{L^2} \lesssim {|\rho|}^{1/2} \e_1 2^{5p_0m} 2^{-l} 
\end{align*}
whenever $|\rho| \lesssim 2^{l-10}$.
Then, using the third bound in \eqref{bn13} it follows that
\begin{equation*}
\begin{split}
|\widehat{f_{k_1}^+}(\xi+\eta,s)\widehat{f_{k_2}^+}(\xi+\sigma,s)\widehat{f_{k_3}^-}(-\xi-\eta-\sigma,s)&-
\widehat{f_{k_1}^+}(\xi,s)\widehat{f_{k_2}^+}(\xi,s)\widehat{f_{k_3}^-}(-\xi,s)|\\
&\lesssim \varepsilon_1^32^{\overline{l}/2}\cdot 2^{-N_0k_+}2^{5p_0m}2^{-k(1+2\beta)},
\end{split}
\end{equation*}
whenever $|\eta|+|\sigma|\leq 2^{\overline{l}+4}$. In addition, using \eqref{csymbols2}, 
\begin{equation*}
|c^\ast_\xi(\eta,\sigma)-c^\ast_\xi(0,0)|\lesssim 2^{3k/2}2^{\overline{l}},
\end{equation*}
provided that  $|\eta|+|\sigma|\leq 2^{\overline{l}+4}$. Therefore
\begin{equation}\label{bn53}
\begin{split}
\Big|J'_{\overline{l},\overline{l}}(\xi,s)-&\int_{\mathbb{R}^2}e^{-is\eta\sigma/(4|\xi|^{3/2})}\widehat{f_{k_1}^+}(\xi,s)\widehat{f_{k_2}^+}(\xi,s)
\widehat{f_{k_3}^-}(-\xi,s)\varphi(2^{-\overline{l}}\eta)\varphi(2^{-\overline{l}}\sigma)c^\ast_\xi(0,0)\,d\eta d\sigma\Big|\\
&\lesssim 2^{2\overline{l}}2^{3k/2}\cdot \varepsilon_1^32^{\overline{l}/2}2^{-N_0k_+}2^{5p_0m}2^{-k(1+2\beta)}\\
&\lesssim \varepsilon_1^3(2^{\beta k}+2^{(N_1+15)k})^{-1}2^{-9m/8}.
\end{split}
\end{equation}

Starting from the general formula
\begin{equation*}
\int_{\mathbb{R}}e^{-ax^2-bx}\,dx=e^{b^2/(4a)}\sqrt{\pi}/\sqrt{a},\qquad a,b\in\mathbb{C},\,\,\Re\,a>0,
\end{equation*}
we calculate, for any $N\geq 1$,
\begin{equation*}
\int_{\mathbb{R}\times\mathbb{R}}e^{-ixy}e^{-x^2/N^2}e^{-y^2/N^2}\,dxdy=\sqrt{\pi}N\int_{\mathbb{R}}e^{-y^2/N^2}e^{-N^2y^2/4}\,dy=2\pi+O(N^{-1}).
\end{equation*}
Therefore, for $N\geq 1$,
\begin{equation*}
\int_{\mathbb{R}\times\mathbb{R}}e^{-ixy}\varphi(x/N)\varphi(y/N)\,dxdy=2\pi+O(N^{-1/2}).
\end{equation*}
Recalling also that $2^{\overline{l}}\approx |\xi|^{3/4}2^{-49m/100}$, it follows that
\begin{equation*}
\Big|\int_{\mathbb{R}^2}e^{-is\eta\sigma/(4|\xi|^{3/2})}\varphi(2^{-\overline{l}}\eta)\varphi(2^{-\overline{l}}\sigma)\,d\eta d\sigma-\frac{4|\xi|^{3/2}}{s}(2\pi)\Big|\lesssim 2^{3k/2}2^{-(1+4p_1)m}.
\end{equation*}
Therefore, using also \eqref{bn13},
\begin{equation}\label{bn54}
\begin{split}
\Big|\int_{\mathbb{R}^2}e^{-is\eta\sigma/(4|\xi|^{3/2})}&\widehat{f_{k_1}^+}(\xi,s)\widehat{f_{k_2}^+}(\xi,s)
\widehat{f_{k_3}^-}(-\xi,s)\varphi(2^{-\overline{l}}\eta)\varphi(2^{-\overline{l}}\sigma)c^\ast_\xi(0,0)\,d\eta d\sigma\\
&-\frac{8\pi|\xi|^{3/2}c^\ast_\xi(0,0)\cdot \widehat{f_{k_1}^+}(\xi,s)\widehat{f_{k_2}^+}(\xi,s)
\widehat{f_{k_3}^-}(-\xi,s)}{s}\Big|\lesssim \varepsilon_1^32^{-(1+4p_1)m}2^{-N_0k_+}.
\end{split}
\end{equation}
and the bound \eqref{bn50} follows from \eqref{bn51}, \eqref{bn53}, and \eqref{bn54}. This completes the proof of the lemma.
\end{proof}

\begin{lem}\label{bb11}
The bounds \eqref{bn20} hold provided that \eqref{ksassump} holds and, in addition,
\begin{equation}\label{bn70}
\begin{split}
&\max(|k_1-k|,|k_2-k|,|k_3-k|)\geq 21,\\
&\min(k_1,k_2,k_3)\geq -19m/20,\qquad\max(|k_1-k_3|,|k_2-k_3|)\geq 5.
\end{split}
\end{equation}
\end{lem}

\begin{proof}[Proof of Lemma \ref{bb11}] Recall the definition
\begin{equation}\label{bn71}
I_{k_1,k_2,k_3}^{++-}(\xi,s)=\int_{\mathbb{R}^2}e^{is\Phi(\xi,\eta,\sigma)}\widehat{f_{k_1}^+}(\xi+\eta,s)\widehat{f_{k_2}^+}(\xi+\sigma,s)
\widehat{f_{k_3}^-}(-\xi-\eta-\sigma,s)c^\ast_\xi(\eta,\sigma)\,d\eta d\sigma,
\end{equation}
where
\begin{equation*}
\Phi(\xi,\eta,\sigma)=\Lambda(\xi)-\Lambda(\xi+\eta)-\Lambda(\xi+\sigma)+\Lambda(\xi+\eta+\sigma).
\end{equation*}
It suffices to prove that, for any $s\in[2^m-1,2^{m+1}]$,
\begin{equation}\label{bn72}
\big|I_{k_1,k_2,k_3}^{++-}(\xi,s)\big|\lesssim 2^{-m}\varepsilon_1^32^{-2p_1m}2^{-(N_1+15)k_+}.
\end{equation}

By symmetry, we may assume that $|k_1-k_3|\geq 5$ and notice that
\begin{equation}\label{bn73}
|(\partial_\eta\Phi)(\xi,\eta,\sigma)|=|-\Lambda'(\xi+\eta)+\Lambda'(\xi+\eta+\sigma)|\gtrsim 2^{-\min(k_1,k_3)/2},
\end{equation}
provided that $|\xi+\eta|\in[2^{k_1-2},2^{k_1+2}]$, $|\xi+\eta+\sigma|\in[2^{k_3-2},2^{k_3+2}]$. As in the proof of Lemma \ref{bb2}, we integrate by parts in $\eta$ to
estimate
\begin{equation*}
|I_{k_1,k_2,k_3}^{++-}(\xi,s)|\leq |J^1_{1}(\xi,s)|+|J^2_{1}(\xi,s)|+|F_{1}(\xi,s)|+|G_{1}(\xi,s)|,
\end{equation*}
where
\begin{equation*}
\begin{split}
&J_{1}^1(\xi,s):=\int_{\mathbb{R}^2}e^{is\Phi(\xi,\eta,\sigma)}\widehat{f_{k_1}^+}(\xi+\eta,s)\widehat{f_{k_2}^+}(\xi+\sigma,s)
\widehat{f_{k_3}^-}(-\xi-\eta-\sigma,s)(\partial_\eta r_3)(\eta,\sigma)c^\ast_\xi(\eta,\sigma)\,d\eta d\sigma,\\
&J_{1}^2(\xi,s):=\int_{\mathbb{R}^2}e^{is\Phi(\xi,\eta,\sigma)}\widehat{f_{k_1}^+}(\xi+\eta,s)\widehat{f_{k_2}^+}(\xi+\sigma,s)
\widehat{f_{k_3}^-}(-\xi-\eta-\sigma,s) r_3(\eta,\sigma)(\partial_\eta c^\ast_\xi)(\eta,\sigma)\,d\eta d\sigma,\\
&F_{1}(\xi,s):=\int_{\mathbb{R}^2}e^{is\Phi(\xi,\eta,\sigma)}(\partial\widehat{f_{k_1}^+})(\xi+\eta,s)\widehat{f_{k_2}^+}(\xi+\sigma,s)
\widehat{f_{k_3}^-}(-\xi-\eta-\sigma,s)r_3(\eta,\sigma)c^\ast_\xi(\eta,\sigma)\,d\eta d\sigma,\\
&G_{1}(\xi,s):=\int_{\mathbb{R}^2}e^{is\Phi(\xi,\eta,\sigma)}\widehat{f_{k_1}^+}(\xi+\eta,s)\widehat{f_{k_2}^+}(\xi+\sigma,s)
(\partial\widehat{f_{k_3}^-})(-\xi-\eta-\sigma,s)r_3(\eta,\sigma)c^\ast_\xi(\eta,\sigma)\,d\eta d\sigma,
\end{split}
\end{equation*}
and
\begin{equation*}
 r_3(\eta,\sigma):=\frac{1}{s(\partial_\eta\Phi)(\xi,\eta,\sigma)}\cdot 
\varphi_{[k_1-1,k_1+1]}(\xi+\eta)\varphi_{[k_3-1,k_3+1]}(\xi+\eta+\sigma).
\end{equation*}
Using also \eqref{bn73}, it follows easily that
\begin{equation*}
 \|\mathcal{F}^{-1}(r_3)\|_{L^1}\lesssim 2^{-m}2^{\min(k_1,k_3)/2}, \qquad \|\mathcal{F}^{-1}(\partial_\eta r_3)\|_{L^1}\lesssim 2^{-m}2^{-\min(k_1,k_3)/2}.
\end{equation*}

We apply first \eqref{touse2} with
\begin{equation*}
\widehat{f}(\eta):=e^{-is\Lambda(\xi+\eta)}\widehat{f_{k_1}^+}(\xi+\eta,s),\,
\widehat{g}(\sigma):=e^{-is\Lambda(\xi+\sigma)}\widehat{f_{k_2}^+}(\xi+\sigma,s),\,
\widehat{h}(\theta):=e^{is\Lambda(\xi-\theta)}\widehat{f_{k_3}^-}(-\xi+\theta,s).
\end{equation*}
Using also \eqref{bn13}, \eqref{csymbols}, and \eqref{csymbols2}, we conclude that
\begin{equation*}
 |J^1_{1}(\xi,s)| +|J^2_{1}(\xi,s)|\lesssim \varepsilon_1^32^{-m/2}2^{3p_0m}2^{-(N_0-20)\max(k_1,k_2,k_3)_+}\cdot 2^{-m}2^{-\min(k_1,k_3)/2}.
\end{equation*}
Similarly, we apply \eqref{touse2} with 
\begin{equation*}
\widehat{f}(\eta):=e^{-is\Lambda(\xi+\eta)}(\partial\widehat{f_{k_1}^+})(\xi+\eta,s),\,
\widehat{g}(\sigma):=e^{-is\Lambda(\xi+\sigma)}\widehat{f_{k_2}^+}(\xi+\sigma,s),\,
\widehat{h}(\theta):=e^{is\Lambda(\xi-\theta)}\widehat{f_{k_3}^-}(-\xi+\theta,s),
\end{equation*}
and use \eqref{bn13} to conclude that
\begin{equation*}
 |F_{1}(\xi,s)|\lesssim \varepsilon_1^32^{-m/2}2^{8p_0m}2^{-k_1}2^{-(N_0/2-30)\max(k_1,k_2,k_3)_+}\cdot 2^{-m}2^{\min(k_1,k_3)/2}.
\end{equation*}
Finally, we apply \eqref{touse2} with 
\begin{equation*}
\widehat{f}(\eta):=e^{-is\Lambda(\xi+\eta)}\widehat{f_{k_1}^+}(\xi+\eta,s),\,
\widehat{g}(\sigma):=e^{-is\Lambda(\xi+\sigma)}\widehat{f_{k_2}^+}(\xi+\sigma,s),\,
\widehat{h}(\theta):=e^{is\Lambda(\xi-\theta)}(\partial\widehat{f_{k_3}^-})(-\xi+\theta,s),
\end{equation*}
and use \eqref{bn13} to conclude that
\begin{equation*}
 |G_{1}(\xi,s)|\lesssim \varepsilon_1^32^{-m/2}2^{8p_0m}2^{-k_3}2^{-(N_0/2-30)\max(k_1,k_2,k_3)_+}\cdot 2^{-m}2^{\min(k_1,k_3)/2}.
\end{equation*}
Therefore
\begin{equation*}
 |J_{1}(\xi,s)|+|F_{1}(\xi,s)|+|G_{1}(\xi,s)|\lesssim \varepsilon_1^32^{-m}2^{-(N_0/2-30)k_+}2^{-m/2+8p_0m}2^{-\min(k_1,k_3)/2},
\end{equation*}
and the desired bound \eqref{bn72} follows from the assumptions $-\min(k_1,k_3)/2\leq 19m/40$, see \eqref{bn70}, and $\max(k_1,k_2,k_3)_+\leq 3m/N_0$ (see the hypothesis of Lemma \ref{bb1}).
\end{proof}

\begin{lem}\label{bb12}
The bounds \eqref{bn20} hold provided that \eqref{ksassump} holds and, in addition,
\begin{equation}\label{bn80}
\begin{split}
&\max(|k_1-k|,|k_2-k|,|k_3-k|)\geq 21,\\
&\min(k_1,k_2,k_3)\geq -19m/20,\qquad\max(|k_1-k_3|,|k_2-k_3|)\leq 4.
\end{split}
\end{equation}
\end{lem}

\begin{proof}[Proof of Lemma \ref{bb12}] We may assume that
\begin{equation}\label{bn81}
 \min(k_1,k_2,k_3)\geq k+10,
\end{equation}
and rewrite
\begin{equation}\label{bn82}
\begin{split}
I_{k_1,k_2,k_3}^{++-}(\xi,s)=\int_{\mathbb{R}^2}e^{is\Phi(\xi,\eta,\sigma)}&\widehat{f_{k_1}^+}(\xi+\eta,s)\widehat{f_{k_2}^+}(\xi+\sigma,s)\\
&\times
\widehat{f_{k_3}^-}(-\xi-\eta-\sigma,s)\varphi_{[k_2-4,k_2+4]}(\sigma)c_\xi^\ast(\eta,\sigma)\,d\eta d\sigma,
\end{split}
\end{equation}
where, as before,
\begin{equation*}
\Phi(\xi,\eta,\sigma)=\Lambda(\xi)-\Lambda(\xi+\eta)-\Lambda(\xi+\sigma)+\Lambda(\xi+\eta+\sigma).
\end{equation*}
It suffices to prove that, for any $s\in[t_1,t_2]$,
\begin{equation}\label{bn83}
\big|I_{k_1,k_2,k_3}^{++-}(\xi,s)\big|\lesssim 2^{-m}\varepsilon_1^32^{-2p_1m}2^{-(N_1+15)k_+}.
\end{equation}

Notice that
\begin{equation}\label{bn84}
|(\partial_\eta\Phi)(\xi,\eta,\sigma)|=|-\Lambda'(\xi+\eta)+\Lambda'(\xi+\eta+\sigma)|\gtrsim 2^{-k_2/2},
\end{equation}
provided that $|\xi+\eta|\in[2^{k_1-2},2^{k_1+2}]$, $|\xi+\eta+\sigma|\in[2^{k_3-2},2^{k_3+2}]$, and $|\sigma|\approx 2^{k_2}$ (recall also that $2^{k_1}\approx 2^{k_2}\approx 2^{k_3}$). The bound \eqref{bn83} follows by integration by parts in $\eta$, as in the proof of Lemma \ref{bb11}. 
 \end{proof}
 
\begin{lem}\label{bb12.5}
The bounds \eqref{bn20} hold provided that \eqref{ksassump} holds and, in addition,
\begin{equation}\label{bn88}
\begin{split}
&\max(|k_1-k|,|k_2-k|,|k_3-k|)\geq 21,\\
&\min(k_1,k_2,k_3)\leq -19m/20,\qquad k\leq -m/5.
\end{split}
\end{equation}
\end{lem}

\begin{proof}[Proof of Lemma \ref{bb12.5}] It follows from the definition and the bounds \eqref{csymbols} and \eqref{bn13} that, for any $s\in[t_1,t_2]$,
\begin{equation*}
\begin{split}
|I_{k_1,k_2,k_3}^{++-}(\xi,s)| & \lesssim \varepsilon_1^3|\xi|^{1/2}2^{(1-\beta)\min(k_1,k_2,k_3)}2^{2p_0m}2^{-(N_0-20)\max(k_1,k_2,k_3,0)}\\
&\lesssim 2^{-m}\varepsilon_1^32^{-2p_1m}2^{-(N_1+15)k_+}.
\end{split}
\end{equation*}
The desired estimate \eqref{bn20} follows in this case.
\end{proof}
 
 \begin{lem}\label{bb13}
The bounds \eqref{bn20} hold provided that \eqref{ksassump} holds and, in addition,
\begin{equation}\label{bn90}
\begin{split}
&\max(|k_1-k|,|k_2-k|,|k_3-k|)\geq 21,\\
&\min(k_1,k_2,k_3)\leq -19m/20,\qquad k\geq -m/5.
\end{split}
\end{equation}
\end{lem}

\begin{proof}[Proof of Lemma \ref{bb13}] In this case we cannot prove pointwise bounds on $\big|I_{k_1,k_2,k_3}^{+,+,-}(\xi,s)|$ and we need to integrate by parts in $s$. For \eqref{bn20} it suffices to prove that
\begin{equation}\label{bn91}
\begin{split}
\Big|\int_{\mathbb{R}^2\times[t_1,t_2]}e^{iL(\xi,s)}e^{is\Phi(\xi,\eta,\sigma)}\widehat{f_{k_1}^+}(\xi+\eta,s)\widehat{f_{k_2}^+}(\xi+\sigma,s)
\widehat{f_{k_3}^-}(-\xi-\eta-\sigma,s)c^\ast_\xi(\eta,\sigma)\,d\eta d\sigma ds\Big|\\
\lesssim \varepsilon_1^32^{-2p_1m}2^{-(N_1+15)k_+},
\end{split}
\end{equation}
where
\begin{equation*}
\begin{split}
&\Phi(\xi,\eta,\sigma)=\Lambda(\xi)-\Lambda(\xi+\eta)-\Lambda(\xi+\sigma)+\Lambda(\xi+\eta+\sigma),\\
&L(\xi,s)=\frac{\widetilde{c}(\xi)}{4\pi^2}\int_0^s|\widehat{f}(\xi,r)|^2 \frac{dr}{r+1}.
\end{split}
\end{equation*}
The assumptions \eqref{bn90} and \eqref{ksassump} show that
\begin{equation}\label{bn91.1}
 \widetilde{k}:=\min(k,\mathrm{med}(k_1,k_2,k_3))\geq -m/5.
\end{equation}
Then we make the simple observation that
\begin{equation*}
 \Lambda(a)+\Lambda(b)-\Lambda(a+b)\geq \Lambda(a)/2\qquad\text{ if }\qquad 0\leq a\leq b
\end{equation*}
to conclude that $\Phi$ satisfies the weakly elliptic bound
\begin{equation}\label{bn91.2}
 |\Phi(\xi,\eta,\sigma)|\geq 2^{\widetilde{k}/2-10}
\end{equation}
provided that $|\xi+\eta|\in[2^{k_1-2},2^{k_1+2}]$, $|\xi+\sigma|\in[2^{k_2-2},2^{k_2+2}]$, $|\xi+\eta+\sigma|\in[2^{k_3-2},2^{k_3+2}]$. 
Letting $\dot{L}(\xi,s):=(\partial_sL)(\xi,s)$, where $L$ is defined in \eqref{bn2},
and using \eqref{bn13}, we notice that, for any $s\in[2^{m}-1,2^{m+1}]$,
\begin{equation}\label{bn93.5}
|\dot{L}(\xi,s)|\lesssim \varepsilon_1^22^{2k}2^{-N_0k_+}2^{-m}.
\end{equation}

We integrate by parts in $s$ to conclude that the integral in the left-hand side of \eqref{bn91} is dominated by
\begin{equation*}
B^0(\xi)+\sum_{j=1}^2B_j(\xi),
\end{equation*}
where
\begin{equation*}
\begin{split}
B^0(\xi):=\int_{t_1}^{t_2}\Big|\int_{\mathbb{R}^2} & e^{is\Phi(\xi,\eta,\sigma)} \frac{1}{\Phi(\xi,\eta,\sigma)}
  \frac{d}{ds} \Big[ e^{i L(\xi,s)}
\\
&\times\widehat{f_{k_1}^+}(\xi+\eta,s)\widehat{f_{k_2}^+}(\xi+\sigma,s)
\widehat{f_{k_3}^-}(-\xi-\eta-\sigma,s)c^\ast_\xi(\eta,\sigma)\Big]\,d\eta d\sigma\Big| ds
\end{split}
\end{equation*}
and, for $j=1,2$,
\begin{equation*}
\begin{split}
B_j(\xi):=\Big|\int_{\mathbb{R}^2}&e^{it_j\Phi(\xi,\eta,\sigma)}\frac{1}{\Phi(\xi,\eta,\sigma)}
\\
&\times\widehat{f_{k_1}^+}(\xi+\eta,t_j)\widehat{f_{k_2}^+}(\xi+\sigma,t_j)
\widehat{f_{k_3}^-}(-\xi-\eta-\sigma,t_j)c^\ast_\xi(\eta,\sigma)\,d\eta d\sigma\Big|
\end{split}
\end{equation*}

Let
\begin{equation*}
 r_4(\eta,\sigma):=\frac{1}{\Phi(\xi,\eta,\sigma)}\cdot 
\varphi_{[k_1-1,k_1+1]}(\xi+\eta)\varphi_{[k_2-1,k_2+1]}(\xi+\sigma)\varphi_{[k_3-1,k_3+1]}(\xi+\eta+\sigma).
\end{equation*}
Using \eqref{bn91.2}--\eqref{bn93.5} and integration by parts it is easy to see that, for any $s\in[t_1,t_2]$,
\begin{equation}\label{bn95}
\|\mathcal{F}^{-1}(r_4)\|_{L^1}\lesssim 2^{-\widetilde{k}/2}.
\end{equation}
Using the first $L^\infty$ bound in \eqref{bn13} we estimate, for $j\in\{1,2\}$,
\begin{equation}\label{bn96}
B_j(\xi)\lesssim \varepsilon_1^32^{(1-\beta)\min(k_1,k_2,k_3)}2^{-(N_0/2-20)k_+}\lesssim \varepsilon_1^32^{-3m/4}2^{-(N_0/2-20)k_+}.
\end{equation}

Expanding the $d/ds$ derivative we estimate
\begin{equation*}
\begin{split}
& B^0(\xi)\lesssim 2^m\sup_{s\in[t_1,t_2]}[B^0_0(\xi,s)+B^0_1(\xi,s)+B^0_2(\xi,s)+B^0_3(\xi,s)],
\\  
& B^0_0(\xi,s):=\int_{\mathbb{R}^2}
  \Big|\frac{\partial_s L(\xi,s)}{\Phi(\xi,\eta,\sigma)}\widehat{f_{k_1}^+}(\xi+\eta,s)\widehat{f_{k_2}^+}(\xi+\sigma,s)
\widehat{f_{k_3}^-}(-\xi-\eta-\sigma,s)c^\ast_\xi(\eta,\sigma)\Big|\,d\eta d\sigma,
\\
& B^0_1(\xi,s):=\Big|\int_{\mathbb{R}^2}e^{is\Phi(\xi,\eta,\sigma)}r_4(\eta,\sigma)(\partial_s\widehat{f_{k_1}^+})(\xi+\eta,s)\widehat{f_{k_2}^+}(\xi+\sigma,s)
\widehat{f_{k_3}^-}(-\xi-\eta-\sigma,s)c^\ast_\xi(\eta,\sigma)\,d\eta d\sigma\Big|,
\\
& B^0_2(\xi,s):=\Big|\int_{\mathbb{R}^2}e^{is\Phi(\xi,\eta,\sigma)}r_4(\eta,\sigma)\widehat{f_{k_1}^+}(\xi+\eta,s)(\partial_s\widehat{f_{k_2}^+})(\xi+\sigma,s)
\widehat{f_{k_3}^-}(-\xi-\eta-\sigma,s)c^\ast_\xi(\eta,\sigma)\,d\eta d\sigma\Big|,
\\
& B^0_3(\xi,s):=\Big|\int_{\mathbb{R}^2}e^{is\Phi(\xi,\eta,\sigma)}r_4(\eta,\sigma)\widehat{f_{k_1}^+}(\xi+\eta,s)\widehat{f_{k_2}^+}(\xi+\sigma,s)
(\partial_s\widehat{f_{k_3}^-})(-\xi-\eta-\sigma,s)c^\ast_\xi(\eta,\sigma)\,d\eta d\sigma\Big|.
\end{split}
\end{equation*}

As before, we combine \eqref{touse2}, \eqref{bn95}, and the bounds \eqref{bn13} and \eqref{touse4} to conclude that
\begin{equation}\label{bn97.1}
\begin{split}
\sup_{s\in[t_1,t_2]}[B^0_1(\xi,s)+B^0_2(\xi,s)+B^0_3(\xi,s)]\lesssim \varepsilon_1^3 2^{-5m/4} 2^{-(N_0/2-20)k_+}.
\end{split}
\end{equation}
In addition, using the definition of the function $L$, we have
\begin{equation}\label{bn97.15}
\sup_{s\in[t_1,t_2]}|\partial_s L(\xi,s)|\lesssim \varepsilon_1^2 2^{-m}.
\end{equation}
Therefore
\begin{equation}\label{bn97.2}
\sup_{s\in[t_1,t_2]}B^0_0(\xi,s)\lesssim \varepsilon_1^3 2^{-5m/4} 2^{-(N_0/2-20)k_+}.
\end{equation}
The desired bound \eqref{bn91} follows from \eqref{bn96}, \eqref{bn97.1}, and \eqref{bn97.2}.
\end{proof}

\subsection{Proof of \eqref{bn21}} After changes of variables, it suffices to prove that
\begin{equation}\label{bn141}
\begin{split}
\Big|\int_{\mathbb{R}^2\times[t_1,t_2]}&e^{iL(\xi,s)}e^{is\Phi^{\iota_1,\iota_2,\iota_3}(\xi,\eta,\sigma)}\widehat{f_{k_1}^{\iota_1}}(\xi+\eta,s)\widehat{f_{k_2}^{\iota_2}}(\xi+\sigma,s)\\
&\times\widehat{f_{k_3}^{\iota_3}}(-\xi-\eta-\sigma,s)c^{\iota_1\iota_2\iota_3}(\xi,-\eta,-\xi-\eta-\sigma)\,d\eta d\sigma ds\Big|\lesssim \varepsilon_1^32^{-2p_1m}2^{-(N_1+15)k_+},
\end{split}
\end{equation}
where $(\iota_1,\iota_2,\iota_3)\in\{(+,+,+),(-,-,+),(-,-,-)\}$ and 
\begin{equation*}
\Phi^{\iota_1,\iota_2,\iota_3}(\xi,\eta,\sigma)=\Lambda(\xi)-\iota_1\Lambda(\xi+\eta)-\iota_2\Lambda(\xi+\sigma)-\iota_3\Lambda(\xi+\eta+\sigma).
\end{equation*}
The main observation is that the phases $\Phi^{\iota_1,\iota_2,\iota_3}$ are weakly elliptic, i.e.
\begin{equation*}
|\Phi^{\iota_1,\iota_2,\iota_3}(\xi,\eta,\sigma)|\geq 2^{\mathrm{med}(k_1,k_2,k_3)/2-100},
\end{equation*}
provided that $|\xi+\eta|\in[2^{k_1-2},2^{k_1+2}]$, $|\xi+\sigma|\in[2^{k_2-2},2^{k_2+2}]$, 
$|\xi+\eta+\sigma|\in[2^{k_3-2},2^{k_3+2}]$, and $(\iota_1,\iota_2,\iota_3)\in\{(+,+,+),(-,-,+),(-,-,-)\}$. 
The proof then proceeds as in the proof of Lemma \ref{bb13}, using integration by parts in $s$.


\section{Lagrangian formulation and ``Wu's good coordinates''}\label{secL}
In this section we review the formulation of the problem used by Wu in \cite{WuAG}.
We use the same notations as in the cited paper to summarize the construction of Wu's modified Lagrangian variables.
We recall the cubic evolution equations associated to these good unknowns, some other (elliptic) relations between these variables, 
and the relative cubic energies.

Let $v$ be the fluid's velocity field, recall that we denote by $z: (t,\a) \in [0,T] \times \R  \rightarrow \C$ 
the Lagrangian map (restricted to the surface parametrized by $\alpha$),
that is the solutions of
\begin{equation}
z_t (t, \a) = v (t, z(\a,t)) \, , \quad z(0,\a) = \a + i y_0 (\a) \, .
\end{equation}
$z(t,\a) = x(t,\a) + i y(t,\a)$ is the equation of the free interface. The imaginary part of $z$,
$\Im z$, measures the height of the interface.

\subsection{The transformation $k$ and modified Lagrangian coordinates}
Define the change of coordinates $k$
\begin{equation}
\label{defk}
k (t,\a) := \bar{z} (t,\a) + \frac{1}{2} (I + \H_z) {(I + \K_z)}^{-1} (z (t,\a) - \bar{z} (t,\a) )  
\end{equation}
where $\H_\g$ denotes the Hilbert transform along the curve $\g$, see \eqref{HT},
and $\K_z = \Re \H_z$.
The explicit change of coordinates \eqref{defk} is the same as the one given by Totz and Wu in \cite[formula (2.3)]{WuNLS}.
$k$ will be shown to be a diffeomorphism on $\R$. 
%
%
Given the change of coordinates $k$ we can define the transformed Lagrangian unknowns as in \cite{WuAG}
\begin{equation*}
\wt{L} (t,\a) := ( \z_\a(t,\a) - 1, u(t,\a) , w(t,\a), \Im \z(t,\a) )
\end{equation*}
with
\begin{align}
\label{defzeta}
\z (t,\a) &:=  z (t,k^{-1}(t,\a))
\, , \quad u(t,\a) := z_t (t,k^{-1}(t,\a)) 
\, , \quad w(t,\a) := z_{tt}   (t,k^{-1}(t,\a)) \, .
\end{align}
Note that
\begin{align*}
w(t,\a) &:= z_{tt} (t,k^{-1}(t,\a)) = (\partial_t + b (t,\a) \partial_\a) u(t,\a) \, ,
\end{align*}
where $b (t,\a) = k_t (t, k^{-1}(t,\a))$.
Also, following \cite{WuAG} we define the ``good quantities'' in terms of the transformed Lagrangian unknowns
\begin{align}
\label{defchi}
\chi & := 2i(I-\H_\z) \Im \z = (I-\H_z)(z-\bar{z}) \circ k^{-1}
\\
\label{defl}
\l & := (I-\H_\z) ( \psi \circ k^{-1} ) = (I-\H_z)\psi \circ k^{-1}
\\
\label{defv}
v & := \partial_t (I-\H_z)(z-\bar{z}) \circ k^{-1} = (\partial_t + b \partial_\a )\chi \quad , \quad v_1 := (I-\H_\z) v \, .
\end{align}
In what follows we will denote $\H$ for $\H_\z$, where this creates no confusion.
We recall that $\psi$ is the trace of the velocity potential in Lagrangian coordinates:
$\psi(t,\a) = \Phi(t , z(t,\a))$. 
The quantities $\chi,\l$ and $v$ are those for which cubic equations are derived and the Energy argument is performed.
The relation between the Eulerian trace of the velocity potential $\phi$ and the surface elevation $h$ with
the Lagrangian quantities is given by the identities:
\begin{align}
\label{hEL}
h(t, \Re z(t,\a)) & = \Im z (t,\a)
\\
\label{phiEL}
\phi(t, \Re z(t,\a)) & = \psi (t,\a) \, .
\end{align}
Precomposing with $k^{-1}$ these become
\begin{align*}
h(t, \Re \z(t,\a)) = \Im \z (t,\a)
\quad , \quad \phi(t, \Re \z(t,\a)) & = \psi (t, k^{-1}(t,\a))  \, .
\end{align*}

\subsection{The cubic equations}\label{seccubiceq}
In \cite{WuAG} cubic equations are derived for the ``good quantities''  $\chi, \l$ and  $v$ given by \eqref{defchi}-\eqref{defv}.
We will not discuss the derivation of these equations here, 
but refer the reader to section 2 in \cite{WuAG} for the details.
In the cited work it is shown that for $F = \chi, \l$ or $v$
\begin{equation}
\label{cubiceq0}
{(\partial_t + b \partial_\a)}^2 F - i A \partial_\a F = G(\wt{L})
\end{equation}
where the right-hand side $G(\wt{L})$ can be thought of as a cubic expression in the variables 
\begin{equation}
\label{vectorwtL}
\wt{L} = (u,w, \Im \z, \z_\a - 1) \, ,
\end{equation}
which involves singular integrals related to the Cauchy integral and to Calder\'{o}n commutators.
The functions $b$ and $A$ are defined by
\begin{align}
\label{defb}
b (t,\a) & = k_t \left( t , k^{-1}(t,\a) \right)
\\
\label{defA}
A (t, \a) & = (a(t, \cdot) k_\a(t, \cdot) ) \circ k^{-1}(t,\a)
\end{align}
where $a(t,\a)$ is the Rayleigh-Taylor coefficient appearing in \eqref{WWL}.
As shown in \cite[Proposition 2.1]{WuAG}, both $b$ and $A-1$ are real valued and quadratically small if $\wt{L}$ is small.
More precisely:
\begin{lem}\label{lemenergy10c}
Let $b$ and $A$ be given by \eqref{defb} and \eqref{defA}. 
Under the a priori assumptions \eqref{aprioriL1} we have
\begin{align}
\label{boundA-1}
{\| A(t)-1 \|}_{H^{N_1}} + {\| b(t) \|}_{H^{N_1}} & \leq \e_1 \, ,
\end{align}
and for $0\leq k\leq N_0$
\begin{align}
\label{boundA-11}
& {\| A(t)-1 \|}_{X_k} \lesssim \e_1 {\| \wt L(t) \|}_{X_k} \, ,
\\
\label{boundbX_k}
& {\| b(t) \|}_{X_k} \lesssim \e_1 {\| \wt L(t) \|}_{X_k} \, .
\end{align}
\end{lem}

\begin{proof}
The identities in (2.30) of \cite{WuAG} (see Proposition 2.4 in \cite{WuAG} for their derivation) read
\begin{align}
\label{formulaA}
& (I-\H)(A-1) = i [u,\H] \frac{\bar{u}_\a}{\z_\a} + i [w,\H] \frac{\bar{\z}_\a - 1}{\z_\a} 
\\
\label{formulab}
& (I-\H)b = - [u,\H] \frac{\bar{\z}_\a - 1}{\z_\a} \, .
\end{align}
The above right-hand sides are all of the form $Q_0(\wt{L}, \wt{L})$ or $Q_0(\wt{L}, \partial_\a \wt{L})$, 
where $Q_0$ is defined in \eqref{Q_0}.
\eqref{formulaA} and \eqref{formulab} then follow from the bounds for operators of the type $Q_0$ given in \eqref{estQ_0L^2},
Sobolev's embedding, the estimate \eqref{estHlow} for $\H$, and the a priori  smallness assumptions on $\wt{L}$.
The bounds \eqref{boundA-11} and \eqref{boundbX_k} follow by combining \eqref{formulaA} and \eqref{formulab} with 
Lemma \ref{lemI-Hf}.
\end{proof}

\vskip10pt
The cubic equations verified by $\chi, \l$ and $v$ are (see formulas (2.27)-(2.29) and (4.33) in \cite{WuAG})
\begin{align}
\label{cubiceqchi}
&  {(\partial_t + b \partial_\a)}^2 \chi - i A \partial_\alpha \chi = G^\chi (u,\Im \z)
\\
\label{cubiceqv}
&  {(\partial_t + b \partial_\a)}^2 v - i A \partial_\alpha v = G^v (u,w,\z_\a - 1,\Im\z,\chi_\a)
\\
\label{cubiceql}
&  {(\partial_t + b \partial_\a)}^2 \l - i A \partial_\alpha \l = G^\l (u,w, \Im \z)
\end{align}
where
\begin{align}
\label{Gchi}
\begin{split}
G^\chi (u,\Im \z) & := \frac{4}{\pi} \int \frac{ (u(\a) - u(\b)) (\Im\z(\a) - \Im\z(\b)) }{ {|\z(\a)-\z(\b)|}^2 } u_\b (\b) \, d\b 
\\
& + \frac{2}{\pi} \int { \left( \frac{ u(\a) - u(\b) }{ \z(\a)-\z(\b) } \right) }^2 \Im \z_\b (\b) \, d\b \, ,
\end{split}
\end{align}
\begin{align}
\label{Gv}
\begin{split}
G^v (u,w,\z_\a - 1,\Im\z, \chi_\a) & := 
  \frac{4}{\pi} \int \frac{ (w(\a) - w(\b)) (\Im\z(\a) - \Im\z(\b)) }{ {|\z(\a)-\z(\b)|}^2 } u_\b (\b) \, d\b 
\\
& + \frac{4}{\pi} \int \frac{ (u(\a) - u(\b)) (\Im\z(\a) - \Im\z(\b)) }{ {|\z(\a)-\z(\b)|}^2 } w_\b (\b) \, d\b 
\\
& +  \frac{2}{i\pi}  \int { \left( \frac{ u(\a) - u(\b) }{ \z(\a)-\z(\b) } \right) }^2 u_\b (\b) \, d\b
- \frac{2}{i\pi}  \int \frac{ {|u(\a) - u(\b)|}^2 }{ {(\bar{\z}(\a)- \bar{\z}(\b))}^2 }  u_\b (\b) \, d\b
\\
& + \frac{2}{i\pi} \int \frac{ (u(\a) - u(\b)) (w(\a) - w(\b)) }{ {(\z(\a)-\z(\b))}^2 } (\z_\b(\b) - \bar{\z}_\b(\b)) \, d\b 
\\
& +  \frac{1}{i\pi}  \int { \left( \frac{ u(\a) - u(\b) }{ \z(\a)-\z(\b) } \right) }^2 (u_\b(\b) - \bar{u}_\b(\b)) \, d\b
\\
& -  \frac{2}{i\pi}  \int { \left( \frac{ u(\a) - u(\b) }{ \z(\a)-\z(\b) } \right) }^3 (\z_\b(\b) - \bar{\z}_\b(\b)) \, d\b
+ i \frac{a_t}{a} \circ k^{-1} A \partial_\a \chi \, ,
\end{split}
\end{align}
\begin{align}
\label{Gl}
\begin{split}
G^\l (u,w,\Im \z) & := 
 \frac{2}{\pi} \int \frac{ (u(\a) - u(\b)) (\Im\z(\a) - \Im\z(\b)) }{ {|\z(\a)-\z(\b)|}^2 } \bar{\z}_\b(\b) w(\b) \, d\b 	 
\\
& + [u,\bar{\H}] \left( \bar{u} \frac{u_\a }{\bar{\z}_\a} \right) +  u [u,\H] \frac{\bar{u}_\a }{\z_\a}
   - 2 [u,\H] \frac{u \cdot u_\a}{\z_\a}
\\
& + \frac{1}{i\pi} \int { \left( \frac{ u(\a) - u(\b) }{ \z(\a)-\z(\b) } \right) }^2 u (\b) \cdot \z_\b(\b) \, d\b \, .
\end{split}
\end{align}
A fourth additional equation is also derived in \cite[formula (4.32)]{WuAG} for the quantity
\begin{align}
\label{defv_1}
 v_1 := (I-\H_\z) v
\end{align}
and has the same form as above
\begin{align} 
\label{cubiceqv_1}
&  {(\partial_t + b \partial_\a)}^2 v_1 - i A \partial_\alpha v_1 = G^{v_1} (u,w,\z_\a - 1, \Im \z, v, \chi)
\end{align}
with
\begin{align}
\label{Gv_1}
\begin{split}
G^{v_1} \left( u, w, \z_\a - 1, \Im \z, v, \chi \right)
& : =  (I-\H) \P v - 2 [u,\H] \frac{\partial_\a}{\z_\a} \P \chi 
- 2[u,\H] \frac{\partial_\a}{\z_\a} \left( w \frac{\partial_\a}{\z_\a} \chi \right)
\\
& -i [(\H + \bar{\H})u, \H] {\left( \frac{\partial_\a}{\z_\a} \right)}^2 \chi
+  \frac{1}{i\pi}  \int { \left( \frac{ u(\a) - u(\b) }{ \z(\a)-\z(\b) } \right) }^2 v_\b (\b) \, d\b \, ,
\end{split}
\end{align}
where $\P := {(\partial_t + b \partial_\a)}^2 - i A \partial_\alpha$.
The above cubic equations allow the construction of energy functionals, see section \ref{secE} below, 
which suitably control the variables $\chi,v$ and $\l$. 
The quantities appearing on the right-hand sides of these cubic equations, such as $u,w$ and $\z_\a-1$, 
are in turned controlled by the energies, see Proposition \ref{proenergy1}, 
using the identities \eqref{u=}--\eqref{I-HImz=}.

In order to simplify our presentation and the estimates performed on the above equations, 
we define the following types of trilinear operators:
\begin{subequations}
\label{opC}
 \begin{align}
\label{C^1} 
T^1 (f,g,h) & :=  \int \frac{ (f(\a) - f(\b)) (g(\a) - g(\b)) }{ {|\z(\a)-\z(\b)|}^2 } h (\b) \, d\b 
\\
\label{C^2} 
T^2 (f,g,h) & :=  \int \frac{ (f(\a) - f(\b)) (g(\a) - g(\b)) }{ {(\z(\a)-\z(\b))}^2 } h (\b) \, d\b
\\
\label{C^3} 
T^3 (f,g,h) & :=  \int \frac{ (f(\a) - f(\b)) (g(\a) - g(\b)) }{ {(\bar{\z}(\a)-\bar{\z}(\b))}^2 } h (\b) \, d\b \, ,
\end{align}
\end{subequations}
and denote by $\bC$ any scalar multiple of them:
\begin{align}
\label{C}
\bC (f,g,h) = c_i T_i (f,g,h) \, 
\end{align}
for some constant $c_i \in \C$, $i=1,2,3$.
We can then write the nonlinearity \eqref{Gchi} appearing in the equation \eqref{cubiceqchi} as:
\begin{align}
\label{Gchi0}
 G^\chi (u,\Im \z) = \bC (u,\Im \z, u_\a) + \bC (u,u,\Im \z_\a) \, .
\end{align}
Similarly we have
\begin{align}
\begin{split}
\label{Gv0}
 G^v (u,w,\z_\a - 1,\Im \z) & = \bC (w,\Im \z, u_\a) + \bC (u,\Im \z, w_\a) + \bC (u,u, u_\a) + \bC (u,w,\Im \z_\a)
\\
& + \bC (u,u, \bar{u}_\a) + u \bC (u,u, \Im \z_\a) + \bC (u,u, u \Im \z_\a) + i \frac{a_t}{a} \circ k^{-1} A \partial_\a \chi \, , 
\end{split}
\end{align}
and
\begin{align}
\begin{split}
\label{Gl0}
G^\l (u,w, \Im\z) & = \bC (u,\Im \z, \bar{\z}_\a w) + \bC (u,u, u \cdot \z_\a)
\\
& + [u,\bar{\H}] \left( \bar{u} \frac{u_\a }{\bar{\z}_\a} \right) +  u [u,\H] \frac{\bar{u}_\a }{\z_\a}
   - 2 [u,\H] \frac{u \cdot u_\a}{\z_\a} \, ,
\end{split}
\end{align}
and
\begin{align}
\begin{split}
\label{Gv_10}
G^{v_1} \left( u, w, \z_\a - 1, \Im \z, v, \chi \right)
& : =  (I-\H) \P v - 2 [u,\H] \frac{\partial_\a}{\z_\a} \P \chi 
- 2[u,\H] \frac{\partial_\a}{\z_\a} \left( w \frac{\partial_\a}{\z_\a} \chi \right)
\\
& -i [(\H + \bar{\H})u, \H] {\left( \frac{\partial_\a}{\z_\a} \right)}^2 \chi
+ \bC (u,u, v_\a) \, .
\end{split}
\end{align}
By writing the nonlinearities \eqref{Gchi},\eqref{Gv},\eqref{Gl} and \eqref{Gv_1}
in terms of operators of the type \eqref{opC} above, we will be able to efficiently estimate them 
by making use of a general Proposition giving $L^2$-type bounds of such operators.
These estimates are given in Proposition \ref{proCCmain}, and are obtained by improving two statements contained in \cite{WuAG} 
(which in turn rely in part on the work of Coifman, McIntosh and Meyer \cite{CMM}).
All the terms that cannot be written as a trilinear operators of the form \eqref{opC}, acting on the components of $\wt{L}$,
need to be treated separately. To bound these remaining nonlinear terms, namely
\begin{align*}
 & i \frac{a_t}{a} \circ k^{-1} A \partial_\a \chi 
\quad , \quad  
[u,\bar{\H}] \left( \bar{u} \frac{u_\a }{\bar{\z}_\a} \right) +  u [u,\H] \frac{\bar{u}_\a }{\z_\a}
\quad , \quad
[u,\H] \frac{u \cdot u_\a}{\z_\a}
\\
& [(\H + \bar{\H})u, \H] {\left( \frac{\partial_\a}{\z_\a} \right)}^2 \chi
\quad , \quad
[u,\H] \frac{\partial_\a}{\z_\a} \left( w \frac{\partial_\a}{\z_\a} \chi \right)
\end{align*}
we will make use of some additional special structure present in them.

\subsection{The Energy}\label{secE}
The total energy for the system is given by the sum of three energies naturally associated 
to the equations \eqref{cubiceqchi}, \eqref{cubiceql} and \eqref{cubiceqv_1}.
Let us define 
\begin{equation}
\label{S_k}
S_k := D^{k} S \, .
\end{equation}
The first term in the energy is given by 
\begin{align}
\label{energychi}
\begin{split}
E^\chi(t) & = 
	\sum_{k=0}^{N_0}  \int_\R \frac{1}{A} {\left| (\partial_t + b \partial_\a) D^k \chi \right|}^2
	+ i {\left(D^k \chi \right)}^h \partial_\alpha \overline{{\left( D^k \chi\right)}^h} \, d\a 
\\
 & + 	\sum_{k=0}^{N_0/2} \int_\R \frac{1}{A} {\left| (\partial_t + b \partial_\a) S_k \chi \right|}^2
	+ i {\left(S_k \chi \right)}^h \partial_\alpha \overline{{\left( S_k \chi\right)}^h} \, d\a \, ,
\end{split}
\end{align}
where $f^h$ denotes the anti-holomorphic part of a function $f$:
\begin{equation}
\label{f^h}
f^h := \frac{I - \H}{2} f \, .
\end{equation}
By considering only the anti-holomorphic parts of $D^k \chi$ and $S_k \chi$ in \eqref{energychi}, 
one obtains that all summands in $E^\chi$ are non-negative, see Lemma \eqref{proevolEf}.
%
%
Similarly one constructs the energy associated to \eqref{cubiceql}:
\begin{align}
\label{energyl}
\begin{split}
E^\l(t) & = 
	\sum_{k=0}^{N_0-2}  \int_\R \frac{1}{A} {\left| (\partial_t + b \partial_\a) D^k \l \right|}^2
	+ i {\left(D^k \l \right)}^h  \partial_\alpha \overline{{\left( D^k\l \right)}^h} \, d\a 
\\
 & + \sum_{k=0}^{N_0/2} \int_\R \frac{1}{A} {\left| (\partial_t + b \partial_\a) S_k \l \right|}^2
	+ i {\left(S_k \l \right)}^h \partial_\alpha \overline{{\left( S_k \l \right)}^h} \, d\a \, ,
\end{split}
\end{align}
%
%
The energy controlling $v$ instead is given in terms of $v_1$, and based on the equation \eqref{cubiceqv_1}:
\begin{align}
\label{energyv_1}
\begin{split}
E^{v} (t) & 
	= \sum_{k=0}^{N_0} \int_\R \frac{1}{A} {\left| (\partial_t + b \partial_\a) D^k v_1 \right|}^2
	+ i D^k v_1 \partial_\alpha D^k \overline{v}_1 \, d\a
\\
& + \sum_{k=0}^{N_0/2} \int_\R \frac{1}{A} {\left| (\partial_t + b \partial_\a) S_k v_1 \right|}^2
	+ i S_k v_1 \partial_\alpha S_k \overline{v}_1 \, d\a \, .
\end{split}
\end{align}
Here in the second summand there is no restriction to $(D^k v_1)^h$ or $(S v_1)^h$.
Therefore the Energy $E^v$ has no definite sign.
Nevertheless, it can be shown that this Energy controls the norms of $v$ and $v_1$ up to cubic lower order contributions.
The total energy is then given by
\begin{equation}
\label{Etot}
E(t) = E^{\chi}(t) + E^{v}(t) + E^{\l}(t) \, .
\end{equation}


\section{Proof of Proposition \ref{proenergy}: energy estimates}\label{secproenergy}

Once the energy $E(t)$ has been defined, we can proceed with the proof of Proposition \ref{proenergy}. 
The main ideas are essentially the same as those used in Wu's paper \cite{WuAG}.
In fact, we use a similar procedure and borrow several identities and estimates from this paper.
However, some arguments there need to be adjusted in order to make the energy estimates
valid for all times, and compatible with the growth of the highest Sobolev and weighted norms. 
In particular, we need to pay special attention to certain dangerous nonlinear terms that can potentially create logarithmic losses, 
compare with Wu's energy estimate \eqref{EinWu},
and show how these losses can be avoided.

Recall that we denoted $\wt{L} = (u,w, \Im \z, \z_\a - 1)$, and define the vectors
\begin{equation*}
L^- := (\z_\a - 1, u, w, \partial_\a \chi, v) \quad , \quad L := (L^-, \Im\z) \, .
\end{equation*}
We separate the proof of Proposition \ref{proenergy} into three main steps.
We first show how $E(t)$ controls the $X_{N_0}$-norm of $\wt{L}(t)$:

\begin{pro}\label{proenergy1}
Under the a priori assumptions \eqref{aprioriL1}, that is
\begin{align}
\label{aprioriL13}
\sup_{t \in [0,T]}  \left( (1+t)^{-p_0} {\| \wt{L}(t) \|}_{X_{N_0}} +  {\| \wt{L}(t) \|}_{H^{N_1+5}} 
    + {\| \wt{L}(t) \|}_{W^{N_1,\infty}}\sqrt{1+t} \right) \leq \e_1 \ll 1 \, ,
\end{align}
 we have
\begin{equation}
\label{estproenergy1}
 {\| \wt{L} (t) \|}_{X_{N_0}} \lesssim \sqrt{E(t)}  \, ,
\end{equation}
for any $t\in[0,T]$ and $\e_1$ sufficiently small.
\end{pro}

This is proven in section \ref{secproenergy1}. Our next proposition shows how to estimate the energy increment.

\begin{pro}\label{proenergy2}
Assume again that \eqref{aprioriL13} holds for $\e_1$ small enough. Then
\begin{align}
\label{estproenergy2}
\frac{d}{dt} \sqrt{E(t)} \lesssim
  \left(  {\| L (t) \|}_{W^{N_1,\infty}}  +  {\| \H L^- (t) \|}_{W^{N_1,\infty}}  \right)^2  \sqrt{E(t)} \, ,
\end{align}
for any $t\in[0,T]$.
\end{pro}
The above Proposition will be proven in section \ref{secproenergy2}.
We eventually need to bound the $L^\infty$-norms on the right hand side of \eqref{estproenergy2} in terms of the $Z^\p$-norm of $h$ and $\phi$:
\begin{pro}\label{proenergy3}
Under the a priori assumption \eqref{aprioriL13} and \eqref{apriori0} we have
\begin{align}
\label{estproenergy3}
{\| L (t) \|}_{W^{N_1,\infty}}  +  {\| \H L^- (t) \|}_{W^{N_1,\infty}} 
  \lesssim  {\| (h,\phi) \|}_{Z^\p} = {\| (h, \Lambda \phi) \|}_{W^{N_1+4,\infty}} \, .
\end{align}
\end{pro}

The proof of this Proposition is in section \ref{secproenergy3}.

In view of the initial assumptions \eqref{initdataa}--\eqref{initdatab} and the discussion in \cite[sec 5.1]{WuAG},
one has $E(0) \lesssim \e_0^2$.
It is then clear that Propositions \ref{proenergy1}--\ref{proenergy3} imply Proposition \ref{proenergy}.

\subsection{Proof of Proposition \ref{proenergy1}: Energy bounds}\label{secproenergy1}
The estimate \eqref{estproenergy1} will be achieved through a sequence of Lemmas which we state below 
and prove in the remaining of this section.
We start by using some formulae derived in \cite{WuAG} to bound ${\| \wt L(t) \|}_{X_{N_0}}$
by the $X_{N_0}$-norms of $(\partial_t + b \partial_\a) \chi$, $(\partial_t + b \partial_\a) v$ and $(\partial_t + b \partial_\a) \l$:

\begin{lem}\label{lemenergy11}
Under the assumption \eqref{aprioriL13} it is possible to write
\begin{align}
\label{utochi}
& 2 u = (\partial_t + b \partial_\a) \chi + Q
\\
\label{wtov}
& 2 w = (\partial_t + b \partial_\a) v + Q
\\
\label{z_a-1tov}
& 2 (\z_\a - 1) = -i (\partial_t + b \partial_\a) v + Q
\\
\label{Imztol}
& (I-\H) \Im\z = - (\partial_t + b \partial_\a) \l + Q
\end{align}
where $Q$ denotes a generic term which is at least quadratic in $\wt{L}$ and satisfies
\begin{align}
\label{estQ}
& {\| Q(t) \|}_{X_{N_0}} \lesssim  \e_1 {\| \wt L(t) \|}_{X_{N_0}}
\\
\label{estQinfty}
& {\| Q(t) \|}_{W^{N_1,\infty}} \lesssim  \e_1 {\| \wt L(t) \|}_{W^{N_1,\infty}}
\end{align}
for $\e_1$ small enough. In particular we see that
\begin{align}
\label{estuwz_a-1}
{\| \wt L(t) \|}_{X_{N_0}} \sim {\| (\partial_t + b\partial_\a) \chi \|}_{X_{N_0}} + {\| (\partial_t + b\partial_\a) v \|}_{X_{N_0}}
  + {\| (\partial_t + b\partial_\a) \l \|}_{X_{N_0}}
\end{align}
and
\begin{align}
\label{estuwz_a-1infty}
{\| \wt{L}(t) \|}_{W^{N_1,\infty}}  \lesssim 
  {\| (\partial_t + b\partial_\a) \chi \|}_{W^{N_1,\infty}} + {\| (\partial_t + b\partial_\a) v \|}_{W^{N_1,\infty}}  
  +  {\| (\partial_t + b\partial_\a) \l \|}_{W^{N_1,\infty}} \, .
\end{align}
\end{lem}
The proof of this Lemma is given in section \ref{prlemenergy11}.
We then establish the following commutators estimates:
\begin{lem}\label{lemenergy12}
Assume again \eqref{aprioriL13} holds for $\e_1$ small enough.
Then, for $Q$ as in \eqref{estQ} above, we have
\begin{align}
\label{chi_atoz_a}
& \partial_\a \chi = (I-\H) (\z_\a - \bar{\z}_\a) + Q
\\
\label{vtou}
& v = 2u +  Q
\\
\label{l_atou}
& \partial_\a \l = u + Q \, .
\end{align}
%
%
Moreover, denoting $D$ for $\partial_\a$, we have for all $0\leq k \leq N_0$ and $f = \chi,v$ or $\l$
\begin{align}
\label{commD}
& {\| (\partial_t + b\partial_\a) D^k f -  D^k (\partial_t + b\partial_\a) f \|}_{L^2} \lesssim  \e_1 {\| \wt L(t) \|}_{X_{N_0}} \, ,
\end{align}
and for all $0\leq k \leq \frac{N_0}{2}$
\begin{align}
\label{commS}
& {\| (\partial_t + b\partial_\a) D^k S f -  D^k S (\partial_t + b\partial_\a) f \|}_{L^2} \lesssim  \e_1 {\| \wt L(t) \|}_{X_{N_0}}
  + {\| (\partial_t + b\partial_\a) f \|}_{H^k} \, .
\end{align}
\end{lem}

The above Lemma is proven in section \ref{prlemenergy12}. Eventually we show how $E(t)$ controls 
$\chi$ and $v$:
\begin{lem}\label{lemenergy13}
Under the assumption \eqref{aprioriL13}, we have
\begin{align}
\label{D_tchi}
& {\| (\partial_t + b\partial_\a) D^k \chi \|}_{L^2}^2 \lesssim E(t) \, ,
\\
\label{D_tv}
& {\| (\partial_t + b\partial_\a) D^k v \|}_{L^2}^2 \lesssim E(t)  + \e_1 {\| \wt L(t) \|}_{H^{N_0}}^2 
\end{align}
for all $0 \leq k \leq N_0$. Also
\begin{align}
\label{D_tl}
& {\| (\partial_t + b \partial_\a) D^k \l \|}^2_{L^2} \lesssim  E(t) \, ,
\end{align}
for all $0 \leq k \leq N_0-2$.
Moreover, for $0 \leq k \leq \frac{N_0}{2}$
\begin{align}
\label{D_tSchi}
& {\| (\partial_t + b\partial_\a) D^k S \chi \|}_{L^2}^2 \lesssim E(t) \, ,
\\
\label{D_tSv}
& {\| (\partial_t + b\partial_\a) D^k S v \|}_{L^2}^2  \lesssim E(t)  + \e_1 {\| \wt L(t) \|}_{X_{N_0}}^2 \, ,
\\
\label{D_tSl}
& {\| (\partial_t + b \partial_\a) D^k S \l \|}^2_{L^2} \lesssim  E(t) \, .
\end{align}
\end{lem}

This Lemma is proven in \ref{prlemenergy13}. One can actually obtain the stronger bound
\begin{align*}
{\| (\partial_t + b\partial_\a) \G v \|}_{L^2}^2 + {\| (\partial_t + b\partial_\a) \G v_1 \|}_{L^2}^2
     \lesssim E(t)  + \e_1 {\| \wt L(t) \|}_{X_{N_0}}^2 \, .
\end{align*}

The bound \eqref{estproenergy1}, and hence Proposition \ref{proenergy1}, follow easily from the above three lemmas.

\subsubsection{Proof of Lemma \ref{lemenergy11}}\label{prlemenergy11}
Equations (2.35), (2.43), (3.38) and (2.51) in \cite{WuAG} respectively read
\begin{align}
\label{u=}
& 2 u  = (\partial_t + b \partial_\a) \chi + (\H + \bar{\H}) u + [u,\H] \frac{\z_\a - \bar{\z}_\a}{\z_\a}
\\
\nn
& 2 w  = (\partial_t + b \partial_\a) v  +  [\bar{u}, \bar{\H}] \frac{u_\a}{\bar{\z}_\a} +
			[u, \H] \frac{2 u_\a - \bar{u}_\a}{\z_\a} +  (\H + \bar{\H}) w
\\ 
\label{w=}
&  \qquad + [w,\H] \frac{\z_\a - \bar{\z}_\a}{\z_\a} 
	- \frac{1}{i \pi} \int {\left( \frac{u(\a) - u(\b)}{\z(\a) - \z(b)} \right)}^2 (\z_\b(\b) - \bar{\z}_\b(\b)) \, d\b \, 
\\
\label{z_a-1=}
& \z_\a - 1  = \frac{w}{i A} - \frac{A - 1}{A} 
\\
\label{I-HImz=}
& (I-\H) \Im \z  = - (\partial_t + b \partial_\a ) \l  -\frac{1}{2}  [u,\H] \frac{\bar{\z}_\a u}{\z_\a} \, .
\end{align}

\paragraph{{\it Proof of \eqref{utochi}}}
To show \eqref{utochi} it is enough to prove that $(\H + \bar{\H}) u$ and $[u,\H] \frac{\z_\a - \bar{\z}_\a}{\z_\a}$
are quadratic terms satisfying \eqref{estQ}.
Estimate \eqref{estH+barH} and the a priori assumption \eqref{aprioriL13} give
\begin{align*}
{\| (\H + \bar{\H}) u \|}_{X_{N_0}} & \lesssim {\| \z_\a - 1 \|}_{H^{\frac{N_0}{2} + 1}} {\| u \|}_{X_{N_0}} 
  +  {\| \z_\a - 1 \|}_{X_{N_0}} {\| u \|}_{H^{\frac{N_0}{2} + 1}}
\lesssim \e_1 {\| \wt{L} \|}_{X_{N_0}} \, .
\end{align*}
Similarly, under the a priori assumptions \eqref{aprioriL13}, estimate \eqref{estQ_0L^2b} implies
\begin{align*}
{\left\| [u,\H] \frac{\z_\a - \bar{\z}_\a}{\z_\a} \right\|}_{X_{N_0}}
  \leq \e_1 {\| u \|}_{X_{N_0}} + \e_1 {\| \Im \z \|}_{X_{N_0}} + \e_1 {\| \z_\a - 1 \|}_{X_{N_0}}   
  \lesssim \e_1 {\| \wt{L} \|}_{X_{N_0}} \, .
\end{align*}

\paragraph{{\it Proof of \eqref{wtov}}}
We need to estimate all of the terms in the difference $2w - (\partial_t + b \partial_\a) v$ from \eqref{w=}.
The terms
\begin{align*}
[w, \H] \frac{\z_\a - \bar{\z}_\a}{\z_\a} \qquad \mbox{and} \qquad (\H + \bar{\H} )w
\end{align*}
can be estimated as above, the only difference being the appearance of $w$ instead of $u$.
The last term in \eqref{w=} is of the form $\bC (u,u, \Im \z_\a)$, where $\bC$ is defined in \eqref{opC}. 
This can be directly estimated using \eqref{estCCmain0}.
The remaining terms are
\begin{align*}
I_1 := [\bar{u}, \bar{\H}] \frac{u_\a}{\bar{\z}_\a} \qquad \mbox{and} \qquad I_2 := [u, \H] \frac{2 u_\a - \bar{u}_\a}{\z_\a} \, .
\end{align*}
These are terms of the form $Q_0 (\wt{L}, \partial_\a \wt{L})$ and can be bounded by making use of \eqref{estQ_0L^2}:
\begin{align*}
{\left\| I_1 \right\|}_{X_{N_0}} + {\left\| I_2 \right\|}_{X_{N_0}} & \lesssim {\left\| Q_0 (u,\partial_\a u) \right\|}_{X_{N_0}} 
  \lesssim \e_1 {\| \wt{L} \|}_{X_{N_0}} \, .
\end{align*}

\paragraph{{\it Proof of \eqref{z_a-1tov}}}
We start from \eqref{z_a-1=} and rewrite as:
\begin{align}
\label{z_a-1=1}
\z_\a - 1  & = -iw + i w \left(1 - \frac{1}{A} \right) - \frac{A - 1}{A}  \, . 
\end{align}
Since $w$ satisfies \eqref{wtov}, in order to show \eqref{z_a-1tov} is suffices to verify that
\begin{align*}
{\left\| w \left(1 - \frac{1}{A} \right) \right\|}_{X_{N_0}} + {\left\| \frac{A - 1}{A} \right\|}_{X_{N_0}} 
    \lesssim \e_1 {\| \wt{L} \|}_{X_{N_0}} \, .
\end{align*}
This follows from the bounds on $A-1$ \eqref{boundA-1} and \eqref{boundA-11}, and the a priori assumptions \eqref{aprioriL13}.

\paragraph{{\it Proof of \eqref{Imztol}}} 
This follows directly from \eqref{I-HImz=} and arguments similar to the ones above.

\paragraph{{\it Proof of \eqref{estuwz_a-1} and \eqref{estuwz_a-1infty}}}
Using \eqref{est1-Hf} in combination with \eqref{Imztol} one can deduce that
\begin{align*}
{\| \Im \z \|}_{X_{N_0}} & \lesssim {\| (\partial_t + b\partial_\a) \l \|}_{X_{N_0}} + \e_1 {\| \wt{L} \|}_{X_{N_0}} \, .
\end{align*}
In view of \eqref{utochi}-\eqref{estQ} 
we have then obtained \eqref{estuwz_a-1}.

From the a priori bound on the $H^{N_1 + 5}$-norm of $\wt{L}$ it is not hard to see that \eqref{estQinfty} holds true
by using similar arguments as above. 
\eqref{estuwz_a-1infty} immediately follows, concluding the proof of Lemma \ref{lemenergy11}. $\hfill \Box$

\subsubsection{Proof of Lemma \ref{lemenergy12}}\label{prlemenergy12}
Identities (2.50), (2.35) and (2.36) in \cite{WuAG} respectively read:
\begin{align}
\nn
\frac{i}{2} \partial_\a \chi  & = w \bar{\z}_\a - \frac{1}{2}  \H \left( u_\a \frac{u \bar{\z}_\a}{\z_\a} \right) + 
  \frac{1}{2} [u,\H] \left( \frac{\partial_\a (u \bar{\z}_\a)}{\z_\a} \right)
\\
\nn 
& - \frac{1}{2\pi i} \int \frac{ (u(\a) - u(\b)) (\z_\a(\a) - \z_\b(\b)) }{ {(\z(\a) - \z(\b))}^2 } u(\b) \bar{\z}_\b(\b) \, d\b 
  + \frac{1}{2}(\z_\a - \bar{\z}_\a) w +\frac{1}{2} \bar{u}_\a u
\\ 
\label{chi_a=}
& -\frac{1}{2} \z_\a \left( \H \frac{1}{\z_\a} + \bar{\H} \frac{1}{\bar{\z}_\a} \right) (w\bar{\z}_\a + u\bar{u}_\a)
  + \frac{\z_\a}{\pi} \int \Im \left( \frac{ u(\a) - u(\b) }{ {(\z(\a) - \z(\b))}^2 } \right) u(\b) \bar{\z}_\b(\b) \, d\b \, , 
\\
\label{v=}
v & = 2u - (\H + \bar{\H}) u - [u, \H] \frac{\z_\a - \bar{\z}_\a}{\z_\a}  \, ,
\\
\label{l_a=}
\partial_\a \l & =  u \bar{\z}_\a - \frac{1}{2} \left( \z_\a \H \frac{1}{\z_\a} + \bar{\z}_\a \bar{\H} \frac{1}{\bar{\z}_\a} \right) (u \bar{\z}_\a) \, .
\end{align}
From these, Proposition \ref{proCCmain} and the a priori assumption \eqref{aprioriL13}, it is not hard to see that 
\begin{align}
\label{chivl1}
& {\| \partial_\a \chi \|}_{H^{N_1}} + {\| v \|}_{H^{N_1}} + {\|  \partial_\a \l \|}_{H^{N_1}}
  \lesssim \e_1 \, ,
\\
\label{chivl2}
& {\| \partial_\a \chi \|}_{X_{N_0}} + {\| v \|}_{X_{N_0}} + {\|  \partial_\a \l \|}_{X_{N_0}}
  \lesssim {\| \wt{L} \|}_{X_{N_0}} \, .
\end{align}
We are now going to use these bounds to control the commutators $[\partial_t + b\partial_\a, D^k] f$ and $[\partial_t + b\partial_\a, S] f$ 
for $f = \chi,v,\l $.
By direct computation one sees that
\begin{align}
\label{commD_tD^k}
& [D^k, \partial_t + b\partial_\a]  = [ D^k, b\partial_\a] = \sum_{j=1}^k \partial^j b \partial^{k-j} \partial_\a
\\
\label{commD_tS}
& [S, \partial_t + b\partial_\a] = \left( Sb -\frac{1}{2}b \right) \partial_\a - \frac{1}{2}(\partial_t + b\partial_\a) \, . 
\end{align}
Moreover, for $S_k = D^k S$, we have
\begin{align}
\label{commD_tS_k}
& [S_k, \partial_t + b\partial_\a] f = D^k [S, \partial_t + b\partial_\a] f + [D^k, \partial_t + b\partial_\a] S f \, . 
\end{align}
From \eqref{commD_tD^k} it follows that for any $0\leq k \leq N_0$
\begin{align*}
{\left\| D^k (\partial_t + b\partial_\a) f - (\partial_t + b\partial_\a) D^k f \right\|}_{L^2} & \lesssim 
  \sum_{j=1}^{N_0/2} {\| \partial^j b \|}_{L^\infty} {\| \partial_\a f \|}_{H^{k-1}}
  + {\| \partial_\a b \|}_{H^{k-1}} \sum_{j=1}^{N_0/2} {\| \partial_\a f \|}_{L^\infty} 
\\
& \lesssim {\| \partial_\a b \|}_{H^{\frac{N_0}{2}+1}} {\| \partial_\a f \|}_{H^{k-2}}
  + {\| \partial_\a b \|}_{H^{k-1}} {\| \partial_\a f \|}_{H^{\frac{N_0}{2}+1}}  \, .
\end{align*}
Using \eqref{chivl1} and \eqref{chivl2} we see that
for $f = \chi,v$ and $0\leq k \leq N_0$, or $f=\l$ and $0\leq k \leq N_0 - 2$, we have
\begin{align*}
{\left\| D^k (\partial_t + b\partial_\a) f - (\partial_t + b\partial_\a) D^k f \right\|}_{L^2} 
& \lesssim  {\| \partial_\a b \|}_{H^{\frac{N_0}{2}+1}} {\| \wt{L} \|}_{X_{N_0}} +  \e_1 {\| \partial_\a b \|}_{H^{N_0-1}} \, .
\end{align*}
This would conclude the proof of \eqref{commD} provided one has
\begin{align*}
& {\| \partial_\a b \|}_{H^{\frac{N_0}{2}+1}} \lesssim \e_1
\qquad \mbox{and} \qquad {\| \partial_\a b \|}_{H^{N_0-1}} \lesssim {\| \wt{L} \|}_{X_{N_0}} \, .
\end{align*}
These two estimates follow from Lemma \ref{lemenergy10c}.

To show the commutator estimate \eqref{commS} we start by using \eqref{commD_tS_k}:
\begin{align}
\label{S_kD_t1}
& {\| [S_k, \partial_t + b\partial_\a] f \|}_{L^2} \lesssim 
  {\| [D^k, \partial_t + b\partial_\a] S f \|}_{L^2} + {\| [S, \partial_t + b\partial_\a] f \|}_{H^k} \, . 
\end{align}
From \eqref{commD_tD^k}, and for $0\leq k \leq \frac{N_0}{2}$,  we get
\begin{align*} 
{\| [D^k, \partial_t + b\partial_\a] S f \|}_{L^2} \lesssim {\| \partial_\a b \|}_{H^{\frac{k}{2}+1}} {\| \partial_\a S f \|}_{H^{k-1}}
  + {\| \partial_\a b \|}_{H^{k-1}} {\| \partial_\a S f \|}_{H^{\frac{k}{2}+1}}  \, . 
\end{align*}
Using \eqref{chivl2} we can bound the above right-hand side to obtain
\begin{align} 
\label{D_kD_tS}
&1 {\| [D^k, \partial_t + b\partial_\a] S f \|}_{L^2} \lesssim {\| \partial_\a b \|}_{H^{\frac{N_0}{2}}} {\| \wt{L} \|}_{X_{N_0}} 
\end{align}
for $f = \chi$, $v$ or $\l$.
Moreover, from \eqref{commD_tS} we see that
\begin{align*}
{\| [S, \partial_t + b\partial_\a] f \|}_{H^k} \lesssim 
  {\| S b \partial_\a f \|}_{H^k} + {\| b \partial_\a f  \|}_{H^k} + {\| (\partial_t + b \partial_\a)f \|}_{H^k}
  \, . 
\end{align*}
With $f = \chi$, $v$ or $\l$, and using \eqref{chivl1} and \eqref{chivl2} we deduce 
\begin{align}
\label{D^kSD_t}
{\| [S, \partial_t + b\partial_\a] f \|}_{H^k} \lesssim  \e_1 \left( {\| S b \|}_{H^k} + {\| b \|}_{H^k} \right)
  + {\| (\partial_t + b \partial_\a)f \|}_{H^k} \, . 
\end{align}

Putting together \eqref{S_kD_t1} with \eqref{D_kD_tS} and \eqref{D^kSD_t} we get 
\begin{align*}
& {\| [S_k, \partial_t + b\partial_\a] f \|}_{L^2} \lesssim {\| \partial_\a b \|}_{H^{\frac{N_0}{2}}} {\| \wt{L} \|}_{X_{N_0}}
 + \e_1 \left( {\| S b \|}_{H^{\frac{N_0}{2}}} + {\| b \|}_{H^{\frac{N_0}{2}}} \right) + {\| (\partial_t + b \partial_\a)f \|}_{H^k} \, . 
\end{align*}
To obtain \eqref{commS} it then suffices to have
\begin{align*}
& {\| S b \|}_{H^{\frac{N_0}{2}}} \lesssim  {\| \wt{L} \|}_{X_{N_0}}
  \qquad \mbox{and} \qquad {\| b \|}_{H^{\frac{N_0}{2}+1}} \lesssim \e_1 \, .
\end{align*}
These two estimates are again direct consequence of Lemma \ref{lemenergy10c}. $\hfill \Box$

\subsubsection{Proof of Lemma \ref{lemenergy13}}\label{prlemenergy13}
Recall the definition of $E^\chi$ given in \eqref{energychi}:
\begin{align}
\label{energychi1a}
E^\chi(t) & = \sum_{k=0}^{N_0}  \int_\R \frac{1}{A} {\left| (\partial_t + b \partial_\a) D^k \chi \right|}^2
  + i {\left(D^k \chi \right)}^h \partial_\alpha  \overline{{\left( D^k\chi \right)}^h} \, d\a 
\\
\label{energychi1b}
& + \sum_{k=0}^{N_0/2} \int_\R \frac{1}{A} {\left| (\partial_t + b \partial_\a) S_k \chi \right|}^2
  + i {\left(S_k \chi \right)}^h \partial_\alpha \overline{ {\left( S_k \chi \right)}^h} \, d\a \, \,
\end{align}
where
$ f^h := (I-\H)f/2$.
As in Lemma 4.1 of \cite{WuAG} (see Lemma \ref{proevolEf} below) 
we know that if $\Theta$ is the boundary value of an holomorphic function in $\Omega(t)^c$, such as $f^h$, then
\begin{align*}
i \int \Theta \partial_\a \bar{\Theta} \, d\a \geq 0 \, .
\end{align*}
Therefore both summands in \eqref{energychi1a} and \eqref{energychi1b} are nonnegative.
In particular
\begin{align}
E^\chi(t) \gtrsim \sum_{k=0}^{N_0}  \int_\R \frac{1}{A} {\left| (\partial_t + b \partial_\a) D^k \chi \right|}^2
  + \, \sum_{k=0}^{N_0/2} \int_\R \frac{1}{A} {\left| (\partial_t + b \partial_\a) S_k \chi \right|}^2 \, .
\end{align}
From this, and since $| A-1 | \leq 1/2$, see \eqref{boundA-11}, we get the desired bounds \eqref{D_tchi} and \eqref{D_tSchi}.
The exact same argument can be used to show \eqref{D_tl} and \eqref{D_tSl}.

To prove \eqref{D_tv} and \eqref{D_tSv} the argument is more complicated since $E^v$ is not nonnegative,
but we just need to adapt the proof of Lemma 4.2 in \cite[pp. 89-92]{WuAG} to see that
\begin{align*}
E^v (t) \geq \frac{1}{4} {\| (\partial_t + b \partial_\a) \G v \|}^2_{L^2} + \frac{1}{8} {\| (\partial_t + b \partial_\a) \G v_1 \|}^2_{L^2}  
	- \e_1 E^\chi - \e_1 {\| \wt{L} \|}^2_{X_{N_0}}
\end{align*}
where $\G = D^k$ for $0 \leq k \leq N_0$, or $\G= D^k S$ for $0\leq k \leq \frac{N_0}{2}$.  $\hfill \Box$

\subsection{Proof of Proposition \ref{proenergy2}: Evolution of the Energy}\label{secproenergy2}
We want to show, under the a priori assumptions \eqref{aprioriL13}, the following bound for the evolution of the Energy
\begin{equation}
\label{EE}
\frac{d}{dt} \sqrt{E(t)} \lesssim
    {\left( {\| L(t) \|}_{W^{N_1,\infty}} + {\| \H L^-(t) \|}_{W^{N_1,\infty}} \right)}^2  \sqrt{E(t)} \, 
\end{equation}
for any $t\in[0,T]$.
From the definition of $E$ in \eqref{Etot}, and the bound \eqref{estproenergy1}, we see that it suffices to prove
\begin{equation}
\label{EEf}
\frac{d}{dt} E^f(t) \lesssim 
    {\left( {\| L(t) \|}_{W^{N_1,\infty}} + {\| \H L^-(t)\|}_{W^{N_1,\infty}} \right)}^2 \left( {\| \wt{L}(t) \|}_{X_{N_0}}^2 + E(t) \right) \, 
\end{equation}
for any $t\in[0,T]$ and $f = \chi,\l$ and $v$.

\subsubsection{Basic energy equality}\label{secproenergy2pre}
Assume that $F$ is a smooth function vanishing sufficiently fast at infinity and satisfying the equation 
\begin{align}
\label{PF=}
\P F := {(\partial_t + b(t,\a) \partial_\a)}^2 F(t,\a) - i A(t,\a) \partial_\a F(t,\a) = G(t,\a) \, .
\end{align}
Define the zero-th energy associated to \eqref{PF=} by
\begin{align}
\label{E_0^F}
E_0^F (t) & = \int_\R \frac{1}{A(t,\a)} {\left| (\partial_t + b(t,\a) \partial_\a) F(t,\a) \right|}^2
  + i F(t,\a) \partial_\alpha \bar{F}(t,\a)  \, d\a \, .
\end{align}
Also define higher order energies
\begin{align}
\label{E_j^F}
E_j^F (t) & = \int_\R \frac{1}{A(t,\a)} {\left| (\partial_t + b(t,\a) \partial_\a) \G^j F(t,\a) \right|}^2
	+ i {(\G^j F)}^h (t,\a) \partial_\alpha \bar{{(\G^j F)}^h} (t,\a)  \, d\a \, .
\end{align}
where $j\geq 1$, $\G^j=D^j$ for $j \leq N_0$, or $\G^j=S_j$ for $j \leq N_0/2$,
and  $f^h$ is defined in \eqref{f^h}.
Then the following holds:

\begin{lem}[\cite{WuAG}]\label{proevolEf}
Assume that $F$ and $E_0$ are as above, then
\begin{align}
\label{d_tE_0^F}
\frac{d}{dt} E_0^F (t) & = \int_\R \frac{2}{A} \Re \left( (\partial_t + b \partial_\a) F \, \bar{G} \right)
	- \frac{1}{A}  \frac{a_t}{a} \circ k^{-1}    {\left| (\partial_t + b \partial_\a) F \right|}^2 \, d\a \, .
\end{align}

Furthermore, if $\Theta$ is the boundary value of an holomorphic function in $\Omega_t^c$, that is $\Theta=\Theta^h$, then
\begin{align*}
i \int_\R \Theta(t,\a) \partial_\alpha \bar{\Theta}(t,\a) \, d\a \geq 0 \, .
\end{align*}

Let $E_j^F (t)$ be as in \eqref{E_j^F}, then
\begin{align}
\label{d_tE_j^F1}
\frac{d}{dt} E_j^F (t) & = \int_\R \frac{2}{A} \Re \left( (\partial_t + b \partial_\a) \G^j F  \bar{G_j} \right)
	- \frac{1}{A}  \frac{a_t}{a} \circ k^{-1}    {\left| (\partial_t + b \partial_\a) \G^j F \right|}^2 \, d\a 
\\
	& - 2\Re \int_\R i\partial_t {(\G^j F)}^h \partial_\a \bar{{(\G^j F)}^r} + i \partial_t {(\G^j F)}^r \partial_\a \bar{{(\G^j F)}^h} 
		+  \partial_t {(\G^j F)}^r \partial_\a \bar{{(\G^j F)}^r} \, d\a \, ,
\label{d_tE_j^F2}
\end{align}
where
\begin{align*}
G_j := \G^j G + [\P,\G^j]F
\end{align*}
and $f^r$ is defined by 
$f^r := (1/2)(I+\H)f$.

\end{lem}
The proof of the above Lemma can be found in \cite[pp. 83-85]{WuAG}.

\subsubsection{Evolution of $E^\chi$}\label{secevolchi}
We want to show
\begin{equation}
\label{evolchi}
\frac{d}{dt} E^\chi(t) \lesssim  {\left( {\| L \|}_{W^{N_1,\infty}} + {\| \H L^- \|}_{W^{N_1,\infty}} \right)}^2  E(t) \, .
\end{equation}
We recall that the energy $E^\chi$ is given by
\begin{align*}
E^\chi (t) & = 
	\sum_{k=0}^{N_0}  \int_\R \frac{1}{A} {\left| (\partial_t + b \partial_\a) D^k \chi \right|}^2
	+ i {\left( D^k \chi \right)}^h \partial_\alpha  \overline{{\left(D^k \chi \right)}^h} \, d\a
\\
& + \sum_{k=0}^{N_0/2} \int_\R \frac{1}{A} {\left| (\partial_t + b \partial_\a) S_k \chi \right|}^2
	+ i {\left( S_k \chi \right)}^h \partial_\alpha \overline{{\left( S_k \chi \right)}^h} \, d\a \, .
\end{align*}
From \eqref{cubiceqchi} and \eqref{Gchi0} we know that $\chi$ satisfies an equation of the form
$\P \chi = G^\chi$ with a cubic nonlinearity $G^\chi = \bC (u,\Im \z, u_\a) + \bC (u,u,\Im \z_\a)$,
where $\bC$ are operators of the type defined in \eqref{C}.
This nonlinearity can be schematically rewritten as
\begin{align}
\label{Gchi01}
G^\chi = \bC (\wt{L}, \wt{L}, \wt{L}_\a) \, .
\end{align}

By using  Lemma \ref{proevolEf} one obtains
\begin{align}
\label{dtEchi1}
\frac{d}{dt} E^\chi(t) & = \sum_{k=0}^{N_0} \int \frac{2}{A} \Re \left( (\partial_t + b \partial_\a) D^k \chi \, \P D^k \chi \right) 
  -  \frac{1}{A} \frac{a_t}{a} \circ k^{-1} {\left|  (\partial_t + b \partial_\a) D^k \chi \right|}^2  \, d\a
  \\
\label{dtEchi2}
& - 2 \sum_{k=0}^{N_0} \Re \int  i\partial_t (D^k \chi)^h \partial_\a \bar{(D^k \chi)^r} 
  +  i\partial_t (D^k \chi)^r \partial_\a \bar{(D^k \chi)^h}  
  + i \partial_t (D^k \chi)^r \partial_\a \bar{(D^k \chi)^r} \, d\a
\\
\label{dtEchi3}
& + \sum_{k=0}^{N_0/2} \int \frac{2}{A} \Re \left( (\partial_t + b \partial_\a) S_k \chi \, \P S_k \chi \right) 
  -  \frac{1}{A} \frac{a_t}{a} \circ k^{-1} {\left|  (\partial_t + b \partial_\a) S_k \chi \right|}^2  \, d\a 
\\ \label{dtEchi4}
& - 2 \sum_{k=0}^{N_0/2} \Re \int  i\partial_t (S_k \chi)^h \partial_\a \bar{(S_k \chi)^r} 
  +  i\partial_t (S_k \chi)^r \partial_\a \bar{(S_k \chi)^h}  
  + i \partial_t (S_k \chi)^r \partial_\a \bar{(S_k \chi)^r} \, d\a \, .
\end{align}
Since ${\| A -1 \|}_{L^\infty} \leq \frac{1}{2}$, we have
\begin{align*}
\eqref{dtEchi1} + \eqref{dtEchi3} & \lesssim 
\sqrt{E^\chi (t)} \sum_{k=0}^{N_0} \left( {\|  D^k G^\chi \|}_{L^2} +  {\|  [\P, D^k] \chi \|}_{L^2}  \right)
\\
& + \sqrt{E^\chi (t)} \sum_{k=0}^{N_0/2} \left( {\| S_k G^\chi \|}_{L^2} +  {\|  [\P, S_k] \chi \|}_{L^2}  \right)
  +  E^\chi (t)  { \left\| \frac{a_t}{a} \circ k^{-1}  \right\| }_{L^\infty} \, .
\end{align*}
The terms in \eqref{dtEchi2} and \eqref{dtEchi4} are remainder terms. 
Since there is no logarithmic loss in the estimates for those terms in \cite[p. 94-98]{WuAG},
they can be estimated exactly as in the cited paper, and therefore we skip them. 
Thus, to obtain the desired bound \eqref{evolchi} on the evolution of $E^\chi$ it suffices to show
\begin{align}
\label{estevolchi1}
& {\| G^\chi \|}_{X_{N_0}} \lesssim {\| L \|}^2_{W^{N_1+3,\infty}}  \sqrt{E}
\\
\label{estevolchi2}
& \sum_{k=0}^{N_0} {\|  [\P, D^k] \chi \|}_{L^2} 
  \lesssim {\left( {\| L \|}_{W^{N_1,\infty}} + {\| \H L^- \|}_{W^{N_1,\infty}} \right)}^2  \sqrt{E}
\\
\label{estevolchi3}
& \sum_{k=0}^{N_0/2} {\|  [\P, S_k] \chi \|}_{L^2} 
  \lesssim {\left( {\| L \|}_{W^{N_1,\infty}} + {\| \H L^- \|}_{W^{N_1,\infty}} \right)}^2  \sqrt{E}
\\
\label{estevolchi4}
& {\left\| \frac{a_t}{a} \circ k^{-1}  \right\|}_{L^\infty} \lesssim {\| \wt{L} \|}^2_{W^{N_1,\infty}} \, .
\end{align}
These estimates are performed in the next four subsections.

\vskip5pt
\paragraph{\it{Proof of \eqref{estevolchi1}}}
The bound \eqref{estevolchi1} follows directly from \eqref{Gchi01} and Proposition \ref{proCCmain}.
Indeed 
applying \eqref{estCCmain0} we see that
\begin{align*}
{\| G^\chi \|}_{X_{N_0}} =  {\| \bC (\wt{L}, \wt{L}, \wt{L}_\a) \|}_{X_{N_0}}
  \lesssim {\| \wt{L} \|}^2_{W^{N_1,\infty}} {\| \wt{L} \|}_{X_{N_0}} \lesssim {\| L \|}^2_{W^{N_1,\infty}} \sqrt{E} \, ,
\end{align*}
having used \eqref{estproenergy1} in the last inequality.

\vskip5pt
\paragraph{\it{Proof of \eqref{estevolchi2}}}\label{secevolchi2}
Recall 
the definition $\P = {(\partial_t + b \partial_\a)}^2 - i A \partial_\a$.
By direct computation we see that
\begin{align}
\label{commD^kP}
\begin{split}
& [D^k, \P ] f =  [D^k, \partial_t + b\partial_\a ] (\partial_t + b\partial_\a )f  + 
   (\partial_t + b\partial_\a ) [D^k, \partial_t + b\partial_\a ] f - i [D^k, A \partial_\a] f
\\
& = [D^k, b\partial_\a ] (\partial_t + b\partial_\a )f  + 
   (\partial_t + b\partial_\a ) [D^k, b\partial_\a ] f - i [D^k, A \partial_\a] f
\\
& = \sum_{j=1}^k c_{k,j} D^j b D^{k-j} \partial_\a (\partial_t + b\partial_\a )f
  + (\partial_t + b\partial_\a ) \sum_{j=1}^k c_{k,j} D^j b D^{k-j} \partial_\a f
  - i \sum_{j=1}^k c_{k,j}  D^j A D^{k-j} \partial_\a  f \, ,
\end{split}
\end{align}
for some coefficients $c_{k,j}$.
It follows that for any $0 \leq k \leq N_0$
\begin{subequations}
\begin{align}
\nn
& {\| [\P, D^k] \chi \|}_{L^2}  \lesssim
\\
\nn
\\
\label{CE1}
& \lesssim  {\| (\partial_t + b\partial_\a) \chi \|}_{H^{N_0}} {\| \partial_\a b \|}_{W^{N_0/2,\infty}}
 + {\| \partial_\a b \|}_{H^{N_0-1}} {\| (\partial_t + b\partial_\a) \chi \|}_{W^{N_0/2,\infty}}
\\
\label{CE2}
& + \sum_{j=0}^{N_0-1} {\| (\partial_t + b\partial_\a) D^j \partial_\a b \|}_{L^2} {\| \partial_\a \chi \|}_{W^{N_0/2,\infty}}
  + {\| \partial_\a b \|}_{H^{N_0-1}}  \sum_{j=0}^{N_0/2}  {\| (\partial_t + b\partial_\a)  D^j \partial_\a \chi \|}_{L^\infty}
\\
\label{CE3}
& + \sum_{j=0}^{N_0/2} {\| (\partial_t + b\partial_\a) D^j \partial_\a b \|}_{L^\infty} {\| \partial_\a \chi \|}_{H^{N_0-1}}
  + {\| \partial_\a b \|}_{W^{N_0/2,\infty}}  \sum_{j=0}^{N_0-1}  {\| (\partial_t + b\partial_\a)  D^j \partial_\a \chi \|}_{L^2}
\\
\label{CE4}
& + {\| \partial_\a \chi \|}_{H^{N_0}} {\| \partial_\a A \|}_{W^{N_0/2,\infty}} 
  + {\| \partial_\a A \|}_{H^{N_0}} {\| \partial_\a \chi \|}_{W^{N_0/2,\infty}} \, .
\end{align}
\end{subequations}

Combining \eqref{estuwz_a-1}, commutator estimates for $[\partial_t + b\partial_\a, D^j]$,  \eqref{chi_atoz_a}
and \eqref{estproenergy1}, it follows that
\begin{align}
\label{CEL2}
{\| (\partial_t + b \partial_\a) \chi \|}_{H^{N_0}} +
  \sum_{j=0}^{N_0-1} {\| (\partial_t + b \partial_\a) D^j \partial_\a \chi \|}_{L^2} + 
  {\| \partial_\a \chi \|}_{H^{N_0}} \lesssim \sqrt{E(t)} \, . 
\end{align}
From Lemma \ref{lemenergy11}, and commutation estimates for $[\partial_t + b\partial_\a, D^j]$, we also control the following $L^\infty$ norms:
\begin{align}
\label{CELinfty}
{\| (\partial_t + b \partial_\a) \chi \|}_{W^{N_0/2,\infty}} + {\| \partial_\a \chi \|}_{W^{N_0/2,\infty}} +
  \sum_{j=0}^{N_0/2} {\| (\partial_t + b \partial_\a) D^j \partial_\a \chi \|}_{L^\infty}
  \lesssim  {\| L \|}_{W^{N_1,\infty}}  \, .
\end{align}

\paragraph{{\it Estimate of \eqref{CE1}}}
Using \eqref{CEL2} and \eqref{CELinfty} we can bound
\begin{align*}
\eqref{CE1} \lesssim \sqrt{E(t)} {\| \partial_\a b \|}_{W^{N_0/2,\infty}}
 + {\| \partial_\a b \|}_{H^{N_0-1}} {\| L \|}_{W^{N_1,\infty}} \, .
\end{align*}
To obtain the desired bound in then suffices to show
\begin{align}
\label{dbLinfty} 
& {\| \partial_\a b \|}_{W^{N_0/2,\infty}} \lesssim \left( {\| \H L^- \|}_{W^{N_1,\infty}} + {\| L \|}_{W^{N_1,\infty}} \right)^2
\\
\label{dbL2} 
& {\| \partial_\a b \|}_{H^{N_0-1}} \lesssim {\| L \|}_{W^{N_1,\infty}} \sqrt{E} \, .
\end{align}
From formula \eqref{formulab} we see that $(I-\H)b = g$ with
\begin{equation}
\label{I-Hb=}
g = - [u,\H] \frac{\bar{\z}_\a - 1}{\z_\a} \, .
\end{equation}
Using \eqref{est1-Hfinfty2} in Lemma \ref{lemI-Hf} we get
\begin{align*}
{\| \partial_\a b \|}_{W^{N_0/2,\infty}} & \lesssim {\| \partial_\a g \|}_{W^{N_0/2,\infty}} 
  + {\| \wt{L} \|}_{W^{N_0/2+1,\infty}} {\| \partial_\a g \|}_{H^{N_0/2+1}} \, .
\end{align*}
Since $g$ above is an operator of the form $Q_0(\wt{L} ,\wt{L})$, with $Q_0$ defined in \eqref{Q_0}, 
we can use the $L^\infty$ bound provided by \eqref{estpartialQ_0} to deduce
\begin{align}
\label{dbLinfty1} 
{\| \partial_\a b \|}_{W^{N_0/2,\infty}} & \lesssim {\| L \|}_{W^{N_0/2+2,\infty}} 
  {\left( {\| \H L^- \|}_{W^{N_0/2+2,\infty}} + {\| L \|}_{W^{N_0/2+2,\infty}} \right)}  \, .
\end{align}
Here we have also used the $L^2$-bounds from \ref{theoCCL^2} to estimate ${\| \partial_\a g \|}_{H^{N_0/2+1}}$, 
and the a priori assumption ${\| \wt{L} \|}_{H^{N_1+5}} \lesssim \e_1$. \eqref{dbLinfty} is proven.
Using again Lemma \ref{lemI-Hf} we can estimate
\begin{align}
\label{dbL21} 
& {\| \partial_\a b \|}_{H^{N_0-1}} \lesssim  {\| \partial_\a g \|}_{H^{N_0-1}} 
  \lesssim {\| \wt{L} \|}_{H^{N_0}} {\| \wt{L} \|}_{W^{N_0/2+2,\infty}} \, ,
\end{align}
which in light of \eqref{estproenergy1} suffices to obtain \eqref{dbL2}.

\paragraph{{\it Estimate of \eqref{CE2}}}
Using \eqref{CEL2} and \eqref{CELinfty} we see that
\begin{align*}
\eqref{CE2} \lesssim \sum_{j=0}^{N_0-1} {\| (\partial_t + b\partial_\a) D^j \partial_\a b \|}_{L^2} {\| L \|}_{W^{N_0/2+2,\infty}}
  + {\| \partial_\a b \|}_{H^{N_0-1}} {\| L \|}_{W^{N_0/2+2,\infty}} \, .
\end{align*}
Since we already have the bound \eqref{dbL2} for ${\| \partial_\a b \|}_{H^{N_0-1}}$,  in order to estimate \eqref{CE2} it suffices to show
\begin{align}
\label{D_tdbL2} 
\sum_{j=0}^{N_0-1} {\| (\partial_t + b\partial_\a) D^j \partial_\a b \|}_{L^2}  
  \lesssim {\| \wt{L} \|}_{W^{N_0/2+2,\infty}}  {\| \wt{L} \|}_{X_{N_0}} \, .
\end{align}
To establish this estimate we use the following identity derived in \cite[formula (2.52)]{WuAG}:
\begin{equation}
 \label{formulaD_tb}
\begin{split}
& (I - \H) (\partial_t + b\partial_\a )b = B_3
\\
& B_3 := [u,\H] \frac{\partial_\a( 2b - \bar{u}) }{\z_\a} - [w,\H] \frac{\bar{\z}_\a - 1}{\z_\a}
+ \frac{1}{i\pi}  \int {\left( \frac{u(\a) - u(\b)}{\z(\a) - \z(\b)} \right)}^2 (\bar{\z}_\b(\b) - 1) \, d\b \, .
\end{split}
\end{equation}
Using this formula, \eqref{est1-Hf0} in Lemma \ref{lemI-Hf}, and the fact that $[\partial_\a, \partial_t + b\partial_\a]b$ 
gives higher order quartic terms, it is not hard to see that \eqref{D_tdbL2} holds


\paragraph{{\it Estimate of \eqref{CE3}}}
Using again \eqref{CEL2} and \eqref{CELinfty} we can bound
\begin{align*}
\eqref{CE3} \lesssim \sum_{j=0}^{N_0/2} {\| (\partial_t + b\partial_\a) D^j \partial_\a b \|}_{L^\infty} \sqrt{E(t)}
  + {\| \partial_\a b \|}_{W^{N_0/2,\infty}} \sqrt{E(t)} \, .
\end{align*}
By virtue of \eqref{dbLinfty} 
it suffices to establish the bound
\begin{align}
\label{D_tdbLinfty} 
\sum_{j=0}^{N_0/2} {\| (\partial_t + b\partial_\a) D^j \partial_\a b \|}_{L^\infty}  
  \lesssim {\left( {\| \H L^- \|}_{W^{N_0/2+2,\infty}} + {\| L \|}_{W^{N_0/2+2,\infty}} \right)}^2 \, .
\end{align}
Commuting $\partial_t + b\partial_\a$ and $\partial_\a^j$, using \eqref{formulaD_tb} and \eqref{est1-Hfinfty2} in Lemma \ref{lemI-Hf},
and \eqref{dbLinfty1},
one can see that for any $0 \leq j \leq N_0/2$
\begin{align*}
{\| (\partial_t + b\partial_\a) D^j \partial_\a b \|}_{L^\infty} 
  & \lesssim {\| \partial_\a (\partial_t + b\partial_\a) b \|}_{W^{j,\infty}} + {\| L \|}_{W^{N_0/2+2,\infty}}^2
\\
& \lesssim {\| \partial_\a B_3 \|}_{W^{N_0/2,\infty}} + {\| L \|}_{W^{N_0/2+2,\infty}} {\| \partial_\a B_3 \|}_{H^{N_0/2+1}} 
  + {\| L \|}_{W^{N_0/2+2,\infty}}^2 \, .
\end{align*}
Since ${\| \partial_\a B_3 \|}_{H^{N_0/2+1}}$ can be estimated by means of Corollary \ref{theoCCL^2},
we only need to bound ${\| \partial_\a B_3 \|}_{W^{N_0/2,\infty}}$.
$B_3$ contains potentially dangerous terms, but since they are of the form $Q_0(L^-, L^-)$
we can use again \eqref{estpartialQ_0} in Proposition \ref{proCCmain} to obtain
\begin{align*}
{\| \partial_\a B_3 \|}_{W^{N_0/2,\infty}} & \lesssim 
  {\left( {\| \H L^- \|}_{W^{N_0/2+2,\infty}} + {\| L \|}_{W^{N_0/2+2,\infty}} \right)}^2 \, .
\end{align*}
This gives \eqref{D_tdbLinfty}, which in turn allows to bound \eqref{CE3} by the right-hand side of \eqref{estevolchi2} as desired.

\paragraph{{\it Estimate of \eqref{CE4}}}
From \eqref{CEL2} and \eqref{CELinfty} we see that
\begin{align*}
\eqref{CE4} & \lesssim \sqrt{E(t)} {\| \partial_\a A \|}_{W^{N_0/2,\infty}} 
  + {\| \partial_\a A \|}_{H^{N_0}} {\| L \|}_{W^{N_0/2+2,\infty}} \, .
\end{align*}
To obtain the desired bound it is sufficient to show the following two estimates:
\begin{align} 
\label{dAL2}
& {\| \partial_\a A \|}_{H^{N_0}} \lesssim \sqrt{E(t)} {\| L \|}_{W^{N_0/2+2,\infty}}
\\
\label{dALinfty} 
& {\| \partial_\a (A-1) \|}_{W^{N_0/2,\infty}}
  \lesssim {\left( {\| \H L^- \|}_{W^{N_0/2+2,\infty}} + {\| L \|}_{W^{N_0/2+2,\infty}} \right)}^2 \, .
\end{align}
Recalling the identity \eqref{formulaA} for $(I-\H)(A-1)$
we see that the two terms in the right hand side of that formula 
are of the same type of the one appearing in the formula \eqref{I-Hb=} for $(I-\H)b$. 
Therefore, in order to show \eqref{dAL2} and \eqref{dALinfty}, one can proceed in the exact same fashion as was done 
before to obtain \eqref{dbLinfty} and \eqref{dbL2}.
Also in this case the presence of a derivative acting on $A-1$ in \eqref{dALinfty} plays a crucial role,
allowing us to use the bound \eqref{estpartialQ_0} on operators of the type $\partial_\a Q_0$.

\vskip5pt
\paragraph{\it{Proof of \eqref{estevolchi3}}}
Since $S_k = D^k S$ for any $0 \leq k \leq N_0/2$, we can write
\begin{align*}
[\P, S_k] f = [\P, D^k] S + D^k [\P, S] \, .
\end{align*}
Thus, to prove \eqref{estevolchi3} it is enough to show
\begin{align}
\label{estevolchi31} 
{\| [\P, D^k] S \chi\|}_{L^2} & 
  \lesssim {\left( {\| L \|}_{W^{N_1,\infty}} + {\| \H L^- \|}_{W^{N_1,\infty}} \right)}^2  \sqrt{E}
\\
\label{estevolchi32}
{\| [\P, S] \chi \|}_{H^k} & 
  \lesssim {\left( {\| L \|}_{W^{N_1,\infty}} + {\| \H L^- \|}_{W^{N_1,\infty}} \right)}^2  \sqrt{E}
\end{align}
for any $0 \leq k \leq N_0/2$.
Recall that the commutation $[\P, D^k]$ is explicitely given in \eqref{commD^kP}, whereas
a direct computation shows that
\begin{equation}
 \label{commS2}
\begin{split}
 [\P,S] & = \P + \left\{ \left( S b - \frac{1}{2}b \right) b_\a - S(\partial_t+b\partial_\a)b  \right\} \partial_\a
\\ 
& - \left( S b - \frac{1}{2}b \right) \left\{ (\partial_t+b\partial_\a)\partial_\a + \partial_\a (\partial_t+b\partial_\a)  \right\}
  + i S A \partial_\a \, . 
\end{split}
\end{equation}

\vskip5pt
\paragraph{\it Proof of \eqref{estevolchi31}}
For $0 \leq k \leq N_0/2$, we first use \eqref{commD^kP} to obtain
\begin{align}
\label{estevolchi311}
\begin{split}
{\| [\P, D^k] S \chi \|}_{L^2}
& \lesssim {\| \partial_\a b \|}_{W^{k-1,\infty}} {\| \partial_\a (\partial_t + b\partial_\a) S \chi \|}_{H^{k-1}} 
+ \sum_{j=1}^{k} {\| (\partial_t + b\partial_\a) D^j b \|}_{L^\infty} {\| \partial_\a S \chi \|}_{H^{k-1}}
\\
& + {\| \partial_\a b \|}_{W^{k-1,\infty}} \sum_{j=1}^{k} {\| (\partial_t + b\partial_\a) D^j S \chi \|}_{L^2}
+ {\| \partial_\a A \|}_{W^{k-1,\infty}} {\| \partial_\a S \chi \|}_{H^{k-1}} \, .
\end{split}
\end{align}
Commuting $\partial_t + b\partial_\a$ with $\partial_\a$ and $S$ in the appropriate fashion,
using \eqref{D_tSchi} to control the $L^2$ norm of $\partial_\a (\partial_t + b\partial_\a) S_k \chi$, 
and \eqref{chi_atoz_a} to control ${\| \partial_\a S \chi \|}_{H^{k-1}}$,
we see that all of the $L^2$-based norms in \eqref{estevolchi311} 
are controlled by $\sqrt{E}$.
To obtain \eqref{estevolchi31} We are then left with proving that for $0 \leq k \leq N_0/2$
\begin{align}
\label{estevolchi321}
& {\| \partial_\a b \|}_{W^{k-1,\infty}} 
  \lesssim {\left( {\| L \|}_{W^{N_1,\infty}} + {\| \H L^- \|}_{W^{N_1,\infty}} \right)}^2
\\
\label{estevolchi322}
& {\| (\partial_t + b\partial_\a) \partial_\a b \|}_{W^{k-1,\infty}} 
  \lesssim {\left( {\| L \|}_{W^{N_1,\infty}} + {\| \H L^- \|}_{W^{N_1,\infty}} \right)}^2
\\
\label{estevolchi323}
& {\| \partial_\a A \|}_{W^{k-1,\infty}} 
  \lesssim {\left( {\| L \|}_{W^{N_1,\infty}} + {\| \H L^- \|}_{W^{N_1,\infty}} \right)}^2 \, .
\end{align}
The bound \eqref{estevolchi321} is implied by \eqref{dbLinfty1} which has been already proven.
Up to commuting $\partial_t + b\partial_\a$ and $\partial_\a^j$, for $0\leq j \leq k-1$, we see that \eqref{estevolchi322} 
would follow from obtaining the same bound for  $(\partial_t + b\partial_\a) \partial_\a^j  \partial_\a b$.
Such an estimate has been already obtained in \eqref{D_tdbLinfty}.
Since also \eqref{estevolchi323} has been shown to hold true before, see \eqref{dALinfty},
we have completed the proof of \eqref{estevolchi31}.

\vskip5pt
\paragraph{\it Proof of \eqref{estevolchi32}}
Using \eqref{commS2}, for any $0 \leq k \leq N_0/2$, we can estimate
\begin{align*}
{\| [\P,S] \chi \|}_{H^k} & \lesssim {\| \P \chi \|}_{H^k} 
  + \left( {\left\| S b - \frac{1}{2}b \right\|}_{H^k} {\|b_\a\|}_{H^{k+1}} +  {\| S(\partial_t+b\partial_\a)b  \|}_{H^k} \right) 
  {\| \partial_\a \chi \|}_{W^{\frac{N_0}{2}, \infty}}
\\ 
& + {\left\| S b - \frac{1}{2}b \right\|}_{H^k} {\| \partial_\a (\partial_t + b\partial_\a) \chi \|}_{W^{\frac{N_0}{2}, \infty}}
  + {\| S (A-1) \|}_{H^k} {\| \partial_\a \chi \|}_{W^{\frac{N_0}{2}, \infty}} \, .
\end{align*}
From \eqref{estuwz_a-1infty} we know that
\begin{align*}
 {\| \partial_\a \chi \|}_{W^{\frac{N_0}{2}, \infty}} +  {\| \partial_\a (\partial_t+b\partial_\a) \chi \|}_{W^{\frac{N_0}{2}, \infty}}
  \lesssim {\| L \|}_{W^{\frac{N_0}{2}, \infty}} \, .
\end{align*}
From Lemma \ref{lemenergy10c} we see that ${\| b \|}_{H^{k+1}} \lesssim 1$. 
To conclude the desired bound it then suffices to show the following $L^2$-estimates:
\begin{align}
\label{estevolchi32a}
& {\| b \|}_{H^k} + {\| S b \|}_{H^k} \lesssim \left( {\| L \|}_{W^{N_1,\infty}} + {\| \H L^- \|}_{W^{N_1,\infty}} \right)  \sqrt{E}
\\ 
\label{estevolchi32b}
& {\| S(\partial_t+b\partial_\a) b \|}_{H^k}  \lesssim 
  \left( {\| L \|}_{W^{N_1,\infty}} + {\| \H L^- \|}_{W^{N_1,\infty}} \right)  \sqrt{E}
\\
\label{estevolchi32c}
& {\| S (A-1) \|}_{H^k} \lesssim \left( {\| L \|}_{W^{N_1,\infty}} + {\| \H L^- \|}_{W^{N_1,\infty}} \right)  \sqrt{E}
\end{align}
for any $0 \leq k \leq N_0/2$.

The $L^2$ estimates above can all be proven in the same fashion, so we just give details for the first one.
As before, from \eqref{formulab} we know that $(I-\H) b = Q_0(\wt{L},\wt{L})$, where $Q_0$ is as in\eqref{Q_0}.
From \eqref{est1-Hf} we then have
\begin{align*}
{\| b \|}_{H^k} + {\| S b \|}_{H^k} & \lesssim {\| Q_0(\wt{L},\wt{L}) \|}_{X_k} 
  + {\| \z_\a -1 \|}_{X_k} \big( {\| Q_0(\wt{L},\wt{L}) \|}_{W^{\frac{k}{2},\infty}}  
  + {\| \Im \z_\a \|}_{W^{k,\infty}} {\| Q_0(\wt{L},\wt{L}) \|}_{H^{\frac{k}{2}+1}} \big) \, .
\end{align*}
Using the estimate \eqref{estQ_0L^2}, and the energy bounds \eqref{estproenergy1}, we see that
${\| Q_0(\wt{L},\wt{L}) \|}_{X_k}$ is bounded by the right-hand side of \eqref{estevolchi32a}.
Moreover 
we can use \eqref{estHlow}, and the a priori assumptions, to deduce that
\begin{align*}
{\| Q_0(\wt{L},\wt{L}) \|}_{H^{\frac{k}{2}+1}} \lesssim  
  {\| L \|}_{W^{\frac{N_0}{2} + 2,\infty}} {\| L \|}_{H^{\frac{N_0}{2} + 3}} \lesssim {\| L \|}_{W^{\frac{N_0}{2} + 2,\infty}} \, .
\end{align*}
This, and ${\| \z_\a -1 \|}_{X_k} \lesssim \sqrt{E}$, suffice to obtain \eqref{estevolchi32a}.
One can easily see that \eqref{estevolchi32b} and \eqref{estevolchi32c} follow analogously, 
by using  respectively the identities \eqref{formulaD_tb} and \eqref{formulaA}.

\vskip5pt
\paragraph{\it{Proof of \eqref{estevolchi4}}}\label{secevolchi4}
In order to complete the Energy estimate for $E^\chi$ we want to prove the $L^\infty$ bound
\begin{align*}
{ \left\|  \frac{a_t}{a} \circ k^{-1}  \right\| }_{L^\infty} \lesssim  {\| \wt{L} \|}_{W^{N_1,\infty}}^2 \, .
\end{align*}
In what follows we are going to establish the stronger bound
\begin{align}
\label{a_t1}
{ \left\|  \frac{a_t}{a} \circ k^{-1}  \right\| }_{W^{\frac{N_0}{2},\infty}} \lesssim  {\| \wt{L} \|}_{W^{N_1,\infty}}^2
\end{align}
and the additional estimate
\begin{align}
\label{a_t2}
{ \left\|  \frac{a_t}{a} \circ k^{-1}  \right\| }_{X_{N_0}} \lesssim  {\| \wt{L} \|}_{W^{N_1,\infty}} {\| \wt{L} \|}_{X_{N_0}} \, .
\end{align}
The above bounds will also be useful later on.
To prove \eqref{a_t1}-\eqref{a_t2} we will use the formula (2.32) from \cite{WuAG}, which reads
\begin{align}
\label{formulaa_t}
\begin{split}
(I - \H) & \left( \frac{a_t}{a} \circ k^{-1} A \bar{\z}_\a \right) = I_1 + I_2
\\
I_1 & := 2i [w, \H] \frac{\bar{u}_\a}{\z_\a} + 2i [u,\H] \frac{\bar{w}_\a}{\z_\a}
\\
I_2 & := - \frac{1}{\pi} \int { \left( \frac{ u(\a) - u(\b) }{ \z(\a)-\z(\b) } \right) }^2 \bar{u}_\b (\b) \, d\b \, ,
\end{split}
\end{align}
in combination with the following Lemma:
\begin{lem}\label{lemI-HfA}
Let $f \in X_k$, $0 \leq k \leq N_0$, let $A$ be as in \eqref{formulaA}, and $w$ as defined in \eqref{defzeta}.
Assume that $f$ and $g$ are related by
\begin{align*}
(I - \H) (f A \bar{\z}_\a) = g \, .
\end{align*}
Then, for any $0 \leq k \leq N_0$
\begin{align}
\label{est1-HfA}
{\| f \|}_{X_k} \lesssim  {\| g \|}_{X_k} +  \left( {\| w \|}_{X_k} + {\| \z_\a - 1 \|}_{X_k} \right)
    \left( {\| g \|}_{W^{\frac{k}{2},\infty}} + {\| \Im \z_\a \|}_{W^{\frac{k}{2}+1,\infty}} {\| (g,w) \|}_{H^{\frac{k}{2}+1,\infty}} \right) \, ,
\end{align}
and for $0 \leq k \leq \frac{N_0}{2} + 2$
\begin{align}
\label{est1-HfAinfty}
{\| f \|}_{W^{k,\infty}} \lesssim  {\| \Re \, g \|}_{W^{k,\infty}}  +  {\| \Im \z_\a \|}_{W^{k+1,\infty}}  {\| g \|}_{H^{k + 1}} 
  +  {\| g \|}_{H^{k+1}} {\| w \|}_{W^{k+1,\infty}} \, .
\end{align}
\end{lem}

\proof From the identity \eqref{id30} we see that $A \bar{\z_\a} = 1 - i w$,  
and therefore
\begin{align*}
(I -\H )f = g + (I-\H) (i f w) \, .
\end{align*}
\eqref{est1-HfA} can then be derived by applying \eqref{est1-Hf} to the above identity, using \eqref{estHb} and the a priori bounds on $w$.
The estimate \eqref{est1-HfAinfty} follows similarly from \eqref{est1-Hfinfty}, together with \eqref{est1-Hflow} and \eqref{estHlow}. 
$\hfill \Box$

Notice that $I_2$ in \eqref{formulaa_t} is of the form $\bC(u,u,\bar{u}_\a)$,
and therefore is easier to estimate, so that we can skip its treatment and focus on $I_1$. 
From Lemma \ref{lemI-HfA} we see that in order to prove \eqref{a_t1} it suffices to obtain the bounds
\begin{align}
\label{a_t11}
& {\| \Re I_1 \|}_{W^{\frac{N_0}{2}+1,\infty}} \lesssim {\| \wt{L} \|}_{W^{N_1,\infty}}^2 
\\
\label{a_t12}
& {\| I_1 \|}_{H^{\frac{N_0}{2}+1}} \lesssim {\| \wt{L} \|}_{W^{N_1,\infty}} \, .
\end{align}
For \eqref{a_t2} it is enough to prove
\begin{align}
\label{a_t13}
{\| I_1 \|}_{X_{N_0}} & \lesssim  {\| \wt{L} \|}_{W^{N_1,\infty}} {\| \wt{L} \|}_{X_{N_0}} \, .
\end{align}

To show \eqref{a_t11} we write explicitely $\Re I_1$ as follows:
\begin{align}
\label{a_t20}
\begin{split}
\Re I_1 & = \Re \left( \frac{2}{\pi} \int \frac{ w(\a) - w(\b) }{ \z(\a)-\z(\b) } \bar{u}_\b (\b) \, d\b 
  + \frac{2}{\pi} \int \frac{ u(\a) - u(\b) }{ \z(\a)-\z(\b) } \bar{w}_\b (\b) \, d\b \right)
\\
& =  \Re \left( w 2i \H \frac{\bar{u}_\a}{\z_\a} + u  2i \H \frac{\bar{w}_\a}{\z_\a} \right)
  - \Re \left( \frac{2}{\pi} \int \frac{w(\b) \bar{u}_\b (\b) + u(\b) \bar{w}_\b (\b)}{ \z(\a)-\z(\b) } \, d\b \right) \, .
\end{split}
\end{align}
The first contribution above is estimated using \eqref{estHdfL^infty}:
\begin{align}
{\left\| w 2i \H \frac{\bar{u}_\a}{\z_\a} + u  2i \H \frac{\bar{w}_\a}{\z_\a}  \right\|}_{W^{\frac{N_0}{2}+1,\infty}}
\lesssim {\| w \|}_{W^{N_1,\infty}} {\| u \|}_{W^{N_1,\infty}}  \, .
\end{align}
To bound the second summand in \eqref{a_t20} we use the identity
\begin{align*}
2 \Re \int \frac{f(\b)}{\z(\a) - \z(\b) }\, d\b = \int \frac{2 \Re f(\b)}{\z(\a) - \z(\b) }\, d\b
 + 2i \int \frac{\bar{f}(\b) (\Im \z(\a) - \Im \z(\b)) }{{|\z(\a) - \z(\b)|}^2}\, d\b
\end{align*}
and notice that
\begin{align*}
2 \Re \left( w \bar{u}_\b + u \bar{w}_\b \right) =  \partial_\b (w \bar{u} + u \bar{w} ) \, .
\end{align*}
It follows that
\begin{align}
\Re \left( \frac{2}{\pi} \int \frac{w(\b) \bar{u}_\b (\b) + u(\b) \bar{w}_\b (\b)}{ \z(\a)-\z(\b) } \, d\b \right)
= 2 i \H \frac{ \partial_\a (w \bar{u} + u \bar{w} ) }{\z_\a} + I_r
\end{align}
where $I_r$ is cubic remainder which can be easily estimated.
Since
\begin{align}
 {\left\| \H \frac{ \partial_\a (w \bar{u} + \bar{u} w ) }{\z_\a} \right\|}_{W^{\frac{N_0}{2}+1,\infty}}
\lesssim {\| w \|}_{W^{\frac{N_0}{2}+3,\infty}} {\| u \|}_{W^{\frac{N_0}{2}+3,\infty}}
\end{align}
we have concluded the proof of \eqref{a_t11}.
The estimates \eqref{a_t12} and \eqref{a_t13} follow respectively from \eqref{estHlow},
and the bound \eqref{estQ_0L^2b} provided by Proposition \ref{proCCmain}.


\subsubsection{Evolution of $E^\l$}\label{secevoll}
The energy associated with $\l$ is given by
\begin{align*}
E^{\l} (t) 
    & = \sum_{k=0}^{N_0-2}  \int \frac{1}{A} {\left| (\partial_t + b \partial_\a) D^k \l \right|}^2
    + i  {\left( D^k \l \right)}^h \partial_\alpha \overline{{\left( D^k \l \right)}^h} \, d\a 
\\
& + \sum_{k=0}^{N_0/2}  \int \frac{1}{A} {\left| (\partial_t + b \partial_\a) S_k \l \right|}^2
    + i{\left( S_k \l \right)}^h \partial_\alpha \overline{{\left( S_k \l \right)}^h} \, d\a  \, .
\end{align*}
From \eqref{cubiceql} and \eqref{Gl} we have
${(\partial_t + b \partial_\a)}^2 \l + i A \partial_\alpha \l = \sum_{j=1}^4 G^{\l}_j$,
where
\begin{align}
\label{cubicl1}
G^{\l}_1 & = - \left[u, \H\frac{1}{\z_\a} + \bar{\H}\frac{1}{\bar{\z}_\a} \right] (\bar{\z}_\a w )
\\
\label{cubicl2}
G^{\l}_2 & = [u,\bar{\H}] \left( \bar{u} \frac{u_\a }{\bar{\z}_\a} \right) +  u [u,\H] \frac{\bar{u}_\a }{\z_\a}
\\
\label{cubicl3}
G^{\l}_3 & = - 2 [u,\H] \frac{u \cdot u_\a}{\z_\a} 
\\
\label{cubicl4}
G^{\l}_4 & = \frac{1}{i\pi} \int { \left( \frac{ u(\a) - u(\b) }{ \z(\a)-\z(\b) } \right) }^2 u (\b) \cdot \z_\b(\b) \, d\b \, .
\end{align}
As already done for $\chi$ before, we use Proposition \ref{proevolEf} to compute 
\begin{align}
\label{dtEl1}
\frac{d}{dt} E^\l (t) & = \sum_{k=0}^{N_0-2} \int \frac{2}{A} \Re \left( (\partial_t + b \partial_\a) D^k \l \, \P D^k \l \right) 
  -  \frac{1}{A} \frac{a_t}{a} \circ k^{-1} {\left|  (\partial_t + b \partial_\a) D^k \l \right|}^2  \, d\a
\\
\label{dtEl2}
& - 2 \sum_{k=0}^{N_0-2} \Re \int  i\partial_t (D^k \l)^h \partial_\a\bar{(D^k \l)^r}
  +  i\partial_t (D^k \l)^r \partial_\a \bar{(D^k \l)^h}  
  + i \partial_t (D^k \l)^r \partial_\a \bar{(D^k \l)^r} \, d\a
\\
\label{dtEl3}
& + \sum_{k=0}^{N_0/2}  \int \frac{2}{A} \Re \left( (\partial_t + b \partial_\a) S_k \l \, \P S_k \l \right) 
  -  \frac{1}{A} \frac{a_t}{a} \circ k^{-1} {\left|  (\partial_t + b \partial_\a) S_k \l \right|}^2  \, d\a 
\\ \label{dtEl4}
& - 2 \sum_{k=0}^{N_0/2}  \Re \int  i\partial_t (S_k \l)^h \partial_\a \bar{(S_k \l)^r}
  +  i\partial_t (S_k \l)^r \partial_\a \bar{(S_k \l)^h}  
  + i \partial_t (S_k \l)^r \partial_\a \bar{(S_k \l)^r} \, d\a \, .
\end{align}
Since ${\| A-1 \|}_{L^\infty} \leq \frac{1}{2}$ we see that
\begin{align}
\eqref{dtEl1} + \eqref{dtEl3} & \lesssim 
\sqrt{E^\l} (t) \sum_{k=0}^{N_0-2} \left( {\|  D^k G^{\l} \|}_{L^2} + {\|  [\P,D^k] \l \|}_{L^2}  \right) 
\\
& + \sqrt{E^\l} (t) \sum_{k=0}^{N_0/2} \left( {\|  S_k G^{\l} \|}_{L^2} + {\|  [\P,S_k] \l \|}_{L^2}  \right) 
\\ 
& +  E^\l (t)  { \left\| \frac{a_t}{a} \circ k^{-1}  \right\| }_{L^\infty}  \, .
\end{align}
As before, the terms in \eqref{dtEl2} and \eqref{dtEl4} are remainder terms: 
there is no logarithmic loss in estimating them already in \cite{WuAG}, so we can disregard them.
Moreover ${ \| a_t/a \circ k^{-1}  \| }_{L^\infty}$ has been already estimated in section \ref{secevolchi4}, see \eqref{a_t1}.
Also, the terms ${\|  [\P,D^k] \l \|}_{L^2}$ and ${\|  [\P,S_k] \l \|}_{L^2}$ can be treated exactly as done in section \ref{secevolchi2}.
Therefore, to control the time evolution of $E^\l$ by the right-hand side of \eqref{EEf} it suffices to show
\begin{align}
\label{estG^l}
{\| G^\l \|}_{X_{N_0}} \lesssim 
    {\| L(t) \|}_{W^{N_1,\infty}}^2  
    {\| \wt{L} \|}_{X_{N_0}} \, .
\end{align}
This is done in the following sections by estimating each of the terms in \eqref{cubicl1}-\eqref{cubicl4}. 

\vskip5pt
\paragraph{\it {Estimate of \eqref{cubicl1}}}
Observe that
\begin{align*}
G_1^\l = - \frac{2}{\pi} \int \frac{ (u(\a) - u(\b)) (\Im\z(\a) - \Im\z(\b)) }{ {|\z(\a)-\z(\b)|}^2 } w(\b) \bar{\z}_\b (\b) \, d\b 
\end{align*}
and therefore it is an operator of the form $\bC(u,\Im \z, w \z_\a)$ which can be estimated by means of Proposition \ref{proCCmain}.

\vskip5pt
\paragraph{\it {Estimate of \eqref{cubicl2}}}
The term $G_2^\l$ is more delicate. In order to estimate it we need to exploit its special structure,
which allows the appearence of the Hilbert transform acting on products 
only when the arguments are perfect derivative of functions that we can control.
Let us start by explicitely rewriting $G_2^\l$ as
\begin{align}
 G_2^\l & = u \bar{\H} \left( \bar{u} \frac{u_\a }{\bar{\z}_\a} \right) - \bar{\H} \left(u \bar{u} \frac{u_\a }{\bar{\z}_\a} \right)
  + u^2 \H \frac{\bar{u}_\a }{\z_\a} -  u \H \left( u \frac{\bar{u}_\a }{\z_\a} \right) \, .
\end{align}
We can then apply $D^k$ for $k=1,\dots,N_0$, and $S_k$, for $k=0,\dots,N_0/2$, 
to the above expression and use the commutation identities to distribute them.
This procedure will give many terms, most of which can be estimated directly using the $L^2$-bounds given by Proposition \ref{proCCmain}.
There will only two be types of dangerous terms:

\noindent
1. Terms for which $D^{N_0}$ or $S_{N_0/2}$ fall on $u_\a$;

\noindent
2. Terms that require the estimate in $L^\infty$ of an Hilbert transform of a product, such as $\H \frac{1}{\z_\a} u \bar{u}_\a$.

\noindent
More precisely, denoting by $\G^{N_0}$ either $D^{N_0}$ or $S_{N_0/2}$,
all the dangerous terms are:
\begin{align*}
A_1 & = \G^{N_0} u \bar{\H} \left( \bar{u} \frac{u_\a }{\bar{\z}_\a} \right) 
\,\, , \,\,
A_2 = u \bar{\H} \left( \bar{u} \frac{\G^{N_0} u_\a }{\bar{\z}_\a} \right)
\,\, , \,\,
A_3  = - \bar{\H} \left( u \bar{u} \frac{\G^{N_0} u_\a }{\bar{\z}_\a} \right)
\\
A_4 & = 2 u \G^{N_0} u \bar{\H} \frac{u_\a }{\bar{\z}_\a} 
\,\, ,\,\,
A_5 = u^2 \H \frac{\G^{N_0} \bar{u}_\a }{\z_\a}
\,\, , \,\,
A_6 = - \G^{N_0} u  \H \left( u \frac{\bar{u}_\a }{\z_\a} \right)
\,\, , \,\,
A_7 = - u \H \left( u \frac{\G^{N_0} \bar{u}_\a }{\z_\a} \right) .
\end{align*}
In particular, using Propositions \ref{proCCmain} and \ref{theoCCL^2}, one can verify that
\begin{align*}
& {\left\| \G^{N_0} \left( u \bar{\H} \left( \bar{u} \frac{u_\a }{\bar{\z}_\a} \right) \right) - (A_1 + A_2) \right\|}_{L^2}
  + {\left\| - \G^{N_0} \bar{\H} \left(u \bar{u} \frac{u_\a }{\bar{\z}_\a} \right) - A_3 \right\|}_{L^2} 
\\
& + {\left\| \G^{N_0} \left( u^2 \H \frac{\bar{u}_\a }{\z_\a} \right) - (A_4 + A_5) \right\|}_{L^2}
  + {\left\| - \G^{N_0} \left( u \H \left( u \frac{\bar{u}_\a }{\z_\a} \right) \right) - (A_6 + A_7) \right\|}_{L^2}
\\
& \lesssim {\| u \|}_{W^{N_1,\infty}}^2 {\| u \|}_{X_{N_0}} \, .
\end{align*}
It follows that
\begin{align*}
{\Big\|  \G^{N_0} G_2^\l \Big\|}_{L^2} \lesssim \Big\| \sum_{j=1}^7 A_j {\Big\|}_{L^2} + {\| u \|}_{W^{N_1,\infty}}^2 {\| u \|}_{X_{N_0}} \, .
\end{align*}
To estimate the terms $A_j$  we need to combine them appropriately. More precisely we look at the combinations
\begin{align}
\label{El21}
A_2 + A_3 & =  -\frac{1}{i\pi} \int \frac{(u(\a) - u(\b)) \bar{u}(\b)}{ \bar{\z}(\a)-\bar{\z}(\b) } \G^{N_0} u_\b (\b) \, d\b 
\\ 
\label{El22}
A_5 + A_7 & =  \frac{1}{i\pi} u(\a)  \int \frac{u(\a) - u(\b)}{ \z(\a)-\z(\b) } \G^{N_0} \bar{u}_\b (\b) \, d\b 
\\
\label{El23}
A_1 + A_6 & = \G^{N_0}  u \left(\bar{\H} \left( \bar{u} \frac{u_\a }{\bar{\z}_\a} \right)
    - \H \left( u \frac{\bar{u}_\a }{\z_\a} \right) \right) \, .
\end{align}
The remaining term $A_4$ can be directly bounded using Lemma \ref{lemHinfty}
\begin{align*}
{\left\| A_4 \right\|}_{L^2} \lesssim {\left\| u \G^{N_0} u \bar{\H} \frac{u_\a }{\bar{\z}_\a} \right\|}_{L^2} 
  \lesssim {\| u \|}_{L^\infty}  {\| \G^{N_0} u \|}_{L^2} {\left\| \bar{\H} \frac{u_\a }{\bar{\z}_\a} \right\|}_{L^\infty}
  \lesssim {\| u \|}_{W^{N_1,\infty}}^2 {\| u \|}_{X_{N_0}} \, .
\end{align*}
The terms \eqref{El21} and \eqref{El22} can be estimated by means of \eqref{estQ_0L^2b} (case $k=0$) in Proposition \ref{proCCmain}
since, upon commuting $\G^{N_0}$ and $\partial_\a$, they are linear combinations of terms of the form
\begin{align*}
Q_0 (u, \partial_\a \bar{u}\G^{N_0} u) \quad , \quad  Q_0 \left(u, \partial_\a(\bar{u}\G^{N_0}u) \right) \quad , \quad Q_0 (u,\bar{u} \G^{N_0} u) 
  \quad \mbox{or} \quad u Q_0 (u, \partial_\a \G^{N_0} u) \, .
\end{align*} 
Notice that there are no singular integrals that need to be estimated in $L^\infty$ here.

To bound \eqref{El23} we rewrite it as:
\begin{align}
\nn
\eqref{El23} & = \G^{N_0}  u\left( \bar{\H} \frac{1}{\bar{\z}_\a} (\bar{u} u_\a) +  \H \frac{1}{\z_\a} ( \bar{u} u_\a ) \right)
  - \G^{N_0}  u\left( \H \frac{1}{\z_\a} (\bar{u} u_\a) +  \H \frac{1}{\z_\a} ( u \bar{u}_\a ) \right)
\\
\label{El31}
& = \G^{N_0} u \left( \bar{\H} \frac{1}{\bar{\z}_\a} +  \H \frac{1}{\z_\a} \right) ( \bar{u} u_\a )
 - \G^{N_0} u \, \H \frac{1}{\z_\a} \partial_\a ( \bar{u} u ) =: B_1 + B_2 \, .
\end{align}
The first term $B_1$ in \eqref{El31} is quartic and can therefore be easily estimated.
The second term $B_2$ can instead be bounded by means of Lemma \ref{lemHinfty}
using the fact that the argument of the Hilbert transform is a perfect derivative:
\begin{align*}
{\| B_2 \|} & \lesssim {\| \G^{N_0} u \|}_{L^2} {\left\| \H \frac{1}{\z_\a} \partial_\a ( \bar{u} u ) \right\|}_{L^\infty}
  \lesssim {\| u \|}_{X_{N_0}} {\| u^2 \|}_{W^{N_1,\infty}} \, .
\end{align*}

\vskip5pt
\paragraph{\it {Estimate of \eqref{cubicl3}}}
The term $G_3^\l$ can be treated similarly to $G_2^\l$.
Using the same notation as above, and indicating with ``$+ \cdots$'' harmless terms
that can be controlled directly by means of Propositions \ref{proCCmain} or \ref{theoCCL^2}, we can write
\begin{align}
\label{El41}
& \G^{N_0} G_3^\l = - 2 \G^{N_0} \left( [u,\H] \frac{u \cdot u_\a}{\z_\a} \right) = 
  - 2 \G^{N_0} u  \H \frac{u \cdot u_\a}{\z_\a}  - 2 [u,\H] \frac{u \cdot \G^{N_0} u_\a}{\z_\a} + \cdots \, .
\end{align}
Since $u \cdot u_\a = \Re( u \bar{u}_\a)$, the first summand in \eqref{El41} is like $B_2$ above.
The second summand in \eqref{El41} 
is of the same form as \eqref{El21} above, and therefore can be estimated similarly.

\vskip5pt
\paragraph{\it {Estimate of \eqref{cubicl4}}}
The term $G_4^\l$ is of the type $\bC(\wt{L},\wt{L},\wt{L})$ and can therefore be treated directly using \eqref{estCCmain00} 
in Proposition \ref{proCCmain}. This concludes the proof of \eqref{estG^l}, hence of \eqref{EEf} for $f = \l$.



\subsubsection{Evolution of $E^v$}\label{secevolv}
The energy associated with $v$ is given by
\begin{align*}
E^{v} (t) & = 
  \sum_{k=0}^{N_0} \int \frac{1}{A} {\left| (\partial_t + b \partial_\a) D^k v_1 \right|}^2 + i D^k v_1 \partial_\alpha \overline{D^k v_1} \, d\a 
  \\
& + \sum_{k=0}^{N_0/2}  \int \frac{1}{A} {\left| (\partial_t + b \partial_\a) S_k v_1 \right|}^2
    + i S_k v_1 \, \partial_\alpha \overline{S_k v_1} \, d\a  \, .
\end{align*}
Here $v_1 = (I-\H) v$ and from \eqref{cubiceqv}, \eqref{Gv}, \eqref{Gv_1} we have
${(\partial_t + b \partial_\a)}^2 v_1 + i A \partial_\alpha v_1 = \sum_{j=1}^5 G^{v_1}_j$,
with
\begin{align}
\label{cubicv1}
G^{v_1}_1 & =  (I-\H) \P v 
\\
\label{cubicv2}
G^{v_1}_2 & = - 2 [u,\H] \frac{\partial_\a}{\z_\a} \P \chi 
\\
\label{cubicv3}
G^{v_1}_3 & = - 2[u,\H] \frac{\partial_\a}{\z_\a} \left( w \frac{\partial_\a}{\z_\a} \chi \right)
\\
\label{cubicv4}
G^{v_1}_4 & = -i [(\H + \bar{\H})u, \H] {\left( \frac{\partial_\a}{\z_\a} \right)}^2 \chi
\\
\label{cubicv5}
G^{v_1}_5 & = \frac{1}{i\pi}  \int { \left( \frac{ u(\a) - u(\b) }{ \z(\a)-\z(\b) } \right) }^2 v_\b (\b) \, d\b \, .
\end{align}

Using \eqref{d_tE_0^F} we compute 
\begin{align*}
\frac{d}{dt} E^v (t) & = \sum_{k=0}^{N_0} \int \frac{2}{A} \Re \left( (\partial_t + b \partial_\a) D^k v_1 \, \P D^k v_1 \right) 
  -  \frac{1}{A} \frac{a_t}{a} \circ k^{-1} {\left|  (\partial_t + b \partial_\a) D^k v_1 \right|}^2  \, d\a
\\
& + \sum_{k=0}^{N_0/2}  \int \frac{2}{A} \Re \left( (\partial_t + b \partial_\a) S_k v_1 \, \P S_k v_1 \right) 
  -  \frac{1}{A} \frac{a_t}{a} \circ k^{-1} {\left|  (\partial_t + b \partial_\a) S_k v_1 \right|}^2  \, d\a \, .
\end{align*}
Since ${\| A-1 \|}_{L^\infty} \leq \frac{1}{2}$ we see that
\begin{align*}
\frac{d}{dt} E^v (t) & \lesssim 
\sqrt{E^{v_1}} (t) \sum_{k=0}^{N_0}  \left( {\|  D^k G^{v_1} \|}_{L^2} + {\|  [\P,D^k] v_1 \|}_{L^2}  \right) 
\\
& + \sqrt{E^{v_1}} (t) \sum_{k=0}^{N_0/2} \left( {\|  S_k G^{v_1} \|}_{L^2} + {\|  [\P,S_k] v_1 \|}_{L^2}  \right) 
  +  E^{v_1} (t)  { \left\| \frac{a_t}{a} \circ k^{-1}  \right\| }_{L^\infty}  \, .
\end{align*}
Since ${ \| a_t/a \circ k^{-1} \| }_{L^\infty}$ has been already estimated in \eqref{a_t1},
and ${\|  [\P,D^k] v_1 \|}_{L^2}$ and ${\|  [\P,S] {v_1} \|}_{L^2}$ can be bounded as in section \ref{secevolchi2},
we only need to suitably control ${\| G^{v_1} \|}_{X_{N_0}}$ by showing
\begin{align}
\label{estG^v}
{\| G^{v_1} \|}_{X_{N_0}} \lesssim {\left( {\| L(t) \|}_{W^{N_1,\infty}} + {\| \H L^-(t) \|}_{W^{N_1,\infty}} \right)}^2
    {\| L \|}_{X_{N_0}} \, .
\end{align}
This bound is proven below by separately estimating each of the terms \eqref{cubicv1}-\eqref{cubicv5}.

\vskip5pt
\paragraph{\it {Estimate of \eqref{cubicv1}}}
From \eqref{estHb} we see that
\begin{align}
{\| G^{v_1}_1 \|}_{X_{N_0}} \lesssim {\| \P v \|}_{X_{N_0}} +  {\| \wt{L} \|}_{X_{N_0}}  {\|\P v\|}_{H^{{N_0/2}+1}} \, .
\end{align}
Since the second summand above is easier to estimate (it is a quartic expression) 
we just show how to control the $X_{N_0}$-norm of $\P v$ in \eqref{Gv}.
With the exception of the last term, $i (a_t/a) \circ k^{-1} \, A \partial_\a \chi$,
these are all terms of the form $\bC(L,L,L_\a)$ for which \eqref{estCCmain0} applies directly.
To estimate the remaining term we first notice that we can essentially replace $A$ with $1$ in view of Lemma \ref{lemenergy10c}
and the estimates \eqref{a_t1} and \eqref{a_t2}.
We then have
\begin{align*}
{\left\| \frac{a_t}{a} \circ k^{-1} \, \partial_\a \chi \right\|}_{X_{N_0}} \lesssim 
  {\left\| \frac{a_t}{a} \circ k^{-1} \right\|}_{X_{N_0}}  {\|  \partial_\a \chi  \|}_{W^{\frac{N_0}{2},\infty}}
  + {\left\| \frac{a_t}{a} \circ k^{-1} \right\|}_{W^{\frac{N_0}{2},\infty}}  {\|  \partial_\a \chi  \|}_{X_{N_0}} \, .
\end{align*}
Using \eqref{a_t1}-\eqref{a_t2} to bound the norms of $(a_t/a) \circ k^{-1}$,
and \eqref{chivl1}-\eqref{chivl2} 
to control the norms of $\partial_\a \chi$ we see
\begin{align*}
{\left\| \frac{a_t}{a} \circ k^{-1} \, A \partial_\a \chi \right\|}_{X_{N_0}} \lesssim 
   {\| L(t) \|}_{W^{N_1,\infty}}^2 {\| L \|}_{X_{N_0}}
\end{align*}
as desired.

\vskip5pt
\paragraph{\it {Estimate of \eqref{cubicv2}}}
Since $P \chi$ is a cubic term, \eqref{cubicv2} is a quartic term. Moreover it is of the form $Q_0(u, \partial_\a \P \chi)$.
Thus we can use \eqref{estQ_0L^2b} to obtain
\begin{align*}
 {\left\|  \eqref{cubicv2} \right\|}_{X_{N_0}} & \lesssim 
  {\| u \|}_{W^{\frac{N_0}{2}+2,\infty}} {\| \P \chi \|}_{X_{N_0}} + {\| \P \chi \|}_{W^{\frac{N_0}{2}+2,\infty}} {\| u \|}_{X_{N_0}} 
\\
& \lesssim \e_1 {\| \P \chi \|}_{X_{N_0}} + {\| \P \chi \|}_{W^{\frac{N_0}{2}+3,\infty}} {\| u \|}_{X_{N_0}} \, .
\end{align*}
One can then use bounds obtained previously on ${\| \P \chi \|}_{X_{N_0}}$ and the estimate
\begin{align*}
{\| \P \chi \|}_{W^{\frac{N_0}{2}+3,\infty}} \lesssim \e_1 {\| L \|}_{W^{\frac{N_0}{2},\infty}}^2 \, ,
\end{align*}
which is easy to derive, to deduce the desired bound for \eqref{cubicv2}.

\vskip5pt
\paragraph{\it {Estimate of \eqref{cubicv3}}}
This term is of the form 
\begin{align*}
\eqref{cubicv3} = Q_0 \left(u, \partial_\a \left( w \frac{\partial_\a}{\z_\a} \chi \right) \right) \, .
\end{align*}
Applying \eqref{estQ_0L^2b}, followed by product estimates, \eqref{chivl1}-\eqref{chivl2}, and the use of \eqref{chi_atoz_a}, 
we see that
\begin{align*}
{\left\|  \eqref{cubicv3} \right\|}_{X_{N_0}} & \lesssim 
  {\| u \|}_{W^{\frac{N_0}{2}+2,\infty}} {\left\|  w \frac{\partial_\a}{\z_\a} \chi \right\|}_{X_{N_0}} 
  + {\left\| w \frac{\partial_\a}{\z_\a} \chi \right\|}_{W^{\frac{N_0}{2}+2,\infty}} {\| u \|}_{X_{N_0}}
\\
& + {\left\| u \right\|}_{W^{\frac{N_0}{2}+1,\infty}} {\left\| w \frac{\partial_\a}{\z_\a} \chi \right\|}_{W^{\frac{N_0}{2}+1,\infty}} 
  {\| \z_\a-1 \|}_{X_{N_0}}
  \lesssim {\| L \|}_{X_{N_0}} {\| L \|}_{W^{N_1,\infty}}^2 \, .
\end{align*}

\vskip5pt
\paragraph{\it {Estimate of \eqref{cubicv4}}}
This contribution can also be written in terms of the operator $Q_0$ as
\begin{align*}
\eqref{cubicv4} = Q_0 \left( (\H + \bar{\H}) u, \partial_\a \left( \frac{\partial_\a}{\z_\a} \chi \right) \right) \, .
\end{align*}
We can then use \eqref{estQ_0L^2b} to obtain
\begin{align*}
{\left\|  \eqref{cubicv4} \right\|}_{X_{N_0}} & \lesssim 
  {\| \partial_\a (\H + \bar{\H}) u \|}_{W^{\frac{N_0}{2}+1,\infty}} {\left\|  \frac{\partial_\a}{\z_\a} \chi \right\|}_{X_{N_0}} 
  + {\left\| \frac{\partial_\a}{\z_\a} \chi \right\|}_{W^{\frac{N_0}{2}+1,\infty}} {\| (\H + \bar{\H}) u \|}_{X_{N_0}} 
\\
& + {\left\| (\H + \bar{\H}) u \right\|}_{H^{\frac{N_0}{2}+2,\infty}} 
  {\left\| \frac{\partial_\a}{\z_\a} \chi \right\|}_{W^{\frac{N_0}{2}+2,\infty}} {\| \z_\a-1 \|}_{X_{N_0}} \, .
\end{align*}
From \eqref{H+barH}, \eqref{estbQL^inftyfg} and \eqref{estHdfL^infty}, and \eqref{estH+barH}, we deduce
\begin{align*}
& {\| \partial_\a (\H + \bar{\H}) u \|}_{W^{\frac{N_0}{2}+1,\infty}} \lesssim {\| L \|}^2_{W^{\frac{N_0}{2}+3,\infty}}
\\
& {\| (\H + \bar{\H}) u \|}_{X_{N_0}} \lesssim  {\| L \|}_{W^{\frac{N_0}{2}+1,\infty}} {\| L \|}_{X_{N_0}} \, .
\end{align*}
Using these bounds we eventually see that
\begin{align*}
{\left\|  \eqref{cubicv4} \right\|}_{X_{N_0}} & \lesssim {\| L \|}_{X_{N_0}} {\| L \|}_{W^{\frac{N_0}{2}+3,\infty}}^2 \, .
\end{align*}

\vskip5pt
\paragraph{\it {Estimate of \eqref{cubicv5}}}
This term is of the form $\bC(L,L,L_\a)$ and therefore \eqref{estCCmain0} can be applied directly.

\subsection{Proof of Proposition \ref{proenergy3}: Control in terms of the $Z^\p$ norm}\label{secproenergy3}
Recall the definitions 
\begin{equation}
\label{vectorL2}
\wt{L} := (\z_\a - 1, u, \Im \z, w) \quad , \quad 
 L := (\z_\a - 1, u, w, \Im\z, \partial_\a \chi, v)
\end{equation}
and 
\begin{equation}
\label{vectorL-2}
 L^- := (\z_\a - 1, u, w, \partial_\a \chi, v)  \, .
\end{equation}
We want to show
\begin{align}
\label{proenergy31}
{\| L (t) \|}_{W^{N_1,\infty}} \lesssim {\| (h(t),\phi(t)) \|}_{Z^\p}
\end{align}
and
\begin{align}
\label{proenergy32}
{\left\| \H  L^- (t) \right\|}_{W^{N_1,\infty}}  \lesssim {\| (h(t),\phi(t)) \|}_{Z^\p} 
\end{align}
where the $Z^\p$ norm is defined in \eqref{defZp}.
These estimates rely on the following Lemmas:

\begin{lem}\label{lemenergy31}
Assume that the a priori estimate \eqref{aprioriL1} holds. 
Then there exists constants $c_i,d_i$, for $i=1,\dots,5$ such that
\begin{align}
\label{appL-}
L_i^-(t,\a) - c_i \partial_\a \chi (t,\a) - d_i \partial_\a \l (t,\a) = Q (t,\a)
\end{align}
where $Q$ denotes a quadratic expression in $\wt{L}$ satisfying
\begin{align}
\label{Qlemenergy31}
{\| Q(t) \|}_{H^{k}} \lesssim  {\| \wt{L}(t) \|}_{H^{k+1}} {\| \wt{L}(t) \|}_{W^{\frac{k}{2}+1,\infty}} \, .
\end{align}
\end{lem}

\begin{lem}\label{lemenergy32}
Assume that the a priori estimates \eqref{apriori0} and \eqref{aprioriL1} hold,
and that
\begin{align}
\label{lemenergy32hyp}
\sup_{[0,T]} {(1+t)}^\frac{1}{8} {\| \phi(t) \|}_{L^\infty} \leq \e_1 \, ,
\end{align}
as guranteed by \eqref{estcorproE4}.
Then for any $0 \leq k \leq N_1$ and any $t \in [0,T]$ we have,
\begin{align}
\label{est1lemenergy3}
{\| \partial_\a \chi (t) \|}_{W^{k,\infty}} & \lesssim {\| (h(t),\phi(t)) \|}_{Z^\p}
\\
\label{est2lemenergy3}
{\| \partial_\a \l (t) \|}_{W^{k,\infty}} & \lesssim {\| (h(t),\phi(t)) \|}_{Z^\p}
\end{align}
and
\begin{align}
\label{est3lemenergy3}
{\| \H \partial_\a \chi (t) \|}_{W^{k,\infty}} & \lesssim {\| (h(t),\phi(t)) \|}_{Z^\p}
\\
\label{est4lemenergy3}
{\| \H \partial_\a \l (t) \|}_{W^{k,\infty}} & \lesssim {\| (h(t),\phi(t)) \|}_{Z^\p} \, .
\end{align}
\end{lem}


\vskip5pt
\begin{proof}[Proof of Proposition \ref{proenergy3}]
The estimate
\begin{align*} 
{\| L^- \|}_{W^{N_1,\infty}} + {\| \H L^- \|}_{W^{N_1,\infty}} \lesssim  {\| (h,\phi) \|}_{Z^\p}
\end{align*}
clearly follows by combining the above Lemmas.
To obtain \eqref{proenergy31} we need to use in addition the identity
\begin{align*} 
\Im \z (t,\a)= h(t ,\Re \z(t,\a))
\end{align*}
to estimate $\Im \z$ in $W^{N_1,\infty}$:
\begin{align*} 
{\| \Im \z \|}_{W^{N_1,\infty}} \lesssim  {\| h \|}_{W^{N_1,\infty}} {\| \Re \z_\a \|}_{W^{N_1,\infty}}
  \lesssim  {\| h \|}_{W^{N_1,\infty}} \left( 1 + {\| \Re \z_\a - 1\|}_{H^{N_1+1}} \right)
  \lesssim  {\| h \|}_{W^{N_1,\infty}} \, .
\end{align*}
\end{proof}

\subsubsection{Proof of Lemma \ref{lemenergy31}}
To prove Lemma \ref{lemenergy31} we use the identities (2.44), (2.50), (3.38) and (2.35) derived by Wu \cite{WuAG},
which relate the components of $L^-$, $\z_\a-1,u,w$ and $v$, to $\partial_\a \chi$ and $\partial_\a \l$.
These are given resepctively by
\begin{align}
\label{id10}
\bar{u} & = \partial_\a \bar{\l} + \bar{u} (1 - \z_\a) + \frac{1}{2} (\z_\a - \bar{\z}_\a) \bar{u}
  + \frac{1}{2} \bar{\z}_\a  \left(  \H \frac{1}{\z_\a}  +  \bar{\H} \frac{1}{\bar{\z}_\a} \right) (\bar{u} \z_\a)
\\
\nn
\bar{w} \z_\a  & = -\frac{i}{2} \partial_\a \bar{\chi} + \frac{1}{2} \bar{\H} \left( \bar{u}_\a \frac{\bar{u}\z_\a}{\bar{\z_\a}} \right)
  - \frac{1}{2} [\bar{u}, \bar{\H}] \frac{\partial_\a ( \bar{u}\z_\a )}{\bar{\z_\a}}
  \\ 
\nn
& - \frac{1}{2\pi i}  \int \frac{ (\bar{u}(\a) - \bar{u}(\b) ) (\z_\a(\a) - \z_\b(\b)) } {{ (\bar{\z}(\a) - \bar{\z}(\b)) }^2} \bar{u}(\b) \z_\b(\b) \,d\b
+ \frac{1}{2} (\z_\a - \bar{\z}_\a) \bar{w} 
\\ \nn
&  - \frac{1}{2} u_\a \bar{u} 
  + \frac{1}{2} \bar{\z}_\a  \left(  \H \frac{1}{\z_\a}  +  \bar{\H} \frac{1}{\bar{\z}_\a} \right) (\bar{w} \z_\a + \bar{u} u_\a)
\\
& -\frac{\bar{\z}_\a}{\pi}  \int \Im {\left( \frac{u(\a) - u(\b)} { \bar{\z}(\a) - \bar{\z}(\b) } \right)}^2 \bar{u}(\b) \z_\b(\b) \, d\b
\label{id20}
\\
\label{id30}
\z_\a - 1  & = \frac{w}{i A} - \frac{A-1}{A} 
\\
\label{id40}
v & = 2u - (\H + \bar{\H}) u - [u,\H] \frac{\z_\a - \bar{\z}_\a}{\z_\a} \, .
\end{align}
From \eqref{id10} we can schematically write
\begin{align*}
u - \partial_\a \l = \wt{L} \cdot \wt{L} + \bQ( \wt{L}, \wt{L} )
\end{align*}
up to cubic and higher order terms whose arguments have the same regularity of $\wt{L}$.
Commuting derivatives via \eqref{commK2} and using \eqref{proCCmain11}, it is then easy to verify that $u - \partial_\a \l = Q$,
where $Q$ is a quadratic term satisfying the estimate \eqref{Qlemenergy31} in the statement.
Similarly, from \eqref{id20} we deduce that up to cubic terms
\begin{align*}
w - \frac{i}{2} \partial_\a \chi =  \H( \wt{L} \cdot \wt{L}_\a ) + Q_0 (\wt{L}, \wt{L}_\a ) +  \bQ (\wt{L}, \wt{L}) \, .
\end{align*}
Arguing as above using the bounds \eqref{estHlow}, \eqref{estQ_0L^2} and \eqref{proCCmain11}, it follows that 
$w - \frac{i}{2} \partial_\a \chi = Q$, for some $Q$ satisfying \eqref{Qlemenergy31}.
Using \eqref{id30}, the last equality above, and the quadratic bounds on $A-1$ given by \eqref{formulaA}, we can write
\begin{align*}
\z_\a - 1  & = \frac{w}{i A} - \frac{A-1}{A} = - i w + i w \frac{A-1}{A} - \frac{A-1}{A} = \frac{1}{2} \partial_\a \chi + Q \, ,
\end{align*}
so that \eqref{appL-} is verified also for the component $\z_\a - 1$.
Combining \eqref{id10} with the identity $u - \partial_\a \l = Q$, and the quadratic bounds on $\H + \bar{\H}$ given in \eqref{H+barH},
we see that $v - 2 \partial_\a \l = Q$, for $Q$ as above.
Thus we have checked that \eqref{appL-} holds true for all $i=1,\dots,5$. $\hfill \Box$

\subsubsection{Proof of Lemma \ref{lemenergy32}}\label{seclemenergy32}
Let $\H = \H_\z$ and let $H_0$ be the flat Hilbert transform. We start by establishing the following estimate:
\begin{align}
 \label{est0lemenergy32}
{\left\| (I-\H) (f \circ \Re \z) - \left[ (I-H_0) f(1+h^\p) \right] \circ \Re \z \right\|}_{W^{k,\infty}} 
  \lesssim {\| f \|}_{W^{k+1,\infty}} {\| h \|}_{W^{k+2,\infty}} \, .
\end{align}
for any $0 \leq k \leq N_1 + 4$.
Notice that since $\Im \z(t,\a) = h (t, \Re \z (t,\a))$ we have $\z  = ( x + i h(x) ) \circ \Re \z$.
Thus by a change of variables one has
\begin{align*}
(I-\H) (f \circ \Re \z) =  (I-\H_{x+ih(x)}) f \circ \Re \z  \, .
\end{align*}
Expanding out the denominator in the expression for $\H_{x+ih(x)} f$ we see that
\begin{align}
\nn
& (I-\H_{x+ih(t,x)}) f = (I-H_0) \left[ f(t,\cdot) (1+ h^\p(t,\cdot)) \right] (x) + R_f (t,x) \, 
\\
\label{R_f}
& \mbox{with} \qquad
R_f (t,x) := \frac{1}{i\pi} \int H\left( \frac{h(x) - h(y)}{x-y} \right) \frac{h(x) - h(y)}{ {(x-y)}^2 } f (y)(1+ h^\p(y)) \, dy
\end{align}
for some smooth function $H$. 
To prove \eqref{est0lemenergy32} it then suffices to show that $R_f$ in $W^{k,\infty}$ is bounded by the right-hand side of \eqref{est0lemenergy32}.
Applying the commutation identity \eqref{commK2} in order to distribute derivatives, 
and the $L^\infty$ estimate \eqref{estinfty1}, 
it is not hard to see that
\begin{align*}
{\| R_f \|}_{W^{k,\infty}} & \lesssim {\| h \|}_{W^{k+2,\infty}}  {\| f(1+h^\p) \|}_{W^{k+1,\infty}} 
  \lesssim {\| f \|}_{W^{k+1,\infty}} {\| h \|}_{W^{k+2,\infty}} \, .
\end{align*}

We will now show the estimates \eqref{est2lemenergy3} and \eqref{est4lemenergy3}.
The estimates \eqref{est1lemenergy3} and \eqref{est3lemenergy3} can be proven in a similar fashion, so we will not detail the proof here.

\vskip5pt
\paragraph{\it {Proof of \eqref{est2lemenergy3}}}
Recall the definition of $\l$
\begin{align*}
\l = (I - \H)\psi = (I - \H) ( \phi \circ \Re \z) \, .
\end{align*}
Applying \eqref{est0lemenergy32} with $f = \phi$, and using the assumption \eqref{lemenergy32hyp}, we see that
\begin{align}
\label{est01lemenergy32}
{\left\| \partial_\a \l - \partial_\a \left[ (I-H_0) \phi (1+h^\p) \circ \Re \z \right] \right\|}_{W^{k,\infty}} 
  \lesssim {\| \phi \|}_{W^{k+1,\infty}} {\| h \|}_{W^{k+2,\infty}} \lesssim  {\| (h,\phi) \|}_{Z^\p}
\end{align}
for any $0 \leq k \leq N_1 + 2$.
Moreover, using ${\| \partial_x H_0 f \|}_{L^\infty} \lesssim {\| \Lambda f \|}_{W^{1,\infty}}$, 
standard product estimates, and the hypothesis \eqref{lemenergy32hyp}, one can estimate
\begin{align}
\label{est02lemenergy32}
\begin{split}
 & {\left\| \partial_\a \left[ (I-H_0) \phi(1+h^\p) \right] \circ \Re \z \right\|}_{W^{k,\infty}} 
  \lesssim {\left\| \partial_x \left[ (I-H_0) \phi(1+h^\p) \right] \right\|}_{W^{k,\infty}} 
\\
  & \lesssim {\| \partial_x (I-H_0) \phi \|}_{W^{k+1,\infty}} + {\| \phi h^\p \|}_{W^{k+2,\infty}} 
  \lesssim  {\| (h,\phi) \|}_{Z^\p} \, .
\end{split}
\end{align}
This and \eqref{est01lemenergy32} give us for all $0 \leq k \leq N_1 + 2$
\begin{align}
\label{est03lemenergy32}
{\left\| \partial_\a \l \right\|}_{W^{k,\infty}} & \lesssim  {\| (h,\phi) \|}_{Z^\p} \, ,
\end{align}
which in particular implies \eqref{est2lemenergy3}.

\vskip5pt
\paragraph{\it {Proof of \eqref{est4lemenergy3}}}
Let us write
\begin{align}
\label{Flemenergy32}
\l = F \circ \Re \z
\qquad & \mbox{with} \qquad
F(t,x) = \left( I - \H_{x+ih(t,x)} \right) \phi(t,\cdot) \, . 
\end{align}
Then
$\partial_\a \l = \partial_x F \circ \Re \z \, \partial_\a \Re \z$,  
and since we know that $\z_\a -1$ has uniformly bounded $H^{N_1+5}$ norm, 
using also \eqref{est03lemenergy32}, we can easily deduce that
\begin{align}
 \label{est10lemenergy32}
{\| \partial_x  F\|}_{W^{k,\infty}} \lesssim  {\| \partial_\a  \l  \|}_{W^{k,\infty}} \lesssim  {\| (h,\phi) \|}_{Z^\p} \, ,
\end{align}
for any $0 \leq k \leq N_1 + 2$.
Using the definition of $\H$, and making a change of variables, we see that
\begin{align*}
\H \frac{1}{\z_\a}  \partial_\a \l(t,\a) 
  & = \frac{1}{i\pi} \int \frac{\partial_x F(t, \Re \z(t,\b))}{\z(\a)-\z(\b)} \partial_\b \Re \z(t,\b) \, d\b  
\\
& = \frac{1}{i\pi} \left. \int \frac{\partial_y F(t,y)}{ x + ih(x) -(y + ih(y)) } \, dy  \right|_{x = \Re \z(t,\a)}  \, .
\end{align*}
Using also \eqref{est03lemenergy32}, it follows that
\begin{align*}
{\left\| \H \partial_\a \l(t,\a) \right\|}_{W^{N_1,\infty}} 
  & \lesssim {\left\| \H \frac{1}{\z_\a}  \partial_\a \l(t,\a) \right\|}_{W^{N_1,\infty}} + {\| \partial_\a  \l \|}_{W^{N_1+1,\infty}}
\\
&  \lesssim {\left\| \int \frac{\partial_y F(t,y)}{ x + ih(x) -(y + ih(y)) } \, dy  \right\|}_{W^{N_1,\infty}} + 
  {\| (h,\phi) \|}_{Z^\p}\, . 
\end{align*}
Setting
\begin{align}
\label{Glemenergy32}
G(t,x) = \int \frac{\partial_y F(t,y)}{ x + ih(x) -(y + ih(y)) } \, dy  \, ,
\end{align}
we see that in order to obtain \eqref{est4lemenergy3} it suffices to show
\begin{align}
\label{est30lemenergy32}
 {\| G \|}_{W^{N_1,\infty}} \lesssim {\| (h,\phi) \|}_{Z^\p} \, .
\end{align}
Expanding the denominator in \eqref{Glemenergy32}, we can write
\begin{align}
\nn
& G(t,x) = (I - H_0) \partial_x F (t,\cdot) (x) + G_1 (t,x) \, 
\\
\label{Gexp}
& \mbox{with} \qquad
G_1(t,x) := \frac{1}{i\pi} \int H\left( \frac{h(x) - h(y)}{x-y} \right) \frac{h(x) - h(y)}{ {(x-y)}^2 } \partial_y F (t,y) \, dy
\end{align}
for some smooth function $H$.
Expanding the denominator in the expression for $F$ we can write:
\begin{align}
\nn
& F(t,x) = (I - H_0) [(1+h^\p(t,\cdot))\phi(t,\cdot)] (x) + R_\phi (t,x) \,
\end{align}
where $R_\phi$ is given by \eqref{R_f}. 
It follows that
\begin{align}
G & = (I - H_0) \partial_x (I-H_0) [\phi (1+h^\p)] + (I - H_0) \partial_x R_\phi + G_1 \, .
\end{align}
To obtain \eqref{est30lemenergy32} it is then enough to have
\begin{align}
\label{est31lemenergy32}
& {\| H_0 \partial_x [ \phi (1+h^\p) ] \|}_{W^{N_1,\infty}} \lesssim {\| (h,\phi) \|}_{Z^\p} \, , 
\\
\label{est32lemenergy32}
& {\| (I-H_0) \partial_x R_\phi \|}_{W^{N_1,\infty}} \lesssim {\| (h,\phi) \|}_{Z^\p} \, ,
\\
\label{est33lemenergy32}
& {\| G_1 \|}_{W^{N_1,\infty}} \lesssim {\| (h,\phi) \|}_{Z^\p} \, .
\end{align}
The bound \eqref{est31lemenergy32} has been already shown to hold true in the above paragraph, see \eqref{est02lemenergy32}. 
We also have
\begin{align*}
 {\| R_\phi \|}_{W^{N_1+2,\infty}} \lesssim {\| \phi \|}_{W^{N_1+3,\infty}} {\| h \|}_{W^{N_1+4,\infty}} \lesssim {\| (h,\phi) \|}_{Z^\p}
\end{align*}
which is stronger than \eqref{est32lemenergy32}.
Using again commutation identities and \eqref{estinfty1}, together with \eqref{est10lemenergy32}, we get
\begin{align*}
{\| G_1 \|}_{W^{N_1,\infty}} \lesssim {\| h \|}_{W^{N_1+2,\infty}} {\| \partial_x F \|}_{W^{N_1+1,\infty}}
  \lesssim {\| h \|}_{W^{N_1+2}} {\| (h,\phi) \|}_{Z^\p}
\end{align*}
which is enough for \eqref{est33lemenergy32}. $\hfill \Box$


\section{Proof of Proposition \ref{prok}: the diffeomorphism $k$}
\label{secprok}

This section contains the proof of Proposition \ref{prok}. 
The main issue is to show that the change of coordinates $k$ is a uniformly controlled diffeomorphism for all times.
This is a substantial improvement of the analogous analysis performed by Wu in \cite[p. 124-127]{WuAG},
and relies on a special null structure present in the transport equation \eqref{formulak_t} for $k$.

We start by assuming a priori that
\begin{align}
\label{apriorik1}
\sup_{t\in [0,T]}  {\| k_\a(t) - 1 \|}_{W^{N_0/2+3,\infty}} \leq 1/2 \, .
\end{align}
Furthermore we assume, see \eqref{apriori0},
\begin{align}
\label{apriori01}
\sup_{t\in [0,T]} \left[(1+t)^{-p_0} {\| (h(t), \partial_x \phi(t)) \|}_{X_{N_0}} 
  + \sqrt{1+t} {\| (h(t), \phi(t)) \|}_{Z^\p}\right] \leq \e_1 \, ,
\end{align}
and \eqref{aprioriL1}, that is
\begin{align}
\label{aprioriL11}
& \sup_{t \in [0,T]} \left[ (1+t)^{-p_0} {\| \wt{L}(t) \|}_{X_{N_0}} 
    + {\| \wt{L}(t) \|}_{H^{N_1+5}} + \sqrt{1+t}  {\| \wt{L}(t) \|}_{W^{N_1,\infty}} \right] \leq \e_1 \, .
\end{align}
We then aim to conclude
\begin{align}
\label{apriorikcon1}
\sup_{[0,T]}  {\| k_\a(t) - 1 \|}_{W^{N_0/2+3,\infty}} \lesssim \e_0 + \e_1^2
\end{align}
as a consequence of the following Lemmas:

\begin{lem}[Approximation of $k_t$]\label{lemk_t}
Let $k$ be defined as in \eqref{defk} then the following formula holds:
\begin{equation}
\label{formulak_t}
(I-\H_z) k_t = - [z_t,\H_z] \frac{\bar{z}_\a - k_\a}{z_\a} \, .
\end{equation}
Under the assumptions \eqref{apriori01} and \eqref{aprioriL11}
there exists $\g>0$, such that for any $t\in[0,T]$
\begin{align}
\label{k_t1}
{\left\| \partial_\a [u,\H_\z] \frac{\bar{\z}_\a - 1}{\z_\a} - T_0 (h,\phi) \circ \Re \z \right\|}_{W^{N_0/2+3,\infty}}
  & \lesssim \e_1^2 {(1+t)}^{-1-\g}
\end{align}
and
\begin{align}
\label{k_t2}
{\left\| \partial_\a  [u,\H_\z] \frac{\bar{\z}_\a - 1}{\z_\a} - T_0 (h,\phi) \circ \Re \z \right\|}_{H^{N_0/2+4}} & \lesssim 
    \e_1^2 {(1+t)}^{-1/2-\g}  \, ,
\end{align}
with $T_0$ given by
\begin{align}
\label{T_0}
 T_0 (f,g) := \partial_x  [(I-H_0) g_x, H_0] (I-H_0) f_x \, ,
\end{align}
where $H_0$ is the flat Hilbert transform, $H_0 = \H_\mathrm{id}$ according to \eqref{HT}.
\end{lem}

\begin{lem}[Estimate for $T_0$]\label{lemT_0}
Under the a priori assumptions \eqref{apriori01}, there exists $\g>0$ such that
\begin{align}
\label{boundT_01}
& {\left\| T_0 (h,\phi) \right\|}_{W^{N_0/2+3,\infty}} \lesssim \e_1^2 {(1+t)}^{-1-\g}  \, ,
\\
\label{boundT_02}
& {\left\| T_0 (h,\phi) \right\|}_{H^{N_0/2+4}} \lesssim \e_1^2 {(1+t)}^{-1/2-\g} \, .
\end{align}
\end{lem}

The proofs of Lemma \ref{lemk_t} and \ref{lemT_0} are in section \ref{seclemk_t} and \ref{seclemT_0} respectively.
We now show how Proposition \ref{prok} follows from them.

\vskip5pt
\begin{proof}[Proof of Proposition \ref{prok}]
Since we know by our a priori assumption that $k$ is diffeomorphism, we can define
\begin{align}
\label{defKk_t}
K(t,\a) := [z_t,\H_z] \frac{\bar{z}_\a - k_\a}{z_\a} \circ k^{-1} \, .
\end{align}
From the properties of the Hilbert transform and the definition of $\z$ and $u$ in \eqref{defzeta}, we see that
\begin{align}
\label{idK}
K(t,\a) = [u,\H_\z] \frac{\bar{\z}_\a - 1}{\z_\a} \, 
\end{align}
and 
\begin{align}
\label{k_t10}
(I-\H_\z) (k_t \circ k^{-1}) = - K(t,\a) \, .
\end{align}
Applying the estimate \eqref{est1-Hfinfty2} for the inversion of $I-\H$, with $f = k_t \circ k^{-1}$ and $g = - K$, we see that
\begin{align*}
{\| \partial_\a (k_t \circ k^{-1}) \|}_{W^{N_0/2+3,\infty}} \lesssim  
  {\| \partial_\a K(t,\a) \|}_{W^{N_0/2+3,\infty}} + {\| \z_\a - 1 \|}_{W^{N_0/2+4,\infty}} {\| \partial_\a K(t,\a) \|}_{H^{N_0/2+4}} \, .
\end{align*}
From the a priori assumption \eqref{aprioriL11} and \eqref{apriorik1} it follows that
\begin{align}
\label{estk_t1}
{\| \partial_\a k_t \|}_{W^{N_0/2+3,\infty}} \lesssim
  {\| \partial_\a K(t,\a) \|}_{W^{N_0/2+3,\infty}} + \e_1 {(1+t)}^{-1/2} {\| \partial_\a K(t,\a) \|}_{H^{N_0/2+4}} \, .
\end{align}

Applying successively \eqref{k_t1} and \eqref{boundT_01} we see that
\begin{align}
\label{estk_t2}
{\| \partial_\a K(t,\a) \|}_{W^{N_0/2+3,\infty}} \lesssim  \e_1^2 {(1+t)}^{-1-\g} + {\| T_0 (h,\phi) \|}_{W^{N_0/2+3,\infty}}
  \lesssim  \e_1^2 {(1+t)}^{-1-\g} \, .
\end{align}
Similarly, from \eqref{k_t2} and \eqref{boundT_02} we have
\begin{align}
\label{estk_t3}
{\| \partial_\a K(t,\a) \|}_{H^{N_0/2+4}} \lesssim  
  \e_1^2 {(1+t)}^{-1/2-\g} + {\| T_0 (h,\phi) \|}_{H^{N_0/2+4}} \lesssim  \e_1^2 {(1+t)}^{-1/2-\g} \, .
\end{align}
Plugging \eqref{estk_t2} and \eqref{estk_t3} into \eqref{estk_t1} gives
\begin{align*}
{\| \partial_\a k_t (t) \|}_{W^{N_0/2+3,\infty}} \lesssim  \e_1^2 {(1+t)}^{-1-\g}
\end{align*}
whence
\begin{align*}
{\| k_\a(t) - 1 \|}_{W^{N_0/2+3,\infty}} \lesssim  {\| k_\a(0) - 1 \|}_{W^{N_0/2+3,\infty}}
  + \int_0^t {\| \partial_s k_\a(s) \|}_{W^{N_0/2+3,\infty}} \, ds \lesssim \e_0 + C_\g \e_1^2 \, .
\end{align*}
\end{proof}

\subsection{Proof of Lemma \ref{lemk_t}}\label{seclemk_t}
The identity \eqref{formulak_t} is proven by Wu in Proposition 2.4 of \cite{WuAG}.
Let $K$ be given by \eqref{defKk_t}:
\begin{align}
\label{KQ_0}
K(t,\a) = [u,\H_\z] \frac{\bar{\z}_\a - 1}{\z_\a} = Q_0 (u, \bar{\z}_\a - 1) \, ,
\end{align}
where $Q_0$ is the bilinear operator defined in \eqref{Q_0}.
We aim to approximate $\partial_\a K$ by $T_0(h,\phi)$ showing
\begin{align}
\label{k_t11}
& {\left\| \partial_\a K - T_0 (h,\phi) \circ \Re \z \right\|}_{W^{N_0/2+3,\infty}} \lesssim  \e_1^2 {(1+t)}^{-1-\g}
\\
\label{k_t21}
& {\left\| \partial_\a  K - T_0 (h,\phi) \circ \Re \z \right\|}_{H^{N_0/2+4}} \lesssim \e_1^2 {(1+t)}^{-1/2-\g} \, .
\end{align}

\vskip5pt
\paragraph{\it{Step 1: Approximation of $u$}.}
Let $H_0$ denote the flat Hilbert transform, $H_0 = \H_\mathrm{id}$ according to \eqref{HT}.
We start by showing
\begin{align}
\label{appu1}
{\left\| u - (I-H_0) \phi_x \circ \Re \z \right\|}_{W^{N_0/2+5,\infty}} & \lesssim  \e_1^2 {(1+t)}^{-1/2-\g} \, .
\end{align}
Using the identity \eqref{id10} as in the proof of Lemma \ref{lemenergy31}, we can schematically write 
\begin{align*}
u - \partial_\a \l = \wt{L} \cdot \wt{L} + \bQ( \wt{L}, \wt{L} ) \, ,
\end{align*}
so that using \eqref{estQL^infty2} to estimate $\bQ$, and interpolating bewtween the a priori decay assumption and the Sobolev bounds, 
we get
\begin{align*}
{\| u - \partial_\a \l \|}_{W^{N_0/2+5,\infty}}
  \lesssim {\| \wt{L} \|}^2_{W^{N_0/2+7,\infty}} \lesssim \e_1^2 {(1+t)}^{-3/4} \, .
\end{align*}
To obtain \eqref{appu1} it then suffices to show
\begin{align}
\label{appu2}
{\left\| \partial_\a \l - (I-H_0) \phi_x \circ \Re \z \right\|}_{W^{N_0/2+5,\infty}} & \lesssim  \e_1^2 {(1+t)}^{-1/2-\g} \, .
\end{align}
Looking at \eqref{est01lemenergy32} and \eqref{est02lemenergy32} in the proof of Lemma \ref{lemenergy32} one can see that
\begin{align*}
{\left\| \partial_\a \l - \partial_\a \left[ (I-H_0) \phi \circ \Re \z \right] \right\|}_{W^{N_0/2+5,\infty}} & 
  \lesssim  {\left\| \phi \right\|}_{W^{N_0/2+7,\infty}} 
  {\left\| h \right\|}_{W^{N_0/2+8,\infty}} \lesssim  \e_1^2 {(1+t)}^{-5/8} \, ,
\end{align*}
having used \eqref{estcorproE4} in the last inequality.
Since we also have
\begin{align*}
& {\left\| \partial_\a \left[ (I-H_0) \phi \circ \Re \z \right] - (I-H_0) \phi_x \circ \Re \z \right\|}_{W^{N_0/2+5,\infty}}
\\
& = {\left\| (I-H_0) \phi_x \circ \Re \z \, (\Re \z_\a - 1) \right\|}_{W^{N_0/2+5,\infty}}
  \lesssim {\left\| \Lambda \phi \right\|}_{W^{N_0/2+6,\infty}} {\| \z_\a - 1 \|}_{W^{N_0/2+5,\infty}} 
  \lesssim \e_1^2 {(1+t)}^{-1} \, ,
\end{align*}
we have verified \eqref{appu2}, hence \eqref{appu1}
with $\b = 1/8$.

\vskip5pt
\paragraph{\it{Step 2: Approximation of $\z_\a - 1$}.}
We want to show that $\z_\a -1$ can be approximated as follows:
\begin{align}
\label{appz_a-1a}
& {\left\| \z_\a - 1 - i(I-H_0) h_x \circ \Re \z \, \partial_\a \Re \z \right\|}_{W^{N_0/2+5,\infty}}  
  \lesssim  \e_1^2 {(1+t)}^{-1/2-\g} 
\\
\label{appz_a-1b}
& {\left\| \H_\z \left( \z_\a - 1 - i(I-H_0) h_x \circ \Re \z \, \partial_\a \Re \z \right) \right\|}_{W^{N_0/2+3,\infty}}
  \lesssim  \e_1^2 {(1+t)}^{-1/2-\g} \, ,
\end{align}
for some $\g>0$.
Putting together the identities \eqref{id30} and \eqref{id20}, as in the the proof of Lemma \ref{lemenergy31}, we can write
\begin{align}
\label{appz_a-12}
\z_\a - 1 - \frac{1}{2} \partial_\a \chi = \H(\wt{L} \cdot \wt{L}) + \wt{L} \cdot \H \wt{L}_\a + \wt{L} \cdot \wt{L}_\a
  + \bQ ( \wt{L}, \wt{L} ) + A-1 + \cdots \,
\end{align}
where operators of the type $\bQ$ are defined by \eqref{Q_1}-\eqref{opQ},
$A$ is defined in \eqref{defA}, and ``$\cdots$'' denotes cubic or higher order terms which are more easily estimated,
and we will therefore disregard.
Notice that interpolating between the bounds provided by the a priori assumptions \eqref{aprioriL11},
for large enough $p$ one has
\begin{align}
\label{appz_a-13}
{\| \wt{L}(t) \|}_{W^{N_0/2+7,p}} \lesssim \e_1 {(1+t)}^{-2/5} \, .
\end{align}
Combining this with the estimates \eqref{estHL^infty} and \eqref{estHdfL^infty} for $\H$ in $W^{k,\infty}$, we see that
\begin{align}
\label{appz_a-14a}
\begin{split}
& {\left\| \H(\wt{L} \cdot \wt{L}) + \wt{L} \cdot \H \wt{L}_\a + \wt{L} \cdot \wt{L}_\a \right\|}_{W^{N_0/2+5,\infty}}
  \\
& \lesssim \left( {\| \wt{L} \|}_{W^{N_0/2+6,p}} + {\| \wt{L} \|}_{W^{N_0/2+6,\infty}} \right) {\| \wt{L} \|}_{W^{N_0/2+6,\infty}}
  \lesssim \e_1^2 {(1+t)}^{-4/5}
\end{split}
\end{align}
and
\begin{align}
\label{appz_a-14b}
\begin{split}
& {\left\| \H \left( \H(\wt{L} \cdot \wt{L}) + \wt{L} \cdot \H \wt{L}_\a + \wt{L} \cdot \wt{L}_\a \right) \right\|}_{W^{N_0/2+3,\infty}}
\\
& \lesssim \left( {\| \wt{L} \|}_{W^{N_0/2+6,p}} + {\| \wt{L} \|}_{W^{N_0/2+5,\infty}} \right) {\| \wt{L} \|}_{W^{N_0/2+6,\infty}}
  \lesssim \e_1^2 {(1+t)}^{-4/5}
 \end{split}
\end{align}
having chosen $p$ large enough so that \eqref{appz_a-13} holds.

We now want to obtain similar bounds for $\bQ ( \wt{L}, \wt{L} )$ and $A-1$, and more precisely show
\begin{align}
\label{appz_a-15a}
& {\big\| \bQ ( \wt{L}, \wt{L} ) \big\|}_{W^{N_0/2+5,\infty}} \lesssim \e_1^2 {(1+t)}^{-4/5}
\\
\label{appz_a-15b}
& {\big\| \H \bQ ( \wt{L}, \wt{L} ) \big\|}_{W^{N_0/2+3,\infty}} \lesssim \e_1^2 {(1+t)}^{-3/5}
\end{align}
and
\begin{align}
\label{appz_a-16a}
& {\left\| A-1 \right\|}_{W^{N_0/2+5,\infty}} \lesssim \e_1^2 {(1+t)}^{-4/5}
\\
\label{appz_a-16b}
& {\left\| \H (A-1) \right\|}_{W^{N_0/2+3,\infty}} \lesssim \e_1^2 {(1+t)}^{-3/5} \, .
\end{align}
The first bound \eqref{appz_a-15a} follows directly from \eqref{estQL^infty} and \eqref{appz_a-13}.
To obtain \eqref{appz_a-15b} first notice that the inequality in \eqref{proCCmain11}, 
which is an application of the $L^2$ estimates in Corollary \ref{theoCCL^2}, gives
\begin{align*}
{\big\| \bQ ( \wt{L}, \wt{L} ) \big\|}_{H^{N_0/2+5}} & \lesssim {\| \wt{L} \|}_{H^{N_0/2+6}} {\| \wt{L} \|}_{W^{N_0/2+5,\infty}} 
  \lesssim \e_1^2 {(1+t)}^{-2/5} \, .
\end{align*}
Interpolating this and the $L^\infty$ bound \eqref{appz_a-15a} gives
\begin{align*}
{\big\| \bQ ( \wt{L}, \wt{L} ) \big\|}_{W^{N_0/2+3,p}} & \lesssim \e_1^2 {(1+t)}^{-3/5} \, ,
\end{align*}
provided $p$ is large enough. \eqref{appz_a-15b} then follows by applying \eqref{estHL^infty}.

Both \eqref{appz_a-16a} and \eqref{appz_a-16b} rely on the identity \eqref{formulaA}, which we can schematically write as 
\begin{align}
\label{appz_a-1A}
(I - \H)(A-1) = [\wt{L}, \H] \wt{L}_\a +  [\wt{L}, \H] \wt{L} + \cdots
\end{align}
where once again ``$\cdots$''stands for cubic order terms which we are going to disregard.
Applying \eqref{est1-Hfinfty} we get
\begin{align*}
{\| A-1 \|}_{W^{N_0/2+5,\infty}} \lesssim {\big\| [\wt{L}, \H] \wt{L}_\a +  [\wt{L}, \H] \wt{L} \big\|}_{W^{N_0/2+5,\infty}}
  + \e_1 {(1+t)}^{-2/5} {\big\| [\wt{L}, \H] \wt{L}_\a +  [\wt{L}, \H] \wt{L} \big\|}_{H^{N_0/2+6}} \, .
\end{align*}
We can then bound the above right-hand side by using \eqref{estHL^infty}, \eqref{appz_a-13}, 
and the boundedness of $\H$ in \eqref{estHlow}:
\begin{align*}
{\| A-1 \|}_{W^{N_0/2+5,\infty}} \lesssim  {\big( {\| \wt{L} \|}_{W^{N_0/2+7,p}} + {\| \wt{L} \|}_{W^{N_0/2+7,\infty}} \big)}^2 
  + \e_1^2 {(1+t)}^{-1+p_0} \lesssim \e_1^2 {(1+t)}^{-4/5} \, .
\end{align*}
We have therefore obtained \eqref{appz_a-16a}.
Since \eqref{appz_a-1A} holds true also for $\H (A-1)$, up to a sign, the estimate \eqref{appz_a-16b} follows as above.

Putting together \eqref{appz_a-12}-\eqref{appz_a-16b} we have
\begin{align*}
& {\Big\| \z_\a - 1 - \frac{1}{2} \partial_\a \chi \Big\|}_{W^{N_0/2+5,\infty}}  \lesssim  \e_1^2 {(1+t)}^{-1/2-\g}
\\
& {\Big\| \H_\z \Big( \z_\a - 1 - \frac{1}{2} \partial_\a \chi \Big) \Big\|}_{W^{N_0/2+3,\infty}} \lesssim \e_1^2 {(1+t)}^{-1/2-\g} \, ,
\end{align*}
for some $\g >0$. To obtain \eqref{appz_a-1a} and \eqref{appz_a-1b} it is then enough to show
\begin{align*}
& {\left\| \partial_\a \chi - 2 i (I-H_0) h_x \circ \Re \z \, \partial_\a \Re \z \right\|}_{W^{N_0/2+5,\infty}}
  \lesssim  \e_1^2 {(1+t)}^{-1/2-\g}
\\
& {\left\| \H_\z \left( \partial_\a \chi - 2 i (I-H_0) h_x \circ \Re \z  \, \partial_\a \Re \z \right) \right\|}_{W^{N_0/2+3,\infty}}
  \lesssim  \e_1^2 {(1+t)}^{-1/2-\g} \, .
\end{align*}
In light of \eqref{estHdfL^infty} both bounds would follow from
\begin{align}
\label{appz_a-17}
& {\left\| \chi - 2 i (I-H_0) h \circ \Re \z \right\|}_{W^{N_0/2+6,\infty}} \lesssim  \e_1^2 {(1+t)}^{-1/2-\g} \, .
\end{align}
From the definition of $\chi$ in \eqref{defchi} and \eqref{hEL} we see that 
\begin{align*}
\chi = 2i(I-\H) (h \circ \Re \z) \, .
\end{align*}
Applying the inequality \eqref{est0lemenergy32} with $f= 2i h$, one gets
\begin{align*}
& {\left\| \chi - 2i (I-H_0) h \circ \Re \z \right\|}_{W^{N_0/2+6,\infty}} 
  \lesssim {\| h \|}_{W^{N_0/2+7,\infty}} {\| h \|}_{W^{N_0/2+8,\infty}} + {\| (I-H_0) (h h^\p) \|}_{W^{N_0/2+8,\infty}} 
\\
& \lesssim {\| h \|}^2_{W^{N_0/2+9,\infty}} + {\| \partial_x H_0 h^2 \|}_{W^{N_0/2+8,\infty}}
  \lesssim {\| h \|}^2_{W^{N_0/2+10,\infty}} \lesssim \e_1^2 {(1+t)}^{-4/5} \, .
\end{align*}
This gives us \eqref{appz_a-17} and concludes the proof of \eqref{appz_a-1a}-\eqref{appz_a-1b}.

\vskip5pt
\paragraph{\it{Step 3: First approximation of $\partial_\a K$}.}
We now want to show
\begin{align}
\label{Kapp1a}
& {\left\| \partial_\a K - \partial_\a Q_0 \left( (I-H_0)\phi_x \circ \Re \z, (I-H_0) h_x \circ \Re \z \,  \Re \z_\a \right)  
  \right\|}_{W^{N_0/2+3,\infty}} \lesssim  \e_1^2 {(1+t)}^{-1-\g}
\\
\label{Kapp1b}
& {\left\| \partial_\a K - \partial_\a Q_0 \left( (I-H_0)\phi_x \circ \Re \z, (I-H_0) h_x \circ \Re \z \,  \Re \z_\a  \right) 
  \right\|}_{H^{N_0/2+4}} \lesssim \e_1^2 {(1+t)}^{-1/2-\g}  \, .
\end{align}

Let us denote
\begin{align*}
A_1 & := u - (I-H_0)\phi_x \circ \Re \z
\\
A_2 & := \z_\a - 1 - i(I-H_0) h_x \circ \Re \z \,  \partial_\a \Re \z \, .
\end{align*}
Using \eqref{KQ_0} we can write
\begin{align*}
\partial_\a K - \partial_\a Q_0 \left( (I-H_0)\phi_x \circ \Re \z, (I-H_0) h_x \circ \Re \z \,  \partial_\a \Re \z \right) 
  = K_1 + K_2
\end{align*}
where
\begin{align*}
K_1 & := \partial_\a Q_0 \left( A_1, \z_\a -1 \right) 
\\
K_2 & := \partial_\a Q_0 \left( (I-H_0)\phi_x \circ \Re \z, A_2 \right) \, .
\end{align*}
From \eqref{appu1}, \eqref{appz_a-1a} and \eqref{appz_a-1b} we know that there exists $\g>0$ such that
\begin{align}
\label{estA1}
& {\left\| A_1 \right\|}_{W^{N_0/2+5,\infty}} \lesssim  \e_1^2 {(1+t)}^{-1/2-\g} \, ,
\\
\label{estA2a}
& {\left\| A_2 \right\|}_{W^{N_0/2+4,\infty}} \lesssim  \e_1^2 {(1+t)}^{-1/2-\g} \, ,
\\
\label{estA2b}
& {\left\| \H A_2 \right\|}_{W^{N_0/2+3,\infty}} \lesssim  \e_1^2 {(1+t)}^{-1/2-\g} \, .
\end{align}
Using the $L^\infty$ type bound \eqref{estpartialQ_00} for operators of the type $\partial_\a Q_0$, 
\eqref{estA1} above, and the bound in Proposition \ref{proenergy3} together with the a priori decay assumption in \eqref{apriori01},
we see that
\begin{align*}
{\left\| K_1 \right\|}_{W^{N_0/2+3,\infty}} & 
  \lesssim {\left\| A_1 \right\|}_{W^{N_0/2+5,\infty}} \left( {\| \H(\z_\a - 1) \|}_{W^{N_0/2+3,\infty}}
  + {\| \z_\a - 1 \|}_{W^{N_0/2+4,\infty}} \right) \lesssim  \e_1^3 {(1+t)}^{-1-\g} \, .
\end{align*}
Similarly, in view of \eqref{estA2a} and \eqref{estA2b} above, one has
\begin{align*}
{\left\| K_2 \right\|}_{W^{N_0/2+3,\infty}}
  & \lesssim {\left\| (I-H_0)\phi_x \circ \Re \z \right\|}_{W^{N_0/2+5,\infty}} \left( {\| \H A_2 \|}_{W^{N_0/2+3,\infty}} 
  + {\| A_2 \|}_{W^{N_0/2+4,\infty}} \right) 
\\ &\lesssim  \e_1^3 {(1+t)}^{-1-\g} \, .
\end{align*}
We have therefore obtained \eqref{Kapp1a}.
The $H^{N_0/2+4}$ estimate \eqref{Kapp1b} can be obtained similarly,
estimating $K_1$ and $K_2$ in $H^{N_0/2+4}$, by using the bounds on $A_1$ and $A_2$ given by \eqref{appu1}, \eqref{appz_a-1a},
\eqref{estA1}, \eqref{estA2a} and \eqref{estA2b},
and boundedness properties of the Cauchy integral on Sobolev spaces, see \eqref{estHlow}.

\vskip5pt
\paragraph{\it{Step 4: Approximation by $T_0$}.}
Let us denote
\begin{align}
\label{phi_0}
\phi_0 & := \partial_x (I-H_0) \phi
\\
\label{h_0}
h_0 & := \partial_x (I-H_0) h \, .
\end{align}
To eventually obtain \eqref{k_t11}-\eqref{k_t21} it suffices to combine \eqref{Kapp1a}-\eqref{Kapp1b} with
\begin{align}
\label{Kapp2a}
{\left\| \partial_\a Q_0 \left( \phi_0 \circ \Re \z, h_0 \circ \Re \z \,  \Re \z_\a \right)  
  - T_0 (h,\phi) \circ \Re \z \right\|}_{W^{N_0/2+3,\infty}} & \lesssim  \e_1^2 {(1+t)}^{-1-\g}
\\
\label{Kapp2b}
{\left\| \partial_\a Q_0 \left( \phi_0 \circ \Re \z, h_0 \circ \Re \z \,  \Re \z_\a  \right) 
  - T_0 (h,\phi) \circ \Re \z  \right\|}_{H^{N_0/2+4}} & \lesssim  \e_1^2 {(1+t)}^{-1/2-\g}  \, .
\end{align}
With the notation \eqref{phi_0}-\eqref{h_0} we can write $T_0(h,\phi)$ in \eqref{T_0} as
\begin{align}
  \label{T_01}
T_0(h,\phi) \circ \Re \z (\a)  = \partial_x [ \phi_0, H_0 ] h_0 \circ \Re \z (\a) 
  = \frac{1}{i\pi} \partial_x  \int \frac{ \phi_0 (x) - \phi_0 (y) }{ x - y} h_0(y) \, dy \circ \Re \z (\a)  \, ,
\end{align}
whereas writing explicitely $Q_0$ and changing variables we can write
\begin{align}
\nn
 \partial_\a Q_0 \left( \phi_0 \circ \Re \z, h_0 \circ \Re \z \,  \Re \z_\a \right)
  & = \frac{1}{i\pi} \partial_\a  \int \frac{ \phi_0 \circ \Re \z(\a) - \phi_0 \circ \Re \z(\b) }{\z(\a) - \z(\b)} 
  h_0 \circ \Re \z(\b) \,  \Re \z_\b(\b) \, d\b
\\
\label{T_02}
  & = \frac{1}{i\pi} \partial_\a  \left( \int \frac{ \phi_0 (x) - \phi_0 (y) }{ x+ih(x) - (y+ih(y))} h_0(y) \, dy  \circ \Re \z(\a) \right) \, .
\end{align}
Then the difference we are interested in is given by
\begin{align*}
 \partial_\a Q_0 \left( \phi_0 \circ \Re \z, h_0 \circ \Re \z \,  \Re \z_\a \right) - T_0(h,\phi) \circ \Re \z (\a)
  = T_1 \circ \Re \z \, (\Re \z_\a-1) + T_2 \circ \Re \z \, \Re \z_\a \, ,
\end{align*}
where
\begin{align}
\label{T_1}
T_1 & := \frac{1}{i\pi} \partial_x \int \frac{ \phi_0 (x) - \phi_0 (y) }{x - y} h_0(y) \, dy
\\
\label{T_2}
T_2 & := \frac{1}{i\pi} \partial_x  \int H \left( \frac{ h(x)-h(y) }{x-y} \right) 
  \frac{ (h(x)-h(y))(\phi_0 (x) - \phi_0 (y)) }{ {(x - y)}^2 } h_0(y) \, dy \, ,
\end{align}
for some smooth function $H$.
We have expanded the denominator in \eqref{T_02} to obtain the above identity.
Then, since $\Re \z_\a-1$ decays like ${(1+t)}^{-1/2}$ in $L^\infty$, in order to get \eqref{Kapp2a} it suffices to prove
\begin{align}
\label{estT_1}
{\| T_1 \|}_{W^{N_0/2+3,\infty}} & \lesssim \e_1^2 {(1+t)}^{-1/2-\g}
\\
\label{estT_2}
{\| T_2 \|}_{W^{N_0/2+3,\infty}} & \lesssim \e_1^2 {(1+t)}^{-1-\g} \, .
\end{align}
Applying $\partial_x$ to the integrand in \eqref{T_1}, and using \eqref{estinfty1}, we see that
\begin{align*}
{\| T_1 \|}_{W^{N_0/2+3,\infty}} & \lesssim {\left\| \int \frac{ \phi_0 (x) - \phi_0 (y) }{{(x - y)}^2} h_0(y) \, dy \right\|}_{W^{N_0/2+3,\infty}}
    + {\left\| \partial_x \phi_0 \, H_0 h_0 \right\|}_{W^{N_0/2+3,\infty}}
\\
& \lesssim  {\left\| \phi_0 \right\|}_{W^{N_0/2+5,\infty}} 
    \left( {\left\| h_0 \right\|}_{W^{N_0/2+4,\infty}} + {\left\| H_0 h_0 \right\|}_{W^{N_0/2+4,\infty}} \right)
\\
& \lesssim {\left\| \Lambda \phi \right\|}_{W^{N_0/2+6,\infty}} {\left\| h \right\|}_{W^{N_0/2+5,\infty}} \lesssim \e_1^2 {(1+t)}^{-1} \, .
\end{align*}
Similarly, again using \eqref{estinfty1}, it is not hard to see that
\begin{align*}
{\| T_2 \|}_{W^{N_0/2+3,\infty}} & 
  \lesssim  {\left\| h \right\|}_{W^{N_0/2+5,\infty}} {\left\| \phi_0 \right\|}_{W^{N_0/2+5,\infty}} 
    {\left\| h_0 \right\|}_{W^{N_0/2+4,\infty}}
  \lesssim \e_1^3 {(1+t)}^{-3/2} \, .
\end{align*}
This gives us \eqref{estT_2} and concludes the proof of \eqref{Kapp2a}.

The remaining estimate \eqref{Kapp2b} can be obtained similarly, 
using the $L^2$ bounds of Theorem \ref{theoCMM} instead of the $L^\infty$ bound \eqref{estinfty1}.
In particular it suffices to show
\begin{align*}
{\| T_1 \|}_{H^{N_0/2+4}} \lesssim \e_1^2 {(1+t)}^{-\g}
\qquad \mbox{and} \qquad
{\| T_2 \|}_{W^{N_0/2+3,\infty}} \lesssim \e_1^2 {(1+t)}^{-1/2-\g} \, .
\end{align*}
We only detail the bound for $T_1$ as the one for $T_2$ can proved similarly 
(notice that $T_2$ is a cubic term and therefore its bounds are ${(1+t)}^{-1/2}$ better than those of $T_1$).
Applying $\partial_x$ to the integrand in \eqref{T_1}, commuting derivatives via \eqref{commK2},
and using the estimates in Theorem  \ref{theoCMM}, one sees that
\begin{align*}
{\| T_1 \|}_{H^{N_0/2+4}} & 
  \lesssim {\left\| \int \frac{ \phi_0 (x) - \phi_0 (y) }{{(x - y)}^2} h_0(y) \, dy \right\|}_{H^{N_0/2+4}} 
   + {\left\| \partial_x \phi_0 \, H_0 h_0 \right\|}_{H^{N_0/2+4}} 
\\
& \lesssim {\left\| \phi_0 \right\|}_{H^{N_0/2+5}} 
  \left( {\left\| h_0 \right\|}_{W^{N_0/2+4,\infty}} + {\left\| H_0 h_0 \right\|}_{W^{N_0/2+4,\infty}} \right)
  \lesssim \e_1^2 {(1+t)}^{-1/2}
\end{align*}
which is more than sufficient.  $\hfill \Box$

\subsection{Proof of Lemma \ref{lemT_0}}\label{seclemT_0}
The proof proceeds in several steps.

\subsubsection{Step 1: The operator in Fourier space}
Introduce the notations
\begin{align*}
\wt{\phi} & := (I-H_0) \phi
\, , \quad 
\wt{h} := (I-H_0) h
\end{align*}
so that from \eqref{T_0} we can write
\begin{align*}
 \partial_x T_0(h,\phi) = \partial_x [ \wt{\phi}_x ,H_0] \wt{h}_x \, .
\end{align*}
By taking Fourier transform, and using the notation \eqref{al0}, we see that
\begin{align}
\label{T_0T_m}
 \partial_x T_0(h,\phi) = M \big( \wt{h}, \wt{\phi} \big) 
\end{align}
where the symbol of the operator $M$ is given by
\begin{align}
\label{symT_0}
m(\xi,\eta) = \xi\eta|\xi-\eta| -|\xi| \eta (\xi-\eta) \, .
\end{align}
We then want to show
\begin{align}
\label{boundT_m1}
& {\big\| M \big( \wt{h}, \wt{\phi} \big) \big\|}_{W^{N_0/2+3,\infty}} \lesssim  \e_1^2 {(1+t)}^{-1-\g}
\\
\label{boundT_m2}
& {\big\| M \big( \wt{h}, \wt{\phi} \big) \big\|}_{H^{N_0/2+4}} \lesssim  \e_1^2 {(1+t)}^{-1/2-\g}
\end{align}
for some $\g>0$, under the a priori assumptions \eqref{apriori01} and \eqref{aprioriL11}.

\subsubsection{Step 2: Approximation}
Let $H$ and $\Psi$ be the functions defined in \eqref{HPsi0} in Proposition \ref{proE1}.
Define
\begin{align*}
\wt{H} := (I-H_0) H
\, , \quad
\wt{\Psi} := (I-H_0) \Psi \, .
\end{align*}
We then claim that the following hold true:
\begin{align}
\label{T_mHPsi1}
& {\big\| M \big( \wt{h}, \wt{\phi} \big) - M \big( \wt{H}, \wt{\Psi} \big) \big\|}_{W^{N_0/2+3,\infty}} \lesssim \e_1^2 {(1+t)}^{-1-\g}  \, ,
\\
 \label{T_mHPsi2}
& {\big\| M \big( \wt{h}, \wt{\phi} \big) - M \big( \wt{H}, \wt{\Psi} \big) \big\|}_{H^{N_0/2+4}} \lesssim \e_1^2 {(1+t)}^{-1/2-\g} \, .
\end{align}
Using \eqref{HPsi0} we see that
\begin{align}
\label{T_mHPsi4}
M \big( \wt{H}, \wt{\Psi} \big) - M \big( \wt{h}, \wt{\phi} \big) & 
  = M \big( (I-H_0)A, \wt{\phi} \big) + M \big( \wt{H}, (I-H_0) B \big) \, .
\end{align}
From the definition of $m$ in \eqref{symT_0}, the definitions \eqref{Al4} and \eqref{al11},
we see that 
\begin{equation}
\label{descm}
{\|m^{k,k_1,k_2}\|}_{\mathcal{S}^\infty} \lesssim 2^k 2^{k_1} 2^{k_2} \, .
\end{equation}
Combining this with the $L^2$ and $L^\infty$ bounds on $A$ and $B$ in \eqref{Al1000} and \eqref{Al1001}, 
the a priori bounds \eqref{Al32},
and using \eqref{mk6} in Lemma \ref{touse}, one can easily verify that
\begin{align*}
& {\big\| M \big( (I-H_0)A, \wt{\phi} \big) \big\|}_{W^{N_0/2+3,\infty}} \lesssim 
  \e_1^3 {(1+t)}^{-5/4}
\\
& {\big\| M \big( \wt{H}, (I-H_0) B \big) \big\|}_{W^{N_0/2+4,\infty}} \lesssim 
  \e_1^3 {(1+t)}^{-5/4} \, .
\end{align*}
These and \eqref{T_mHPsi4} give us \eqref{T_mHPsi1}.
Similarly one can show \eqref{T_mHPsi2} again by using \eqref{descm}, Lemma \ref{touse}, \eqref{Al32}
and the bounds on the Sobolev norms of $A$ and $B$ provided by \eqref{Al1000}.

\subsubsection{Step 3: Reduction to bilinear estimates}
We are left with proving
\begin{align}
\label{T_mHPsi10}
& {\big\| M \big( \wt{H}, \wt{\Psi} \big) \big\|}_{W^{N_0/2+3,\infty}} \lesssim \e_1^2 {(1+t)}^{-1-\g}  \, ,
\\
\label{T_mHPsi11}
& {\big\| M \big( \wt{H}, \wt{\Psi} \big) \big\|}_{H^{N_0/2+4}} \lesssim \e_1^2 {(1+t)}^{-1/2-\g} \, .
\end{align}
Recalling the definition of $V = H+i\Lambda\Psi$ from \eqref{V0},
we have $\widetilde{V}_\pm = \wt{H} \pm i\Lambda \wt{\Psi}$,
so that $\wt{H} = (\wt{V}_+ + \wt{V}_-)/2$ and $\wt{\Psi} = (i\Lambda)^{-1} (\wt{V}_+ - \wt{V}_-)/2$.
We then have
\begin{align}
M \big( \wt{H},\wt{\Psi} \big) = \sum_{\e_1,\e_2 \in \{+,-\}} c_{\e_1,\e_2} Q \big( \wt{V}_{\e_1}, \wt{V}_{\e_2} \big) 
\end{align}
where, again according to the notation \eqref{al0},
\begin{align}
\label{symq}
 q (\xi,\eta) = {|\eta|}^{-1/2} \big[ \xi\eta|\xi-\eta| -|\xi| \eta (\xi-\eta) \big] \, ,
\end{align}
and $c_{\e_1,\e_2}$ are some constants.
With the notation \eqref{Al4} and \eqref{al11} we have
\begin{equation}
\label{descq}
{\|q^{k,k_1,k_2}\|}_{\mathcal{S}^\infty} \lesssim 2^k 2^{k_1} 2^{k_2/2} \, .
\end{equation}

From the definition of $\wt{V}$ and the bounds provided by \eqref{vhphiZp}, \eqref{vhphiSobolev} and  \eqref{xd_xF}
on $V = H + i \Lambda \Psi$, we see that
the desired bounds \eqref{T_mHPsi10}-\eqref{T_mHPsi11} reduce to showing the following bilinear estimates:
\begin{align}
\label{T_mHPsi20}
& {\left\| Q (v_\pm(t), v_\pm(t)) \right\|}_{W^{N_0/2+3,\infty}} \lesssim  \e_1^2 {(1+t)}^{-1-\g}  \, ,
\\
\label{T_mHPsi21}
& {\left\| Q (v_\pm(t), v_\pm(t)) \right\|}_{H^{N_0/2+4}} \lesssim  \e_1^2 {(1+t)}^{-1/2-\g} \, ,
\end{align}
where $q$ is as in \eqref{symq}, and $v_\pm = e^{\mp it\Lambda}f_\pm$ satisfies, for all $k \in \Z$,
\begin{align}
 \label{boundf1}
& {\| P_k v_\pm(t) \|}_{W^{N_1,\infty}} \lesssim \e_1 {(1+t)}^{-1/2} \, ,
\\
 \label{boundf2}
& {\| P_k f_\pm(t) \|}_{H^{N_0-10}} \lesssim \e_1 {(1+t)}^{5p_0} \, ,
\\
 \label{boundf3}
& {\| x \partial_x P_k f_\pm(t) \|}_{H^{\frac{N_0}{2}-20}} \lesssim \e_1 {(1+t)}^{5p_0} 
\, .
\end{align}

\subsubsection{Step 4: Proof of the bilinear estimates \eqref{T_mHPsi20}-\eqref{T_mHPsi21}}

\vskip5pt
\paragraph{{\it Step 1: Frequency decomposition}}
It suffices to show that for all $t \in [2^m-2, 2^{m+1}]$ and $m \in \{1,2,\dots\}$, there exists a constant $\g>0$ such that
\begin{align}
\label{T_mHPsi100}
& \sum_{k,k_1,k_2} {\left\| P_k Q (P_{k_1} v_\pm(t), P_{k_2} v_\pm(t)) \right\|}_{W^{N_0/2+3,\infty}} \lesssim \e_1^2 2^{-(1+\g)m} \, ,
\\
\label{T_mHPsi101}
& \sum_{k,k_1,k_2} {\left\| P_k Q (P_{k_1} v_\pm(t), P_{k_2} v_\pm(t)) \right\|}_{H^{N_0/2+4}} \lesssim \e_1^2 2^{-(1/2+\g)m} \, .
\end{align}
By symmetry and conjugation it is clear that we can reduce matters to 
proving the estimates \eqref{T_mHPsi100}-\eqref{T_mHPsi101} for the two bilinear operators
\begin{align}
\label{B_+}
 T_+(f,f)(t) & := \F^{-1} \int_{\R} e^{it\Phi_+} q(\xi,\eta) \what{f}_+(t,\xi-\eta) \what{f}_+(t,\eta) \, d\eta \, ,
\\
\label{B_-}
 T_-(f,f)(t) & := \F^{-1} \int_{\R} e^{it\Phi_-} q(\xi,\eta) \what{f}_+(t,\xi-\eta) \what{f}_-(t,\eta) \, d\eta \, ,
\end{align}
where $f = (f_+, f_-)$, and 
\begin{align}
\label{Phiq}
\Phi_\pm (\xi,\eta) = {|\eta|}^{1/2} \pm  {|\xi - \eta|}^{1/2} 
  \, , \quad  q (\xi,\eta) = {|\eta|}^{-1/2} \left[ \xi\eta|\xi-\eta| -|\xi| \eta (\xi-\eta) \right] .
\end{align}
Notice that under the a priori assumptions these bilinear terms have decay rates which barely fail to give \eqref{T_mHPsi100} and \eqref{T_mHPsi101}.
The key to obtaining the extra necessary decay is the vanishing of the symbol $q(\xi,\eta)$ on the space-resonant sets, 
i.e. for those $(\xi,\eta)$ such that $\partial_\eta \Phi_\pm(\xi,\eta) = 0$.
One can then use integration by parts in frequency, and the weighted bound \eqref{boundf3}, to derive the desired estimate.

First let us observe that using \eqref{descq}, Lemma \ref{touse}, and \eqref{boundf1}-\eqref{boundf3},
one can bound as desired the sums in \eqref{T_mHPsi100} and \eqref{T_mHPsi101}, for all those frequencies
$(k,k_1,k_2)$ such that $\min(k,k_1,k_2) \leq -m/N_0$ and $\max(k,k_1,k_2) \geq 3m/N_0$.
The remaining sums have only $C m^3$ terms.
Therefore it suffices to show the estimates for each $(k,k_1,k_2)$ fixed satisfying
\begin{align}
\label{freqQ}
k,k_1,k_2 \in [-m/N_0,3m/N_0] \cap \Z \quad , \quad \max(k_1, k_2) \geq k - 10 \, . 
\end{align}

\vskip5pt
\paragraph{{\it Step 2: Spatial decomposition}} 
Let us define $\rho:\mathbb{R}\to[0,1]$ to be an even compactly supported function which equals $1$ on $[0,1]$ and vanishes on $[2,\infty)$. 
Let $R := 2^{3m/4}$. We decompose the profiles $f = (f_+,f_-)$ into two pieces: 
$f = f_{\geq R} + f_{\leq R} = (f_{+,\geq R}, f_{-,\geq R}) + (f_{+,\leq R}, f_{-,\leq R})$ where
\begin{align*}
 f_{\leq R}(x) = f(x) \rho \left( \frac{x}{R} \right) \quad , \quad f_{\geq R}(x) = f(x) - f_{\leq R}(x) \, .
\end{align*}
We then want to show that for all $t \in [2^m-2, 2^{m+1}]$, $m \in \{1,2,\dots\}$,
and $k,k_1,k_2 \in [-m/N_0,3m/N_0] \cap \Z$
\begin{align}
\label{Tpm1} 
& 
{\left\| P_k T_\pm (P_{k_1} f_{\geq R}(t), P_{k_2} f(t)) \right\|}_{L^\infty} + 
{\left\| P_k T_\pm (P_{k_1} f_{\leq R}(t), P_{k_2} f_{\geq R}(t)) \right\|}_{L^\infty} \lesssim  \e_1^2 2^{-(1+\g)m} 2^{-(N_0/2+3)k_+}
  \, ,
\\
\label{Tpm2}
& {\left\| P_k T_\pm (P_{k_1} f_{\leq R}(t), P_{k_2} f_{\leq R}(t)) \right\|}_{L^\infty} \lesssim  \e_1^2 2^{-(1+\g)m} 2^{-(N_0/2+3)k_+} \, .
\end{align}
We also need to prove the $H^{N_0/2+4}$ versions of the above estimates
corresponding to \eqref{T_mHPsi101}, but since those can be obtained analogously we will skip them.

\vskip5pt
\paragraph{{\it Step 3: Proof of \eqref{Tpm1}}}
First notice that both terms in \eqref{Tpm1} have (at least) one profile supported at a distance $R$ from the origin. 
Since this is the only important aspect that we will use to gain the necessary decay,
we only show how to estimate one of the terms, the other being analogous.
We then want to prove
\begin{align*}
{\left\| P_k T_\pm (P_{k_1} f(t), P_{k_2} f_{\geq R}(t)) \right\|}_{L^\infty} \lesssim \e_1^2 2^{-9m/8} 2^{-(N_0/2+3)k_+}
\end{align*}
for any $k,k_1,k_2 \in [-m/N_0, 3m/N_0] \cap \mathbb{Z}$, and any $t \in [2^m-2, 2^{m+1}]$.
Using \eqref{descq}, Sobolev's embedding, and the bounds \eqref{boundf1} and \eqref{boundf3}, we see that
\begin{align*}
{\left\| P_k T_\pm (P_{k_1} f(t), P_{k_2} f_{\geq R}(t)) \right\|}_{L^\infty} 
  & \lesssim 2^k 2^{k_1} 2^{k_2/2} {\left\| P_{k_1} v_\pm(t) \right\|}_{L^\infty} {\left\| e^{\pm it\Lambda} P_{k_2} f_{\geq R}(t) \right\|}_{L^\infty}  
\\
& \lesssim 2^k 2^{k_1} \e_1 2^{-m/2} 2^{-N_1 \max(k_1,0)} {\left\| \partial_x f_{\geq R}(t) \right\|}_{L^2}  
\\
& \lesssim \e_1^2 2^k 2^{k_1} 2^{-m/2} R^{-1} 2^{-(N_0/2-20) \max(k_1,k_2,0)} 2^{mp_0}
\\
& \lesssim \e_1^2 2^{-9m/8} 2^{-(N_0/2+3)k_+} \, ,
\end{align*}
in view of the frequency constraints \eqref{freqQ}, $p_0 \leq 1/1000$ and $R = 2^{3m/4}$.

\vskip5pt
\begin{proof}[Step 4: Proof of \eqref{Tpm2}]
It suffices to show
\begin{align*}
{\left\| P_k T_\pm (P_{k_1} f_{\leq R}(t), P_{k_2} f_{\leq R}(t)) \right\|}_{L^\infty} \lesssim \e_1^2 2^{-9m/8} 2^{-(N_0/2+3)k_+} \, ,
\end{align*}
for any $k,k_1,k_2$ as in \eqref{freqQ} and $t \in [2^m-2, 2^{m+1}]$.
We distinguish the two cases of $T_+$ and $T_-$.
In the first case we introduce and extra cutoff in $\xi - 2\eta$ by writing
\begin{align}
\begin{split}
 \label{T+l} 
& P_k T_+ ( P_{k_1} f_{\leq R}(t), P_{k_2} f_{\leq R}(t) )
  = \sum_{l \in \Z} P_k T^l_+ ( P_{k_1} f_{\leq R}(t), P_{k_2} f_{\leq R}(t) )
\\
  & := \sum_{l \in \Z} \F^{-1} \int_{\R} e^{it\Phi_+(\xi,\eta)} \varphi_{l}(\xi-2\eta) 
    \varphi_{k}(\xi) \varphi_{k_1}(\xi-\eta) \varphi_{k_2}(\eta) q(\xi,\eta) \what{f_{+,\leq R}}(t,\xi-\eta) \what{f_{+,\leq R}}(t,\eta) \, d\eta
\end{split}
\end{align}
where $\Phi_+$ is as in \eqref{Phiq}.
First notice that the contribution in the summation over $l$ in \eqref{T+l} is zero if $l \geq 3m/N_0 + 100$.
\eqref{T+l} also vanishes if $l \leq -m/N_0 - 100$, because in this case $|\xi - 2\eta| \leq |\eta|/10$,
which implies that $\xi - \eta$ and $\xi$ have the same sign, and therefore $q(\xi,\eta) = 0$, see \eqref{Phiq}.
The summation over $l$ can then be disregarded and it is enough to show
\begin{align*}
{\left\| P_k T_+^l (P_{k_1} f_{\leq R}(t), P_{k_2} f_{\leq R}(t)) \right\|}_{L^\infty} \lesssim \e_1^2 2^{-5m/4} 2^{-(N_0/2+3)k_+}
\end{align*}
for any $l,k,k_1,k_2 \in [-m/N_0 - 100, 3m/N_0 + 100] \cap \mathbb{Z}$ and $t \in [2^m-2, 2^{m+1}]$.
Observe that for any integer $j$ we have
\begin{align*}
{\big| \partial_\eta^j \partial_\eta \Phi_+(\xi,\eta) \big|} 
  \gtrsim 2^{- j m/10}
\end{align*}
and
\begin{align*}
{\big\|\partial_\eta^j \varphi_{k_i}(\cdot) \what{f_{\leq R}}(t,\cdot) \big\|}_{L^1} \lesssim R^j 2^{k_i/2} 2^{mp_0} 
\quad , \quad \left|\partial_\eta^j ( q(\xi,\eta)  \varphi_{l}(\xi-2\eta) )\right| \lesssim 2^{(j + 5/2)m/N_0} \, .
\end{align*}
Integrating by parts $L$ times in $\eta$ in the integral in \eqref{T+l} we see that
\begin{align*}
{\big\| P_k T_+^l (P_{k_1} f_{\leq R}(t), P_{k_2} f_{\leq R}(t)) \big\|}_{L^\infty} 
  & \lesssim {\left( \frac{1}{2^m 2^{-m/10}} \right)}^{L} R^{L} \, 2^{k_1/2} 2^{k_2/2} 2^{2mp_0} 2^{- (N_0 - 10)\max(k_1,k_2,0)} \e_1^2
\\
& \lesssim \e_1^2 2^{-3m/2} 2^{- (N_0/2 + 3)k_+} \, ,
\end{align*}
where the second inequality holds provided $L$ is large enough.

In the case of the operator $T_-$ it is enough to observe that $\partial_\eta \Phi_-$ vanishes (linearly) only when $\xi = 0$, 
and, also in this case, we have
\begin{align*}
{|\partial_\eta^j \partial_\eta \Phi_-|} 
  \gtrsim 2^{- j m/10} \, ,
\end{align*}
under the frequency constraints \eqref{freqQ}.
One can then use the same integration by parts argument as above (without the need to resort to a further splitting) and obtain
\begin{align*}
{\left\| P_k T_- (P_{k_1} f_{\leq R}(t), P_{k_2} f_{\leq R}(t)) \right\|}_{L^\infty} \lesssim \e_1^2 2^{-3m/2} 2^{- (N_0/2+3)k_+} \, .
\end{align*}
This concludes the proof of \eqref{Tpm2} and therefore of Lemma \ref{lemT_0}.
\end{proof}


\section{Proof of Proposition \ref{proLE}: transition to Eulerian coordinates}\label{secproLE}
Here we want to transfer the a priori bounds from the modified Lagrangian coordinates to Eulerian coordinates.
Recall that
\begin{equation*}
 \wt{L}(t,\a) = ( \z_\a(t,\a) - 1, u(t,\a), w(t,\a), \Im \z(t,\a) ) \, .
\end{equation*}
Also, recall the a priori assumptions \eqref{aprioriL1} and \eqref{apriori0}, that is
\begin{align}
 \label{aprioriLE1}
\sup_{t\in[0,T]} \left(  (1+t)^{-p_0} {\| \wt{L}(t) \|}_{X_{N_0}} +  {\| \wt{L}(t) \|}_{H^{N_1+5}} + 
   \sqrt{1+t} {\| \wt{L}(t) \|}_{W^{N_1,\infty} } \right) \leq \e_1 \, ,
\end{align}
and
\begin{align}
\label{aprioriLE2}
\sup_{[0,T]} \left( (1+t)^{-p_0} {\| (h(t), \phi_x(t)) \|}_{X_{N_0}} + {\| h(t) + i\Lambda \phi(t)  \|}_{H^{N_1 + 10}}
  + \sqrt{1+t} {\| (h(t), \Lambda \phi(t)) \|}_{W^{N_1+4,\infty}} \right) \leq \e_1 \, ,
\end{align}
where
${\| f \|}_{X_{N_0}} := {\| f \|}_{H^{N_0}} + {\| Sf \|}_{H^{N_0/2}}$.
To prove Proposition \ref{proLE} we need to show, under the above a priori assumptions, that
\begin{align}
\label{LE3}
& {\| \wt{L}(t) \|}_{W^{N_1,\infty}} \lesssim {\| (h(t) , \partial_x \phi(t)) \|}_{Z^\p}
\\
\label{LE2}
& {\| \wt{L}(t) \|}_{H^{N_1+5}}  \lesssim {\| (h(t) , \partial_x \phi(t)) \|}_{H^{N_1+7}} + \e_1^2
\\
\label{LE1}
& {\| (h(t) , \partial_x \phi(t)) \|}_{X_{N_0}}  \lesssim {\| \wt{L}(t) \|}_{X_{N_0}} \, .
\end{align}

\subsection{Proof of \eqref{LE3}}
From Proposition \ref{proenergy3} we have
\begin{align*}
{\| L(t) \|}_{W^{N_1,\infty}} \lesssim  {\| (h(t),\partial_x \phi(t)) \|}_{Z^\p}
\end{align*}
which is stronger than \eqref{LE3}.

\subsection{Proof of \eqref{LE2}}
Using $h (\Re \z) = \Im \z$ it is clear that
\begin{align*}
{\| \Im \z \|}_{H^{N_1+5}} & \lesssim {\| h \|}_{H^{N_1+5}}  \, .
\end{align*}
Also, from \eqref{appL-}-\eqref{Qlemenergy31} in Lemma \ref{lemenergy31}, and the a priori assumptions \eqref{aprioriL13}, we see that
\begin{align*}
{\| (\z_\a(t)-1,u(t),w(t)) \|}_{H^{N_1+5}} & \lesssim {\| (\partial_\a \chi(t), \partial_\a \l(t)) \|}_{H^{N_1+5}} 
  + {\| \wt{L} (t) \|}_{H^{N_1+6}} {\| \wt{L} (t) \|}_{W^{\frac{N_1+5}{2},\infty}} 
\\
& \lesssim {\| (\partial_\a \chi(t), \partial_\a \l(t)) \|}_{H^{N_1+5}} + \e_1 {(1+t)}^{p_0} \e_1 {(1+t)}^{-1/2}
\\
&	\lesssim {\| (\partial_\a \chi(t), \partial_\a \l(t)) \|}_{H^{N_1+5}} + \e_1^2 \, .
\end{align*}
Thus, to obtain \eqref{LE2} it is enough to show
\begin{align}
\label{LE2chil}
{\| (\partial_\a \chi(t), \partial_\a \l(t)) \|}_{H^{N_1+5}} & \lesssim {\| (h(t) , \partial_x \phi(t)) \|}_{H^{N_1+7}} + \e_1^2  \, .
\end{align}

From the definition of $\chi$ in \eqref{defchi}, and $\Im \z = h \circ \Re \z$, we see that 
\begin{align*}
\chi = (I-\H) (\z-\bar{\z}) = 2i (I-\H) \Im \z = 2i (I-\H) (h \circ \Re\z) \, .
\end{align*}
Using the bounds on the Hilbert transform \eqref{estH} with the a priori assumptions \eqref{aprioriLE1}-\eqref{aprioriLE2}, we get
\begin{align*}
{\| \partial_\a \chi(t) \|}_{H^{N_1+5}} & 
  \lesssim {\| h \circ \Re \z(t) \|}_{H^{N_1+6}} + \e_1^2  \lesssim {\| h(t) \|}_{H^{N_1+6}} + \e_1^2 \,
\end{align*}
which gives the bound \eqref{LE2chil} for the component $\chi_\a$.

From \eqref{Flemenergy32} and \eqref{R_f} in the proof of Lemma \ref{lemenergy32},
we see that
\begin{align*}
{\| \partial_\a \l \|}_{H^{N_1+5}} \lesssim {\| \partial_x [\phi(1+h^\p)] \|}_{H^{N_1+5}} + {\| R_\phi \|}_{H^{N_1+6}} 
  \lesssim {\| \partial_x \phi \|}_{H^{N_1+5}} + {\| \phi h^\p \|}_{H^{N_1+6}} + {\| R_\phi \|}_{H^{N_1+6}} \, ,
\end{align*}
where we recall
\begin{align*}
R_\phi (t,x)  := \frac{1}{i\pi} \int H\left( \frac{h(x) - h(y)}{x-y} \right) \frac{h(x) - h(y)}{ {(x-y)}^2 } \phi(t,y) \, dy  \, .
\end{align*}
Applying $N_1 + 6$ derivatives to the above expression, commuting them via \eqref{commK2}, and using Theorem \ref{theoCMM}, one can obtain
\begin{align*}
{\| R_\phi \|}_{H^{N_1+6}} \lesssim  {\| h \|}_{H^{N_1+7}} {\| \phi \|}_{W^{N_1+4,\infty}}
  +  {\| h \|}_{W^{N_1+4,\infty}} {\| \phi_x \|}_{H^{N_1+5}} \lesssim  \e_1^2  \, .
\end{align*}
Here we have used the a priori assumptions \eqref{aprioriLE2}
and the bound \eqref{estcorproE4} in Corollary \ref{corproE4}. 
A similiar estimate can be easily obtained for ${\| \phi h^\p \|}_{H^{N_1+6}}$.
It follows that 
\begin{align*}
{\| \partial_\a \l \|}_{H^{N_1+6}} \lesssim  {\| (h, \phi _x) \|}_{H^{N_1+7}} + \e_1^2 \, ,
\end{align*}
which gives \eqref{LE2chil}, completing the proof of \eqref{LE2}.

\subsection{Proof of \eqref{LE1}}
To show \eqref{LE1} we will exploit the identities
\begin{align*}
h(t, \Re \z(t,\a)) = \Im \z (t,\a)
\qquad \mbox{and}
\qquad (I-\H_\z) \phi(t, \Re \z(t,\a)) = \l (t,\a)
\end{align*}
to prove the following:
\begin{lem}\label{lemLE1}
Let $S$ be the scaling vector field. Then
\begin{align}
\label{LE10}
& {\| h(t) \|}_{X_{N_0}}  \lesssim {\| \Im\z(t) \|}_{X_{N_0}} + {\| \z_\a(t) - 1 \|}_{X_{N_0}}
\end{align}
and
\begin{align}
\label{LE20}
& {\| \partial_x \phi(t) \|}_{X_{N_0}}  \lesssim {\| \l_\a(t) \|}_{X_{N_0}} + {\| \z_\a(t) - 1 \|}_{X_{N_0}} \, .
\end{align}
\end{lem}
The estimates in the above Lemma, together with \eqref{chivl2}, imply \eqref{LE1}.
To prove the estimates involving $S$ we will need the two auxiliary Lemmas below:

\begin{lem}
\label{lemcomp}
Let $S$ be the scaling vector field. 
Then for any two functions $f: \R_t \times \R_x \rightarrow \R$ and $g: \R_t \times \R_\alpha \rightarrow \R_x $, the following formula holds:
\begin{equation}
\label{Sformula}
S f \circ g  =  S (f \circ g) - (f^\p \circ g) (S g - g) \, .
\end{equation}
Here, for a function $c:\R_t \times \R_x \rightarrow \R$, the operation $c \circ g$ is to be understood as composition in the space variable:
$$ (c \circ g) (t,\alpha) = c(t,g(t,\alpha)) \, ,$$
and $f^\p$ denotes the derivative with respect to the space variable.
\end{lem}
The proof of the above statement is by direct computation.

\begin{lem}
\label{lemRez}
Let $\z$ be the Lagrangian map in the modified Lagrangian coordinates, then
\begin{equation}
\label{Rez-a}
(I-\H)(\Re \z - \a ) = i (I-\H) \Im \z \, . 
\end{equation}
Under the a priori assumptions \eqref{aprioriLE1}-\eqref{aprioriLE2}, it follows that
\begin{align}
\label{Rez-aX}
{\| \Re \z - \a  \|}_{X_{N_0}} &  \lesssim  {\| \Im \z \|}_{X_{N_0}} + {\| \z_\a-1 \|}_{X_{N_0}} \, .
\end{align}
In particular
\begin{align}
\label{SRez-RezL^2}
{\| S \Re \z - \Re \z  \|}_{H^{\frac{N_0}{2}}} &  \lesssim  {\| \Im \z \|}_{X_{N_0}} + {\| \z_\a-1 \|}_{X_{N_0}} \lesssim {\| \wt{L} \|}_{X_{N_0}} \, .
\end{align}
\end{lem}
\begin{proof}
The identity \eqref{Rez-a} follows from the fact that $(I-\H_\z)(\bar{\z}-\a) = 0$, which
comes from the identity $(I-\H_z)(k-\bar{z}) = 0$ upon composition with $k^{-1}$.

The estimate \eqref{Rez-aX} follows from an application of Lemma \ref{lemI-Hf} with $f = \Re\z -\a$ and $g = i (I-\H) \Im \z$,
Sobolev's embedding, and from the boundedness properties of $\H$ in Lemma \ref{lemHf}:
\begin{align*}
& {\| \Re \z - \a \|}_{X_{N_0}} \lesssim  {\| (I-\H) \Im \z \|}_{X_{N_0}} + {\| \z_\a - 1 \|}_{X_{N_0}} {\| (I-\H) \Im \z \|}_{H^{N_1}}
\\
& \lesssim  {\| \Im \z \|}_{X_{N_0}} + {\| \z_\a - 1 \|}_{X_{N_0}} ( {\|\Im \z \|}_{H^{N_1}} + {\| (I-\H) \Im \z \|}_{H^{N_1}} )
  \lesssim {\| \Im \z \|}_{X_{N_0}} + {\| \z_\a - 1 \|}_{X_{N_0}} \, .
\end{align*}
\eqref{SRez-RezL^2} immediately follows since $S\Re\z - \Re \z = S(\Re\z - \a) + (\a -\Re\z)$.
\end{proof}

\vskip15pt
\begin{proof}[Proof of Lemma \ref{lemLE1}]
Recall that $h(t, \Re \z(t,\a)) = \Im \z (t,\a)$.
Since ${\| \z_\a - 1 \|}_{H^{N_1+5}} \leq \e_1$, in particular we see that for $\e_1 \leq 1/2$, 
the map $\Re \z$ is a diffeomorphism with  
\begin{equation*}
\big| \partial_\a \Re \z \big| \geq 1/2 \, , \quad \big| \partial_\a^k \Re \z \big| \leq 3/2 
\end{equation*} 
for $1 \leq k \leq N_1$.
It immediately follows that
\begin{align*}
{\| h \|}_{H^{N_0}} = {\| \Im \z \circ \Re \z^{-1} \|}_{H^{N_0}} 
  & \lesssim {\| \Im \z(t) \|}_{H^{N_0}} + {\| \Re \z_{\a\a}(t) \|}_{H^{N_0-1}} 
  \lesssim {\| \Im \z \|}_{H^{N_0}} + {\| \z_\a (t) - 1\|}_{H^{N_0}} \, .
\end{align*}
This takes care of the Sobolev component of the norm to bound in \eqref{LE10}.
To estimate the weighted component we apply Lemma \ref{lemcomp} to get
\begin{equation}
S h \circ \Re \z  =  S \Im \z - (h^\p \circ \Re\z) (S \Re\z - \Re\z) \, .
\end{equation}
Using \eqref{SRez-RezL^2} we have
\begin{align*}
{\| S h \|}_{H^\frac{N_0}{2}} & 
  \lesssim  {\| S \Im \z \|}_{H^\frac{N_0}{2}} + {\| h^\p \|}_{H^\frac{N_0}{2}} {\| S \Re\z - \Re\z \|}_{H^\frac{N_0}{2}} 
  \lesssim {\| \Im \z \|}_{X_{N_0}} + {\| \z_\a-1 \|}_{X_{N_0}} \, .
\end{align*}
This gives us \eqref{LE10}.

Recall the relation between the trace of the velocity potential in Eulerian variables and in modified Lagrangian variables $\l$:
\begin{align}
\label{phiEL1}
(I-\H_\z) \phi(t, \Re \z(t,\a)) & = \l (t,\a) \, .
\end{align}
Since
\begin{align*}
[\partial_\a, \H] f = [\z_\a,\H] \frac{f_\a}{\z_\a} = \frac{1}{i\pi} Q_0 (\z_\a -1, f_\a) \, ,
\end{align*}
where $Q_0$ is defined in \eqref{Q_0}, it follows that
\begin{align}
\label{phi_xEL1}
(I-\H) \partial_\a (\phi \circ \Re \z) & = \partial_\a \l + Q_0 \left(\z_\a-1,  \partial_\a (\phi \circ \Re \z) \right) \, .
\end{align}
Denoting $f = \phi \circ \Re \z$ and $g = \partial_\a \l + \frac{1}{i\pi} Q_0 \left(\z_\a-1,  \partial_\a (\phi \circ \Re \z) \right)$ we have
$(I-\H) \partial_\a f = g$,  so that \eqref{est1-Hflow} gives
\begin{align}
\label{lemLE2est1}
{\| \partial_\a f \|}_{H^{N_1}} & \lesssim {\| g \|}_{H^{N_1}}
  \lesssim  {\| \partial_\a \l \|}_{H^{N_1}} + {\| Q_0 (\z_\a-1, \partial_\a f) \|}_{H^{N_1}} \, ,
\end{align}
while \eqref{est1-Hf} gives
\begin{align}
\label{lemLE2est2}
\begin{split}
{\| \partial_\a f \|}_{X_{N_0}} & \lesssim {\| g \|}_{X_{N_0}} + {\| \z_\a -1 \|}_{X_{N_0}} {\| g \|}_{H^{N_1}}
  \lesssim  {\| \partial_\a \l \|}_{X_{N_0}} + {\| Q_0 (\z_\a-1, \partial_\a f) \|}_{X_{N_0}}
\\ 
& + {\| \z_\a -1 \|}_{X_{N_0}} \left( {\| \partial_\a \l \|}_{H^{N_1}} + {\| Q_0 (\z_\a-1, \partial_\a f) \|}_{H^{N_1}} \right) \, .
\end{split}
\end{align}
From \eqref{chivl1}-\eqref{chivl2} we have
\begin{align*}
{\| \partial_\a \l \|}_{H^{N_1}} \lesssim {\| \wt{L} \|}_{H^{N_1}} \lesssim \e_1
\quad , \quad 
{\| \partial_\a \l \|}_{X_{N_0}} \lesssim {\| \wt{L} \|}_{X_{N_0}} \, .
\end{align*}
From estimate \eqref{estQ_0L^2b} we get
\begin{align*}
{\| Q_0 (\z_\a-1, \partial_\a f) \|}_{X_{N_0}} & \lesssim {\| \z_\a - 1 \|}_{W^{\frac{N_0}{2}+1,\infty}} {\| \partial_\a f \|}_{X_{N_0}}
  + {\| (\partial_\a f, \H \partial_\a f) \|}_{W^{\frac{N_0}{2}+1,\infty}} {\| \z_\a - 1 \|}_{X_{N_0}} 
\\
& + {\| \z_\a - 1 \|}_{W^{\frac{N_0}{2}+1,\infty}} {\| \partial_\a f \|}_{W^{\frac{N_0}{2}+1,\infty}} {\| \z_\a-1 \|}_{X_{N_0}} 
\\
& \lesssim \e_1 {\| \partial_\a f \|}_{X_{N_0}} + {\| (\partial_\a f, \H \partial_\a f) \|}_{W^{\frac{N_0}{2}+1,\infty}} 
  {\| \wt{L} \|}_{X_{N_0}} \, .
\end{align*}
Next we claim that
\begin{align} 
\label{lemLE2est10}
{\| (\partial_\a f, \H \partial_\a f) \|}_{W^{\frac{N_0}{2}+1,\infty}} \lesssim \e_1 \, .
\end{align}
Assuming this estimate for now, 
and using the fact that ${\| Q_0 (\z_\a-1, \partial_\a f) \|}_{H^{N_1}}$ can be bounded uniformly in time in a straightforward fashion,
we can use the last four bounds in \eqref{lemLE2est2} to deduce that
\begin{align}
\label{lemLE2est20}
& {\| \partial_\a f \|}_{X_{N_0}} \lesssim  {\| \wt{L}\|}_{X_{N_0}} \, .
\end{align}
Them, since $\phi_x = (\partial_\a f/\Re \z_\a) \circ \Re \z^{-1}$, we have
\begin{align}
& {\| \phi_x \|}_{H^{N_0}} \lesssim  {\| \partial_\a f \|}_{H^{N_0}} +  {\| \Re \z_{\a\a} \|}_{H^{N_0-1}}  
  \lesssim  {\| \wt{L}\|}_{X_{N_0}} \, .
\end{align}
This takes care of the Sobolev component of the $X_{N_0}$-norm. 
To bound the weighted component we use $\phi_x \circ \Re \z = \partial_\a f / \Re \z_\a$ in combination with \eqref{Sformula}, 
estimate \eqref{SRez-RezL^2}, and \eqref{lemLE2est20}, to get
\begin{align*}
{\| S \phi_x \|}_{H^\frac{N_0}{2}} & \lesssim  {\| (S \phi_x) \circ \Re \z \|}_{H^\frac{N_0}{2}}
  \lesssim  {\left\| S (f_\a /\Re \z_\a) \right\|}_{H^\frac{N_0}{2}} 
  + {\| \phi_x \|}_{H^{\frac{N_0}{2}+2}} {\| S \Re\z - \Re\z \|}_{H^\frac{N_0}{2}} \lesssim {\| \wt{L} \|}_{X_{N_0}}.
\end{align*}
This shows \eqref{LE1} provided we verify \eqref{lemLE2est10}.
Observe that
\begin{align*} 
{\| \partial_\a f \|}_{W^{\frac{N_0}{2}+1,\infty}} \lesssim {\| \partial_x \phi \|}_{W^{\frac{N_0}{2}+1,\infty}} \lesssim \e_1 \, ,
\end{align*}
directly from the a priori assumptions.
To bound $\H \partial_\a f$ instead, we use Lemma \ref{lemHinfty}, the a priori assumptions and \eqref{estcorproE4} to obtain
\begin{align*} 
{\| \H \partial_\a f \|}_{W^{\frac{N_0}{2}+1,\infty}} \lesssim {\| f \|}_{W^{\frac{N_0}{2}+3,\infty}} 
  \lesssim {\| \phi \|}_{W^{\frac{N_0}{2}+3,\infty}} \lesssim \e_1 \, .
\end{align*}
This shows \eqref{lemLE2est10} and conlcudes the proof of \eqref{LE1}, hence of Proposition \ref{proLE}.
\end{proof}

\appendix
\section{Supporting material}\label{appWu}

In this first appendix we gather some useful Lemmas that are used several times in sections \ref{secprok} and \ref{secproLE}
and in the course of the energy estimates.
First, in section \ref{appWu1} we give some variants of estimates proven in \cite{WuAG} related to the Hilbert transform on curves.
In section \ref{secop} we first recall some Theorems about multilinear operators of ``Calder\'{o}n's commutators''-type,
and then prove some additional bounds on them that are used for the energy estimates.

\subsection{Estimates for the Cauchy integral}\label{appWu1}
In what follows we will always be under the a priori assumption that \eqref{aprioriL1} holds. 

\begin{lem}\label{lemHf}
Let $\H = \H_\z$. 
Then, for any $f$ in $X_k$ with $0 \leq k \leq N_0$, we have
\begin{align}
\label{estH}
{\| \H f \|}_{H^k} + {\left\| \H \frac{1}{\z_\a} f \right\|}_{H^k} & 
  \lesssim {\| f \|}_{H^k} + {\| \z_\a - 1 \|}_{H^k} {\| f \|}_{W^{\frac{k}{2},\infty}} \, ,
\\
\label{estHb}
{\| \H f \|}_{X_k} + {\left\| \H \frac{1}{\z_\a} f \right\|}_{X_k} & 
  \lesssim {\| f \|}_{X_k} + {\| \z_\a - 1 \|}_{X_k} {\| f \|}_{W^{\frac{k}{2},\infty}} \, .
\end{align}
In particular, if $k \leq N_1+5$, one has
\begin{equation}
\label{estHlow}
{\| \H f \|}_{H^k}  + {\left\| \H \frac{1}{\z_\a} f \right\|}_{H^k} \lesssim  {\| f \|}_{H^k} \, .
\end{equation}
Furthermore, for any $0 \leq k \leq N_0$
\begin{align}
\label{estH+barH}
{\| (\H + \bar{\H})  f \|}_{X_k} & \lesssim {\| \z_\a - 1 \|}_{W^{\frac{k}{2},\infty}} {\| f \|}_{X_k} 
    + {\| \z_\a - 1 \|}_{X_k} {\| f \|}_{W^{\frac{k}{2}, \infty}} \, ,
\end{align}
and for $0 \leq k \leq N_1$
\begin{align}
\label{estH+barHlow}
{\| (\H + \bar{\H})  f \|}_{H^k} & \lesssim {\| \Im \z_\a  \|}_{W^{k,\infty}} {\| f \|}_{H^k}   \, .
\end{align}
\end{lem}

\proof
The $L^2$ case in \eqref{estH} follows directly from Theorem \ref{theoCMM}.
The $H^k$, respectively $X_k$, estimates can be proven by induction using the commutation identities \eqref{commK2} to distribute derivatives, 
respectively \eqref{commK3} to distribute the vector field $S$, 
and the bounds given in Theorem \ref{theoCMM} for operators of the type $C_1$, as defined in \eqref{C_11}.




To prove \eqref{estH+barH} one notices that
\begin{equation}
\label{H+barH}
(\H + \bar{\H})  f  =  -\frac{2}{\pi} \int \frac{\Im\z(\a) - \Im\z(\b)}{ {|\z(\a) - \z(\b)|}^2 } \, f(\b) \z_\b(\b) d\b 
  + \frac{2}{\pi} \int \frac{f(\b) \Im\z_\b(\b) }{\bar{\z}(\a) - \bar{\z}(\b)} d\b \, ,
\end{equation}
which is the sum of two operator of the form $C_1(H, \Im \z, f\z_\a)$ and $C_1(H, \id, f\Im \z_\a)$, for some smooth $H$,
see \eqref{C_11} below.
Applying the commutation identities \eqref{commK2} and \eqref{commK3}, followed by the $L^2$-estimates of Theorem \ref{theoCMM},
one can then verify the validity of \eqref{estH+barH} and \eqref{estH+barHlow}. $\hfill \Box$

The next Lemma is a variant of Lemma 3.8 in \cite{WuAG} and gives estimates of real valued functions $f$ in terms of the norms of $(I-\H)f$.

\begin{lem}\label{lemI-Hf}
Let $f \in X_k$, $0 \leq k \leq N_0$, be real-valued with 
\begin{align*}
(I - \H) f = g \, . 
\end{align*}
Then, for $0 \leq k \leq N_1+5$, one has
\begin{align}
\label{est1-Hflow}
{\| f \|}_{H^k} &  \lesssim {\| g \|}_{H^k} \, .
\end{align}
Furthermore, for $0 \leq k \leq N_0$
\begin{align}
\label{est1-Hf0}
{\| f \|}_{H^k} & \lesssim {\| g \|}_{H^k}  
  +  {\| \z_\a - 1 \|}_{H^k} \left( {\| g \|}_{ W^{\frac{k}{2},\infty}} + {\| \Im \z_\a \|}_{W^{\frac{k}{2}+1,\infty}}  
  {\| g \|}_{H^{\frac{k}{2}+1}} \right)
\\
\label{est1-Hf}
{\| f \|}_{X_k} & \lesssim {\| g \|}_{X_k}  
  +  {\| \z_\a - 1 \|}_{X_k} \left( {\| g \|}_{ W^{\frac{k}{2},\infty}} + {\| \Im \z_\a \|}_{W^{\frac{k}{2}+1,\infty}}  
  {\| g \|}_{H^{\frac{k}{2}+1}} \right) \, .
\end{align}
Moreover, for $0 \leq k \leq N_1$, we have
\begin{align}
\label{est1-Hfinfty}
{\| f \|}_{W^{k,\infty}} \lesssim  {\| \Re \, g \|}_{W^{k,\infty}}   & +  {\| \Im \z_\a \|}_{W^{k+1,\infty}}  {\| g \|}_{H^{k + 1}} \, ,
\end{align}
and a similar estimate for $\partial_\a f$:
\begin{align}
\label{est1-Hfinfty2}
{\| \partial_\a f \|}_{W^{k,\infty}} \lesssim  {\| \Re \, \partial_\a g \|}_{W^{k,\infty}}  
  +  {\| \z_\a - 1\|}_{W^{k+1,\infty}}  {\| \partial_\a g \|}_{H^{k + 1}} \, .
\end{align}
\end{lem}

\proof
Since $f$ is real-valued we have $(I - \K) f = \Re g$, where $\K = \Re \H$. Then
\begin{align*}
(I - \K) \partial_\a^j f = \Re \partial_\a^j g - \left[ \K, \partial_\a^j \right] f 
  =  \Re \partial_\a^j g - \sum_{k=1}^j \partial_\a^{j-k} \left[ \K, \partial_\a \right] \partial_\a^{k-1} f \, .
\end{align*}
Notice that 
\begin{align}
\label{[D,H]}
[\partial_\a,\H] f =  [\z_\a, \H] \frac{f_\a}{\z_\a} = C_2(H, \z_\a-1, f)  \, ,
\end{align}
for some smooth $H$, and where $C_2$ is defined in \eqref{C_21}.
We can then use the fact that the inverse of $I - \K$ is bounded on $L^2$ with an operator norm depending only on $\e_1$,
\eqref{[D,H]}, and Theorem \ref{theoCMM}, to obtain
\begin{align*}
{\| \partial_\a^j f \|}_{L^2} & \lesssim {\left\| \partial_\a^j g \right\|}_{L^2} 
  +  {\| \z_\a - 1 \|}_{ H^{\frac{j}{2} + 1} }  {\| f \|}_{H^j}
  +  {\| \z_\a - 1 \|}_{H^j} {\| f \|}_{W^{\frac{j}{2},\infty}}
\\ 
& \lesssim {\left\| \partial_\a^j g \right\|}_{L^2} + \e_1 {\| f \|}_{H^j}  + {\| \z_\a - 1 \|}_{H^j} {\| f \|}_{W^{\frac{j}{2},\infty}} \, .
\end{align*}
After summing over $j$, the second term in the right-hand side above can be absorbed to the left hand-side for $\e_1$ small enough.
We have therefore obtained that for any $0 \leq k \leq N_0$
\begin{align}
\label{A1}
{\| f \|}_{H^k} & \lesssim {\left\| g \right\|}_{H^k} + {\| \z_\a - 1 \|}_{H^k} {\| f \|}_{W^{\frac{k}{2}, \infty}} \, .
\end{align}
If $k \leq N_1+5$ the last term above can be also absorbed to the left hand-side thus yielding \eqref{est1-Hflow}.

In order to prove \eqref{est1-Hf0} let us focus on the term ${\| f \|}_{W^{\frac{k}{2}, \infty}}$.
Using the identity $f = \K f + \Re g$, Sobolev's embedding, the estimate \eqref{estH+barHlow}, and \eqref{estHlow}, we get
\begin{align*}
{\| f \|}_{W^{\frac{k}{2}, \infty}} \leq {\| \K f \|}_{H^{\frac{k}{2} + 1}} + {\| \Re g \|}_{W^{\frac{k}{2}, \infty}}
& \lesssim {\| \Im \z_\a \|}_{W^{\frac{k}{2}+1,\infty}} {\| f \|}_{H^{\frac{k}{2} + 1}} + {\| \Re g \|}_{W^{\frac{k}{2}, \infty}}
\\
& \lesssim {\| \Im \z_\a \|}_{W^{\frac{k}{2}+1,\infty}} {\| g \|}_{H^{\frac{k}{2} + 1}} + {\| \Re g \|}_{W^{\frac{k}{2}, \infty}} \, .
\end{align*}
Plugging this last inequality into \eqref{A1} gives \eqref{est1-Hf0}.
Substituting $k$ with $2k$ we obtain \eqref{est1-Hfinfty}. \eqref{est1-Hfinfty2} can be obtained similarly.

From above we see that \eqref{est1-Hf} would follow if we show
\begin{align}
\label{A2}
{\| S f \|}_{H^k} & \lesssim {\left\| S g \right\|}_{H^k} + {\| S(\z_\a - 1) \|}_{H^k} {\| f \|}_{W^{k, \infty}} \, . 
\end{align}
Starting from $(I-\H)f = g$ one can commute derivatives using \eqref{[D,H]} and commute $S$ by using
\begin{align}
\label{[S,H]}
[S,\H] f =  [S\z - \z, \H] \frac{f_\a}{\z_\a} = C_2(H, S\z - \z, f)  \, .
\end{align}
Applying Theorem \ref{theoCMM} one can then obtain
\begin{align*}
{\| \partial_\a^j S f \|}_{L^2} & \lesssim {\left\| \partial_\a^j S g \right\|}_{L^2} 
  +  {\| \z_\a - 1 \|}_{H^j}  {\| S f \|}_{H^j} +  {\| S(\z_\a - 1) \|}_{H^j} {\| f \|}_{W^{j,\infty}}
\\ 
& \lesssim {\left\| S g \right\|}_{H^j} + \e_1 {\| S f \|}_{H^j}  + {\| S(\z_\a - 1) \|}_{H^j} {\| f \|}_{W^{j,\infty}} \, .
\end{align*}
Summing over $j$ and absorbing the second summand above in the left-hand side, we obtain \eqref{A2} and hence \eqref{est1-Hf}.$\hfill \Box$


\subsection{Estimates for Multilinear Operators}\label{secop}

In this section we study some singular integrals that appear when performing the energy estimates. 
These integral operators are well known objects, which are usually referred to as Calder\'{o}n's commutators.
We first state some $L^2$-bounds like the ones already given in \cite{WuAG}.

Let $H \in C^1$, $A_i \in C^1$ for $i=1,\dots,m$, and $F \in C^\infty$. Using the same notation in \cite{WuAG} we define
\begin{align}
\label{C_11} 
C_1 (H,A,f) & :=  \mbox{p.v.} \int F \left(  \frac{ H(x) - H(y) }{x-y} \right) 
    \frac{ \prod_{i=1}^m ( A_i(x) - A_i(y) ) }{ {(x-y)}^{m+1} }  f(y) \, dy 
\\
\label{C_21}
C_2 (H,A,f) & :=  \mbox{p.v.} \int F \left(  \frac{ H(x) - H(y) }{x-y} \right) 
    \frac{ \prod_{i=1}^m ( A_i(x) - A_i(y) ) }{ {(x-y)}^{m} } \partial_y f(y) \, dy \, .
\end{align}
Integrals like the ones above are always to be understood in the principal value sense, but, as before,
for simplicity we will often omit the p.v. notation.

We also define the quadratic bilinear operators
\begin{align}
\label{Q_0} 
& Q_0 (f,g) :=  \int \frac{ f(\a) - f(\b) }{ \z(\a)-\z(\b) } g(\b) \, d\b \, ,
\\
\label{Q_1} 
& Q_1 (f,g) :=  \int \frac{ f(\a) - f(\b) }{ {(\z(\a)-\z(\b))}^2 } g(\b) \, d\b \, ,
\\
\label{Q_2} 
& Q_2 (f,g) :=  \int \frac{ f(\a) - f(\b) }{ {|\z(\a)-\z(\b)|}^2 } g(\b) \, d\b \, .
\end{align}
We denote by $\bQ$ indistictly any scalar multiple of the operators $Q_1$ or $Q_2$:
\begin{align}
\label{opQ}
\bQ (f,g) := c_i Q_i(f,g)
\end{align}
for $c_i \in \C$, $i=1,2$.
$Q_0$ causes some difficulties because it does not admit standard $L^2 \times L^\infty \rightarrow L^2$ estimates.
Moreover, it does not admit $L^\infty$ type estimates like those in Lemma \ref{lemCCL^infty} below for $Q_1$ and $Q_2$;
in order to bound it we need to resort to a stronger space than $L^\infty$.

We recall the following:
\begin{theo}[Coifman-McIntosh-Meyer \cite{CMM}, Wu \cite{WuAG}]\label{theoCMM}
There exists $c = c(F, {\| H^\p \|}_{L^\infty})$ such that the operators $C_j$, for $j=1,2$, satisfy the bounds
\begin{align}
{\| C_j (H,A,f) \|}_{L^2} & \leq c \prod_{i=1}^m {\| \partial A_i \|}_{L^\infty} {\| f \|}_{L^2}
\\
{\| C_j (H,A,f) \|}_{L^2} & \leq c {\| \partial A_1 \|}_{L^2} \prod_{i=2}^m {\| \partial A_i \|}_{L^\infty} {\| f \|}_{L^\infty} \, .
\end{align}
\end{theo}

From the above Theorem we can infer the following bounds on the operators of the type $\bC$ defined in \eqref{C} and $\bQ$ in \eqref{opQ}:
\begin{cor}\label{theoCCL^2}
There exists a constant $c = c({\| \partial_\a \z \|}_{L^\infty})$ such that
\begin{align*}
{\| \bQ (f, g) \|}_{L^2} & \leq c {\| \partial_\a f \|}_{L^\infty} {\| g \|}_{L^2}
\\
{\| \bQ (f, g) \|}_{L^2} & \leq c {\| \partial_\a f \|}_{L^2} {\| g \|}_{L^\infty}
\end{align*}
and
\begin{align*}
{\| \bC (f, g, h_\a) \|}_{L^2} & \leq c {\| \partial_\a f \|}_{L^\infty} {\| \partial_\a g \|}_{L^\infty} {\| h \|}_{L^2}
\\
{\| \bC (f, g, h_\a) \|}_{L^2} & \leq c {\| \partial_\a f \|}_{L^2} {\| \partial_\a g \|}_{L^\infty} {\| h \|}_{L^\infty} \, .
\end{align*}
\end{cor}

In \ref{prlemCCL^infty} we will prove the following simple Lemma:
\begin{lem}\label{lemCCL^infty}
There exists a constant $c = c( {\| \partial_\a \z \|}_{W^{1,\infty}} )$ such that the operators $\bQ$ satisfy the bound
\begin{align}
\label{estlemCCL^infty}
{\| \bQ (f, g) \|}_{L^\infty} & \leq c {\| f \|}_{W^{2,\infty}} {\| g \|}_{W^{1,\infty}} \, .
\end{align}
\end{lem}

We will also need to bound operators of the type $Q_0$ in $L^\infty$. However, they will only appear with a derivative in front,
so that we can use the following Lemma:
\begin{lem}\label{lemCCL^infty2}
There exists a constant $c = c( {\| \z_\a - 1 \|}_{H^3} )$ such that
\begin{align}
\label{estlemL^infty2}
{\| \partial_\a Q_0 (f, g) \|}_{L^\infty} & \leq c {\| f \|}_{W^{2,\infty}} \left( {\| \H g \|}_{L^\infty} + {\| g \|}_{W^{1,\infty}} \right) \, .
\end{align}
\end{lem}

The above results, together with some commutation identities, will give us the following Proposition:
\begin{pro}\label{proCCmain}
Recall the definitions
\begin{equation}
\label{vectorL}
 L := (\z_\a - 1, u, w, \Im\z, \partial_\a \chi, v) \in \C^6
\end{equation}
and 
\begin{equation}
\label{vectorL-}
 L^- := (\z_\a - 1, u, w, \partial_\a \chi, v)   \in \C^5 \, .
\end{equation}
Let $\bQ$ and $\bC$ be given by \eqref{opQ} and \eqref{opC}--\eqref{C}. 
Then

\begin{enumerate}

\item  There exists a constant $c = c({\| \z_\a-1 \|}_{H^{N_1+4}})$ such that 
for any integer 
$k \leq N_1$
\begin{align}
\label{estbQL^inftyfg}
& {\| \bQ (f, g) \|}_{W^{k,\infty}} \leq c {\| f \|}_{W^{k+2,\infty}} {\| g \|}_{W^{k+2,\infty}}  \, .
\end{align}
In particular
\begin{align}
\label{estQL^infty}
& {\| \bQ (L_i, L_j) \|}_{W^{[\frac{m}{2}]+1,\infty}} \leq c {\| L \|}^2_{W^{[\frac{m}{2}]+3,\infty}}
\end{align}
for any $i,j \in \{ 1, \dots, 6 \}$ and $0 \leq m \leq N_0$.

\item  There exists a constant $c = c({\| \z_\a -1 \|}_{H^{N_1+4}})$ such that for any integer $0 \leq k \leq N_0$
\begin{align}
\nn
{\| Q_0 (f, g) \|}_{X_k} & \lesssim_c {\| f \|}_{W^{\frac{N_0}{2}+1,\infty}} {\| g \|}_{X_k}
  + {\| (g, \H g) \|}_{W^{\frac{N_0}{2}+1,\infty}} {\| f \|}_{X_k} 
\\
\label{estQ_0L^2}
& + {\| f \|}_{W^{\frac{N_0}{2}+1,\infty}} {\| g \|}_{W^{\frac{N_0}{2}+1,\infty}} {\| \z_\a-1 \|}_{X_k}
\end{align}and
\begin{align}
\nn
{\| Q_0 (f, \partial_\a g) \|}_{X_k} & \lesssim_c {\| \partial_\a f \|}_{W^{\frac{N_0}{2}+1,\infty}} {\| g \|}_{X_k}
  + {\| g \|}_{W^{\frac{N_0}{2}+2,\infty}} {\| f \|}_{X_k} 
\\
\label{estQ_0L^2b}
& + {\| f \|}_{W^{\frac{N_0}{2}+1,\infty}} {\| g \|}_{W^{\frac{N_0}{2}+1,\infty}} {\| \z_\a-1 \|}_{X_k} \, .
\end{align}

Furthermore, for $k \leq N_1$,
\begin{align}
\label{estpartialQ_00}
{\| \partial_\a Q_0 (f, g) \|}_{W^{k,\infty}} & \leq c {\| f \|}_{W^{k+2,\infty}}
  \left( {\| \H g \|}_{W^{k,\infty}} + {\| g \|}_{W^{k+1,\infty}} \right) \, ,
\end{align}
so that
\begin{align}
\label{estpartialQ_0}
{\| \partial_\a Q_0 (L_i, L_j^-) \|}_{W^{[\frac{m}{2}]+1,\infty}} & \leq c {\| L \|}_{W^{[\frac{m}{2}]+3,\infty}} 
  \left( {\| \H L^- \|}_{W^{[\frac{m}{2}]+2,\infty}} + {\| L^- \|}_{W^{[\frac{m}{2}]+2,\infty}} \right)
\end{align}
for any $i \in \{ 1, \dots, 6 \}$, $j \in \{ 1, \dots, 5 \}$ and $0 \leq m \leq N_0$.

\item  There exists a constant $c$ as above
such that for any triple $(f,g,h)$ with ${\| (f,g,h) \|}_{H^{N_1-2}} \leq 1$, and any integer $m$, one has
\begin{align}
\label{estCCmain00}
{\| \bC (f, g, h) \|}_{X_m} + {\| \bC (f, g, \partial_\a h) \|}_{X_m} & \leq c {\| (f,g,h,\z_\a-1) \|}_{X_m} 
  {\| (f,g,h) \|}_{W^{\frac{N_0}{2}+3,\infty}}^2 \, .
\end{align}
In particular
\begin{align}
\label{estCCmain0}
{\| \bC (L_i, L_j, L_k) \|}_{X_m} + {\| \bC (L_i, L_j, \partial_\a L_k) \|}_{X_m} 
  & \leq c {\| L \|}_{X_m} {\| L \|}_{W^{\frac{N_0}{2}+3,\infty}}^2 
\end{align}
for any $i,j,k \in \{ 1, \dots, 6 \}$ and $0 \leq m \leq N_0$. 
\end{enumerate}
\end{pro}

The proof of the above Proposition is given in \ref{secproCCmain}.
We will also need the following simple Lemma:

\begin{lem}\label{lemHinfty}
Let $\H = \H_\z$ denote the Hilbert transform along a curve $\z$ satisfying 
${\| \z_\a - 1\|}_{H^{N_1+4}} \leq \frac{1}{2}$.
Then for any $f$ with ${\| f \|}_{H^{k+2}} \leq 1$, and $k \leq N_1$, we have
\begin{align}
\label{estHdfL^infty}
& {\| \H \partial_\a f(t) \|}_{W^{k,\infty}} + {\left\| \H \frac{1}{\z_\a} \partial_\a f(t) \right\|}_{W^{k,\infty}}  
  \lesssim {\| f(t) \|}_{W^{k+2,\infty}} \, 
\end{align}
and for any $2 \leq p < \infty$
\begin{align}
\label{estHL^infty}
& {\| \H  f(t) \|}_{W^{k,\infty}} + {\left\| \H \frac{1}{\z_\a} f(t) \right\|}_{W^{k,\infty}}  
  \lesssim {\| f(t) \|}_{W^{k+1,p}} + {\| f(t) \|}_{W^{k+1,\infty}}  \, .
\end{align}
\end{lem}

\subsubsection{Commutator identities}
Let $\bK$ be an integral operator of the form
\begin{align}
 \label{defK}
\bK f (\a,t) = \mbox{p.v.} \int K(\a,\b;t) f(\b,t) \, d\b
\end{align}
with kernel $K(\a,\b;t)$ or $(\a-\b)K(\a,\b;t)$ continuous and bounded, and $K$ smooth away from the diagonal $\a=\b$.
One can easily verify that
\begin{subequations}
\label{commK}
\begin{align}
\label{commK1}
[\partial_t, \bK] f(\a,t) &= \int \partial_t K(\a,\b;t) f(\b,t) \, d\b \, ,
\\
\label{commK2}
[\partial_\a, \bK] f(\a,t) &= \int (\partial_\a + \partial_\b) K(\a,\b;t) f(\b,t) \, d\b \, ,
\\
\label{commK3}
[S, \bK] f(\a,t) &= \int \left( \a \partial_\a + \b \partial_\b + \frac{1}{2}t \partial_t \right) K(\a,\b;t) f(\b,t) \, d\b +  \bK f (\a,t)  \, ,
\end{align}
\end{subequations}
for any sufficiently smooth and decaying $f$.

\subsubsection{Proof of Lemma \ref{lemCCL^infty}}\label{prlemCCL^infty}
It is enough to just look at the case of $Q_1$, as the treatment of $Q_2$ is identical.
Expanding out the denominator in \eqref{Q_1} we can write
\begin{align*}
 \frac{1}{{(\z(\a)-\z(\b))}^2} = F \left(  \frac{ \z(\a) - \a - (\z(\b) - \b) }{\a-\b} \right) \frac{1}{ {(\a-\b)}^2 }
\end{align*}
where $F(x) = \sum_{k\geq0} {(-1)}^k (k+1) x^k$.
Then one can see that proving \eqref{estlemCCL^infty} 
can be reduced to proving the following estimate for operators of the type $C_1$ as in \eqref{C_11}:
\begin{align}
\label{estinfty1} 
{\left\| \mbox{p.v.} \int F \left(  \frac{ H(x) - H(y) }{x-y} \right) \frac{ A(x) - A(y) }{ {(x-y)}^2 }  f(y) \, dy \right\|}_{L^\infty}
	& \lesssim \prod_{i=1}^m {\| A \|}_{W^{2,\infty}} {\| f \|}_{W^{1,\infty}} \, ,
\end{align}
where the implicit constant depends on ${\| H^{\p\p} \|}_{L^\infty}$.
To show this we split the integral into two pieces:
\begin{align*}
& \int F \left(  \frac{ H(x) - H(y) }{x-y} \right) \frac{ A(x) - A(y) }{ {(x-y)}^2 }  f(y) \, dy = I_1(x) + I_2(x)
\\
 & I_1(x) = \int_{|x-y| \leq 1} F \left(  \frac{ H(x) - H(y) }{x-y} \right) \frac{ A(x) - A(y) }{ {(x-y)}^2 }  f(y) \, dy
\\
 & I_2(x) = \int_{|x-y| \geq 1} F \left(  \frac{ H(x) - H(y) }{x-y} \right) \frac{ A(x) - A(y) }{ {(x-y)}^2 }  f(y) \, dy \, .
\end{align*}
We write
\begin{align*}
I_1(x) & = \int_{|x-y| \leq 1} \left[ F \left(  \frac{ H(x) - H(y) }{x-y} \right) - F(H^\p(x)) \right] \frac{ A(x) - A(y)}{ {(x-y)}^2 }  f(y) \, dy
	\\
	& + F(H^\p(x)) \int_{|x-y| \leq 1} \frac{ A(x) - A(y) - A^\p(x) (x-y)}{ {(x-y)}^2 }  f(y) \, dy
	\\
	& + F(H^\p(x)) \int_{|x-y| \leq 1} \frac{A^\p(x)}{x-y}  ( f(y) - f(x) )\, dy
	=: I_{1,1}(x) + I_{1,2}(x) + I_{1,3}(x) \, .
\end{align*}
It is then easy to see that we can then estimate
\begin{align*}
| I_{1,1}(x) | & \lesssim {\| F^\p \|}_{L^\infty} {\| H^{\p\p} \|}_{L^\infty} {\| A^\p \|}_{L^\infty} {\| f \|}_{L^\infty}
\\
| I_{1,2}(x) | & \lesssim {\| F \|}_{L^\infty} {\| A^{\p\p} \|}_{L^\infty} {\| f \|}_{L^\infty}
\\
| I_{1,3}(x) | & \lesssim {\| F^\p \|}_{L^\infty} {\| A^\p \|}_{L^\infty} {\| f^\p \|}_{L^\infty}
\end{align*}
so that
${\| I_1 \|}_{L^\infty} \lesssim c\left( {\| F \|}_{W^{1,\infty}}, {\| H^{\p\p} \|}_{L^\infty} \right) 
  {\| A^\p \|}_{W^{1,\infty}} {\| f \|}_{W^{1,\infty}}$.
Since $|x-y|^{-2}$ is integrable for $|x-y| \geq 1$ one has
\begin{align*}
{\| I_2 \|}_{L^\infty} & \lesssim c\left( {\| F \|}_{L^\infty} \right) {\| A \|}_{L^\infty} {\| f \|}_{L^\infty} \, .
\end{align*}
The bound \eqref{estinfty1} follows. $\hfill \Box$

\subsubsection{Proof of Lemma \ref{lemCCL^infty2}}
We start by calculating
\begin{align*}
\partial_\a Q_0 (f,g) & = \partial_\a \int \frac{ f(\a) - f(\b) }{ \z(\a)-\z(\b) } g(\b) \, d\b 
\\
& = - \partial_\a \z(\a) \int \frac{ f(\a) - f(\b) }{ {(\z(\a)-\z(\b))}^2 } g(\b) \, d\b 
  + \int \frac{ \partial_\a f(\a) }{ \z(\a)-\z(\b) } g(\b) \, d\b 
  =: Q_0^1(\a) + Q_0^2(\a) \, .
\end{align*}
Since the integral operators in $Q^0_1$ is of the type $\bQ$, we can use Lemma \ref{lemCCL^infty} to bound
\begin{align*}
{\| Q_0^1 \|}_{L^\infty} & \lesssim {\| \z_\a \|}_{L^\infty} {\| f \|}_{W^{2,\infty}} {\| g \|}_{W^{1,\infty}} \, .
\end{align*}
The second contribution to $\partial_\a Q_0 (f,g)$ is
\begin{align*}
Q_0^2 =  \partial_\a f \left( \H \frac{1}{\z_\a} g \right) =  \partial_\a f \H g +  \partial_\a f \H \left(\frac{1}{\z_\a} - 1\right) g
\end{align*}
Thus, using also \eqref{estH}, we see that
\begin{align*}
{\| Q_0^2 \|}_{L^\infty} & \lesssim {\| \partial_\a f \|}_{L^\infty} {\| \H g \|}_{L^\infty} 
	+  {\| \partial_\a f \|}_{L^\infty} {\left\| \H \left(\frac{1}{\z_\a} - 1\right) g \right\|}_{H^1}
\\
& \lesssim {\| \partial_\a f \|}_{L^\infty} {\| \H g \|}_{L^\infty} 
	+  {\| \partial_\a f \|}_{L^\infty} c( {\| \z_\a - 1 \|}_{H^3}) {\| g \|}_{W^{1,\infty}} \, .
\end{align*}
We conclude that $Q_0^2$ satisfies the desired bound and so does $\partial_\a Q_0 (f,g)$. $\hfill \Box$

\subsubsection{Proof of Proposition \ref{proCCmain}}\label{secproCCmain}

\paragraph{{\it Proof of \eqref{estbQL^inftyfg}}}
We want to show that for any two functions $f$ and $g$
\begin{align}
\label{estQL^infty2}
{\| \bQ (f, g) \|}_{W^{k,\infty}} & \leq c {\| f \|}_{W^{k+2,\infty}} {\| g \|}_{W^{k+2,\infty}} \, .
\end{align}
This can be shown by induction, using \eqref{estlemCCL^infty} as the base of the induction. 
Again, it is enough to just look at the case of $Q_1$.
Let us assume that \eqref{estQL^infty2} holds true for some $1 \leq k\leq \frac{N_0}{2}+1$.
We want to show the estimate for $k+1$.
Notice that we can write $\bQ$ as an operator of the form $\bK$, see \eqref{defK}, with Kernel
\begin{align*}
 K(\a,\b;t) = \frac{f(\a) - f(\b) }{ {(\z(\a)-\z(\b))}^2 } \, .
\end{align*}
Using the commutation identity \eqref{commK2} we see that
\begin{align*}
\partial_\a \bQ (f,g) & = \bQ (f, \partial_\a g) + \bQ (\partial_\a f,g) + I(\a)
\end{align*}
where
\begin{align*}
I(\a) = - \int \frac{ (f(\a) - f(\b))( \partial_\a \z(\a) - \partial_\b \z(\b)) }{ {(\z(\a)-\z(\b))}^3 } g(\b) \,d\b \, .
\end{align*}
Using the inductive hypothesis we have
\begin{align*}
{\| \bQ (\partial_\a f, g) \|}_{W^{k,\infty}} + {\| \bQ (f, \partial_\a g) \|}_{W^{k,\infty}} 
  & \leq c {\| \partial_\a f \|}_{W^{k+2,\infty}} {\| g \|}_{W^{k+2,\infty}} 
  +  c {\| f \|}_{W^{k+2,\infty}} {\| \partial_\a g \|}_{W^{k+2,\infty}}\, .
\end{align*}
By expanding the denominator in the integral defining $I$, we see that 
$I$ is an operator of the form  $C_1(H , A, g )$, see \eqref{C_11}, with $A = (f, \z_\a-1)$ and $H = \z - \id$.
Letting $k+1 = k_1 + k_2 + k_3$, and using \eqref{commK2}, we see that $D^{k+1} I $ is a sum of operators of the form 
\begin{align*}
C_1 \left(\z - \a, A_{k_2,k_3}, D^{k_1} g) \right)
\end{align*}
where $A_{k_2,k_3} = (\wt{A}_{k_3}, D^{k_2} f)$, and $\wt{A}_{k_3}$ is a vector with at most $k_3$ components 
satisfying $${\| A_{k_3}^\p \|}_{L^p} \lesssim {\| D^{k_3 + 1} \z_\a \|}_{L^p}$$ for $p=2,\infty$.
Applying Theorem \ref{theoCMM} we see that
\begin{align*}
{\| I \|}_{H^{k+2,\infty}} & \leq
  c {\| \partial_\a  f \|}_{W^{k+2,\infty}} {\| \z_\a -1\|}_{H^{k+2}} {\| g \|}_{W^{k+2,\infty}}
  \leq c {\| f \|}_{W^{k+3,\infty}} {\| g \|}_{W^{k+2,\infty}} \, ,
\end{align*}
where the constant $c$ depends only on ${\| \z_\a \|}_{H^{N_1}}$.
We can then deduce
\begin{align*}
{\| \bQ (f,g) \|}_{W^{k+1,\infty}} 
  & \leq {\| \bQ (\partial_\a  f,g) \|}_{W^{k+1,\infty}} + {\| \bQ (f,  \partial_\a g) \|}_{W^{k+1,\infty}} + {\| I \|}_{W^{k+1,\infty}} 
  \leq c {\| f \|}_{W^{k+3,\infty}} {\| g \|}_{W^{k+3,\infty}} 
\end{align*}
which is exactly \eqref{estQL^infty2} with $k+1$ replacing $k$.

\vskip10pt
\paragraph{\it{Proof of \eqref{estQ_0L^2}}}
Let us first look at the $H^k$ component of the $X_k$ norm. Since 
\begin{align}
\label{Q_0form}
Q_0(f,g) = \H \frac{1}{\z_\a} (fg) - f \H \frac{1}{\z_\a} g
\end{align}
we can use product Sobolev estimates and the $H^k$ bounds on the Hilbert transform \eqref{estH} to obtain
\begin{align*}
{\| Q_0(f,g) \|}_{H^k} 
& \lesssim {\| f g \|}_{H^k} + {\| \z_\a - 1 \|}_{H^k}  {\| f g \|}_{W^{\frac{k}{2}+1,\infty}}
  + {\| f \|}_{H^k} {\| \H g \|}_{L^\infty} + {\| f \|}_{L^\infty} {\| \H g \|}_{H^k}
\\
& \lesssim {\| f \|}_{H^k} {\| (\H g, g) \|}_{L^\infty}  + {\| f \|}_{L^\infty} {\| g \|}_{H^k}
  + {\| \z_\a - 1 \|}_{H^k}  {\| f \|}_{W^{\frac{k}{2}+1,\infty}} {\| g \|}_{W^{\frac{k}{2}+1,\infty}}
\end{align*}
where the implicit constants depend only on ${\| \z_\a -1\|}_{H^{N_1}}$.

A similar argument can be used to bound the $S^{-1} H^k$ norm of $Q_0(f,g)$ for $0 \leq k \leq \frac{N_0}{2}$.
First we observe that for any $0 \leq k \leq \frac{N_0}{2}$ one has
\begin{align}
\nn
{\| S Q_0(f,g) \|}_{H^k} & \leq {\left\| S \H \frac{1}{\z_\a} (f g) \right\|}_{H^k} + {\left\| S f \H \frac{1}{\z_\a} g \right\|}_{H^k} 
  + {\left\| f  S \H \frac{1}{\z_\a} g \right\|}_{H^k}
\\
\label{estQ_0L^210}
& \leq {\left\| \H \frac{1}{\z_\a} (f g) \right\|}_{X_k} + {\| f \|}_{X_k} {\left\| \H \frac{1}{\z_\a} g \right\|}_{W^{k,\infty}} 
  + {\| f \|}_{W^{k,\infty}} {\left\| \H \frac{1}{\z_\a} g \right\|}_{X_k} \, .
\end{align}
We can then use \eqref{estHb} to obtain
\begin{align*}
{\left\| \H \frac{1}{\z_\a} (f g) \right\|}_{X_k} & \leq c {\| f g \|}_{X_k} + {\| \z_\a - 1 \|}_{X_k}  {\| f g \|}_{W^{\frac{k}{2}+1,\infty}}
\\
{\left\| \H \frac{1}{\z_\a} g \right\|}_{X_k} & \leq c {\| g \|}_{X_k} + {\| \z_\a - 1 \|}_{X_k}  {\| g \|}_{W^{\frac{k}{2}+1,\infty}} \, .
\end{align*}
Since we also have
\begin{align*}
{\| f g \|}_{X_k} & \leq c {\| f \|}_{X_k}  {\| g \|}_{W^{k,\infty}} + {\| f \|}_{W^{k,\infty}} {\| g \|}_{X_k}  
\end{align*}
we can plug the above bounds in \eqref{estQ_0L^210} and get the desired conclusion.

\vskip10pt
\paragraph{\it{Proof of \eqref{estQ_0L^2b}}}
Let us start again with the $H^k$ component of the $X_k$ norm.
First of all observe that $Q_0 (f, \partial_\a g)$ is an operator of the form $C_2(\z-\a, f ,g)$, see \eqref{C_21}.
Distributing derivatives on $Q_0(f, \partial_\a g)$ by using \eqref{commK2}, 
we see that for any integer $k = k_1 + k_2 + k_3$, we have that $D^k Q_0(f, \partial_\a g)$ is a sum of operators of the form 
\begin{align*}
C_2 \left(\z - \a, A_{k_2,k_3}, D^{k_1} g) \right)
\end{align*}
where $A_{k_2,k_3} = (\wt{A}_{k_3}, D^{k_2} f)$, and $\wt{A}_{k_3}$ is a vector with at most $k_3$ components 
satisfying $${\| A_{k_3}^\p \|}_{L^p} \lesssim {\| D^{k_3 + 1} \z_\a \|}_{L^p}$$ for $p=2,\infty$.
One can then apply Theorem \ref{theoCMM} to deduce that the $H^k$-norm of $Q_0(f, \partial_\a g)$
is bounded by the right-hand side of \eqref{estQ_0L^2b} for any $0 \leq k \leq N_0$.

The estimate for ${\| S Q_0(f, \partial_\a g) \|}_{H^k}$, for $0 \leq k \leq \frac{N_0}{2}$,
follows similarly by using the commutation identity \eqref{commK3}. 
Indeed, applying $S$ to $Q_0(f, \partial_\a g) \sim C_2 (\z-\a, f, g)$, and commuting $S$ and $\partial_\a$
when $S$ falls on $\partial_\a g$, 
one obtains operators of the form $C_2 (\z-\a, S f, g)$, $C_2 (\z-\a, (f, S \z_\a), g)$, $C_2 (\z-\a, f, S g)$ or $C_2 (\z-\a, f, g)$ itself.
Applying and distributing $k$ derivatives as above, one can then estimate the resulting expressions in $L^2$
via Theorem \ref{theoCMM}, eventually obtaining the desired bound.

\vskip15pt
\paragraph{\it{Proof of \eqref{estpartialQ_00}}}

This estimate follows from the same proof of Lemma \ref{lemCCL^infty2}, which is the case $l=0$, 
after applying and commuting $k$ derivatives similarly to what has been already done before.
Since the proof is straightforward, we skip it.

\vskip15pt
\paragraph{\it{Proof of \eqref{estCCmain00}}}
Let us start by showing the $H^m$ estimate
\begin{align}
\label{estCCmain01}
{\| \bC (f, g, h) \|}_{H^m} + {\| \bC (f, g, \partial_\a h) \|}_{H^m} & \leq c {\| (f,g,h,\z_\a-1) \|}_{H^m} 
  {\| (f,g,h,\z_\a-1) \|}_{W^{\frac{N_0}{2}+3,\infty}}^2 \, ,
\end{align}
for all integers $0 \leq m \leq N_0$.
Again we will us induction and commutation identities. The base for the induction is given by the estimates
\begin{align}
{\| \bC (f, g, h) \|}_{L^2} & \leq c {\| (f,g,h) \|}_{L^2} {\| (f,g,h) \|}_{W^{\frac{N_0}{2}+3,\infty}}^2 \, ,
\\
\label{estCCmain012}
{\| \bC (f, g, \partial_\a h) \|}_{L^2} & \leq c {\| (f,g,h)\|}_{L^2} {\| (f,g,h) \|}_{W^{\frac{N_0}{2}+3,\infty}}^2 \, .
\end{align}
To verify these we cannot use directly Theorem \ref{theoCMM}. We instead write
\begin{align*}
\bC (f, g, h) & = f \bQ (g, h) - \bQ (g, f h)
\\
\bC (f, g, \partial_\a h) & = f \bQ (g, \partial_\a h) - \bQ (g, f \partial_\a h) \, .
\end{align*}
From Theorem \ref{theoCCL^2} we have
\begin{align*}
{\| \bQ (a, b) \|}_{L^2} \leq c {\| \partial_\a a \|}_{L^\infty} {\| b \|}_{L^2} \, .
\end{align*}
Thus, using \eqref{estQL^infty}, we obtain
\begin{align*}
{\| \bC (f,g,h) \|}_{L^2} & \leq {\| f \bQ (g, h) \|}_{L^2} + {\| \bQ (g, f h) \|}_{L^2}
\\
& \lesssim  {\| f \|}_{L^2} {\| \bQ (g, h) \|}_{L^\infty} + {\| \partial_\a g \|}_{L^\infty}  {\| f h \|}_{L^2}
\lesssim {\| (f,g,h) \|}_{L^2} {\| (f,g,h) \|}^2_{W^{2,\infty}} \, .
\end{align*}
Similarly we have
\begin{align*}
{\| \bC (f,g,\partial_\a h) \|}_{L^2} & \leq {\| f \bQ (g, \partial_\a h) \|}_{L^2} + {\| \bQ (g, f \partial_\a h) \|}_{L^2}
\\
& \lesssim  {\| f \|}_{L^2} {\| \bQ (g, \partial_\a h) \|}_{L^\infty} + {\| \partial_\a g \|}_{L^\infty}  {\| f \partial_\a h \|}_{L^2}
\lesssim {\| (f,g,h) \|}_{L^2} {\| (f,g,h) \|}^2_{W^{3,\infty}} \, .
\end{align*}

Now let us assume that \eqref{estCCmain01} holds true for some integer $0 \leq l \leq m-1$.
Using the commutation identity \eqref{commK2} we see that
\begin{align*}
\partial_\a \bC (f,g,h) & = \bC (\partial_\a f,g,h) + \bC (f, \partial_\a g, h) + \bC (f, g, \partial_\a h) + J_1(\a)
\end{align*}
where
\begin{align*}
J_1(\a) = - \int \frac{ (f(\a) - f(\b))(g(\a) - g(\b)) ( \partial_\a \z(\a) - \partial_\b \z(\b)) }{ {(\z(\a)-\z(\b))}^3 } h(\b) \,d\b  \, .
\end{align*}
Since
\begin{align*}
\bC (\partial_\a f,g,h) = \partial_\a f \bQ (g, h) - \bQ (g, \partial_\a f h)
\end{align*}
we have
\begin{align}
\label{proCCmain10}
{\| \bC (\partial_\a  f,g,h) \|}_{H^l} & \leq {\| \partial_\a f \|}_{H^l} {\| \bQ (g, h) \|}_{L^\infty} 
  + {\| \partial_\a f \|}_{L^\infty} {\| \bQ (g, h) \|}_{H^l} + {\| \bQ (g, \partial_\a f h) \|}_{H^l} \, .
\end{align}
From Theorem \ref{theoCCL^2}, and commutation identities, it is not hard to see that
\begin{align}
\label{proCCmain11}
{\| \bQ (a, b) \|}_{H^l} \leq c {\| a \|}_{W^{\frac{l}{2}+2,\infty}} {\| b \|}_{H^l}
  + c {\| \partial_\a a \|}_{H^l} {\| b \|}_{W^{\frac{l}{2}+2,\infty}}
  + {\| a \|}_{W^{\frac{l}{2}+2,\infty}} {\| b \|}_{W^{\frac{l}{2}+2,\infty}} {\| \z_\a - 1 \|}_{H^l} \, .
\end{align}
We can then use the above estimate and \eqref{estQL^infty} to bound the right-hand side of \eqref{proCCmain10} and obtain
\begin{align*}
& {\| \bC (\partial_\a f, g, h) \|}_{H^l} \lesssim {\| \partial_\a f \|}_{H^l} {\| (f,g,h) \|}^2_{W^{2,\infty}} 
  + {\| \partial_\a  f \|}_{L^\infty} {\| (g,h) \|}_{H^{l+1}} {\| (g,h) \|}_{W^{\frac{l}{2}+2,\infty}} 
\\
& + {\| (\partial_\a f,g,h) \|}^2_{W^{\frac{l}{2}+2,\infty}} {\| \z_\a - 1 \|}_{H^l}
  + {\| g \|}_{W^{\frac{l}{2}+2,\infty}} {\| \partial_\a f h \|}_{H^l} 
  + {\| \partial_\a g \|}_{H^l} {\| \partial_\a f  h \|}_{W^{\frac{l}{2}+2,\infty}} 
\\
& \lesssim {\| (f,g,h,\z_\a-1) \|}_{H^{l+1}} {\| (f,g,h) \|}^2_{W^{\frac{l}{2}+3,\infty}}  \, .
\end{align*}
An identical bound clearly holds for $\bC (f, \partial_\a g, h)$.
Since $l \leq N_0$ we have then obtained
\begin{align*}
{\| \partial_\a \bC (f, g, h) - \bC (f, g, \partial_\a h) - J_1 \|}_{H^l} 
  & \leq c {\| (f,g,h,\z_\a-1) \|}_{H^{l+1}} {\| (f,g,h) \|}_{W^{\frac{N_0}{2}+3,\infty}}^2 \, .
\end{align*}

To estimate
\begin{align*}
\bC (f, g, \partial_\a h) =  \int \frac{ (f(\a) - f(\b))(g(\a) - g(\b)) }{ {(\z(\a)-\z(\b))}^2 } \partial_\b h(\b) \,d\b
\end{align*}
we need to get rid of the extra derivative falling on $h$. Integrating by parts in $\b$ we have
\begin{align*}
\bC (f, g, \partial_\a h) & = \bQ (g, h \partial_\a f) +  \bQ (f, h \partial_\a g) + J_2(\a)
\end{align*}
where
\begin{align*}
 J_2(\a) & = - 2 \int \frac{ (f(\a) - f(\b))(g(\a) - g(\b)) }{ {(\z(\a)-\z(\b))}^3 } \z_\b(\b) h(\b) \,d\b \, .
\end{align*}
Using \eqref{proCCmain11}  we can bound
\begin{align*}
{\| \bQ (g, h\partial_\a f) \|}_{H^l} +  {\| \bQ (f, h\partial_\a g) \|}_{H^l} 
 \leq c {\| (f,g,h,\z_\a-1) \|}_{H^{l+1}} {\| (f,g,h) \|}_{W^{\frac{N_0}{2}+3,\infty}}^2
\end{align*}
as desired.
To bound $J_2$, which is an operator of the form $C_1(\z-\id, (f,g), \z_\a h)$, 
we can again commute derivatives via \eqref{commK2} and apply Theorem \ref{theoCMM} to obtain:
\begin{align*}
{\| J_2 \|}_{H^l} & \leq c {\| (f,g,h,\z_\a-1)\|}_{H^{l+1}} {\| (f,g,h) \|}_{W^{\frac{N_0}{2}+3,\infty}}^2 \, .
\end{align*}
We have then shown
\begin{align*}
{\| \partial_\a \bC (f, g, h) - J_1 \|}_{H^l} & \leq c {\| (f,g,h,\z_\a-1) \|}_{H^{l+1}} {\| (f,g,h) \|}_{W^{\frac{N_0}{2}+3,\infty}}^2 \, .
\end{align*}

To eventually estimate $J_1$ we notice that
\begin{align*}
J_1(\a) = - \z_\a \bC (f, g, h) + \bC (f, g, \z_\a h) \, ,
\end{align*}
so that 
\begin{align*}
{\| J_1 \|}_{H^l} \leq c {\| \bC (f, g, h) \|}_{H^l} + {\| \bC (f, g, \z_\a h) \|}_{H^l} \, .
\end{align*} 
Using the inductive hypotheses we see that
\begin{align*}
{\| J_1 \|}_{H^l} & \leq c {\| (f,g,h,\z_\a-1) \|}_{H^l} {\| (f,g,h) \|}_{W^{\frac{N_0}{2}+3,\infty}}^2  
+  c {\| (f,g, \z_\a h,\z_\a-1) \|}_{H^l}  {\| (f,g, \z_\a h) \|}_{W^{\frac{N_0}{2}+3,\infty}}^2 
\\
& \leq c {\| (f,g,h,\z_\a-1) \|}_{H^l} {\| (f,g,h)\|}_{W^{\frac{N_0}{2}+3,\infty}}^2  
\end{align*} 
where the constant $c$ depends only on lower Sobolev norms of $(f,g,h,\z_\a-1)$, which are unifromly bounded by assumption.
It follows that
\begin{align*}
{\| \partial_\a \bC (f, g, h) \|}_{H^l} & \leq c {\| (f,g,h,\z_\a-1) \|}_{H^{l+1}} {\| (f,g,h) \|}_{W^{\frac{N_0}{2}+3,\infty}}^2 \, ,
\end{align*}
which gives the bound on the first term on the left-hand side of \eqref{estCCmain01}.

To complete the proof of \eqref{estCCmain01} we need to show
\begin{align*}
{\| \bC (f, g, \partial_\a h) \|}_{H^m} & \leq c {\| (f,g,h,\z_\a-1) \|}_{H^m} {\| (f,g,h) \|}^2_{W^{\frac{N}{2}+3,\infty}} \, .
\end{align*}
Again we proceed by induction, the base being given by \eqref{estCCmain012} which has already been verified.
The argument is similar to those above.
Applying a derivative to $\bC (f, g, \partial_\a h)$ we get
\begin{align*}
\partial_\a \bC (f, g, \partial_\a h) & = \bC (\partial_\a f, g, \partial_\a h) 
  + \bC (f, \partial_\a g, \partial_\a h) + \bC (f, g, \partial^2_\a h) + J_3(\a)
\end{align*}
where
\begin{align*}
J_3(\a) = - 2 \int \frac{ (f(\a) - f(\b))(g(\a) - g(\b)) ( \partial_\a \z(\a) - \partial_\b \z(\b)) }{ {(\z(\a)-\z(\b))}^3 } 
  \partial_\b h(\b) \,d\b  \, .
\end{align*}
Since
\begin{align*}
\bC (\partial_\a f, g, \partial_\a h) = \partial_\a f \bQ (g, \partial_\a h) - \bQ (g, \partial_\a f \partial_\a h)
\end{align*}
this term can be directly estimated using \eqref{estbQL^inftyfg} and \eqref{proCCmain11}. 
One can bound similarly $\bC (f, \partial_\a g, \partial_\a h)$.

To control $\bC (f, g, \partial^2_\a h)$ we need to resort again to an integration by parts to remove the presence of the extra derivative.
More precisely we have
\begin{align*}
\bC (f, g, \partial_\a^2 h) & = \bQ (g, \partial_\a h \partial_\a f) +  \bQ (f, \partial_\a h \partial_\a g) + J_4(\a)
\end{align*}
where
\begin{align*}
 J_4(\a) & = - 2 \int \frac{ (f(\a) - f(\b))(g(\a) - g(\b)) }{ {(\z(\a)-\z(\b))}^3 } \z_\b(\b) \partial_\b h(\b) \,d\b \, .
\end{align*}
The terms $\bQ (g, \partial_\a h \partial_\a f)$ and $\bQ (f, \partial_\a h \partial_\a g)$ can be estimated via \eqref{proCCmain11}:
\begin{align*}
 {\| \bQ (g, \partial_\a h \partial_\a f) \|}_{H^l} + {\| \bQ (f, \partial_\a h \partial_\a g) \|}_{H^l} 
& \leq c {\| (f,g,h,\z_\a-1) \|}_{H^{l+1}} {\| (f,g,h) \|}^2_{W^{\frac{N_0}{2}+3,\infty}} \, .
\end{align*}
Similarly to what has been done before, we can expand the factor ${(\z(\a)-\z(\b))}^{-3}$, 
and write $J_4$ as an operator of the type $C_1$ as in \eqref{C_11}.
By using the commutation identity \eqref{commK2} we can then bound it by
\begin{align*}
 {\| J_4 \|}_{H^l} \leq c {\| \partial_\a (f,g,h,\z_\a-1) \|}_{H^l} {\| \partial_\a (f,g,h) \|}^2_{W^{\frac{l}{2}+2,\infty}} 
\leq c {\| (f,g,h,\z_\a-1) \|}_{H^{l+1}} {\| (f,g,h) \|}^2_{W^{\frac{N_0}{2}+3,\infty}} \, .
\end{align*}
This shows that
\begin{align*}
{\| \partial_\a \bC (f, g, \partial_\a h) - J_3 \|}_{H^l} 
  & \leq c {\| (f,g,h,\z_\a-1) \|}_{H^{l+1}} {\| (f,g,h) \|}_{W^{\frac{N_0}{2}+3,\infty}}^2 \, .
\end{align*}
To eventually bound $J_3$ in $H^l$ notice that it can be written as follows:
\begin{align*}
J_3 = - 2 \partial_\a \z  \bC (f, g, \partial_\a h) - J_4 \, .
\end{align*}
Using the inductive hypothesis for the first summand above, and the bound we have already obtained for $J_4$,
one can easily see how the desired bound for $J_3$ follows. This eventually yields
\begin{align*}
{\| \partial_\a \bC (f, g, \partial_\a h) \|}_{H^l} & \leq c {\| (f,g,h,\z_\a-1) \|}_{H^{l+1}} {\| (f,g,h) \|}_{W^{\frac{N_0}{2}+3,\infty}}^2 \, ,
\end{align*}
thereby completing the proof of \eqref{estCCmain01}.

\vskip10pt
We now prove the estimate in the space $S^{-1} H^k$ with $k := [\frac{m}{2}]$, $0 \leq m \leq N_0$:
\begin{align}
\label{estCCmain02}
{\| S \bC (f, g, h) \|}_{H^k} + {\| S \bC (f, g, \partial_\a h) \|}_{H^k} 
  & \leq c {\| S L \|}_{H^k} {\| L \|}_{W^{\frac{N_0}{2}+2,\infty}}^2 \, .
\end{align}
For simplicity we just show the proof of the bound for the second term in the above right-hand side.
The first term can be bounded similarly, and it is actually easier to estimate, since there is one less derivative on $h$ to worry about.
Let us start by computing $S \bC (f, g, \partial_\a h)$ in $L^2$.
By using the commutation identity \eqref{commK3}, and $[S, \partial_\a] =  - \partial_\a$, we see that
\begin{align*}
S \bC (f, g, \partial_\a h) & = \bC (S f, g, \partial_\a h) 
  + \bC (f, S g, \partial_\a h) + \bC (f, g, \partial_\a  S h) +  K_1(\a)
\end{align*}
where
\begin{align*}
 K_1 (\a) & = - 2 \int \frac{ (f(\a) - f(\b))(g(\a) - g(\b)) (S\z(\a) - S\z(\b)) }{ {(\z(\a)-\z(\b))}^3 } \partial_\b h(\b) \,d\b \, .
\end{align*}
Notice that we can write
\begin{align*}
\bC (S f, g, \partial_\a h) = S f \bQ (g,\partial_\a h) - \bQ (g, S f \partial_\a h) \, .
\end{align*}
Then, using the $W^{l,\infty}$ estimate \eqref{estQL^infty} and the $H^l$ estimate \eqref{proCCmain11} for $\bQ$, we see that
for any $l \leq k$:
\begin{align*}
{\| \bC (S f, g, \partial_\a h) \|}_{H^l} & \leq {\| S f \|}_{H^l} {\| \bQ (g,\partial_\a h) \|}_{W^{l,\infty}} 
  +  {\| \bQ (g, S f \partial_\a h) \|}_{H^l}
\\
& \leq c {\| S f \|}_{H^l} {\| (g ,\partial_\a h) \| }_{W^{l+2,\infty}}^2 + 
  {\| \partial_\a g \|}_{W^{l,\infty}} {\| S f \partial_\a h \|}_{H^l}
  \lesssim {\| S f \|}_{H^l} {\| (g,h) \| }_{W^{l+3,\infty}}^2 \, .
\end{align*}
An analogous bound holds for $\bC (f, S g, \partial_\a h)$.

To control  $\bC (f, g, \partial_\a S h)$ we need to integrate by parts in order to remove the derivative from $S h$.
This integration by parts gives:
\begin{align*}
\bC (f, g, \partial_\a S h) = \bQ (g, \partial_\a f S h) + \bQ (f, \partial_\a g S h) + K_2(\a)
\end{align*}
where
\begin{align*}
 K_2 (\a) & = - 2 \int \frac{ (f(\a) - f(\b))(g(\a) - g(\b)) }{ {(\z(\a)-\z(\b))}^3 } \partial_\b \z(\b) h(\b) \,d\b \, .
\end{align*}
The $\bQ$ terms can be treated as before, and therefore satisfy the desired bound.
Thus, so far we have obtained
\begin{align*}
{\| S \bC (f, g, \partial_\a h) - K_1 - K_2 \|}_{H^l} & \leq c {\| S (f,g,h) \|}_{H^l} {\| (f,g,h) \|}_{W^{l+2,\infty}}^2
\end{align*}
for any $l \leq k$.

To conclude we notice that $K_1$, respectively $K_2$, are operators of the form $C_1(H,A,f)$ as in \eqref{C_11},
for some smooth $F$, $H = \z - \id$, and $(A,f) = (f,g,S\z , \partial_\a h)$, respectively $(A,f)=(f,g, \z_\a h)$.
Commuting derivatives, using Theorem \ref{theoCMM} and the assumptions, we can deduce that
\begin{align*}
{\| K_1 \|}_{H^l} + {\| K_2 \|}_{H^l} & \leq 
  c {\| \partial_\a f \|}_{H^l} {\| \partial_\a g \|}_{W^{l,\infty}} {\| \partial_\a S \z \|}_{H^l} {\| \partial_\a  h \|}_{W^{l,\infty}}
  + c {\| \partial_\a f \|}_{H^l} {\| \partial_\a g \|}_{W^{l,\infty}} {\| \z_\a h \|}_{W^{l,\infty}}
\\
& 
\leq c {\| (f,g,h) \|}_{W^{l+1,\infty}}^2 {\| (\partial_\a f, S(\z_\a-1)) \|}_{H^l}
\end{align*}
where $c$ depends only on the $H^{l+2}$-norm of $(f,g,h, \z_\a-1)$, which is uniformly bounded by assumptions.
Here we  have used $\partial_\a S \z = S (\z_\a - 1) + \partial_\a \z$ for the second inequality.
We can then conclude that
\begin{align}
{\| S \bC (f, g, \partial_\a h) \|}_{H^l} & \leq c {\| S (f,g,h,\z_\a-1) \|}_{H^l} {\| (f,g,h) \|}_{W^{l+3,\infty}}^2
\end{align}
for any $l \leq k$. This shows the validity of \eqref{estCCmain02} and finishes the proof of \eqref{estCCmain00}. $\hfill \Box$

\subsubsection{Proof of Lemma \ref{lemHinfty}}
We want to show that for any $f$ with ${\| f \|}_{H^{k+2}} \leq 1$, $0\leq k \leq N_1$,  we have
\begin{align}
\label{lemHinfty1}
{\| \H \partial_\a f \|}_{W^{k,\infty}}  \lesssim {\| f \|}_{W^{k+2,\infty}} \, . 
\end{align}
From the definition of $\H$ we can write $i\pi \H \partial_\a f = I_1 + I_2$ with
\begin{align*}
I_1 (\a) & = \int \frac{ \partial_\b f(\b) }{ \z(\a)-\z(\b) } \,d\b
\\
I_2 (\a) & = \int \frac{ \partial_\b f(\b) }{ \z(\a)-\z(\b) } (\partial_\b \z(\b) - 1) \,d\b = \H \frac{\partial_\a f (\z_\a -1)}{\z_\a} \, .
\end{align*}
$I_2$ is a quadratic term and can be directly estimated using Sobolev's embedding and the boundedness of $\H$ on Sobolev spaces \eqref{estH}:
\begin{align*}
{\| I_2 \|}_{W^{k,\infty}} & \lesssim  {\| I_2 \|}_{H^{k+1}} \lesssim  
    {\| \partial_\a f (\z_\a -1) \|}_{H^{k+1}} + {\| \z_\a - 1 \|}_{H^{k+1}} {\| \partial_\a f (\z_\a -1) \|}_{W^{\frac{k}{2}+1,\infty}}
\\
& \lesssim {\| \z_\a -1 \|}_{H^{k+1}} {\| f \|}_{W^{k+2,\infty}} 
  \lesssim  {\| f \|}_{W^{k+2,\infty}}
\end{align*}
having used ${\| f \|}_{H^{k+2}} \leq 1$ and the assumption on the Sobolev norm of $\z_\a -1$.

To estimate $I_1$ we expand the expression $ (\z(\a)-\z(\b))^{-1}$ in a geometric sum as follows:
\begin{align*}
\frac{1}{ \z(\a)-\z(\b) } & = \frac{1}{\a-\b} \sum_{k \geq 0} {\left( \frac{\z(\a) - \a - (\z(\b) - \b)}{\a-\b} \right)}^k 
  \\
& =  \frac{1}{\a-\b} + \frac{H(\a)- H(\b)}{\a-\b} F \left( \frac{H(\a)- H(\b)}{\a-\b} \right)  \, ,
\end{align*}
where $H := \z - \id$ and $F$ is a smooth function.
We can then write
\begin{align*}
I_1 (\a) & = I_0(\a) + C_2 (F,H,f) (\a)
\end{align*}
where $C_2$ is as in \eqref{C_21}, and
\begin{align}
I_0 (\a) & = \int \frac{ \partial_\b f(\b) }{ \a-\b } \, d\b 
\end{align}
is a constant multiple of the (flat) Hilbert transform $H_0 := \H_{\mbox{\tiny id}}$.
To estimate $C_2(F,H,\partial_\a f)$ we can use Sobolev's embedding, 
the commutation identity \eqref{commK2}, and the bounds provided by Theorem \ref{theoCMM}:
\begin{align*}
{\| C_2(F,H,\partial_\a f) \|}_{W^{k,\infty}} & \lesssim  {\| C_2(F,H,\partial_\a f) \|}_{H^{k+1}} \lesssim  
    {\| f \|}_{H^{k+1}} {\| H_\a \|}_{W^{k+1,\infty}} \lesssim  {\| f \|}_{W^{k+1,\infty}} \, . 
\end{align*}
In the last inequality above we have used again the assumptions ${\| f \|}_{H^{k+1}} \leq 1$
and ${\| \z_\a -1 \|}_{H^{k+1}} \leq \frac{1}{2}$.
So far we have shown
\begin{align*}
{\| \H \partial_\a f - H_0 \partial_\a f \|}_{W^{k,\infty}}  \lesssim {\| f \|}_{W^{k+2,\infty}} \, . 
\end{align*}
Applying the Littlewood-Paley decomposition to $f$, and using the boundedness of $H_0 P_l$ on $L^\infty$, we see that
\begin{align*}
{\| H_0 \partial_\a f \|}_{W^{k,\infty}} \lesssim \sum_l {\| H_0 P_l \partial_\a f \|}_{W^{k,\infty}}
  \lesssim \sum_l 2^l {\| f \|}_{W^{k,\infty}} \lesssim {\| f \|}_{W^{k+2,\infty}} \, .
\end{align*}
This concludes the proof of \eqref{lemHinfty1}. 
The bound for the second summand in the left-hand side of \eqref{estHdfL^infty} can be obtain similarly.

To prove \eqref{estHL^infty} one can use an argument similar to the one just showed,
replacing $\partial_\a  f$ with $f$. The same estimates as above will show:
\begin{align*}
{\| \H f - H_0 f \|}_{W^{k,\infty}}  \lesssim {\| f \|}_{W^{k+1,\infty}} \, .
\end{align*}
To conclude it is then enough to observe that for any $ 2 \leq p < \infty$
\begin{align*}
{\| H_0 f \|}_{W^{k,\infty}}  \lesssim {\| H_0 f \|}_{W^{k+1,p}}  \lesssim {\| f \|}_{W^{k+1,p}} \, .
\end{align*}
The second summand in the left-hand side of \eqref{estHL^infty} can be estimated analogously. $\hfill \Box$

%
%
%
%
\section{The symbols $c^{\iota_1\iota_2\iota_3}$}\label{secsym}
In this section we calculate explicitly the symbols $c^{\iota_1\iota_2\iota_3}$ defined in \eqref{fd5} and prove the bounds \eqref{csymbols} and \eqref{csymbols2}. 

With $V^+(t)=V(t)$ and $V^-(t)=\overline{V}(t)$, recall that
\begin{equation*}
H(t)=\frac{V^+(t)+V^-(t)}{2},\qquad \Psi(t)=\frac{i[\Lambda^{-1}V^-(t)-\Lambda^{-1}V^+(t)]}{2}.
\end{equation*}
Starting from the formula \eqref{p_3},
\begin{equation*}
M_3(H,H,\Psi)=- (1/2) \La \left[  H^2 \La^2 \Psi + \La (H^2 \La \Psi)
	  - 2 H \La (H \La \Psi ) \right],
\end{equation*}
we calculate easily
\begin{equation*}
\mathcal{F}[M_3(H,H,\Psi)](\xi)=\frac{i}{4\pi^2}\sum_{(\iota_1\iota_2\iota_3)}\int_{\mathbb{R}^2}c_1^{\iota_1\iota_2\iota_3}(\xi,\eta,\sigma)\widehat{V^{\iota_1}}(\xi-\eta,t)\widehat{V^{\iota_2}}(\eta-\sigma,t)\widehat{V^{\iota_3}}(\sigma,t)\,d\eta d\sigma,
\end{equation*}
where the sum is taken over $(\iota_1\iota_2\iota_3)\in\{(++-),(--+),(+++),(---)\}$, and
\begin{equation}\label{sy1}
\begin{split}
c^{++-}_1(\xi,\eta,\sigma)&=\frac{2|\xi||\eta-\sigma|^{3/2}-|\xi||\sigma|^{3/2}+2|\xi|^2|\eta-\sigma|^{1/2}-|\xi|^2|\sigma|^{1/2}}{16}\\
&+\frac{-|\xi||\xi-\sigma||\eta-\sigma|^{1/2}-|\xi||\eta||\eta-\sigma|^{1/2}+|\xi||\eta||\sigma|^{1/2}}{8},\\
c^{+++}_1(\xi,\eta,\sigma)&=\frac{|\xi||\sigma|^{3/2}+|\xi|^2|\sigma|^{1/2}-2|\xi||\eta||\sigma|^{1/2}}{16},\\
c^{--+}_1(\xi,\eta,\sigma)&=-c^{++-}_1(\xi,\eta,\sigma),\\
c^{---}_1(\xi,\eta,\sigma)&=-c^{+++}_1(\xi,\eta,\sigma).
\end{split}
\end{equation}

Using now the formula \eqref{q_3},
\begin{equation*}
Q_3(\Psi,H,\Psi)=|\partial_x|\Psi\big[H|\partial_x|^2\Psi-|\partial_x|(H|\partial_x|\Psi)\big]
\end{equation*}
we calculate easily
\begin{equation*}
\mathcal{F}[i\Lambda Q_3(\Psi,H,\Psi)](\xi)=\frac{i}{4\pi^2}\sum_{(\iota_1\iota_2\iota_3)}\int_{\mathbb{R}^2}c_2^{\iota_1\iota_2\iota_3}(\xi,\eta,\sigma)\widehat{V^{\iota_1}}(\xi-\eta,t)\widehat{V^{\iota_2}}(\eta-\sigma,t)\widehat{V^{\iota_3}}(\sigma,t)\,d\eta d\sigma,
\end{equation*}
where the sum is taken over $(\iota_1\iota_2\iota_3)\in\{(++-),(--+),(+++),(---)\}$, and
\begin{equation}\label{sy2}
\begin{split}
c^{++-}_2&(\xi,\eta,\sigma)=\frac{|\xi|^{1/2}|\xi-\eta|^{3/2}|\sigma|^{1/2}+|\xi|^{1/2}|\xi-\eta|^{1/2}|\sigma|^{3/2}-|\xi|^{1/2}|\xi-\eta|^{1/2}|\eta-\sigma|^{3/2}}{8}\\
&+\frac{-|\xi|^{1/2}|\xi-\sigma||\eta-\sigma|^{1/2}|\sigma|^{1/2}-|\xi|^{1/2}|\xi-\eta|^{1/2}|\eta||\sigma|^{1/2}+|\xi|^{1/2}|\xi-\eta|^{1/2}|\eta||\eta-\sigma|^{1/2}}{8},\\
c^{+++}_2&(\xi,\eta,\sigma)=\frac{-|\xi|^{1/2}|\xi-\eta|^{1/2}|\eta-\sigma|^{3/2}+|\xi|^{1/2}|\xi-\eta|^{1/2}|\eta||\sigma|^{1/2}}{8},\\
c^{--+}_2&(\xi,\eta,\sigma)=c^{++-}_2(\xi,\eta,\sigma),\\
c^{---}_2&(\xi,\eta,\sigma)=c^{+++}_2(\xi,\eta,\sigma).
\end{split}
\end{equation}

Let
\begin{equation}\label{sy3}
\widetilde{q}_2(\xi,\eta)=\Lambda^{-1}(\xi-\eta)\Lambda^{-1}(\eta)q_2(\xi,\eta),\qquad \widetilde{m}_2(\xi,\eta)=\Lambda^{-1}(\eta)m_2(\xi,\eta).
\end{equation}
Using the fact that $A,Q_2$ are symmetric we calculate
\begin{equation*}
\mathcal{F}[2A(M_2(H,\Psi),H)](\xi)=\frac{i}{4\pi^2}\sum_{(\iota_1\iota_2\iota_3)}\int_{\mathbb{R}^2}c_3^{\iota_1\iota_2\iota_3}(\xi,\eta,\sigma)\widehat{V^{\iota_1}}(\xi-\eta,t)\widehat{V^{\iota_2}}(\eta-\sigma,t)\widehat{V^{\iota_3}}(\sigma,t)\,d\eta d\sigma,
\end{equation*}
where the sum is taken over $(\iota_1\iota_2\iota_3)\in\{(++-),(--+),(+++),(---)\}$, and
\begin{equation}\label{sy5}
\begin{split}
c^{++-}_3&(\xi,\eta,\sigma)=\frac{a(\xi,\eta)\widetilde{m}_2(\eta,\sigma)-a(\xi,\eta)\widetilde{m}_2(\eta,\eta-\sigma)-a(\xi,\xi-\sigma)\widetilde{m}_2(\xi-\sigma,\xi-\eta)}{4},\\
c^{+++}_3&(\xi,\eta,\sigma)=\frac{-a(\xi,\eta)\widetilde{m}_2(\eta,\sigma)}{4},\\
c^{--+}_3&(\xi,\eta,\sigma)=-c^{++-}_3(\xi,\eta,\sigma),\\
c^{---}_3&(\xi,\eta,\sigma)=-c^{+++}_3(\xi,\eta,\sigma).
\end{split}
\end{equation}

Similarly we calculate
\begin{equation*}
\mathcal{F}[i\Lambda B(M_2(H,\Psi),\Psi)](\xi)=\frac{i}{4\pi^2}\sum_{(\iota_1\iota_2\iota_3)}\int_{\mathbb{R}^2}c_4^{\iota_1\iota_2\iota_3}(\xi,\eta,\sigma)\widehat{V^{\iota_1}}(\xi-\eta,t)\widehat{V^{\iota_2}}(\eta-\sigma,t)\widehat{V^{\iota_3}}(\sigma,t)\,d\eta d\sigma,
\end{equation*}
where the sum is taken over $(\iota_1\iota_2\iota_3)\in\{(++-),(--+),(+++),(---)\}$, and
\begin{equation}\label{sy6}
\begin{split}
c^{++-}_4(\xi,\eta,\sigma)&=\frac{|\xi|^{1/2}b(\xi,\xi-\eta)|\xi-\eta|^{-1/2}\widetilde{m}_2(\eta,\sigma)-|\xi|^{1/2}b(\xi,\xi-\eta)|\xi-\eta|^{-1/2}\widetilde{m}_2(\eta,\eta-\sigma)}{8}\\
&+\frac{|\xi|^{1/2}b(\xi,\sigma)|\sigma|^{-1/2}\widetilde{m}_2(\xi-\sigma,\xi-\eta)}{8},\\
c^{+++}_4(\xi,\eta,\sigma)&=\frac{-|\xi|^{1/2}b(\xi,\xi-\eta)|\xi-\eta|^{-1/2}\widetilde{m}_2(\eta,\sigma)}{8},\\
c^{--+}_4(\xi,\eta,\sigma)&=c^{++-}_4(\xi,\eta,\sigma),\\
c^{---}_4(\xi,\eta,\sigma)&=c^{+++}_4(\xi,\eta,\sigma).
\end{split}
\end{equation}

Finally we calculate
\begin{equation*}
\mathcal{F}[i\Lambda B(H,Q_2(\Psi,\Psi))](\xi)=\frac{i}{4\pi^2}\sum_{(\iota_1\iota_2\iota_3)}\int_{\mathbb{R}^2}c_5^{\iota_1\iota_2\iota_3}(\xi,\eta,\sigma)\widehat{V^{\iota_1}}(\xi-\eta,t)\widehat{V^{\iota_2}}(\eta-\sigma,t)\widehat{V^{\iota_3}}(\sigma,t)\,d\eta d\sigma,
\end{equation*}
where the sum is taken over $(\iota_1\iota_2\iota_3)\in\{(++-),(--+),(+++),(---)\}$, and
\begin{equation}\label{sy7}
\begin{split}
c^{++-}_5&(\xi,\eta,\sigma)=\frac{2|\xi|^{1/2}b(\xi,\eta)\widetilde{q}_2(\eta,\sigma)-|\xi|^{1/2}b(\xi,\xi-\sigma)\widetilde{q}_2(\xi-\sigma,\xi-\eta)}{8},\\
c^{+++}_5&(\xi,\eta,\sigma)=\frac{-|\xi|^{1/2}b(\xi,\eta)\widetilde{q}_2(\eta,\sigma)}{8},\\
c^{--+}_5&(\xi,\eta,\sigma)=c^{++-}_5(\xi,\eta,\sigma),\\
c^{---}_5&(\xi,\eta,\sigma)=c^{+++}_5(\xi,\eta,\sigma).
\end{split}
\end{equation}

The following lemma gives the desired bounds \eqref{csymbols} and \eqref{csymbols2}.

\begin{lem}\label{cprop}
The symbols $c^{\iota_1\iota_2\iota_3}$ satisfy the uniform bounds
\begin{equation}\label{sy10}
\big\|\mathcal{F}^{-1}[c^{\iota_1\iota_2\iota_3}(\xi,\eta,\sigma)\cdot\varphi_l(\xi)\varphi_{k_1}(\xi-\eta)\varphi_{k_2}(\eta-\sigma)\varphi_{k_3}(\sigma)]\big\|_{L^1(\mathbb{R}^3)}\lesssim 2^{l/2}2^{2\max(k_1,k_2,k_3)},
\end{equation}
for any $(\iota_1\iota_2\iota_3)\in\{(++-),(--+),(+++),(---)\}$ and $l,k_1,k_2,k_3\in\mathbb{Z}$. Moreover, for any $\mathbf{k}=(k_1,k_2,k_3),\mathbf{l}=(l_1,l_2,l_3)\in\mathbb{Z}^3$ let
\begin{equation*}
\begin{split}
&c^\ast_\xi(x,y)=c^{++-}(\xi,-x,-\xi-x-y),\\
&(\partial_xc^\ast_\xi)_{\mathbf{k},\mathbf{l}}(x,y)=(\partial_xc^\ast_\xi)(x,y)\cdot \varphi_{k_1}(\xi+x)\varphi_{k_2}(\xi+y)\varphi_{k_3}(\xi+x+y)\varphi_{l_1}(x)\varphi_{l_2}(y)\varphi_{l_3}(2\xi+x+y),\\
&(\partial_yc^\ast_\xi)_{\mathbf{k},\mathbf{l}}(x,y)=(\partial_yc^\ast_\xi)(x,y)\cdot \varphi_{k_1}(\xi+x)\varphi_{k_2}(\xi+y)\varphi_{k_3}(\xi+x+y)\varphi_{l_1}(x)\varphi_{l_2}(y)\varphi_{l_3}(2\xi+x+y).
\end{split}
\end{equation*}
Then, for any $\mathbf{k},\mathbf{l}\in\mathbb{Z}^3$, and $\xi\in\mathbb{R}$
\begin{equation}\label{sy11}
\begin{split}
&\|(\partial_xc^\ast_\xi)_{\mathbf{k},\mathbf{l}}\|_{\mathcal{S}^\infty}\lesssim 2^{-\min(k_1,k_3)}2^{5\max(k_1,k_2,k_3)/2},\\
&\|(\partial_yc^\ast_\xi)_{\mathbf{k},\mathbf{l}}\|_{\mathcal{S}^\infty}\lesssim 2^{-\min(k_2,k_3)}2^{5\max(k_1,k_2,k_3)/2}.
\end{split}
\end{equation}
\end{lem}

\begin{proof}[Proof of Lemma \ref{cprop}] Clearly, for any $(\iota_1\iota_2\iota_3)\in\{(++-),(--+),(+++),(---)\}$
\begin{equation*}
c^{\iota_1\iota_2\iota_3}=\sum_{l=1}^5c^{\iota_1\iota_2\iota_3}_l.
\end{equation*}
The bound \eqref{sy10} follows from the explicit formulas \eqref{sy1}--\eqref{sy7}, the symbol bounds in Lemma \ref{description}, and the algebra properties in Lemma \ref{touse} (ii).

Let $\iota:\mathbb{R}\setminus\{0\}\to\{-1,1\}$, $\iota(x):=x/|x|$. Recalling the formulas in Lemma \ref{proE1+} and using \eqref{sy1}--\eqref{sy7} we calculate
\begin{equation*}
c^\ast_\xi(x,y)=c^\ast_{\xi,1}(x,y)+c^\ast_{\xi,2}(x,y)+c^\ast_{\xi,3}(x,y)+c^\ast_{\xi,4}(x,y)+c^\ast_{\xi,5}(x,y),
\end{equation*}
where
\begin{equation*}
\begin{split}
c^\ast_{\xi,1}(x,y)&=\frac{|\xi||\xi+y|^{3/2}}{8}-\frac{|\xi||\xi+x+y|^{3/2}}{16}+\frac{|\xi|^2|\xi+y|^{1/2}}{8}-\frac{|\xi|^2|\xi+x+y|^{1/2}}{16}\\
&-\frac{|\xi||2\xi+x+y||\xi+y|^{1/2}}{8}-\frac{|\xi||x||\xi+y|^{1/2}}{8}+\frac{|\xi||x||\xi+x+y|^{1/2}}{8},
\end{split}
\end{equation*}
\begin{equation*}
\begin{split}
c^\ast_{\xi,2}(x,y)&=\frac{|\xi|^{1/2}|\xi+x|^{3/2}|\xi+x+y|^{1/2}}{8}+\frac{|\xi|^{1/2}|\xi+x|^{1/2}|\xi+x+y|^{3/2}}{8}-\frac{|\xi|^{1/2}|\xi+x|^{1/2}|\xi+y|^{3/2}}{8}\\
&-\frac{|\xi|^{1/2}|2\xi+x+y||\xi+y|^{1/2}|\xi+x+y|^{1/2}}{8}\\
&-\frac{|\xi|^{1/2}|\xi+x|^{1/2}|x||\xi+x+y|^{1/2}}{8}+\frac{|\xi|^{1/2}|\xi+x|^{1/2}|x||\xi+y|^{1/2}}{8},
\end{split}
\end{equation*}
\begin{equation*}
\begin{split}
c^\ast_{\xi,3}(x,y)&=\frac{|\xi||\xi+x+y|^{1/2}\iota(\xi+x)[|x|\iota(\xi+x+y)-x]}{8}+\frac{|\xi||\xi+y|^{1/2}\iota(\xi+x)[|x|\iota(\xi+y)+x]}{8}\\
&-\frac{|\xi||\xi+x|^{1/2}\iota(\xi+x+y)[|2\xi+x+y|\iota(\xi+x)-(2\xi+x+y)]}{8},
\end{split}
\end{equation*}
\begin{equation*}
\begin{split}
c^\ast_{\xi,4}(x,y)&=\frac{|\xi|^{1/2}|\xi+x|^{1/2}|\xi+x+y|^{1/2}\iota(\xi)[|x|\iota(\xi+x+y)-x]}{8}\\
&+\frac{|\xi|^{1/2}|\xi+x|^{1/2}|\xi+y|^{1/2}\iota(\xi)[|x|\iota(\xi+y)+x]}{8}\\
&-\frac{|\xi|^{1/2}|\xi+x|^{1/2}|\xi+x+y|^{1/2}\iota(\xi)[|2\xi+x+y|\iota(\xi+x)-(2\xi+x+y)]}{8},
\end{split}
\end{equation*}
and
\begin{equation*}
\begin{split}
c^\ast_{\xi,5}(x,y)&=-\frac{|\xi|^{1/2}|\xi+y|^{1/2}|\xi+x+y|^{1/2}\iota(\xi)\iota(\xi+x)|x|[1-\iota(\xi+y)\iota(\xi+x+y)]}{8}\\
&-\frac{|\xi|^{1/2}|\xi+y|^{1/2}|\xi+x|^{1/2}\iota(\xi)\iota(\xi+x+y)|2\xi+x+y|[1+\iota(\xi+y)\iota(\xi+x)]}{16}.
\end{split}
\end{equation*}
The desired bounds \eqref{sy11} are verified easily for every term in these formulas.

Using these formulas, we also calculate
\begin{equation}\label{c0}
c^\ast_{\xi}(0,0)=-|\xi|^{5/2}/2.
\end{equation}
\end{proof}


\section{Estimate of remainder terms}
\label{appR}
In the first section below we give estimates for the Dirichlet-Neumann operator $\N$ in $L^2$, weighted $L^2$, and $L^1$-based Sobolev spaces.
We will the use these to establish several bounds for $R_1$ and $R_2$ in \eqref{defR_1}--\eqref{defR_2}.
We then proceed to estimate all quartic and higher order remainder terms, and in particular prove \eqref{bn22}.

\subsection{Dirichlet-to-Neumann operator: multilinear estimates}\label{secDN}
Here we recall that if $N$ denotes the outward normal vector of the interface $S_0$
\begin{align}
\label{DN}
\N(h) \phi := N \cdot \nabla  \Phi = N \cdot \nabla  \phi_\H \qquad , \qquad G(h) \phi = \sqrt{1 + {|h^\p|}^2} \N(h) \phi \, .
\end{align}
We are interested in particular in estimating quartic and higher order terms in the expansion of the Dirichlet-Neumann operator.
The $L^2$ and weighted $L^2$ estimates are needed to obtain the improved weighted bounds on $V$ in Proposition \ref{proE3}.
The $L^1$ estimates are used to bound these higher order terms in the $Z$-norm.

The first Proposition below gives estimates in $L^2$-based spaces and its proven in \ref{prproDNL^2}:
\begin{pro}\label{proDNL^2}
The Dirichlet-Neumann operator $G$ can be expanded in a series
\begin{align}
\label{DN0}
G(h) f = \sum_{n\geq 0} M_{n+1}(h,\dots, h, f) \, ,
\end{align}
where $M_{n+1}$ is an $n+1$-linear operator satisfying the following $L^2$ bounds:
\begin{align}
\label{DN1}
{\| M_{n+1} (h_1, \dots, h_n, f) \|}_{L^2}  \leq C_0^n \min \left\{ \prod_{i=1}^n {\| h_i^\p \|}_{L^\infty} {\| f^\p \|}_{L^2},
    \min_{j\in\{1,\dots,n\}} \prod_{i\neq j} {\| h_i^\p \|}_{L^\infty} {\| h_j^\p \|}_{L^2} {\| f^\p \|}_{L^\infty} \right\} \, ,
\end{align}
for some absolute constant $C_0$.

Moreover, $G$ is invariant under translation and scaling symmetries, and the following identities hold:
\begin{align}
\label{DN2}
\partial_x M_{n+1}(h_1,\dots, h_n, f) & = \sum_{i=1}^n M_{n+1}(h_1, \dots, \partial_x h_i, \dots, h_n, f) + M_{n+1}(h_1,\dots, h_n, \partial_x f) 
\\
\label{DN3}
\begin{split}
S M_{n+1}(h_1,\dots, h_n, f) & = \sum_{i=1}^n M_{n+1}(h_1,\dots, S h_i, \dots, h_n, f) + M_{n+1}(h_1,\dots, h_n, S f) 
\\
 & - \sum_{i=1}^n M_{n+1} ( h_1, \dots, h_n, f)  \, ,
\end{split}
\end{align}
where $S$ denotes the scaling vector field.
As a consequence, for any integer $l \geq 0$ one has:
\begin{align}
\label{DN4}
{\| M_{n+1}(h_1,\dots, h_n, f) \|}_{H^l} & \lesssim  
  \sum_{i=1}^n  {\| h_i^\p \|}_{H^l} \prod_{j \neq i} {\| h_j \|}_{W^{N_1,\infty}} {\| f_x \|}_{W^{N_1,\infty}} 
  + \prod_{i=1}^n {\| h_i \|}_{W^{N_1,\infty}} {\| f_x \|}_{H^l} \, ,
\\
\nn
{\| S M_{n+1}(h_1,\dots, h_n, f) \|}_{H^l} & \lesssim \sum_{i,j = 1, i\neq j}^n {\| (S h_i)^\p \|}_{H^l} {\| h_j^\p \|}_{W^{l,\infty}} 
  \prod_{k\neq i,j} {\| h_k \|}_{W^{N_1,\infty}} {\| f_x \|}_{W^{l,\infty}} 
\\
\label{DN5}
& + \sum_{i=1}^n {\| h_i^\p \|}_{H^l} \prod_{j \neq i} {\| h_j \|}_{W^{N_1,\infty}} 
  \left( {\| {(Sf)}^\p \|}_{H^l} + {\| f^\p \|}_{H^l} \right)  \, ,
\end{align}
where the implicit constants are bounded by $C_0^n$ for some absolute constant $C_0$.
\end{pro}

Let us denote
\begin{align}
\label{DN10}
{[ G (h) \phi ]}_{\geq 4}(t) := \sum_{n\geq 3} M_{n+1}(h(t),\dots, h(t), \phi(t))
\end{align}
to be the quartic and higher order terms (in $h$ and $\phi$) in the expansion of $G$.
A corollary of this expansion and Proposition \ref{proDNL^2} is the following:

\begin{cor}\label{corDN1}
Under the a priori assumptions \eqref{apriori0} on $h$ and $\phi$ one has 
\begin{align}
\label{DN11}
{\| {[ G (h) \phi ]}_{\geq 4}(t) \|}_{H^{N_0 - 2}} +  {\| S {[ G (h) \phi ]}_{\geq 4} (t) \|}_{H^{\frac{N_0}{2} - 2}}
  & \lesssim \e_1^4 {(1+t)}^{3p_0 -3/2} \, .
\end{align}
Moreover, for  $R_1$ and $R_2$ defined in \eqref{defR_1}-\eqref{defR_2} we have
\begin{align}
\label{estRL^2}
{\| (R_1 + i\Lambda R_2) (t) \|}_{H^{N_0-5}} +  {\| S (R_1 + i\Lambda R_2)(t) \|}_{H^{\frac{N_0}{2}-5}}
  & \lesssim \e_1^4 {(1+t)}^{3p_0 -3/2} \, .
\end{align}
\end{cor}

The next Proposition establishes $L^1$-type estimates:
\begin{pro}\label{proDNL^1}
With the same notations of Proposition \ref{proDNL^2}, and for any $n \geq 3$, we have
\begin{align}
\label{DN13}
{\left\| {|\partial_x|}^\frac{\b}{4} M_{n+1} (h_1, \dots, h_n, f) \right\|}_{W^{1,l}} 
  \lesssim \sum_{i=1}^{n}  {\| h_i\|}_{H^{l+3}} \prod_{j\neq i} {\| h_j \|}_{W^{N_1,\infty}} {\| f_x \|}_{H^l} \, .
\end{align}
As a consequence, under the a priori  assumption \eqref{apriori0} on $h$ and $\phi$,
\begin{align}
\label{DN13.1}
{\left\| {|\partial_x|}^\frac{\b}{4} {[ G (h) \phi ]}_{\geq 4}(t) \right\|}_{W^{1,N_0-10}} 
  \lesssim {(1+t)}^{3p_0 - 1} \, .
\end{align}
Here $\beta = 1/100$ is the parameter that appears in the definition of the $Z$ norm \eqref{defZ}.
\end{pro}

\subsubsection{Proof of Proposition \ref{proDNL^2}}\label{prproDNL^2}
This Proposition follows from some standard potential theory, arguments similar to those in \cite[sec. 7.2]{GMSC}, 
and Theorem \ref{theoCMM}. We sketch the proof below.

\vskip5pt
\paragraph{{\it Expansion of the Dirichlet-Neuman operator}}
In order to find an explicit formula for the Dirichlet-Neumann operator we start with an ansatz for the harmonic extension 
of a function $f$ to the domain $\{ (x,z) \, : \, z \leq h(x) \}$:
\begin{align}
 \label{ansatz}
\Phi (x,z) = \int \frac{1}{2} \log \left( {|x-y|}^2 + {|z-h(y)|}^2 \right) \rho(y) \, dy \, .
\end{align}
By standard potential theory one has
\begin{align}
\nn
\N(h) f(x) & = \lim_{z \rightarrow h(x)} \nabla \Phi (x,z) \cdot N(x)
\\
\label{N1}
& = \frac{1}{2} \frac{\rho(x) }{\sqrt{1+{|h^\p(x) |}^2}} + \int \frac{ h(x) - h(y) + h^\p(x) (y-x) }{ {|x-y|}^2 + {|h(x)-h(y)|}^2} \rho(y) \, dy \, .
\end{align}
We then aim at determining $\rho$ in terms of $h$ and $f$.
Using \eqref{ansatz} and $f(x) = \Phi (x,h(x))$, one has
\begin{align}
f(x) = \int \frac{1}{2} \log \left( {|x-y|}^2 + {|h(x)-h(y)|}^2 \right) \rho(y) \, dy \, .
\end{align}
It follows that
\begin{align*}
|\partial_x| f(x) & = iH_0 \int \frac{x-y + (h(x) - h(y))h^\p(x)}{{|x-y|}^2 + {|h(x)-h(y)|}^2} \rho(y) \, dy 
\\
& = \rho(x)  +  \sum_{n = 1}^{\infty} i H_0 \int \left( \frac{h(x)-h(y)}{x-y} \right)^{2n}
		\frac{x-y + (h(x) - h(y))h^\p(x)}{{(x-y)}^2} \rho(y) \, dy   
\\
& =: \rho(x) + \sum_{n = 1}^{\infty} P_{n}(h) \rho (x) \, .
\end{align*}
One can then invert the above series expansion and write
\begin{align}
 \label{exprho}
\rho = \sum_{k \geq 0} {(-1)}^k {\left[ \sum_{n=1}^\infty P_{n}(h) \right]}^k |\partial_x| f  
\end{align}
where
\begin{align}
\label{P_n}
P_{n}(h)g(x)  & := iH_0 \int \frac{ {\left( h(x) - h(y) \right)}^{2n} h^\p(x) }{ {(x-y)}^{2n+1} } g(y) \, dy 
	+ iH_0 h^\p(x) \int \frac{ {\left( h(x) - h(y) \right)}^{2n+1}  }{ {(x-y)}^{2n+2} } g(y) \, dy  \, .
\end{align}
Expanding the second summand in \eqref{N1} one can write
\begin{align}
\nn
\N(h) f(x) & = \frac{1}{2} \frac{\rho(x) }{\sqrt{1+{|h^\p(x) |}^2}} + \sum_{n=0}^\infty Q_{n}(h) \rho
\\
\label{expN}
Q_{n}(h) & := \int \left( \frac{ h(x) - h(y) }{x-y}\right)^{2n} \frac{h(x) - h(y) + h^\p(x) (y-x)}{{(x-y)}^2} \rho(y) \, dy  \, .
\end{align}
Putting together \eqref{DN}, \eqref{exprho} and \eqref{expN} we eventually obtain \eqref{DN0}.

\vskip5pt
\paragraph{{\it Symmetries and $L^2$-bounds}}
The basic $L^2$-type bounds \eqref{DN1} follow directly from the expansion \eqref{exprho}-\eqref{expN} and Theorem \ref{theoCMM}.
The formulas \eqref{DN2} and \eqref{DN3} follow from the space translation and scaling invariances of the 
basic operators $P_n$ and $Q_n$ in \eqref{P_n} and \eqref{expN}.
More precisely, for any $\d \in \R$ and $\l>0$
\begin{align*}
& [G( h (\cdot+\d) ) f(\cdot+\d)] (x) =  [G(h) f] (x+\d) \, ,
\\
& G \left( \frac{1}{\l} h (\l \cdot) \right) f(\l \cdot)  (x) =  \l [G(h) f] (\l x) \, .
\end{align*}
These identities hold true for each operator $M_n$ in the expansion \eqref{DN0}, that is
\begin{align}
\label{DNt}
& M_n ( h_1 (\cdot+\d), \dots, h_n (\cdot+\d) , f(\cdot+\d) ) (x) = M_n ( h_1,\dots, h_n, f) (x+\d)  \, ,
\\
\label{DNs}
& M_n \left( \frac{1}{\l} h_1 (\l \cdot), \dots, \frac{1}{\l} h_n (\l \cdot), f(\l \cdot) \right) (x) =  M_n (h_1,\dots, h_n, f) (\l x) \, ,
\end{align}
and can be verified directly on the operators $P_n$ and $Q_n$ defined above.
Differentiating with respect to the parameters in \eqref{DNt} and \eqref{DNs},
one sees that
\begin{align*}
\partial_x M_n ( h_1,\dots, h_n, f ) & = \sum_{i=1}^n M_n ( h_1, \dots, \partial_x h_i, \dots, h_n, f) 
  + M_n (h_1, \dots, h_n, \partial_x f)  \, ,
\\
x \partial_x M_n (h,\dots, h, f) (x) & = \sum_{i=1}^n M_n ( h_1, \dots, x \partial_x h_i, \dots, h_n, f)
  + M_n \left( h_1, \dots, h_n,  x \partial_x  f \right) 
\\ & - \sum_{i=1}^n M_n ( h_1, \dots, h_n, f) \, .
\end{align*}
The first identity is \eqref{DN2}.
If $h$ and $f$ depend on time, one can similarly derive \eqref{DN3} from the last identity above.
The estimates \eqref{DN4} and \eqref{DN5} follow by repeated applications of \eqref{DN2} and \eqref{DN3} 
and the $L^2$ estimate \eqref{DN1}. $\hfill \Box$

\subsubsection{Proof of Corollary \ref{corDN1}}\label{prcorDN1}
The estimate \eqref{DN11} is an immediate consequence of the bounds \eqref{DN4} and \eqref{DN5}.
To prove \eqref{estRL^2} it then suffices to prove
\begin{align}
\label{estRL^2pr}
{\| \Lambda R_2 (t) \|}_{H^{N_0-5}} +  {\| S \Lambda R_2(t) \|}_{H^{\frac{N_0}{2}-5}} & \lesssim \e_1^4 {(1+t)}^{3p_0 -3/2} \, .
\end{align}
From the definition of $R_2$ \eqref{defR_2} we see that
\begin{align}
\label{expR_2}
\begin{split}
R_2 & = {\left[ {(|\partial_x|\phi + M_2(h,\phi) + M_3 (h,h,\phi) + R_1(h,\phi)+ h_x\phi_x)}^2 \right]}_{\geq 4} 
  \\
& + {\left[ 2 {(1+{|h_x|}^2)}^{-1/2} \right] }_{\geq 2} {(|\partial_x| \phi + M_2(h,\phi) + M_3 (h,h,\phi) + R_1(h,\phi)+ h_x\phi_x)}^2 
\\
& = {(M_2(h,\phi) + M_3 (h,h,\phi) + R_1(h,\phi)+ h_x\phi_x)}^2 
  + 2 |\partial_x| \phi (M_3 (h,h,\phi) + R_1(h,\phi))
\\
& + {\left[ 2 {(1+{|h_x|}^2)}^{-1/2} \right] }_{\geq 2} {\left(|\partial_x| \phi + M_2(h,\phi) + M_3 (h,h,\phi) + R_1(h,\phi)+ h_x\phi_x\right)}^2 \, .
\end{split}
\end{align}
To obtain the desired bound it suffices to apply appropriately H\"{o}lder's inequality 
in combination with the a priori estimates \eqref{Eapriori}, the $L^2$ estimates \eqref{Al40} for $M_2$, and \eqref{Al45} for $M_3$,
and the following $L^\infty$ estimates:
\begin{align}
& {\|P_k M_2(h,\phi)\|}_{L^\infty} \lesssim \e_1^2 (1+t)^{-1} 2^k 2^{-N_0 k_+/2 },
\\
& {\|P_k M_3(h,h,\phi)\|}_{L^\infty} \lesssim \e_1^2 (1+t)^{-3/2} 2^k 2^{-N_0 k_+/2 } \, .
\end{align}
The last two estimates above can be obtained by inspection of \eqref{p_2} and \eqref{p_3}
using the a priori bounds \eqref{Al32}.

\subsubsection{Proof of Proposition \ref{proDNL^1}}\label{prproDNL^1}
Given the expansion of $\N$ in \eqref{exprho}-\eqref{expN}, the already established $L^2$-based estimates, 
and the commutation property \eqref{DN2},
it is not hard to see that \eqref{DN13} would follow if one can show that operators of the form
\begin{align*}
C_1 \left( h_1, \dots, h_n, f \right) 
  & :=  \mbox{p.v.} \int \frac{ \prod_{i=1}^n ( h_i(x) - h_i(y) ) }{ {(x-y)}^{n+1} }  f(y) \, dy
\end{align*}
satisfy
\begin{align}
\label{DN100}
{\left\| {|\partial_x|}^\frac{\b}{4} H_0 C_1 (h_1, \dots, h_n, f) \right\|}_{L^1} \lesssim \min_{i=1,\dots, n}  
  {\| h_i\|}_{H^3} \prod_{j\neq i} {\| h_j \|}_{W^{2,\infty}} {\| f \|}_{H^2} \, .
\end{align}
The above estimate is in turn implied by
\begin{align*}
{\left\| C_1 (h_1, \dots, h_n, f) \right\|}_{W^{1,1}} \lesssim \min_{i=1,\dots, n}  
  {\| h_i\|}_{H^3} \prod_{j\neq i} {\| h_j \|}_{W^{2,\infty}} {\| f \|}_{H^2} \, .
\end{align*}
Since the action of derivatives on operators of the type $C_1$ produces operators of the same type 
(acting on derivatives of the arguments), it is enough to obtain
\begin{align*}
{\left\| C_1 (h_1, \dots, h_n, f) \right\|}_{L^1} \lesssim \min_{i=1,\dots, n}  
  {\| h_i\|}_{H^2} \prod_{j\neq i} {\| h_j \|}_{W^{1,\infty}} {\| f \|}_{H^1} \, .
\end{align*}
We only provide details of the proof of the above estimate in the case $n=1$, that is
\begin{align}
\label{DN105}
{\left\| \mbox{p.v.} \int \frac{ h(y) - h(x) }{ {(y-x)}^{2} }  f(y) \, dy  \right\|}_{L^1}
  \lesssim {\| h \|}_{H^2} {\| f \|}_{H^1} \, ,
\end{align}
as the case $n \geq 2$ can be treated similarly. Let us write
\begin{align}
\nn
& \int \frac{ h(x) - h(y) }{ {(x-y)}^{2} }  f(y) \, dy = I_1 + I_2 + I_3
\\
\label{DN106}
& I_1(x) := \int_{|y-x| \geq 1} \frac{ h(y) - h(x) }{ {(y-x)}^{2} }  f(y) \, dy
\\
\label{DN107}
& I_2(x) := \int_{|y-x| \leq 1} \frac{ h(y) - h(x) - h^\p(x) (y-x)}{ {(y-x)}^{2} }  f(y) \, dy
\\
\label{DN108}
& I_3(x) := h^\p(x) \int_{|y-x| \leq 1} \frac{f(y)}{ y-x } \, dy \, .
\end{align}
Notice that $I_1$ can be written as $I_1 = K \ast (hf) - h K \ast f$,
where $K(x) := {|x|}^{-2} \chi_{|x| \geq 1}$ is an $L^1$ kernel. It follows that
\begin{align*}
& {\| I_1 \|}_{L^1} \lesssim {\| K \ast (hf) \|}_{L^1} + {\|h\|}_{L^2} {\|K \ast f\|}_{L^2}  \lesssim {\|h\|}_{L^2} {\|f\|}_{L^2} \, ,
\end{align*}
Using Taylor's formula and a change of variables we can write
\begin{align*}
& I_2 = - \int_{|y| \leq 1} \int_0^1 t h^{\p\p}(x + ty) \, dt  f(y+x) \, dy \, .
\end{align*}
It follows that
\begin{align*}
& {\| I_2 \|}_{L^1} \lesssim \int_0^1  \int_{|y| \leq 1} \int |h^{\p\p}(x + ty)|  |f(y+x)| \, dx \, dy \, dt 
\lesssim {\|h^{\p\p} \|}_{L^2} {\|f\|}_{L^2} \, .
\end{align*}
For the last term \eqref{DN108} we first write
\begin{align*}
& I_3 = h^\p(x) \int_{|y-x| \leq 1} \frac{f(y)-f(x)}{ y-x } \, dy  = h^\p(x) \int_{|y| \leq 1} \int_0^1 t f^{\p}(x + ty) \, dt \, dy
\end{align*}
and then estimate
\begin{align*}
& {\| I_3 \|}_{L^1} \lesssim \int_0^1  \int_{|y| \leq 1} \int  |h^\p(x)| |f^{\p}(x + ty)| \, dx \, dy \, dt \lesssim {\|h^\p\|}_{L^2} {\|f^\p\|}_{L^2} \, .
\end{align*}
This shows that \eqref{DN105} holds and completes the proof of Proposition \ref{proDNL^1}.  $\hfill \Box$


\subsection{Proof of \eqref{bn22}}

For $m\in\mathbb{Z}\cap[20,\infty)$, $k \in \Z \cap [-m/2,m/50-1000]$, $|\xi|\in[2^k,2^{k+1}]$, $t_1\leq t_2\in[2^{m-1},2^{m+1}]\cap[0,T]$, 
we want to show
\begin{equation}\label{bn22app}
\Big|\int_{t_1}^{t_2} e^{iH(\xi,s)}e^{is\Lambda(\xi)}\widehat{R}(\xi,s)\,ds\Big|
\lesssim \varepsilon_1^32^{-p_1m}(2^{\beta k}+2^{(N_1+15)k})^{-1}.
\end{equation}
where
\begin{equation}\label{fd3app}
R:=\mathcal{N}_3+\mathcal{N}_4-\widetilde{\mathcal{N}}_3.
\end{equation}
with $\N_3$, $\N_4$ and $\widetilde{\N}_3$ defined respectively in \eqref{defN_3}, \eqref{defN_4} and \eqref{fd2}.
To prove \eqref{bn22app} we will use Lemma \ref{lemR} and \ref{lemrem} below.

\begin{lem}\label{lemR}
Let $R$ be defined as in \eqref{fd3app}.
Under the a priori assumptions \eqref{Al32} on $h$ and $\phi$, we have for $k\in\Z$
\begin{align}
\label{lemR1}
\big| \what{P_k R}(\xi,t) \big|  \lesssim \e_1^4 {(1+t)}^{10p_0 - 1} 2^{-(N_0-20)k_+} 2^{-\b k/4}
\end{align}
and
\begin{align}
\label{lemR2}
{\| P_k R(t) \|}_{L^2} + {\| P_k S R(t) \|}_{L^2}  \lesssim \e_1^4 {(1+t)}^{20p_0 - 3/2} 2^{-(N_0/2-20)k_+} \, .
\end{align}
\end{lem}


\begin{lem}\label{lemrem}
Assume that a function $D = D(\xi,t)$ satisfies for all $t \in [0,T]$
\begin{align}
\label{asslemrem}
\begin{split}
& {\| D (t) \|}_{L^2} +  {\| S D(t) \|}_{L^2} \lesssim \d {(1+t)}^{-11/8 } \, ,
\\
&
{\| \what{D} (t) \|}_{L^\infty} \lesssim \d {(1+t)}^{20 p_0 -1} \, .
\end{split}
\end{align}
It follows that for $k\in\Z$, $|\xi| \in [2^k,2^{k+1}]$, $m\in\{1,2,\dots\}$ and $t_1\leq t_2 \in [2^m-2,2^{m+1}] \cap [0,T]$
\begin{align}
\label{conclemrem}	
\left| \int_{t_1}^{t_2} e^{i H(\xi,s)} e^{is\Lambda(\xi)} \what{D}(\xi,s) \, ds \right| & 
    \lesssim \d (1 + 2^{-k}) 2^{-m/16} \, . 
\end{align}
\end{lem}

We now show how \eqref{bn22app} follows from Lemma \ref{lemR} and \ref{lemrem}. 

\begin{proof}[Proof of \eqref{bn22app}]
From \eqref{lemR1} we see that for $|\xi| \in [2^k,2^{k+1}]$ one has
\begin{align*}
\Big|\int_{t_1}^{t_2} e^{iH(\xi,s)}e^{is\Lambda(\xi)}\what{R}(\xi,s)\,ds\Big| \lesssim 
  2^m \sup_{s \in [2^m-2,2^{m+1}]}  \big| \what{R}(\xi,s) \big| \lesssim 
  \e_1^4  2^{10p_0 m} 2^{-(N_0-20)k_+} 2^{-\b k/4} \, .
\end{align*}
Given our choice of $N_0$ and $N_1$, the desired bound \eqref{bn22app} follows for $k \geq 22p_0/(N_0-80)$ and $ k \leq - 44p_0/3\b$,
with any $p_1 \leq p_0$.
For the remaining frequencies 
\begin{align}
\label{freq1}
k \in [-44p_0/3\b, 22p_0/(N_0-80)] 
\end{align}
we want to apply Lemma \ref{lemrem} with 
$$D(\xi,t) = \left( 2^{\b k} + 2^{(N_1+15)k} \right) P_k R (\xi,t)$$ 
and $\d = \e_1^4$.
From \eqref{lemR1} and \eqref{lemR2} we see that
\begin{align*}
& \big| \what{D} (\xi,t) \big| \lesssim \e_1^4 {(1+t)}^{10p_0 - 1} 2^{-(N_0/2-50)k_+} 2^{3\b k/4} \lesssim \e_1^4 {(1+t)}^{10p_0 - 1} \, ,
\\
& {\| D (t) \|}_{L^2} +  {\| S D(t) \|}_{L^2} \lesssim \e_1^4 {(1+t)}^{20p_0 - 3/2} 2^{40 k_+} \lesssim  {(1+t)}^{21p_0 - 3/2} \, ,
\end{align*}
under the restriction \eqref{freq1}.
The hypotheses of Lemma \ref{lemrem} are then satisfied, and the conlcusion \eqref{conclemrem} implies 
\begin{align*}
\left| \int_{t_1}^{t_2} e^{i H(\xi,s)} e^{is\Lambda(\xi)} \what{R}(\xi,s) \, ds \right| & 
    \lesssim \e_1^4 {\left(2^{\b k} + 2^{(N_1 +15) k} \right)}^{-1} \left( 1 + 2^{-k} \right) 2^{-m/16} \, . 
\end{align*}
This gives \eqref{bn22app} in the considered frequency range \eqref{freq1}, by choosing $p_1 \leq 1/16 - 44p_0/3\b$.
\end{proof}

\subsubsection{Proof of Lemma \ref{lemR}}
Since $R =\mathcal{N}_3+\mathcal{N}_4-\widetilde{\mathcal{N}}_3$, from \eqref{defN_3}, \eqref{defN_4} and \eqref{fd2} we can write
\begin{align*}
R = \mathcal{N}_4 + \sum_{j=1}^5 \wt{\mathcal{N}}_{3,j}
\end{align*}
where we recall that
\begin{align}
\label{N_4app}
\N_4 & = R_1(h,\phi) +2A(M_3(h,h,\phi)+R_1(h,\phi),h)
\\
\nn
& + i \Lambda \left[R_2(h,\phi) + B(h,Q_3(\phi,h,\phi) + R_2(h,\phi))+ B(M_3(h,h,\phi) + R_1(h,\phi),\phi) \right],
\end{align}
and we have defined
\begin{align}
\label{N31}
\wt{\mathcal{N}}_{3,1} & := M_3(h,h,\phi) - M_3(H,H,\Psi) \, ,
\\
\label{N32}
\wt{\mathcal{N}}_{3,2} & := 2A(M_2(h,\phi),h) - 2A(M_2(H,\Psi),H) \, ,
\\
\label{N33}
\wt{\mathcal{N}}_{3,3} & := i \Lambda \left[ Q_3(\phi,h,\phi)  - Q_3(\Psi,H,\Psi) \right] \, ,
\\
\label{N34}
\wt{\mathcal{N}}_{3,4} & := i \Lambda \left[ B(M_2(h,\phi),\phi) - B(M_2(H,\Psi),\Psi) \right] \, ,
\\
\label{N35}
\wt{\mathcal{N}}_{3,5} & := i \Lambda \left[ B(h, Q_2(\phi,\phi)) - B(H, Q_2(\Psi,\Psi)) \right] \, .
\end{align}

\vskip5pt
\paragraph{{\it Proof of \eqref{lemR1}}}
We start by proving that each term in $\N_4$ is bounded by the right hand side of \eqref{lemR1}.
The bound for $R_1 = {[ G (h) \phi ]}_{\geq 4}(t)$ is an immediate consequence of the $L^1$ estimate \eqref{DN13.1}.
The bound for $\Lambda R_2$ can be obtained from \eqref{expR_2} using Cauchy's inequality and the $L^2$ bounds for $M_2$, $M_3$ and $R_1$
given respectively in \eqref{Al40}, \eqref{Al45} and \eqref{DN11}.
From the definition of $A$ in \eqref{al0}-\eqref{a_10} we see that for any integer $l$
\begin{align}
\label{PkwhatA}
 \Big| \what{P_k A}(F,G) \Big| \lesssim 2^{-l k_+} {\| F \|}_{H^{l+1}} {\| G \|}_{H^{l+1}} \, .
\end{align}
Using the $L^2$ bounds \eqref{Al45} on $M_3$, \eqref{DN11} on $R_1$, and the a priori assumptions, it immediately follows that
\begin{align*}
\Big| \what{P_k A}(M_3(h,h,\phi) + R_1(h,\phi),h) \Big| & \lesssim 2^{-(N_0-10) k_+} {\| M_3 + R_1 \|}_{H^{N_0-9}} {\| h \|}_{H^{N_0-9}}
  \\
& \lesssim \e_1^4 2^{-(N_0-10) k_+} {(1+t)}^{4p_0 -1} \, .
\end{align*}
Similarly, from the definition of $B$ in \eqref{al0}-\eqref{b_10} we have
\begin{align}
\label{PkwhatB}
 \Big| \what{P_k B}(F,G) \Big| \lesssim 2^{-l k_+} {\| F \|}_{H^l} {\| \partial_x G \|}_{H^l} \, .
\end{align}
Using again \eqref{Al45} and \eqref{estRL^2} we get
\begin{align*}
& \Big|  \F \left[ P_k \Lambda B(h,Q_3(\phi,h,\phi) + R_2(h,\phi))+ P_k \Lambda B(M_3(h,h,\phi) + R_1(h,\phi),\phi) \right] \Big|
  \\
& \lesssim 2^{-(N_0-15) k_+} \left[ {\| h \|}_{H^{N_0-10}} {\| \partial_x (Q_3 + R_2) \|}_{H^{N_0-10}} 
  + {\| M_3 + R_1 \|}_{H^{N_0-10}} {\| \partial_x \phi \|}_{H^{N_0-10}}  \right]
\\
& \lesssim \e_1^4 2^{-(N_0-15) k_+} {(1+t)}^{4p_0 -1} \, .
\end{align*}

We now estimate the terms \eqref{N31}-\eqref{N35}.
From \eqref{HPsi0} we see that
\begin{align}
\label{N311}
-\wt{\mathcal{N}}_{3,1} & = M_3(A,h,\phi) + M_3(H,A,\phi) + M_3(H,H,B) \, .
\end{align}
From the definition of $M_3$ in \eqref{p_3} we see that for any integer $0\leq l\leq N_0-10$
\begin{align}
\label{whatP_kM_3}
\begin{split}
 \Big| \what{P_k M_3}(E,F,G) \Big| & \lesssim 2^{-l k_+} 2^k \Big[ {\| E \|}_{W^{N_0/2 - 5,\infty}} {\| F \|}_{H^{l+2}} {\| \partial_x G \|}_{H^{l+2}} 
  \\
& + {\| E \|}_{H^{l+2}} {\| F \|}_{W^{N_0/2 - 5,\infty}} {\| \partial_x G \|}_{H^{l+2}} 
  + {\| E \|}_{H^{l+2}} {\| F \|}_{H^{l+2}} {\| |\partial_x| G \|}_{W^{N_0/2 - 5,\infty}} \Big] \, . 
\end{split}
\end{align}
Applying this together with the $L^\infty$ bounds \eqref{Al1001} on $A$ and $\Lambda B$,
the a priori  bounds \eqref{Al32}, \eqref{vhphiZp} and \eqref{vhphiSobolev}, 
one can obtain the desired bound for each of the three terms in \eqref{N311}.

To estimate \eqref{N32} we write
\begin{align*}
-\frac{1}{2} \wt{\mathcal{N}}_{3,2} & = A(M_2(A,\phi),h) + A(M_2(H,B),h) + A(M_2(H,\Psi),A) \, .
\end{align*}
Notice that for any integer $0\leq l\leq N_0-10$
\begin{align}
\label{PkM_2FG}
{\| P_k M_2(F,G) \|}_{L^2}  \lesssim 2^{-l k_+} 2^k \Big[ {\| F \|}_{H^l} {\| |\partial_x| G \|}_{W^{N_0/2-5,\infty}}
  + {\| F \|}_{W^{N_0/2-5,\infty}} {\| \partial_x G \|}_{H^l} \Big] \, .
\end{align}
Using \eqref{PkwhatA}, \eqref{PkM_2FG}, the estimates for $A$ in \eqref{Al1000} and \eqref{Al1001}, and Proposition \ref{proE2}, 
we get
\begin{align*}
& \Big|  \F [ P_k A (M_2(A,\phi),h) ] \Big| \lesssim 2^{-(N_0-15) k_+} {\| M_2(A,\phi) \|}_{H^{N_0-10}} {\| h \|}_{H^{N_0-10}}
\\
& \lesssim 2^{-(N_0-15) k_+} {(1+t)}^{p_0} \Big[ {\| A \|}_{H^{N_0-8}} {\| |\partial_x| \phi \|}_{W^{N_0/2-5,\infty}} 
  + {\| A \|}_{W^{N_0/2-5,\infty}}  {\| \partial_x \phi \|}_{H^{N_0-8}}   \Big]
\\
& \lesssim \e_1^4 2^{-(N_0-15) k_+} {(1+t)}^{3p_0 -1} \, .
\end{align*}

To bound $\wt{N}_{3,3}$ in \eqref{N33} we first write it as 
\begin{align}
\label{N331}
- \wt{\N}_{3,3} & = \Lambda \Big[ Q_3(B,h,\phi) + Q_3(\Psi,A,\phi) + Q_3(\Psi,H,B)  \Big] \, .
\end{align}
We then notice that for any integer $0\leq l\leq N_0-10$ one has
\begin{align}
\label{whatP_kQ_3}
\begin{split}
\Big| \what{P_k Q_3}(E,F,G) \Big| & \lesssim 2^{-l k_+} 
  \Big[ {\| |\partial_x| E \|}_{W^{N_0/2 - 5,\infty}} {\| F \|}_{H^{l+2}} {\| \partial_x G \|}_{H^{l+3}} 
  \\
& + {\| \partial_x E \|}_{H^l} {\| F \|}_{W^{N_0/2 - 5,\infty}} {\| \partial_x G \|}_{H^{l+2}} \Big] \, . 
\end{split}
\end{align}
One can that then bound each one of the three summands in \eqref{N331} by using the above estimate
together with Proposition \ref{proE2}, \eqref{Al32}, \eqref{Al1000} and \eqref{Al1001}.

\eqref{N34} can be estimated in a similar fashion to what we have done above
by writing out the difference as sums of quartic terms, and using \eqref{PkwhatB}
together with \eqref{PkM_2FG}, \eqref{Al1000}, \eqref{Al1001} and Proposition \ref{proE2}.
The term \eqref{N35} can also be estimated similarly by using in addition 
\begin{align}
\label{PkQ_2FG}
\begin{split}
{\| P_k Q_2(F,G) \|}_{L^2}  & \lesssim 2^{-l k_+} 
  \Big[ {\| \partial_x F \|}_{H^l} ( {\| \partial_x G \|}_{W^{N_0/2-5,\infty}} + {\| |\partial_x| G \|}_{W^{N_0/2-5,\infty}} )
 \\
& + ( {\| \partial_x F \|}_{W^{N_0/2-5,\infty}} + {\| |\partial_x| F \|}_{W^{N_0/2-5,\infty}} ) {\| |\partial_x| G \|}_{H^l} \Big] \, ,
\end{split}
\end{align}
for any $0 \leq l \leq N_0 -10$.

\vskip5pt
\paragraph{{\it Proof of \eqref{lemR2}}}
First observe that from \eqref{estRL^2} we already have the desired bound for $R_1$ and $\Lambda R_2$.
To bound the three remaining contributions from $\wt{N}_4$ in \eqref{N_4app} and the five terms \eqref{N31}-\eqref{N35}
we first observe that for $\G = 1$ or $S$ we have the following $L^2$ estimates:
\begin{align}
\label{SAFG}
\begin{split}
& {\| \G P_k A(P_{k_1}F, P_{k_2}G) \|}_{L^2} 
\\ 
& \lesssim 2^k 2^{-(N_0/2 - 10)k_+} 
  \Big[ \left( {\| \G P_{k_1}F \|}_{H^{N_0/2-10}} + {\| P_{k_1}F \|}_{H^{N_0/2-10}} \right) {\|P_{k_2} G \|}_{W^{N_0/2 -10,\infty}}
\\
& + {\|P_{k_1} F \|}_{W^{N_0/2 -10,\infty}} {\| \G P_{k_2}G \|}_{H^{N_0/2-10}} \Big] \, ,
\end{split}
\\
\label{SBFG}
\begin{split}
& {\| \G P_k B(P_{k_1}F, P_{k_2}G) \|}_{L^2} 
\\
& \lesssim 2^{-(N_0/2 - 10)k_+} 
  \Big[ \left( {\| \G P_{k_1}F \|}_{H^{N_0/2-10}} + {\| P_{k_1}F \|}_{H^{N_0/2-10}} \right) 2^{k_2} {\|P_{k_2} G \|}_{W^{N_0/2 -10,\infty}}
\\
& + {\|P_{k_1} F \|}_{W^{N_0/2 -10,\infty}} 2^{k_2} \left( {\| \G P_{k_2}G \|}_{H^{N_0/2-10}} + {\| P_{k_2}G \|}_{H^{N_0/2-10}} \right) \Big] \, .
\end{split}
\end{align}
We also have the following $L^\infty$ estimates for $M_3$ and $Q_3$:
\begin{align}
\label{M_3EFG}
\begin{split}
& {\| P_k M_3(P_{k_1}E, P_{k_2}F, P_{k_3}G) \|}_{L^\infty}
\\
& \lesssim  2^{-(N_0/2 - 15)k_+} 2^k 2^{k_2} 2^{\max(k_1,k_2,k_3)}
  {\| P_{k_1}E \|}_{W^{N_0/2 -10,\infty}} {\| P_{k_2}F \|}_{W^{N_0/2 -10,\infty}} {\| P_{k_3}G \|}_{W^{N_0/2 -10,\infty}} \, ,
\end{split}
\\
\label{Q_3EFG}
\begin{split}
& {\| P_k Q_3(P_{k_1}E, P_{k_2}F, P_{k_3}G) \|}_{L^\infty} 
  \\
& \lesssim 2^{-(N_0/2 - 15)k_+} 2^{k_1} 2^{k_3} 2^{\max(k_2,k_3)}
  {\| P_{k_1}E \|}_{W^{N_0/2 -10,\infty}} {\| P_{k_2}F \|}_{W^{N_0/2 -10,\infty}} {\| P_{k_3}G \|}_{W^{N_0/2 -10,\infty}} \, . 
\end{split}
\end{align}
From the homogeneity of degree $2$ of $M_2$ and $Q_2$, and of degree $3$ of $M_3$ and $Q_3$,
one can obtain identities similar to \eqref{Al25} for the symbols of these operators, 
and deduce the following analogues of the commutation identities \eqref{Al26}:
\begin{equation}
\label{commapp}
\begin{split}
& S M_2(F,G) = M_2(SF,G) + M_2(F,SG) - 2M_2(F,G) \, ,
\\
& S Q_2(F,G) = Q_2(SF,G) + Q_2(F,SG) - 2Q_2(F,G) \, ,
\\
& S M_3(E,F,G) = M_3(SE,F,G) + M_3(E,SF,G) + M_3(E,F,SG) - 3M_2(E,F,G) \, ,
\\
& S Q_3(E,F,G) = Q_3(SE,F,G) + Q_3(E,SF,G) + Q_3(E,F,SG) - 3Q_2(E,F,G) \, .
\end{split}
\end{equation}
One can then use \eqref{SAFG}-\eqref{Q_3EFG} together with the commutation identities \eqref{Al26} and \eqref{commapp},
the estimates \eqref{Al32}, \eqref{Al40} and \eqref{Al45}, \eqref{estRL^2},
and argumets similar to those used above and in section \ref{secproE3}, in particular in the proof of Lemma \ref{Al60}, to obtain
\begin{align*}
& {\| P_k \mathcal{N}_4 \|}_{L^2} + {\| P_k S \mathcal{N}_4 \|}_{L^2}  
  + \sum_{j=1}^5 {\| P_k \wt{\mathcal{N}}_{3,j} \|}_{L^2} + {\| P_k S \wt{\mathcal{N}}_{3,j} \|}_{L^2}
  \lesssim \e_1^4 {(1+t)}^{20p_0 - 3/2} 2^{-(N_0/2-20)k_+}
\end{align*}
which is the desired conclusion. $\hfill \Box$

\subsubsection{Proof of Lemma \ref{lemrem}}
For $t_1 \leq t_2 \in [2^m-2, 2^{m+1}]$ and $|\xi| \in [2^k, 2^{k+1}]$ let us define
\begin{align*}
F(\xi) = \int_{t_1}^{t_2} e^{i H(\xi,s)} e^{is\Lambda(\xi)} \what{D} (\xi,s) \, ds \, .
\end{align*}
We then have
\begin{align*}
| F(\xi) | \lesssim {\| \F^{-1} F \|}_{L^1(|x| \leq 2^{m/2} )} + {\| \F^{-1} F \|}_{L^1(|x| \geq 2^{m/2} )}
  \lesssim  2^{m/4} {\| F \|}_{L^2} + 2^{-m/4} 2^{-k} {\| \xi \partial_\xi F \|}_{L^2} \, .
\end{align*}
Thus, to obtain \eqref{conclemrem} it suffices to show the following two estimates:
\begin{align}
\label{lemrem10}
& {\| F \|}_{L^2} \lesssim \d 2^{-3m/8}
\\
\label{lemrem20}
& {\| \xi \partial_\xi F \|}_{L^2} \lesssim \d 2^{21 m p_0} \, . 
\end{align}
\eqref{lemrem10} can be easily verified using the $L^2$ bound in \eqref{asslemrem}.
To prove \eqref{lemrem20} we write:
\begin{align}
\begin{split}
\xi \partial_\xi F(\xi) & = F_1(\xi) + F_2(\xi) + F_3(\xi)
\\
F_1(\xi) & = \int_{t_1}^{t_2} e^{i H(\xi,s)} \left( i \xi \partial_\xi H(\xi,s) \right) e^{is\Lambda(\xi)} \what{D}(\xi,s) \, ds \, ,
\\
F_2(\xi) & = \int_{t_1}^{t_2} e^{i H(\xi,s)} S(\xi) e^{is\Lambda(\xi)} \what{D}(\xi,s) \, ds \, ,
\\
F_3(\xi) & = \frac{1}{2} \int_{t_1}^{t_2} e^{i H(\xi,s)} s\partial_s \left( e^{is\Lambda(\xi)} \what{D}(\xi,s) \right) \, ds \, ,
\end{split}
\end{align}
having denoted $S(\xi) := \xi \partial_\xi - \frac{1}{2} s \partial_s$;
notice that $S(\xi) \what{f}(\xi) = - \what{S f}(\xi) -\what{f}(\xi)$, where $S$ is the scaling vector field.
Using the definition of $H$ in \eqref{bn2} and the a priori  assumptions, it is easy to see that
for $s \in [2^m-2, 2^{m+1}]$ one has
\begin{align}
\label{lemrem30}
\begin{split}
& {\| \xi \partial_\xi H(\xi,s) \|}_{L^2} \lesssim 2^{m p_0} \, ,
\\
& {\| \partial_s H(\xi,s) \|}_{L^\infty} \lesssim 2^{-m} \, . 
\end{split}
\end{align}
Using the first bound above and the $L^\infty$ bound in \eqref{asslemrem} we see that
\begin{align*}
& {\| F_1 \|}_{L^2} \lesssim  \int_{t_1}^{t_2} {\| \xi \partial_\xi H(s) \|}_{L^2} {\| \what{D}(s) \|}_{L^\infty} \, ds
  \lesssim \d 2^m 2^{m p_0} 2^{m(20 p_0 - 1)} \lesssim \d 2^{21 m p_0} \, ,
\end{align*}
as desired.
Since $[\left( \xi \partial_\xi - \frac{1}{2} s \partial_s \right), e^{is\Lambda(\xi)}] = 0$, we can use 
the $L^2$ bounds in \eqref{asslemrem} to deduce
\begin{align*}
& {\| F_2 \|}_{L^2} \lesssim  \int_{t_1}^{t_2} {\| D(s) \|}_{L^2} + {\| S D(s) \|}_{L^2} \, ds \lesssim \d \, ,
\end{align*}
which is more than sufficient. To estimate $F_3$ we integrate by parts in $s$, use the second bound in \eqref{lemrem30}
and \eqref{asslemrem} to obtain:
\begin{align*}
& {\| F_3 \|}_{L^2} \lesssim  2^m \sup_{s\in [2^m-2, 2^{m+1}] }  {\| D(s) \|}_{L^2}  + 
  \int_{t_1}^{t_2} s {\| \partial_s H(\xi,s) \|}_{L^\infty} {\| D(s) \|}_{L^2} \, ds \lesssim \d \, .
\end{align*}
This proves \eqref{lemrem20} and concludes the proof of the Lemma. $\hfill \Box$

\vskip10pt
\bibliographystyle{plain}

\end{document}